%% file: schubert_bialgebra_separate_paper.tex
\documentclass{amsart}
\pdfoutput=1
\usepackage{algorithm}
\usepackage{algpseudocode}
\usepackage{hyperref}
\usepackage{fancyvrb}
\usepackage{amsmath}
\usepackage{amsfonts}
\usepackage{amsthm}
\usepackage{amssymb}
\usepackage{array}
\usepackage{cases}
\usepackage{caption}
\usepackage{xcolor}
\usepackage{tikz}
\usepackage{tikz-cd}
\usepackage{mathtools}
\usetikzlibrary{calc}
\usepackage{centernot}
\usepackage{mathtools}
\usepackage{graphicx}
\usepackage[margin=1in]{geometry}
\usepackage[makeroom]{cancel}
\usepackage{stackengine}
\usepackage{mathrsfs}
\usepackage{enumitem}

\makeatletter
\newcommand{\sqcupplus@inner}[2]{%
  \vcenter{\hbox{\ooalign{%
    \hfil$\m@th#1\sqcup$\hfil\cr
    \hfil\raisebox{0.15ex}{$\m@th#1\scriptscriptstyle+$}\hfil\cr
  }}}%
}
\newcommand{\sqcupplus}{\mathbin{\mathpalette\sqcupplus@inner\relax}}
\makeatother

\newtheorem{lemma}{Lemma}[subsection]
\newtheorem{corollary}[lemma]{Corollary}

\newtheorem{proposition}[lemma]{Proposition} 
\newtheorem{proposition2}{Proposition}[section]
\newtheorem{theorem}{Theorem}[section]
\newtheorem{conjecture}[theorem]{Conjecture}

\theoremstyle{definition}
\newtheorem{definition}[lemma]{Definition}
\newtheorem{example}[lemma]{Example}

\theoremstyle{remark}

\newtheorem*{remark}{Remark}

\numberwithin{equation}{subsection}

\DeclareMathOperator{\sch}{\mathfrak{S}}
\newcommand{\shup}{{\uparrow}}
\newcommand{\downvar}[1]{\stackrel{#1}{\searrow}}
\newcommand{\wof}[1]{w_{#1}}

\newcommand{\wtt}{\mathtt{wt}}
\newcommand{\QY}{\ensuremath{\mathbb{QY}}}
\newcommand{\tom}[1]{\xrightarrow{#1}}
\newcommand{\Tom}[1]{\xRightarrow{#1}}

\newcommand{\rtt}{\mathtt{rt}}
\newcommand{\trm}{\mathtt{trim}}
\newcommand{\zeromap}{\mathcal{Z}}
\newcommand{\hr}{\mathtt{ht}}
\newcommand{\clp}{\mathtt{clip}}
\newcommand{\BRC}{\mathcal{BRC}}
\newcommand{\RC}{\mathcal{RC}}
\newcommand{\ltb}{\mathrm{little}}
\newcommand{\slide}{\mathfrak{F}}

\newcommand{\dsch}{\ensuremath{\Xi}}
\newcommand{\coma}{\mathcal{A}}
\newcommand{\dcoma}{\mathcal{D}}

\title{The Schubert bialgebra}
\author{Matthew J. Samuel}

\newcommand{\dom}{\mathfrak{d}}

\newcommand{\code}{\mathfrak{c}}

\newcommand{\forest}{\mathfrak{P}}
\newcommand{\fcode}{\mathfrak{c}_\forest}
\newcommand{\finv}{P_{\forest}}

\newcommand{\maxd}{\mathfrak{m}}
\newcommand{\ddv}[1]{\delta}

\begin{document}

\begin{abstract}
We introduce a \emph{Schubert bialgebra} $\coma=\bigoplus_{n\geq 0}\mathbb{Z}[x_1,\ldots,x_n]$, in which polynomials in different numbers of variables multiply to zero and the coproduct splits the variables into two consecutive blocks, and study its graded dual $\dcoma$, whose product is concatenation of weak compositions. The Schubert polynomials $\sch_u^n$ form a basis of $\coma$; we give an explicit formula for the dual Schubert basis $\dsch_u^n$ of $\dcoma$ via a transition formula to an elementary symmetric basis, and show that the Schubert structure constants $c_{u,v}^w$ are the coefficients of the coproduct of $\dcoma$ in the dual Schubert basis. The product of $\dcoma$ is then lifted to an associative product $\sqcupplus$ on the free abelian group $\BRC$ of bounded RC graphs, built from a transition map $\zeromap$ on RC graphs that realizes the Lascoux--Sch\"utzenberger transition formula and generalizes Little's bumping algorithm. Quotienting $\BRC$ by ideals coming from Demazure crystal equivalence, Nadeau--Tewari forest equivalence, and quasi-Yamanouchi equivalence yields manifestly positive Littlewood--Richardson rules for the dual Schubert, dual key, dual forest, and dual slide bases of $\dcoma$. Read on $\coma$, these are branching rules: enumerative formulas for splitting a Schubert, key, forest, or slide polynomial across two blocks of variables, and in the Schubert case they compute the Levi branching coefficients of Ressayre--Richmond for $\mathrm{GL}_p\times\mathrm{GL}_q\hookrightarrow\mathrm{GL}_{p+q}$.
\end{abstract}

	\maketitle

\section{Introduction}

Schubert polynomials $\sch_u$, $u\in S_\infty$, form a $\mathbb{Z}$-basis of $\mathbb{Z}[x_1,x_2,\ldots]$ whose structure constants $c_{u,v}^w$ are the intersection numbers of Schubert varieties in the flag variety. A positive combinatorial formula for the $c_{u,v}^w$ in general remains open. A related, and more tractable, family of coefficients arises when one asks not how Schubert polynomials multiply but how they \emph{split}: expanding $\sch_w(x_1,\ldots,x_{p+q})$ in the basis $\{\sch_u(x_1,\ldots,x_p)\sch_v(x_{p+1},\ldots,x_{p+q})\}$ produces the branching coefficients for the Levi subgroup $\mathrm{GL}_p\times\mathrm{GL}_q\hookrightarrow\mathrm{GL}_{p+q}$ studied in the branching Schubert calculus of Ressayre--Richmond \cite{ressayre2009branching}. This paper introduces an algebraic framework in which these splitting problems become multiplication problems in the dual of a bialgebra, and a combinatorial framework, a ring of RC graphs, in which they acquire manifestly positive enumerative solutions for the Schubert basis and for the key, forest, and slide bases as well.

\subsection*{The Schubert bialgebra and its dual} In initial sections, we introduce a bialgebra $\coma$ that we call the Schubert bialgebra. This bialgebra has a basis of Schubert polynomials $\sch_u^n$ for $u\in S_\infty$, where the last descent of $u$ is not past index $n$. The structure constants for multiplication in this basis are the same as the structure constants for multiplication of Schubert polynomials. The difference from the usual realm of study of Schubert polynomials, which is the polynomial ring in infinitely many variables, is that this algebra is graded into length components that interact trivially. That is,
$$\coma = \bigoplus_{n=0}^\infty \mathbb{Z}[x_1,\ldots,x_n]\text{.}$$
This allows us to consistently define a ``splitting'' coproduct that is compatible with the usual product of polynomials but actively requires the triviality of products between length components to achieve a bialgebra structure.

We focus primarily on the dual $\dcoma$ of this algebra, where the product is derived from this splitting (deconcatenation) coproduct on $\coma$. Specifically, let $\mathrm{Comp}_n$ be the set of weak compositions of length $n$. Then, additively,
$$\dcoma = \bigoplus_{n=0}^\infty \mathbb{Z}^{\oplus \mathrm{Comp}_n}$$
and multiplicatively the product of two basis elements indexed by $a\in \mathrm{Comp}_p$ and $b\in \mathrm{Comp}_q$ is the basis element indexed by the concatenation of $a$ and $b$, which is an element of $\mathrm{Comp}_{p+q}$. We will denote this product by $\sqcupplus$, writing $a\sqcupplus b$ for the concatenation. The coproduct on $\dcoma$ is given by splitting sums of exponents, that is to say
$$\Delta(c) = \sum_{a+b=c} a\otimes b$$
where $a$ and $b$ are weak compositions of the same length as $c$ and $a+b$ denotes the pointwise sum. The counit is given by $\varepsilon(c)=0$ unless $c=[000\cdots 0]$ for some number of $0$'s, in which case $\varepsilon(c)=1$. With these definitions, $\coma$ and $\dcoma$ are dual bialgebras over $\mathbb{Z}$ (with $\coma$ non-unital: the constant $1\in\coma_n$ is a unit for $\coma_n$ only), a fact that is straightforward to verify. Neither $\coma$ nor $\dcoma$ is connected, and neither is a Hopf algebra. A ``Littlewood--Richardson rule'' for a basis of the dual algebra is equivalent to a branching/splitting formula for the corresponding polynomial basis under the partition of variables into two blocks $\{x_1,\ldots,x_p\}$ and $\{x_{p+1},\ldots,x_{p+q}\}$.

The algebraic structure of $\dcoma$ is not itself new; what is new is the product on the dual Schubert basis $\dsch_u^n$ of $\dcoma$ indexed by $u\in S_\infty$ and $n$ at least as large as the last descent of $u$, defined by
$$\langle \dsch_u^m, \sch_v^n\rangle = \delta_{u,v}\delta_{m,n}$$
where we take $\langle -,-\rangle$ to be the pairing where for each composition $c\in\mathrm{Comp}_n$,
$$\langle c, x^a\rangle = \delta_{c,a}$$
where $a$ is any composition. (It is worth pointing out that $c$ and $c$ padded with an additional $0$ are considered different here.) We present an explicit formula for the dual Schubert basis in $\dcoma$ with integer coefficients.

A dual basis to the Schubert polynomials is not new either: Postnikov and Stanley \cite{postnikov2009chains} introduced a family of dual Schubert polynomials $\mathfrak{D}_w$ in the polynomial ring with respect to a particular integral pairing, characterized by a formula that mimics the action of divided difference operators on Schubert polynomials. The relationship between the two bases is that the coefficients of $\mathfrak{D}_u$ in the divided power algebra are, up to a scalar factor, the same as the coefficients of $\dsch_u^n$ in the weak composition basis. 

In addition, the dual forest basis $\{\forest_a^*\}$ of $\dcoma$ is defined by $\langle\forest_a^*,\forest_b\rangle=\delta_{a,b}$. As we record in Proposition~\ref{proposition:forestdualcoefficients}, this dual basis coincides, up to a divided-power normalization, with the volume polynomials of Nadeau--Spink--Tewari \cite{nadeau2024quasisymmetric} and, geometrically, with the Kronecker dual of the homology basis $\{[X(F)]\}\subset H_\bullet(\mathrm{QFl}_n)$ of Bergeron--Gagnon--Nadeau--Spink--Tewari \cite{bergeron2025qsymflag}. We also consider the dual bases of key polynomials and slide polynomials.

\subsection*{Littlewood--Richardson rules for dual bases}

Each of these dual bases in fact has nonnegative structure constants with $\sqcupplus$, for which we also give rules in terms of RC graphs. For an RC graph $R$ with $n$ rows and an integer $0\leq p\leq n$, we associate to $R$ two RC graphs $\clp^p(R)\in\RC_p$ and $\trm^p(R)\in\RC_{n-p}$, which the reader should picture as a decomposition of $R$ across the horizontal cut between rows $p$ and $p+1$. The map $\trm^p$ is simple: it discards the top $p$ rows and keeps the remaining $n-p$ rows (suitably reindexed). The map $\clp^p$ is subtler. It is built by iterating an auxiliary map $\zeromap:\RC_n^0\to\RC_{n-1}$ (Definition \ref{definition:zeromap}), where $\RC_n^0$ denotes the set of RC graphs whose last row is empty. The map $\zeromap$ realizes, on the level of RC graphs, the Lascoux--Schützenberger transition formula for Schubert polynomials (see \cite{notes}), and ends up being a generalization of the main result of \cite{little2003combinatorial}, though this was not intentional. The whole construction preserves a wealth of additional structure (crystal operators, forest invariants) that we exploit later. Indeed, the product of $\dcoma$ is lifted to a product $\sqcupplus$ on the free abelian group $\BRC$ of bounded RC graphs using $\clp^p$ and $\trm^p$, for which $\dcoma$ is a quotient (\S\ref{section:brc}). The nonnegativity of the structure constants of all of these bases of $\dcoma$ is a consequence of this quotient structure.

For the dual Schubert basis $\{\dsch_u^n\}$ for permutations $u$ with last descent at most $n$, we have the following result.

\begin{theorem}[LR rule for dual Schubert polynomials]\label{theorem:LR}
Suppose $u,v\in S_\infty$ and $p,q\geq 0$ are integers such that the last descent of $u$ is at most $p$ and the last descent of $v$ is at most $q$. Define integers $d_{u,v}^w(p, q)$ by the equation
$$\dsch_u^p \sqcupplus \dsch_v^q = \sum_{w} d_{u,v}^w(p, q)\,\dsch_w^{p+q}.$$
Then for any fixed choices of RC graphs $U\in \RC_p(u)$, $V\in\RC_q(v)$, and a permutation $w\in S_\infty$, we have that $d_{u, v}^w(p, q)$ is the number of RC graphs $R\in \RC_{p+q}(w)$ such that $\clp^p(R)=U$ and $\trm^p(R)=V$, and is in particular nonnegative.
\end{theorem}
Cast as a splitting rule for $\sch_w$, Theorem~\ref{theorem:LR} fits the framework of \emph{branching Schubert calculus} of Ressayre--Richmond \cite{ressayre2009branching} for the Levi inclusion $\mathrm{GL}_p\times \mathrm{GL}_q\hookrightarrow \mathrm{GL}_{p+q}$ (with $P=B$): the Ressayre--Richmond branching coefficients $d_w^{\tilde w}$ in this case are exactly our $d_{u,v}^w(p,q)$, and Theorem~\ref{theorem:LR} supplies for them an explicit positive cardinality formula. The associativity of $\sqcupplus$ encodes the compatibility of these splittings across all $p,q\geq 0$ simultaneously, packaging the entire family of Levi comorphisms into the single algebra $(\dcoma,\sqcupplus)$.

The dual forest basis has nonnegative structure constants for the same structural reason. The coefficient $f_{a,b}^c$ in $\forest_a^*\sqcupplus\forest_b^*=\sum_c f_{a,b}^c\,\forest_c^*$ counts, within the forest equivalence class of composition $c$, the RC graphs whose row-splitting lands in forest-classes of compositions $a$ and $b$; the count is well-defined on equivalence classes and is nonnegative (Theorem~\ref{theorem:LRforest}). Read on the quasisymmetric flag variety $\mathrm{QFl}_n$ of Bergeron--Gagnon--Nadeau--Spink--Tewari \cite{bergeron2025qsymflag}, whose cohomology has the forest polynomials as a basis, this is a branching rule for the variable-splitting comorphism $H^\bullet(\mathrm{QFl}_{p+q})\to H^\bullet(\mathrm{QFl}_p)\otimes H^\bullet(\mathrm{QFl}_q)$ (\S\ref{subsection:qflbranch}).

The key polynomials $\kappa_a$ (Demazure characters in type $A$) form a $\mathbb{Z}$-basis of $\mathbb{Z}[x_1,x_2,\ldots]$ indexed by weak compositions; we defer definitions to \S\ref{section:crystals}. The dual key basis $\{\kappa_a^*\}$ of $\dcoma$ is defined by $\langle\kappa_a^*,\kappa_b\rangle=\delta_{a,b}$; its structure as a bialgebra basis appears to be new. The dual key rule (Theorem~\ref{theorem:LRkey}) has the same shape as the dual Schubert and dual forest rules but replaces forest equivalence classes with Demazure crystal components: $k_{a,b}^c$ counts, within the crystal component of highest weight $c$, those RC graphs whose row-splitting produces highest-weight pieces of compositions $a$ and $b$. We later found that nonnegativity of the $k_{a,b}^c$ also follows from Mathieu's existence theorem for excellent filtrations on tensor products of Demazure modules \cite{mathieu1990filtrations}; an explicit positive combinatorial rule had not previously appeared.

In a similar vein, the dual slide basis $\{\slide_a^*\}$ of $\dcoma$ is defined by $\langle\slide_a^*,\slide_b\rangle=\delta_{a,b}$, where $\slide_b$ is the slide polynomial of Assaf--Searles \cite{assaf2017slide}. We give a Littlewood--Richardson rule for this basis as well.

These formulas bear a striking similarity to each other, which is no coincidence: they are all consequences of the same underlying machinery, and in particular of the same maps $\clp^p$ and $\trm^p$ on RC graphs. A general ``dual LR rule factory'' is given in Theorem \ref{theorem:general_lr_template}, which gives a recipe for producing such rules whenever one has a family of polynomials with a positive combinatorial formula in terms of RC graphs that is compatible with the maps $\clp^p$ and $\trm^p$. The dual Schubert, dual key, dual forest, and dual slide bases all fit into this framework, and thus yield the four rules above as special cases.

\subsection*{Relation to the forest Littlewood--Richardson rule}

The first version of this work, arXiv:2606.23876v1, also contained a Littlewood--Richardson rule for the \emph{product} of forest polynomials themselves, that is, for the structure constants $\beta_{a,b}^c$ in $\forest_a\forest_b=\sum_c\beta_{a,b}^c\forest_c$. That result has been split off into the companion paper \cite{samuelforest}, which will appear as a revised version of that arXiv submission containing only the forest rule. The companion paper builds on the row-decomposition maps $\clp^p$, $\trm^p$, and $\zeromap$ and on the crystal and forest-equivalence properties of $\zeromap$ established here; it restates the definitions and results it needs from the present paper with precise references, and the present paper does not depend on anything proved there.

\subsection*{Further directions}

Beyond this, there is potentially much to explore. For example, the coproduct of $\dcoma$ has as its structure constants the structure constants of Schubert polynomials, which have evaded expression in terms of positive formulas for decades, and at the time of this writing no positive, combinatorial formula is known for the structure constants of Schubert polynomials that applies in all cases. At the very least, this coproduct is a new potential method of attack on this and related problems, since, to our knowledge, it was previously unknown. An expression for $\dsch_w^n$ in the composition basis leads to a simple, purely mechanical method for computing all $c_{u,v}^w$ (generalized Littlewood-Richardson coefficients for Schubert polynomials) by leveraging the extremely simple form the coproduct takes on the weak composition basis, as well as a neat positive formula for the expansion of a weak composition $[a]$ in the Schubert basis.

For the reader familiar with bumpless pipe dreams \cite{lam2021back}, our transition map for RC graphs is realized trivially for bumpless pipe dreams. Specifically, while (reduced) bumpless pipe dreams are typically defined in such a way that they cannot deviate from a square shape, there is indeed a consistent way to remove rows from a bumpless pipe dream unambiguously. This was observed by Yu in \cite[Lemma~3.18]{Yu2025EmbeddingBP} in the definition of a $k$-row BPD. Weigandt in \cite{Weigandt2021} defines a transition formula $\mathbf{t}:\mathrm{BPD}_n\to\mathrm{BPD}_{n-1}$, and while the image is realized as square, it agrees with Yu's notion of an $n-1$-row BPD. Under the Gao--Huang bijection \cite{GaoHuang2023} between RC graphs and bumpless pipe dreams, $\zeromap$ is the deletion of the last row (Proposition~\ref{proposition:zerotrunc}); this is how the properties of $\zeromap$ and the associativity of our product are proved in Section~\ref{section:bpd}.

This simplicity does not come for free, however; the $\trm^p$ operation on BPDs is no longer trivial.

We conjecture a far stronger result that would result in a generalization of the Schubert bialgebra to $K$-theory. Let $\mathrm{MBPD}(n)$ denote the set of $n$-row marked bumpless pipe dreams and let $\mathrm{Pipe}(n)$ denote the set of $n$-row not-necessarily-reduced pipe dreams. Huang, Shimozono, and Yu in \cite{HuangShimozonoYu2024} present a bijection $\Phi:\mathrm{MBPD}(n)\to\mathrm{Pipe}(n)$.

For $P\in\mathrm{Pipe}(n)$ and $p\leq n$, let 
$$\clp^p(P) = \Phi(\mathbf{t}^{n-p}(\Phi^{-1}(P))).$$
Define a function $\sqcupplus:\mathrm{Pipe}(m)\times\mathrm{Pipe}(n)\to\mathbb{Z}^{\mathrm{Pipe}(m+n)}$ by
$$P\sqcupplus Q = \sum_{R\in \mathrm{Pipe}(m+n), \clp^m(R) = P, \trm^m(R) = Q} R.$$
Then
\begin{conjecture}
Define
$$\mathrm{BPipe} = \bigoplus_{n=0}^\infty \mathbb{Z}^{\mathrm{Pipe}(n)}.$$
Then we have the following:
\begin{itemize}
\item $\sqcupplus:\mathrm{Pipe}(n)\times\mathrm{Pipe}(m)\to\mathbb{Z}^{\mathrm{Pipe}(n+m)}$ extends to an associative product on $\mathrm{BPipe}$ for which it is a graded ring.
\item The map $\mathbf{f}:\mathrm{BPipe}\to\BRC$ such that
$$\mathbf{f}(P)=\begin{cases}
  P & \text{if } P\text{ is reduced} \\
  0 & \text{otherwise}
\end{cases}$$
induces a surjective homomorphism of rings 
$$\mathbf{f}:\left(\mathrm{BPipe}, \sqcupplus\right) \to \left(\BRC, \sqcupplus\right).$$

\item Let $G_w^n$ for $w\in S_\infty$ be the basis of $\dcoma$ dual to the basis of Grothendieck polynomials (with the $n$ representing the max number of variables). Then the map $\mathrm{BPipe}\to \dcoma$ defined on the $n$-row pipe dream $P\in\mathrm{Pipe}(w,n)$ for $w\in S_\infty$ by
$$P\mapsto G_w^n$$
extends to a surjective homomorphism of rings.
\end{itemize}
\end{conjecture}

In effect this would generalize Little's bumping algorithm to the $0$-Hecke words of permutations that come from pipe dreams representing the monomials of Grothendieck polynomials, yielding word-level combinatorial transition algorithm.
\subsection*{Outline of the paper}

Sections \ref{section:bialgebra} and \ref{section:dualalgebra} establish the algebraic foundations: \S\ref{section:bialgebra} (\nameref{section:bialgebra}) constructs $\coma$, identifies its Schubert basis $\sch_u^n$, and exhibits the elementary basis $\mathcal{E}_n$ of $\coma_n$; \S\ref{section:dualalgebra} (\nameref{section:dualalgebra}) builds the graded dual $\dcoma$, identifies the dual Schubert basis $\dsch_u^n$, and records the coproduct.

Section \ref{section:brc} (\nameref{section:brc}) defines the module $\BRC$ of bounded RC graphs, the row-decomposition maps $\zeromap$, $\clp^p$, $\trm^p$, and the associative product $\sqcupplus$ on $\BRC$; this culminates in Theorem \ref{theorem:LR}, the Littlewood--Richardson rule for the dual Schubert basis, and in Theorem \ref{theorem:general_lr_template}, a general Littlewood--Richardson principle for admissible equivalence relations on $\BRC$ (Definition~\ref{definition:admissible}) from which the rules of the later sections are derived.

Sections \ref{section:crystals}--\ref{section:slide} produce three further LR rules by quotienting $\BRC$ by successively coarser equivalence relations. Section \ref{section:crystals} (\nameref{section:crystals}) imports the Demazure crystal structure on RC graphs from \cite{assaf2018demazure} and identifies an ideal $\Theta\subset\BRC$ whose quotient yields Theorem \ref{theorem:LRkey}, the LR rule for the dual key basis. Section \ref{section:forest} (\nameref{section:forest}) does the analogous construction for the forest invariant of \cite{nadeau2024forest}, producing an ideal $\Gamma\subset\BRC$ whose quotient yields Theorem \ref{theorem:LRforest}, the LR rule for the dual forest basis, and interprets it as a branching rule on the quasisymmetric flag variety. Section \ref{section:slide} (\nameref{section:slide}) introduces the quasi-Yamanouchi map $\QY$ and the corresponding ideal $\Sigma\subset\BRC$, yielding Theorem \ref{theorem:lrfslide}, a positive product rule for the dual slide basis.

The appendix is devoted to the proof of Proposition \ref{proposition:pullindex}, which would otherwise be distracting.

\clearpage

\section{The Schubert Bialgebra and its dual}\label{section:bialgebra}
\subsection{Some preliminaries}

\begin{definition}[Two Pieri relations $\tom{k}$ and $\Tom{k}$]\label{definition:pieri}
The relations $\tom{k}$ and $\Tom{k}$ are defined on $S_\infty$ as follows, with different notation, in \cite{sottile}. For $v,w\in S_\infty$, we have $v\tom{k} w$ if there is a sequence of transpositions $t_{a_1b_1},\ldots,t_{a_mb_m}$ such that the following hold:
\begin{enumerate}
\item  $a_i \leq k < b_i$ for all $i$.
\item $\ell(ut_{a_1b_1}\cdots t_{a_i b_i}) = \ell(u)+i$ for all $i$.
\item $w = ut_{a_1b_1}\cdots t_{a_m b_m}$.
\item \label{item:thintom} The integers $a_1,\ldots,a_m$ are pairwise distinct.
\end{enumerate}
For $\Tom{k}$, only (\ref{item:thintom}) changes:
\begin{itemize}
  \item[(\ref*{item:thintom}')] \label{item:varthintom} The integers $b_1,\ldots,b_m$ are pairwise distinct.
\end{itemize}

Notice that necessarily $m\leq k$ for $\tom{k}$, but there is no upper bound on $m$ for $\Tom{k}$. This reflects the fact that the following theorem holds:
\begin{theorem}[{\cite[Theorem~1]{sottile}}]
  Let $p\geq 0$, $k>0$ be integers and let $u$ be a permutation.
  \begin{itemize}
    
    \item[\textrm{I}. ] We have the formula
    $$\sch_u(x)h_{p,k}(x) = \sum_{\substack{u\Tom{k} w \\ \ell(u,w) = p}} \sch_w(x)$$
    where $h_{p,k}(x)$ is the complete homogeneous symmetric polynomial of degree $p$ in the first $k$ variables.
    \item[\textrm{II}. ] We have the formula
    $$\sch_u(x)e_{p,k}(x) = \sum_{\substack{u\tom{k} w \\ \ell(u,w) = p}} \sch_w(x)$$
    where $e_{p,k}(x)$ is the elementary symmetric polynomial of degree $p$ in the first $k$ variables.
  \end{itemize}
\end{theorem}

\end{definition}

\begin{definition}[Divided differences and skew divided differences]\label{definition:divdiff}
For $i\geq 1$ let $s_i$ act on polynomials in $x$ by exchanging $x_i$ and $x_{i+1}$, and define the divided difference operator
$$\partial^i = \frac{1}{x_i-x_{i+1}}(1-s_i)\text{.}$$
The $\partial^i$ satisfy the nil-Coxeter relations, so for $w\in S_\infty$ with reduced word $w=s_{i_1}\cdots s_{i_k}$ the operator $\partial^w=\partial^{i_1}\cdots\partial^{i_k}$ is well defined. We write $\delta^i$ and $\delta^w$ for the same operators acting on the $y$ variables of a double Schubert polynomial, so $\delta^i=\frac{1}{y_i-y_{i+1}}(1-s_i)$ with $s_i$ now exchanging $y_i$ and $y_{i+1}$. On double Schubert polynomials,
$$\partial^i\sch_w(x;y)=\begin{cases}\sch_{ws_i}(x;y)&\text{if }\ell(ws_i)<\ell(w),\\ 0&\text{otherwise,}\end{cases}\qquad\qquad \delta^i\sch_w(x;y)=\begin{cases}-\sch_{s_iw}(x;y)&\text{if }\ell(s_iw)<\ell(w),\\ 0&\text{otherwise,}\end{cases}$$
so that $\delta^v\sch_w(x;y)=(-1)^{\ell(v)}\sch_{vw}(x;y)$ whenever $\ell(vw)=\ell(w)-\ell(v)$. Since each $\delta^i$ changes sign under the substitution $y\mapsto -y$, the signs cancel on $\sch_w(x;-y)$:
$$\delta^v\bigl[\sch_w(x;-y)\bigr]=\sch_{vw}(x;-y)\qquad\text{whenever }\ell(vw)=\ell(w)-\ell(v)\text{.}$$

The \emph{skew divided difference operators} $\delta_u^w$ for $u\leq w$ in Bruhat order were defined originally in \cite{notes}; we use the conventions of \cite{samuelleibniz}, in which they are characterized by the Leibniz formula
$$\delta^w(pq) = \sum_{u\leq w} \delta^u(p)\,\delta_u^w(q)$$
for arbitrary polynomials $p$ and $q$, together with its iterated form $\delta_u^w(pq)=\sum_{u\leq v\leq w}\delta_u^v(p)\,\delta_v^w(q)$. The operators $\partial_u^w$ on the $x$ variables are defined identically.
\end{definition}

\begin{definition}[Lehmer code, dual code, dominant permutations]\label{definition:codestuff}
For a permutation $w\in S_\infty$, the \emph{(right) Lehmer code} of $w$ is the sequence $\code(w)=(c_1,c_2,\ldots)$ defined by
$$c_i \;=\; \#\{\,j>i : w(j)<w(i)\,\}.$$
This is a weak composition with $\sum_i c_i=\ell(w)$, and the map $w\mapsto\code(w)$ is a bijection from $S_\infty$ onto the set of weak compositions of finite support. We will sometimes write $\code_i(w):=c_i$. The \emph{dual code} of $w$ is
$$\code^*(w) \;:=\; \code(w^{-1}).$$

A permutation $\mu\in S_\infty$ is called \emph{dominant} if $\code(\mu)$ is a partition (a weakly decreasing sequence of nonnegative integers, padded by $0$s). For a partition $\lambda$, we write $\dom_\lambda$ for the dominant permutation with
$$\code^*(\dom_\lambda) \;=\; \lambda.$$
\end{definition}

\begin{definition}[Pulling out variables from double Schubert polynomials, the relation $\downvar{i}$] \label{definition:polysequence}
		For a permutation $v'$ such that $v'\in S_{n}$, define
		$$\varphi_{i,n}(v')(j)=\begin{cases}
			v'(j)&\mbox{ if }j<i\\
			n+1&\mbox{ if }j=i\\
			v'(j-1)&\mbox{ if }i<j\leq n+1\\
			j&\mbox{ if }j>n+1
		\end{cases}$$
		Now fix $v\in S_n$ and $i\geq 1$ an integer, we define a relation $\downvar{i}$ by declaring that $v\downvar{i} v'$ if 
		$$v\Tom{i}\varphi_{i,n}(v')$$
		Equivalently, $v\downvar{i} v'$ if whenever $v\in S_n$ and $n$ is minimal, we have that $v$ satisfies the relation $\Tom{i}$ with respect to the permutation obtained from $v'$ by inserting $n+1$ at position $i$. We note that this concept was introduced by Bergeron and Sottile in \cite{bsskew}.
		
		If $v\downvar{i} v'$, we define a set of integers $Q_i(v',v)$ by
		$$Q_i(v',v)=\{v(j)\mid j>i\mbox{ and }v'(j-1)=v(j)\}$$
		Given the fact that any element of $S_\infty$ fixes all but finitely many positive integers, it follows that $Q_i(v',v)$ is a finite set.
		
		The permutations $v'$ such that $v\downvar{i} v'$ arise when pulling a variable out of a Schubert polynomial and expressing the coefficients of the powers of this variable as Schubert polynomials in the remaining variables, as is done in \cite{coprod}, from which this definition essentially comes.

  \end{definition}
  
\begin{definition}[Upward shifting of permutations, maximum right descent]
		For a permutation $v\in S_\infty$, we define $\shup{v}$ to be the permutation defined by
		$$\shup{v}(i) = \begin{cases}
			v(i-1)+1&\mbox{ if }i>1\\
			1&\mbox{ if }i=1
		\end{cases}$$
		Recursively, we define $\shup^k(v)$ to be $\shup(\shup^{k-1}(v))$. In the literature, this is sometimes written as
    $$\shup^kv = 1^k\times v$$

		For a permutation $w$, we define $\maxd(w)$ to be the maximum right descent of $w$. That is
		$$\maxd(w) = \max\{i\mid w(i)>w(i+1)\}$$

	\end{definition}

The next proposition specializes, in the case of ordinary Schubert polynomial $\sch_v(x)$, to a formula that isolates a chosen index $i$: one can express $\sch_v(x)$ as a sum of terms of the form $x_i^p\sch_{v'}(x^{(i)})$, thereby effectively extracting the variable $x_i$ and leaving Schubert polynomials in the remaining variables \cite[Theorem~5.1]{bsskew}. Proposition \ref{proposition:pullindex} is the double Schubert polynomial version of this, which, as far as we know, is new.

\begin{proposition} \label{proposition:pullindex}
Let $v\in S_\infty$ and let $i>0$ be an integer. Then we have
$$\sch_v(x;y)=\sum_{v\downvar{i} v'}\sch_{v'}(x^{(i)};y)\prod_{b\in Q_i(v',v)}(x_i - y_b)$$
\end{proposition}
See the Appendix for a proof.

\subsection{Definition}

We define a commutative algebra $\coma$ over the integers as follows. For each $n\geq 0$, define $\coma_n$ to be the polynomial ring over $\mathbb{Z}$ in the variables $x_1,\ldots,x_n$ (so $\coma_0=\mathbb{Z}$). Then define
$$\coma =\bigoplus_{n=0}^\infty \coma_n$$
The multiplication within each $\coma_n$ is the usual multiplication of polynomials, and if $a\in \coma_m$ and $b\in \coma_n$ with $m\neq n$, then
$$ab = 0\text{.}$$
That is, $\coma$ is the direct sum of the rings $\coma_n$, including $\coma_0$; it is a commutative ring without unit, and the constant polynomial $1\in\coma_n$ is a unit for the summand $\coma_n$ only. The $\mathbb{Z}$-module structure is the evident one.

\begin{remark}[Fat identities]\label{remark:fatidentity}
The constant polynomial $1\in\coma_n$ acts as a unit on $\coma_n$ itself but is \emph{not} a unit of $\coma$ as a whole: by the rule above it annihilates every $\coma_m$ with $m\neq n$. We will refer to these length-$n$ units informally as ``fat identities''; $\coma$ itself has no unit. This is the conceptual hinge that makes $\coma$ a bialgebra rather than a Hopf algebra. The triviality of cross-length products is what allows the deconcatenation coproduct $\Delta$ defined below to be a homomorphism of rings (if $1\in\coma_0$ acted as a global unit, $\Delta$ would fail to be multiplicative already on $x_1\cdot x_1x_2=0$); conversely, it forces $\coma$ to fail to be connected, so $\coma$ admits no antipode.
\end{remark}

Each $\coma_n$ has a basis consisting of elements $x_a$, where $a$ is a weak composition of length $n$, and the notation indicates that
$$x_a = x_1^{a_1}\cdots x_n^{a_n}$$
The direct sum therefore has a basis that can canonically be identified with the union of these.

We define a coproduct $\Delta:\coma\to\coma\otimes\coma$ on the basis element $x_c$ by
$$\Delta(x_c) = \sum_{ab=c} x_a\otimes x_b$$
where the equation $ab=c$ indicates that $a$ concatenated with $b$ is equal to $c$. We also define a counit $\varepsilon:\coma\to \mathbb{Z}$ by $\varepsilon(x_a)=0$ unless $a$ is the empty composition.

\begin{lemma}
	With $\Delta$ and $\varepsilon$, $\coma$ is a coassociative, counital coalgebra.
\end{lemma}
\begin{proof}
	We have 
	$$\Delta(x_d^n)=\sum x_a^p\otimes x_c^q$$
	Applying $\Delta$ to either tensor factor results in
	$$\sum x_a^p\otimes x_b^q\otimes x_c^r$$
	The symmetry of this is exactly the coassociativity condition. Seeing that we may choose $p=n$ or $q=n$, the definition of the counit gives us the result that $\coma$ is counital as well under $\Delta$ and $\varepsilon$.
\end{proof}

\begin{lemma}
	$\Delta:\coma\to\coma\otimes\coma$ is a homomorphism of rings.
\end{lemma}
\begin{proof}
Let $a,b$ be weak compositions of lengths $m$ and $n$, and for a weak composition $c$ of length $r$ and $0\leq k\leq r$ write $c_{\leq k}$ and $c_{>k}$ for the first $k$ and the remaining $r-k$ entries of $c$. Then
$$\Delta(x_a)\Delta(x_b)=\sum_{j=0}^{m}\sum_{k=0}^{n}x_{a_{\leq j}}x_{b_{\leq k}}\otimes x_{a_{>j}}x_{b_{>k}}\text{.}$$
The $(j,k)$ term vanishes unless $j=k$ and $m-j=n-k$, that is, unless $m=n$ and $j=k$. Hence $\Delta(x_a)\Delta(x_b)=0=\Delta(x_ax_b)$ when $m\neq n$, while for $m=n$
$$\Delta(x_a)\Delta(x_b)=\sum_{k=0}^n x_{a_{\leq k}+b_{\leq k}}\otimes x_{a_{>k}+b_{>k}}=\sum_{k=0}^n x_{(a+b)_{\leq k}}\otimes x_{(a+b)_{>k}}=\Delta(x_{a+b})=\Delta(x_ax_b)\text{.}$$
\end{proof}

\begin{lemma}
	$\nabla:\coma\otimes \coma\to \coma$ is a homomorphism of coalgebras.
\end{lemma}
\begin{proof}
  Consider the basis element $x_a\otimes x_b$ of $\coma\otimes \coma$. The product is
  $$x_{a+b}$$
  Applying $\Delta$ results in
$$\sum_{a'b'=a+b} x_{a'}\otimes x_{b'}$$
Conversely, applying $\Delta\otimes \Delta$ yields
$$\sum_{a_1a_2=a,b_1b_2=b}x_{a_1}\otimes x_{a_2}\otimes x_{b_1}\otimes x_{b_2}$$
and applying $\nabla\otimes \nabla$ yields
$$\sum x_{a_1 + b_1}\otimes x_{a_2 + b_2}$$
The definition of $\nabla$ ensures that $a_1$ and $b_1$ are the same length for a nonzero term, and similarly for $a_2$ and $b_2$. Setting $a_1+b_1=a'$, $a_2+b_2=b'$ yields the result.
\end{proof}

\begin{corollary}
	$\coma$ is a bialgebra over $\mathbb{Z}$.
\end{corollary}

$\coma$ is afforded a grading into finite-rank components by observing that each $\coma_n$ itself is a graded ring with each homogeneous component being a finitely generated free module. Considering the pair $(n, d)$, where $n$ is the number of variables and $d$ is the degree, as a $\mathbb{Z}^2$ grading, we may take the graded dual module $\dcoma$, which, by virtue of the grading, is isomorphic as a free module to $\coma$. We identify the dual basis element $x_a^*$ with the sequence of nonnegative integers $a$. That is,
$$\langle a, x_b\rangle = \delta_{ab}\text{.}$$
Inspection of the coproduct reveals that the product of $a$ and $b$ in $\dcoma$ is simply the concatenation $ab$. Hence $\dcoma$ is isomorphic to the free associative algebra on a countable set indexed by nonnegative integers.

$\dcoma$ also has a coproduct compatible with its product, namely
$$c\mapsto \sum_{a+b=c} a\otimes b$$
Thus $\coma$ and $\dcoma$ are dual bialgebras.

Calling $\coma$ the ``Schubert bialgebra'' may seem unnecessarily grandiose; however, the reason will become clear below.

\subsection{The Schubert basis}

In $\coma$, each $\coma_n$ has a basis of Schubert polynomials $\sch_u^n$, where the largest right descent of $u$ is at most $n$, ensuring that $\sch_u(x)$ has at most $n$ variables. Schubert polynomials limited to a specific number of variables have well-defined structure constants $c_{u,v}^w$, independent of $n$, given by
$$\sch_u^n\sch_v^n = \sum_{w}c_{u,v}^w\sch_w^n$$
These are known to be nonnegative integers by virtue of the fact that they count points of transverse intersections of generic translates of Schubert varieties; however, except in special cases no positive combinatorial formula is known.

\subsection{The elementary basis}\label{section:elementary}

An $n$-elementary weak composition is a finite sequence of nonnegative integers $\alpha=(\alpha_1,\alpha_2,\ldots,\alpha_N)$ such that $\alpha_i\leq i$ for all $i<n$ and $\alpha_i\leq n$  for all $i\geq n$, satisfying the additional restriction that $\alpha_i\geq \alpha_{i+1}$ if $i\geq n$.

An $n$-elementary monomial is an element of $\coma_n$ of the following form: 
$$e_{\alpha_1}^{1}\cdots e_{\alpha_n}^ne_{\alpha_{n+1}}^n\cdots e_{\alpha_m}^n$$
where $\alpha$ is an $n$-elementary weak composition. Define $\mathcal{E}_n$ to be the set of $n$-elementary monomials.

\begin{proposition}
	$\mathcal{E}_n$ forms a $\mathbb{Z}$-basis for $\coma_n$, and $\bigcup_{n}\mathcal{E}_n$ forms a $\mathbb{Z}$-basis for $\coma$.
\end{proposition}
\begin{proof}
	It is a theorem of Macdonald that the polynomial ring $\mathbb{Z}[x_1,\ldots,x_n]$ is a free module over $\Lambda_n$ with basis the Schubert polynomials $\sch_u(x)$ such that $u\in S_n$. These Schubert polynomials are uniquely expressed in terms of strict elementary symmetric monomials with fewer than $n$ variables. Since $\Lambda_n$ is a polynomial ring in the $e_{i,n}$ by the standard characterization of $\Lambda_n$ as a polynomial ring in $e_1,\ldots,e_n$, $\Lambda_n$ has a basis of monomials $e_{\lambda}^n$ as $\lambda$ ranges over all partitions with parts bounded by $n$. The multiplicative combination of these two bases is therefore a $\mathbb{Z}$-basis of $\coma_n$ by freeness of the module. This is exactly the description of the monomials $e_{\alpha}^n$ for $\alpha$ an $n$-elementary weak composition.
\end{proof}

We can ask for transition formulas between the Schubert and the elementary basis. From elementary to Schubert, we may use Sottile's Pieri formula iteratively \cite{sottile}, which provides a positive formula for multiplying a Schubert polynomial by an elementary symmetric polynomial. That is,

$$e_{\alpha}^n = \sum_{1\xrightarrow{1,2,\ldots,n}_{\alpha} u}\sch_u^n$$

For the reverse transition, we need a lemma on the skew divided difference operators $\delta_u^w$ of Definition~\ref{definition:divdiff}. Recall that $\delta^i=\frac{1}{y_i-y_{i+1}}(1-s_i)$ acts on the $y$ variables, that $\delta^w=\delta^{i_1}\cdots\delta^{i_k}$ for a reduced word $w=s_{i_1}\cdots s_{i_k}$, and that the skew operators are characterized by the Leibniz formula $\delta^w(pq)=\sum_{u\leq w}\delta^u(p)\,\delta_u^w(q)$. The lemma computes the application of a skew divided difference operator of degree $-n$ to a polynomial of degree $n$, which yields an integer for integer-coefficient polynomials. Specifically, it is a formula for application of such an operator with $\ell(u,w)=n$ to $y_p^n$ for some $p>0$.

\newcommand{\achain}[1]{\mathcal{C}^{(#1)}}
\begin{definition}
For $u,w\in S_\infty$ such that $u\leq w$, we define $\achain{p}(u,w)$ to be the set of saturated chains $C$ in the Bruhat interval $[u,w]$ of the form
$$u=u_0\prec u_1\prec \cdots \prec u_n=w$$  
such that $u_{i-1}(p)\neq u_i(p)$ for all $i$. We also define $\sigma(C)$ to be the number of indices $i$ for which the index $j\neq p$ such that $u_{i-1}(j)\neq u_i(j)$ satisfies $j<p$.
\end{definition}

\begin{lemma} \label{lemma:powerdiff}
Let $n\geq 0$ be an integer, let $p>0$ be an integer, and let $u, w\in S_\infty$ be permutations such that $\ell(u,w)=n$. Then we have that 
$$\delta_u^w(y_p^n) = \sum_{C\in \achain{p}(u,w)} (-1)^{\sigma(C)}.$$
\end{lemma}
\begin{proof}

Consider the case where $\ell(u,w)=1$. Then there is a transposition $t_{ij}$ such that $w = ut_{ij}$. Assume that $i<j$. If $p\notin\{i,j\}$, then $\delta_u^w(y_p)=0$. If $p=i$, then $\delta_u^w(y_p) = \delta^{ij}(y_p) = 1$. If instead $j=p$, then $\delta_u^w(y_p) = \delta^{ij}(y_p) = -1$. In all cases, the result holds.

Now, consider the case where $\ell(u,w)=n>1$. We have that
$$y_p^n = y_p\cdot y_p^{n-1}$$
Applying the Leibniz formula, we have
$$\delta_u^w(y_p^n) = \sum_{u\prec u'}\delta_u^{u'}(y_p)\delta_{u'}^w(y_p^{n-1})$$
By the above, $\delta_u^{u'}(y_p)$ is nonzero if and only if $u\prec u'$ is a cover in the Bruhat order such that $u(p)\neq u'(p)$. In this case, $\delta_u^{u'}(y_p)$ is either $1$ or $-1$ depending on the number of indices $j<p$ such that $u(j)\neq u'(j)$, which is either $0$ or $1$, respectively. By induction, 
$$\delta_{u'}^w(y_p^{n-1}) = \sum_{C'\in \achain{p}(u',w)} (-1)^{\sigma(C')}.$$
Hence, 
$$\delta_u^{u'}(y_p)\delta_{u'}^w(y_p^{n-1})=\begin{cases}
-\displaystyle\sum_{C'\in \achain{p}(u',w)} (-1)^{\sigma(C')}&\mbox{ if }u(j)\neq u'(j)\mbox{ for some }j<p\\
\displaystyle\sum_{C'\in \achain{p}(u',w)} (-1)^{\sigma(C')}&\mbox{ otherwise,}
\end{cases}$$
which, by definition, is
$$\sum_{C=C'\cup\{u\}, C'\in \achain{p}(u',w)} (-1)^{\sigma(C)}$$
Summing over all $u'$ such that $u\prec u'\leq w$, we obtain the result.
\end{proof}
Using a minor sleight of hand, this affords us an explicit formula for transition of a Schubert polynomial into the elementary basis. We first name the combinatorial data that will index the terms of the formula.

\begin{definition}
For $k,p>0$ let 
$$E_k(x;y_p)=\prod_{i=1}^k (x_i - y_p)\text{,}$$
which is a factorial elementary symmetric polynomial, a special case of a double Schubert polynomial.
\end{definition}

\begin{lemma} \label{lemma:doublediff}
  Let $u,w\in S_\infty$ satisfy $u\leq w$ and let $p>0,k>0$ be integers. Then we have
  $$\delta_u^w(E_k(x;-y_p))=\sum_{C\in\achain{p}(u,w)}(-1)^{\sigma(C)}e_{k-\ell(u,w)}^k+\mathcal{O}(y)$$
\end{lemma}

\begin{proof}
This follows from the observation that
$$E_k(x;-y_p) = \sum_{i=0}^k e_{k-i}^k y_p^i$$
and Lemma \ref{lemma:powerdiff}.
\end{proof}

\begin{definition}\label{definition:elemcoeff}
Fix $w\in S_\infty$ with $\maxd(w)\leq n$. Let
$$\lambda = (n, n, \ldots, n, n - 1,\ldots, 1)$$
be the partition with parts $n-1,n-2,\ldots,1$ preceded by some number of parts equal to $n$, chosen large enough that $\ell(ww_\lambda^{-1})=|\lambda|-\ell(w)$, where $w_\lambda$ is the dominant permutation with $\code^*(w_\lambda)=\lambda$. Write $N=\ell(\lambda)$.

A \emph{path} for $w$ is a sequence $P=(C_1,\ldots,C_N)$ of chains $C_i\in\achain{i}(u_{i-1},u_i)$ for some sequence of permutations $1=u_0\leq u_1\leq\cdots\leq u_N=ww_\lambda^{-1}$; denote the set of such paths by $\mathrm{Path}_\lambda(w)$. For $P\in\mathrm{Path}_\lambda(w)$ set
$$\sigma(P) = \sum_{i=1}^{N} \sigma(C_i),\qquad \alpha(P) = \bigl(\lambda_{N + 1 - i}-|C_{N + 1 - i}|\bigr)_{i=1}^{N},$$
so that $\alpha(P)$ is an $n$-elementary weak composition. Finally, for each weak composition $\alpha$ define
$$E^{\alpha;n}_w = \sum_{\substack{P \in \mathrm{Path}_\lambda(w)\\ \alpha(P)=\alpha}}(-1)^{\sigma(P)}.$$
\end{definition}

\begin{theorem}[Transition to the elementary basis] \label{theorem:elemtrans}
	For each $w$ with $\maxd(w)\leq n$, we have
	$$\sch_w^n = \sum_{\alpha}E_w^{\alpha;n}e_{\alpha}^n.$$
In particular, since $\mathcal{E}_n$ is a basis of $\coma_n$, the integers $E_w^{\alpha;n}$ do not depend on the number of parts equal to $n$ in $\lambda$ (identifying $\alpha$ with its extensions by trailing zeros).
\end{theorem}
\begin{proof}
Let $\lambda$, $w_\lambda$, and $N$ be as in Definition~\ref{definition:elemcoeff}. By \cite[(6.14)]{notes}, the double Schubert polynomial of the dominant permutation $w_\lambda$ factors as
$$\sch_{w_{\lambda}}(x;-y)=\prod_{i=1}^N E_{\lambda_i}(x;-y_i).$$
Recall from Definition~\ref{definition:divdiff} that $\delta^v$ acts on the $y$ variables with $\delta^v\sch_w(x;y)=(-1)^{\ell(v)}\sch_{vw}(x;y)$ whenever $\ell(vw)=\ell(w)-\ell(v)$, and that the sign disappears on $\sch_w(x;-y)$ because each $\delta^i$ is odd under $y\mapsto -y$: $\delta^v[\sch_w(x;-y)]=\sch_{vw}(x;-y)$. By the choice of $\lambda$, this gives
$$\sch_w(x;-y) = \delta^{ww_\lambda^{-1}}\prod_{i=1}^N E_{\lambda_i}(x;-y_i),$$
with $\delta^{ww_\lambda^{-1}}$ acting on the variables $y_i$ appearing in the product.
Expand the right-hand side using the Leibniz formula iterated across the $N$ factors: the result is a sum over sequences $1=u_0\leq u_1\leq\cdots\leq u_N=ww_\lambda^{-1}$ of the products $\prod_i\delta_{u_{i-1}}^{u_i}E_{\lambda_i}(x;-y_i)$. Setting $y_i=0$ for all $i$ and invoking Lemma~\ref{lemma:doublediff} for each factor, this becomes
$$\sch_w(x)=\sum_{P=(C_1,\ldots,C_N)}(-1)^{\sigma(C_1)+\cdots+\sigma(C_N)}e_{\lambda_{1}-|C_1|}^{\lambda_1}e_{\lambda_2-|C_2|}^{\lambda_2}\cdots e_{\lambda_{N}-|C_{N}|}^{\lambda_{N}},$$
the sum running over $P\in\mathrm{Path}_\lambda(w)$. The monomial in the $P$-th term is $e_{\alpha(P)}^n$, and collecting terms by $\alpha(P)$ gives the stated expansion with coefficients $E_w^{\alpha;n}$.
\end{proof}
These numbers are stable for fixed $w$ as soon as the number of variables is at least as large as the value of $n$ such that $w\in S_n$, and are well-studied but poorly understood.

\begin{example}
	Let $n=5$ and let $w$ be the permutation such that $\code(w)=(0,2,0,2,3)$. Then
	\begin{align*}
		\sch_w^n &= e^5_{(0,0,0,0,5,1,1)} - e^5_{(0,0,0,0,4,2,1)} + e^5_{(0,0,0,0,4,3)}- e^5_{(0,1,0,0,4,1,1)} +e^5_{(0,1,0,0,3,2,1)}\\
		&+e^5_{(0,1,0,0,5,1)} - e^5_{(0,1,0,0,3,3)}
	\end{align*}
\end{example}

\begin{remark}
This transition formula from the Schubert basis to the elementary basis can be used to obtain an extremely efficient formula for multiplication of Schubert polynomials. If we want to compute $\sch_u^n\sch_v^n$, we can first express $\sch_v^n$ in the elementary basis using Theorem \ref{theorem:elemtrans}. Then we can multiply by $\sch_v^n$ by iterating the Pieri formula. This is the method used in the python package \verb|schubmult| written by the author \cite{schubmult}. In fact, there is also a version for double Schubert polynomials, quantum Schubert polynomials, and quantum double Schubert polynomials, which we will not go into here but are implemented in the same package.
\end{remark}

\section{The dual algebra}\label{section:dualalgebra}

We examine the graded dual algebra $\dcoma$ more closely now. This is a graded ring generated by countably many elements that we denote by $[i]$, where $i$ is a nonnegative integer. The product, as mentioned previously, is concatenation of sequences. Thus
$$[a_1\cdots a_p][b_1\cdots b_q]=[a_1\cdots a_pb_1\cdots b_q]$$
It is not hard to see that $\dcoma$ is a free algebra on these generators. Thus the set of sequences $[a_1\cdots a_n]$ forms a $\mathbb{Z}$-basis for $\dcoma$. We may realize this as the dual of $\coma$ by declaring that
$$\langle [a_1\cdots a_n],x_1^{b_1}\cdots x_n^{b_n}\rangle = \prod_i\delta_{a_i,b_i}$$

We have a $\mathbb{Z}\times \mathbb{Z}$ grading such that
$$\deg([a_1\cdots a_n])=(n,-a_1-a_2\cdots-a_n)$$
The set of elements such that $\deg(a) = (n,-)$ is $\dcoma_n$, dual as a graded module to $\coma_n$.

\subsection{The dual Schubert basis}
       
There is a basis $\dsch_u^{n}$ dual to the Schubert basis for $\dcoma_n$. Specifically, with the unique pairing $\langle -,-\rangle:\dcoma\times \coma\to\mathbb{Z}$ such that
$$\langle \alpha,x_\beta\rangle = \delta_{\alpha\beta}$$
we define $\dsch_u^n$ to be the unique basis of $\dcoma_n$ such that
$$\langle \dsch_u^n,\sch_v^n\rangle = \delta_{uv}$$
We characterize it with an explicit formula, in terms of the chains $\achain{p}(u,w)$ and the statistic $\sigma$ of \S\ref{section:elementary}.

\begin{definition}\label{definition:dualpath}
Let $u\in S_\infty$ with $\maxd(u)\leq n$. A \emph{dual path} for $(u,n)$ is a sequence $P=(C_1,\ldots,C_n)$ of chains $C_i\in\achain{i}(u_{i-1},u_i)$ for some sequence of permutations $1=u_0\leq u_1\leq\cdots\leq u_n=u$; write $\mathrm{Path}^n(u)$ for the set of dual paths. For $P\in\mathrm{Path}^n(u)$ set
$$\sigma(P)=\sum_{i=1}^n\sigma(C_i),\qquad |P|=(|C_1|,\ldots,|C_n|)\text{,}$$
so that $|P|$ is a weak composition of length $n$ with $\bigl||P|\bigr|=\ell(u)$.
\end{definition}

\begin{theorem}\label{theorem:dualformula}
	For each permutation $u$ and integer $n$ with $\maxd(u)\leq n$,
	$$\dsch_u^{n} = \sum_{P\in\mathrm{Path}^n(u)} (-1)^{\sigma(P)}\,[\,|P|\,]\text{.}$$
	That is, the coefficient of $[\alpha]$ in $\dsch_u^n$ is the signed number of dual paths $P$ for $(u,n)$ with $|C_i|=\alpha_i$ for all $i$.
\end{theorem}
\begin{proof}
	By definition, the coefficient of $[\alpha]$ in $\dsch_u^n$ is the coefficient $c_{\alpha,u}$ of $\sch_u^n$ in the Schubert expansion of $x_\alpha$. Since $\partial^u\sch_v=\delta_{u,v}$ for $\ell(v)=\ell(u)$ and $\partial^u$ kills polynomials of degree less than $\ell(u)$ \cite{notes}, we have $c_{\alpha,u}=\partial^u(x_\alpha)$ when $|\alpha|=\ell(u)$, and $c_{\alpha,u}=0$ otherwise. Writing $x_\alpha=x_1^{\alpha_1}\cdots x_n^{\alpha_n}$ and applying the Leibniz formula of Definition~\ref{definition:divdiff} (in its $x$-variable form) iteratively across the $n$ factors,
	$$\partial^u(x_\alpha)=\sum_{1=u_0\leq u_1\leq\cdots\leq u_n=u}\ \prod_{i=1}^n\partial_{u_{i-1}}^{u_i}\bigl(x_i^{\alpha_i}\bigr)\text{.}$$
	Each factor $\partial_{u_{i-1}}^{u_i}(x_i^{\alpha_i})$ is homogeneous of degree $\alpha_i-\ell(u_{i-1},u_i)$ (zero if this is negative), and the degrees sum to $|\alpha|-\ell(u)=0$, so a term is a nonzero constant only if $\ell(u_{i-1},u_i)=\alpha_i$ for every $i$. For such terms Lemma~\ref{lemma:powerdiff}, which holds verbatim for $\partial_u^w(x_p^k)$, gives $\partial_{u_{i-1}}^{u_i}(x_i^{\alpha_i})=\sum_{C_i\in\achain{i}(u_{i-1},u_i)}(-1)^{\sigma(C_i)}$. Multiplying out yields the sum over dual paths $P$ with $|P|=\alpha$ of $(-1)^{\sigma(P)}$.
\end{proof}

\begin{remark}
Theorem~\ref{theorem:dualformula} can also be derived from the Cauchy formula $\sch_\mu(x;-y)=\sum_{v}\sch_v(x)\sch_{v\mu^{-1}}(y)$ for a dominant permutation $\mu$ with $\code(\mu)$ strictly decreasing and $\code_n(\mu)\geq\ell(u)$, $\code_{n+1}(\mu)=0$, by expanding $\sch_{u\mu^{-1}}(y)$ in the products $\prod_i e^{\code_i(\mu)}_{\code_i(\mu)-\alpha_i}(y)$ exactly as in the proof of Theorem~\ref{theorem:elemtrans}; the coefficients so obtained are independent of $\mu$.
\end{remark}

This is not stable for fixed $u$ as $n$ increases, and this is expected.

\begin{example}
Suppose $w = [1,4,5,2,3]$. Then
$$\dsch_w^3 = [022] - [013]$$
If we choose $[1,3,5,7,2,4,6]$ instead, we have
$$\dsch_{1357246}^4 = [0015] - [0033] - [0114] + [0123]$$
For $w=[4,2,7,1,3,5,6]$,
$$\dsch_{4271356}^3 = -[026] + [035] + [125] - [134] + [206] - [215] - [305] + [314]$$
\end{example}

Thus the product structure of $\dcoma$ encodes splitting of the Schubert polynomial into two sets of variables. In particular,
\begin{lemma} \label{lemma:dpieri}
Let $a\geq 0$ be an integer and $w\in S_\infty$. Then we have
$$[a]\cdot \dsch_{w}^{n} = \sum_{\substack{w'\downvar{1} w\\ \ell(w,w')=a}} \dsch_{w'}^{n+1}$$
\end{lemma}
\begin{proof}
  This follows from Proposition \ref{proposition:pullindex} with $i=1$.
\end{proof}

More generally, the product of $\dcoma$ is dual to splitting the variables of $\coma$ into two consecutive blocks. For permutations $w,u,v$ with $\maxd(w)\leq p+q$, $\maxd(u)\leq p$, $\maxd(v)\leq q$, define the \emph{branching coefficients} $b_{u,v}^w(p,q)$ by
\begin{equation}\label{eq:branching}
\sch_w(x_1,\ldots,x_{p+q})=\sum_{u,v}b_{u,v}^w(p,q)\,\sch_u(x_1,\ldots,x_p)\,\sch_v(x_{p+1},\ldots,x_{p+q})\text{;}
\end{equation}
these are well defined since the products $\sch_u(x_1,\ldots,x_p)\sch_v(x_{p+1},\ldots,x_{p+q})$ form a basis of $\mathbb{Z}[x_1,\ldots,x_{p+q}]$.

\begin{lemma}\label{lemma:dualbranch}
With notation as above,
$$\dsch_u^p\sqcupplus\dsch_v^q=\sum_{w}b_{u,v}^w(p,q)\,\dsch_w^{p+q}\text{.}$$
\end{lemma}
\begin{proof}
For $\alpha\in\dcoma_p$, $\beta\in\dcoma_q$ and $f\in\coma_{p+q}$ we have $\langle\alpha\sqcupplus\beta,f\rangle=\langle\alpha\otimes\beta,\Delta_{p,q}(f)\rangle$, where $\Delta_{p,q}:\coma_{p+q}\to\coma_p\otimes\coma_q$ is the component of $\Delta$; indeed both sides equal $\delta_{ab,c}$ on $\alpha=[a]$, $\beta=[b]$, $f=x_c$. Identifying $\coma_p\otimes\coma_q$ with polynomials in the two blocks of variables, $\Delta_{p,q}(f)=f(x_1,\ldots,x_{p+q})$ read as a polynomial in $(x_1,\ldots,x_p)$ and $(x_{p+1},\ldots,x_{p+q})$, so $\Delta_{p,q}(\sch_w^{p+q})=\sum_{u,v}b_{u,v}^w(p,q)\sch_u^p\otimes\sch_v^q$ by \eqref{eq:branching}. Hence the coefficient of $\dsch_w^{p+q}$ in $\dsch_u^p\sqcupplus\dsch_v^q$ is
\[\langle\dsch_u^p\sqcupplus\dsch_v^q,\sch_w^{p+q}\rangle=\langle\dsch_u^p\otimes\dsch_v^q,\Delta_{p,q}(\sch_w^{p+q})\rangle=b_{u,v}^w(p,q)\text{.}\qedhere\]
\end{proof}

\subsection{The coproduct}

Consider the coproduct dual to the product on $\coma$. We will denote it by
$$\Delta^*:\dcoma\to \dcoma\otimes\dcoma,$$
and we recall that it is characterized by the pairing identity
$$\langle \Delta^*(\alpha), x\otimes y\rangle = \langle \alpha, xy\rangle$$
for all $\alpha\in\dcoma$ and $x,y\in\coma$. Since the product on $\coma$ preserves the length grading (multiplication of polynomials within a fixed $\coma_n$, zero across $\coma_p\otimes\coma_q$ for $p\neq q$ with $p,q>0$), $\Delta^*$ sends $\dcoma_n$ into $\dcoma_n\otimes\dcoma_n$ for each $n$.

The Schubert structure constants now reappear as the coefficients of $\Delta^*$ in the dual Schubert basis.

\begin{theorem}[Schubert LR coefficients as dual Schubert coproduct]\label{theorem:dschcoproduct}
Let $w\in S_\infty$ with $\maxd(w)\leq n$. Then in $\dcoma_n\otimes\dcoma_n$,
$$\Delta^*(\dsch_w^n) \;=\; \sum_{u,v} c_{u,v}^w\, \dsch_u^n\otimes \dsch_v^n,$$
where $c_{u,v}^w$ are the Schubert structure constants $\sch_u^n\sch_v^n=\sum_w c_{u,v}^w\sch_w^n$, and the sum runs over pairs $u,v\in S_\infty$ with $\maxd(u),\maxd(v)\leq n$.
\end{theorem}
\begin{proof}
Expand $\Delta^*(\dsch_w^n) = \sum_{u,v}\lambda_{u,v}\,\dsch_u^n\otimes\dsch_v^n$ in the basis $\{\dsch_u^n\otimes\dsch_v^n\}$ of $\dcoma_n\otimes\dcoma_n$. Pairing with $\sch_u^n\otimes\sch_v^n$ and using the defining identity of $\Delta^*$,
$$\lambda_{u,v}\;=\;\langle \Delta^*(\dsch_w^n), \sch_u^n\otimes\sch_v^n\rangle\;=\;\langle \dsch_w^n,\sch_u^n\sch_v^n\rangle\;=\;\langle \dsch_w^n,\textstyle\sum_z c_{u,v}^z \sch_z^n\rangle\;=\;c_{u,v}^w.$$
\end{proof}

\begin{remark}
In particular, the Schubert Littlewood--Richardson coefficients are the structure constants of the coproduct $\Delta^*$ on $\dcoma$ in the dual Schubert basis. Thus the otherwise mysterious nonnegativity of $c_{u,v}^w$ has the bialgebra-theoretic content that $\Delta^*$ has nonnegative integer coefficients in the basis $\{\dsch_u^n\}$. The same argument applies verbatim to any basis of $\coma$ with nonnegative structure constants, such as the forest polynomials of \S\ref{section:forest}: the structure constants of $\Delta^*$ in the corresponding dual basis are the structure constants of the original basis.
\end{remark}

\begin{example}\label{example:dschcoproduct}
Take $w=24135\in S_5$ (so $\maxd(w)=2$) and $n=2$. Then in $\dcoma_2\otimes\dcoma_2$,
\begin{align*}
\Delta^*(\dsch_{24135}^2) \;=\;& 1\otimes \dsch_{2413}^2 + \dsch_{2413}^2\otimes 1\\
&{}+ \dsch_{132}^2\otimes \dsch_{1423}^2 + \dsch_{1423}^2\otimes \dsch_{132}^2\\
&{}+ \dsch_{132}^2\otimes \dsch_{231}^2 + \dsch_{231}^2\otimes \dsch_{132}^2\\
&{}+ \dsch_{21}^2\otimes \dsch_{1423}^2 + \dsch_{1423}^2\otimes \dsch_{21}^2.
\end{align*}
By Theorem~\ref{theorem:dschcoproduct}, the coefficient of $\dsch_u^2\otimes\dsch_v^2$ is the Schubert structure constant $c_{u,v}^{24135}$; for instance, $c_{132,\,1423}^{24135} = c_{21,\,1423}^{24135} = 1$, while $c_{u,v}^{24135}=0$ for every pair $(u,v)$ not appearing above.
\end{example}

\section{The ring of bounded RC graphs}\label{section:brc}

\subsection{The module of bounded RC graphs}

\begin{definition}\label{definition:rcgraph}
	Let $R\subseteq \mathbb{P}\times \mathbb{P}$ be a finite set. To each such set we associate an element $\wof{R}$ of $S_\infty$ as follows. Given a pair $(i, j)$ where $i,j>0$ are integers, define $s(i,j) = s_{i+j - 1}$. Totally order the grid such that $(i, j) < (a, b)$ if and only if $i < a$ or $i = a$ and $j > b$ (in other words, lexicographical order, except that the order on the second coordinate is reversed). By this ordering, index $R$ as $r_1, r_2,\ldots, r_m$ in increasing order. Then
	$$\wof{R} = s_{r_1}\cdots s_{r_m}$$
	If $\ell(\wof{R}) = m$, then we say that $R$ is an \emph{RC-graph}.
	
	A \emph{bounded RC-graph} is an RC-graph $R$ together with a specified number of rows $n\geq \max\{i: (i,j)\in R\text{ for some }j\}$, with the added condition that $\maxd(\wof{R})\leq n$. The definition is as an ordered pair $(R, n)$, but it will be far more convenient to abuse notation and consider the number of rows a property of $R$ itself. To denote the number of rows, we define the \emph{height} of $R$ to be $\hr(R)=n$. A bounded RC graph has an associated weight vector $\wtt(R)$ where $\wtt_i(R)$ is the number of elements in row $i$.

	To a bounded RC graph $R$ with $\hr(R)\geq 1$, we define the \emph{trim} operation $\trm(R)$ to be the bounded RC graph with height $\hr(R)-1$ and underlying set $\{(i-1, j): (i,j)\in R, i\geq 2\}$. We also define
	$$\shup^m R=\{(i+m, j)| (i,j)\in R\}$$
\end{definition}

Let $\BRC$ be the free abelian group spanned by all bounded RC graphs. If $\BRC_n$ is the subgroup spanned by all bounded RC graphs $R$ with $\hr(R)=n$, then we have a grading
$$\BRC = \bigoplus_{n=0}^\infty \BRC_n$$
There is an evident evaluation map $\mathrm{ev}: \BRC\to \coma$ defined on basis elements as
$$\mathrm{ev}(R) = x_{\wtt(R)}$$
which is a surjective homomorphism of graded modules. There is also a map $\alpha:\BRC\to\dcoma$ defined by
$$\alpha(R) = \dsch_{\wof{R}}^{\hr(R)}$$
which is also a surjective homomorphism of graded modules. In addition, there is the weight map $\omega:\BRC\to \dcoma$ defined by
$$\omega(R) = [\wtt(R)]$$
which is similarly a surjective homomorphism of graded modules.

We can define elements $\mathcal{S}_w(n)$ as
$$\mathcal{S}_w(n) = \sum_{\substack{\wof{R}=w\\\hr(R)=n}}R$$

For a generating element $[a]$ of $\dcoma$ and a bounded RC graph $R$, we define 
$$[a]\cdot R=\sum_{\substack{\trm(R') = R\\\wtt_1(R') = a}}R'$$
which is an element of $\BRC$. By virtue of the fact that $\dcoma$ is a free algebra generated by these elements, it is nearly a trivial observation that this is a left module action on $\BRC$. The consequences of this, however, are not at all trivial.

\begin{lemma}[{\cite[Corollary~3.11]{bbrc}}] \label{lemma:d1desc}
We have that the set $\{w'\mid w\downvar{1} w'\}$ is equal to the set of permutations $w'$ such that there exists a permutation $v=s_{a_1}\cdots s_{a_m}$ for some integers $a_1>a_2>\cdots>a_m\geq 1$ such that $w =v\shup{w'}$ and $\ell(w)=\ell(v)+\ell(w')$.
\end{lemma}

\begin{lemma}\label{lemma:factor}
Let $v$ be a permutation. If $v\downvar{1} v'$, let $a_1,\ldots, a_k$ be the elements of $Q_1(v',v)$ in decreasing order. Then
$$v=s_{a_1}\cdots s_{a_k}\shup{v'}$$
\end{lemma}
\begin{proof}
	Suppose $v\in S_{n}$. We have by definition that there is a sequence of integers $b_1,\ldots,b_p$, all distinct and greater than $1$, such that
	$$vt_{1,b_1}\cdots t_{1, b_p}(1) = n + 1$$
	and
	$$vt_{1,b_1}\cdots t_{1, b_p}(i+1) = v'(i)$$
	for all $i<n-1$. This means that
	$$vt_{1,b_1}\cdots t_{1, b_p} = s_ns_{n-1}\cdots s_1\shup{v'}$$
	This is because if $v''=\shup{v'}$, then
	$$v''(1)=1$$
	and
	$$v''(i) = v(i-1)+1$$
	The cycle $s_n\cdots s_1$ sends $1\mapsto n+1$ and $i\mapsto i-1$ if $1<i\leq n+1$, hence
	$$vt_{1,b_1}\cdots t_{1, b_p} = s_ns_{n-1}\cdots s_1\shup{v'}$$
	In particular,
	$$v = s_ns_{n-1}\cdots s_1\shup{v'}t_{1,b_p}\cdots t_{1, b_1}$$
	This results in the factorization
	$$v = s_n\cdots s_1t_{1,v'(b_p-1)+1}\cdots t_{1,v'(b_1-1)+1}\shup{v'}$$
	The value $v'(b_j-1)$ necessarily decreases as $j$ decreases, since applying the corresponding reflection in the reverse order strictly increases the length with each application.  Consequently, multiplying $s_n\cdots s_1$ on the right by these reflections removes the simple reflections $s_{v'(b_p-1)},\ldots, s_{v'(b_1-1)}$. The indices removed are precisely the complement of the elements of $Q_1(v',v)$, hence the elements of $Q_1(v',v)$ in decreasing order are what remain, as desired.
\end{proof}

\begin{definition}
For a set of positive integers $A=\{a_1,\ldots, a_m\}$ with $a_1>a_2>\cdots > a_m\geq 1$ and a positive integer $i$, define
$$\mathrm{row}_i(A) = \{(i, a_j) \mid a_j\in A\}$$
\end{definition}

\begin{theorem} \label{theorem:modactrc}
Let $a\geq 0$ be an integer and let $R\in \BRC$. Then
$$[a]\cdot R = \sum_{\substack{w'\downvar{1}\wof{R}\\\ell(\wof{R},w')=a}} \left(\mathrm{row}_1(Q_1(\wof{R},w'))\cup \shup{R}\right)$$
where the right side denotes bounded RC graphs with $\hr(R)+1$ rows.
\end{theorem}
\begin{proof}
	If $\trm(R') = R$, then by definition $\wof{R'} = s_{a_1}\cdots s_{a_m}\shup \wof{R}$. It follows by Lemma~\ref{lemma:d1desc} then $\wof{R'}\downvar{1}\wof{R}$. By Lemma~\ref{lemma:factor}, the integers $a_1,\ldots, a_m$ are precisely the elements of $Q_1(\wof{R},\wof{R'})$ in decreasing order. This establishes the result.
\end{proof}

\begin{lemma} \label{lemma:drcpieri}
	Let $a\geq 0$ be an integer and let $w\in S_\infty$. For any valid $n$, we have
	$$[a]\cdot \mathcal{S}_w(n) = \sum_{\substack{w\downvar{1} w'\\\ell(w,w')=a}} \mathcal{S}_{w'}(n+1)$$
\end{lemma}
\begin{proof}
Theorem \ref{theorem:modactrc} establishes an explicit bijection between $[a]\cdot R$ and the set of bounded RC graphs $R'$ such that $\wof{R'}\downvar{1}\wof{R}$ and $\ell(\wof{R},\wof{R'})=a$. The result follows by summing over all $R$ such that $\wof{R}=w$ and $\hr(R)=n$.
\end{proof}

Let $t$ be an indeterminate commuting with all elements of $\dcoma$. Write
$$\sch_\dcoma(t) = \sum_{a=0}^\infty [a]t^a$$
Define
$$\sch_\dcoma(x_1,x_2,\ldots,x_n) = \sch_\dcoma(x_1)\sch_\dcoma(x_2)\cdots \sch_\dcoma(x_n)$$

\begin{theorem}
	Let $\emptyset\in \BRC$ denote the empty bounded RC graph with $\hr(\emptyset)=0$. Then
  $$\sch_\dcoma(x_1,\ldots,x_n)\cdot \emptyset=\sum_{\substack{w\in S_\infty\\ \maxd{w}\leq n}}\sch_w(x_1,\ldots,x_n)\mathcal{S}_w(n)$$
\end{theorem}
\begin{proof}
	The result for $n=1$ is Lemma~\ref{lemma:drcpieri} together with Theorem~\ref{theorem:modactrc}. The general result follows by induction on $n$. Consider the inductive hypothesis
	$$\sch_\dcoma(x_2,\ldots,x_n)\cdot \emptyset=\sum_{w\in S_\infty}\sch_w(x_2,\ldots,x_n)\mathcal{S}_w(n-1)$$
	Then applying $\sch_\dcoma(x_1)$ to both sides, we have
	$$\sch_\dcoma(x_1)\sch_\dcoma(x_2,\ldots,x_n)\cdot \emptyset=\sum_{w\in S_\infty}\sch_w(x_2,\ldots,x_n)\sch_\dcoma(x_1)\mathcal{S}_w(n-1)$$
	which is equal to
	$$\sum_{a=0}^\infty\sum_{w\in S_\infty}x_1^a\sch_w(x_2,\ldots,x_n)[a]\mathcal{S}_w(n-1)$$
	Applying Lemma~\ref{lemma:dpieri} to each term, we have
	$$\sum_{a=0}^\infty\sum_{w\in S_\infty}x_1^a\sch_w(x_2,\ldots,x_n)\sum_{\substack{w'\downvar{1} w\\\ell(w,w')=a}} \mathcal{S}_{w'}(n)$$
	Bringing the polynomial inside the inner sum, this is
	$$\sum_{w'\in S_\infty}\left(\sum_{w:w'\downvar{1}w} x_1^{\ell(w,w')}\sch_w(x_2,\ldots,x_n)\right)\mathcal{S}_{w'}(n)$$
	which simplifies to
	$$\sum_{w'\in S_\infty}\sch_{w'}(x_1,\ldots,x_n)\mathcal{S}_{w'}(n)$$
	as desired.
\end{proof}

\begin{corollary}\label{corollary:evschub}
	For each $w\in S_\infty$ and valid $n$, we have
	$$\mathrm{ev}(\mathcal{S}_w(n)) = \sch_w^n(x_1,\ldots,x_n)$$
\end{corollary}

\subsection{The pipe dream visualization and roots}

Given a pair of positive integers $i,j$ and an RC graph $R$, define 
$$R[i, j]=\{(a, b)\in R\mid (a, b) > (i, j)\}$$
Given an RC graph $R$ and a pair of positive integers $i, j$, we define an ordered pair
$$\rtt_R(i,j)=(\wof{R[i, j]}^{-1}(i + j - 1), \wof{R[i,j]}^{-1}(i+j))$$

There is a common visualization of RC graphs as pipe dreams. In this visualization, we draw an infinite grid in the first quadrant, and in each position $(i,j)$ we draw either a crossing (if $(i,j)\in R$) or an elbow (if $(i,j)\notin R$). Then we draw pipes entering from the left edge of the grid, with the pipe entering at row $i$ labeled $i$. The pipes travel through the grid, turning at elbows and crossing at crossings, and exit at the top of the grid, which is labeled with the same number.

A modification to this common visualization that we use has the following additional features:
\begin{itemize}
\item We write the index of the simple reflections corresponding to the crossings in the grid area itself. Thus, at position $(i,j)$ we write the number $i+j-1$ if there is a crossing at that position.
\item For a bounded RC graph $R$ with $\hr(R)=n$, we only draw the first $n$ pipes entering from the left side of the grid and clip features outside of the first $n$ rows. The pipes exiting at the top are still labeled since the width is not limited.
\item \emph{Note:} The labels at the top are not of the permutation $\wof{R}$, but rather of the permutation $\wof{R}^{-1}$.
\end{itemize}

See Figure~\ref{figure:rcexamppipe} for an example of this visualization.	

\begin{figure}[h]
\centering

\caption{The pipe dream visualization of the bounded RC graph $R$ with $\hr(R)=5$ where $R=\{(1,1),(1,2),(2,1),(3,1),(3,3)\}$}
\label{figure:rcexamppipe}
\begin{tikzpicture}[scale=1.0]
  \draw[lightgray, very thin, opacity=0.3] (0,1) -- (6,1);
  \draw[lightgray, very thin, opacity=0.3] (0,2) -- (6,2);
  \draw[lightgray, very thin, opacity=0.3] (0,3) -- (6,3);
  \draw[lightgray, very thin, opacity=0.3] (0,4) -- (6,4);
  \draw[lightgray, very thin, opacity=0.3] (0,5) -- (6,5);
  \draw[lightgray, very thin, opacity=0.3] (0,6) -- (6,6);
  \draw[lightgray, very thin, opacity=0.3] (0,1) -- (0,6);
  \draw[lightgray, very thin, opacity=0.3] (1,1) -- (1,6);
  \draw[lightgray, very thin, opacity=0.3] (2,1) -- (2,6);
  \draw[lightgray, very thin, opacity=0.3] (3,1) -- (3,6);
  \draw[lightgray, very thin, opacity=0.3] (4,1) -- (4,6);
  \draw[lightgray, very thin, opacity=0.3] (5,1) -- (5,6);
  \draw[lightgray, very thin, opacity=0.3] (6,1) -- (6,6);
  \draw[blue, line width=1.0pt, line cap=round] (0,5.5) -- (1,5.5);
  \draw[blue, line width=1.0pt, line cap=round] (0.5,5) -- (0.5,6);
  \draw[blue, line width=1.0pt, line cap=round] (1,5.5) -- (2,5.5);
  \draw[blue, line width=1.0pt, line cap=round] (1.5,5) -- (1.5,6);
  \draw[blue, line width=1.0pt] (2,5.5) .. controls (2.3,5.5) and (2.5,5.7) .. (2.5,6);
  \draw[blue, line width=1.0pt] (2.5,5) .. controls (2.5,5.3) and (2.7,5.5) .. (3,5.5);
  \draw[blue, line width=1.0pt] (3,5.5) .. controls (3.3,5.5) and (3.5,5.7) .. (3.5,6);
  \draw[blue, line width=1.0pt] (3.5,5) .. controls (3.5,5.3) and (3.7,5.5) .. (4,5.5);
  \draw[blue, line width=1.0pt] (4,5.5) .. controls (4.3,5.5) and (4.5,5.7) .. (4.5,6);
  \draw[blue, line width=1.0pt] (4.5,5) .. controls (4.5,5.3) and (4.7,5.5) .. (5,5.5);
  \draw[blue, line width=1.0pt] (5,5.5) .. controls (5.3,5.5) and (5.5,5.7) .. (5.5,6);
  \draw[blue, line width=1.0pt, line cap=round] (0,4.5) -- (1,4.5);
  \draw[blue, line width=1.0pt, line cap=round] (0.5,4) -- (0.5,5);
  \draw[blue, line width=1.0pt] (1,4.5) .. controls (1.3,4.5) and (1.5,4.7) .. (1.5,5);
  \draw[blue, line width=1.0pt] (1.5,4) .. controls (1.5,4.3) and (1.7,4.5) .. (2,4.5);
  \draw[blue, line width=1.0pt] (2,4.5) .. controls (2.3,4.5) and (2.5,4.7) .. (2.5,5);
  \draw[blue, line width=1.0pt] (2.5,4) .. controls (2.5,4.3) and (2.7,4.5) .. (3,4.5);
  \draw[blue, line width=1.0pt] (3,4.5) .. controls (3.3,4.5) and (3.5,4.7) .. (3.5,5);
  \draw[blue, line width=1.0pt] (3.5,4) .. controls (3.5,4.3) and (3.7,4.5) .. (4,4.5);
  \draw[blue, line width=1.0pt] (4,4.5) .. controls (4.3,4.5) and (4.5,4.7) .. (4.5,5);
  \draw[blue, line width=1.0pt, line cap=round] (0,3.5) -- (1,3.5);
  \draw[blue, line width=1.0pt, line cap=round] (0.5,3) -- (0.5,4);
  \draw[blue, line width=1.0pt] (1,3.5) .. controls (1.3,3.5) and (1.5,3.7) .. (1.5,4);
  \draw[blue, line width=1.0pt] (1.5,3) .. controls (1.5,3.3) and (1.7,3.5) .. (2,3.5);
  \draw[blue, line width=1.0pt, line cap=round] (2,3.5) -- (3,3.5);
  \draw[blue, line width=1.0pt, line cap=round] (2.5,3) -- (2.5,4);
  \draw[blue, line width=1.0pt] (3,3.5) .. controls (3.3,3.5) and (3.5,3.7) .. (3.5,4);
  \draw[blue, line width=1.0pt] (0,2.5) .. controls (0.3,2.5) and (0.5,2.7) .. (0.5,3);
  \draw[blue, line width=1.0pt] (0.5,2) .. controls (0.5,2.3) and (0.7,2.5) .. (1,2.5);
  \draw[blue, line width=1.0pt] (1,2.5) .. controls (1.3,2.5) and (1.5,2.7) .. (1.5,3);
  \draw[blue, line width=1.0pt] (1.5,2) .. controls (1.5,2.3) and (1.7,2.5) .. (2,2.5);
  \draw[blue, line width=1.0pt] (2,2.5) .. controls (2.3,2.5) and (2.5,2.7) .. (2.5,3);
  \draw[blue, line width=1.0pt] (0,1.5) .. controls (0.3,1.5) and (0.5,1.7) .. (0.5,2);
  \draw[blue, line width=1.0pt] (0.5,1) .. controls (0.5,1.3) and (0.7,1.5) .. (1,1.5);
  \draw[blue, line width=1.0pt] (1,1.5) .. controls (1.3,1.5) and (1.5,1.7) .. (1.5,2);
  \node[font=\Large\bfseries, text=green!60!black, fill=white, inner sep=0.5pt, circle, transform shape] at (1.5,5.5) {2};
  \node[font=\Large\bfseries, text=green!60!black, fill=white, inner sep=0.5pt, circle, transform shape] at (0.5,4.5) {2};
  \node[font=\Large\bfseries, text=green!60!black, fill=white, inner sep=0.5pt, circle, transform shape] at (0.5,3.5) {3};
  \node[font=\Large\bfseries, text=green!60!black, fill=white, inner sep=0.5pt, circle, transform shape] at (0.5,5.5) {1};
  \node[font=\Large\bfseries, text=green!60!black, fill=white, inner sep=0.5pt, circle, transform shape] at (2.5,3.5) {5};
  \node[font=\scriptsize, text=black, anchor=south] at (0.5,6.3) {4};
  \node[font=\scriptsize, text=black, anchor=south] at (1.5,6.3) {2};
  \node[font=\scriptsize, text=black, anchor=south] at (2.5,6.3) {1};
  \node[font=\scriptsize, text=black, anchor=south] at (3.5,6.3) {3};
  \node[font=\scriptsize, text=black, anchor=south] at (4.5,6.3) {6};
  \node[font=\scriptsize, text=black, anchor=south] at (5.5,6.3) {5};
  \node[font=\scriptsize, text=black, anchor=east] at (-0.3,5.5) {1};
  \node[font=\scriptsize, text=black, anchor=east] at (-0.3,4.5) {2};
  \node[font=\scriptsize, text=black, anchor=east] at (-0.3,3.5) {3};
  \node[font=\scriptsize, text=black, anchor=east] at (-0.3,2.5) {4};
  \node[font=\scriptsize, text=black, anchor=east] at (-0.3,1.5) {5};
  \draw[black, line width=1.5pt] (0,1) rectangle (6,6);
\end{tikzpicture}
\end{figure}
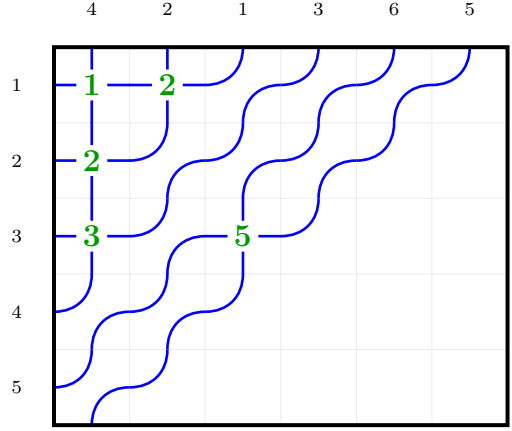

The main benefit of this is that visualizing $\rtt_R(i,j)$ is easy. The pipes that pass through position $(i,j)$ are labeled $s$ and $q$ for some $s,q>0$. For a positive root in an unoccupied square, we will have the following labeling, where $q<s$:

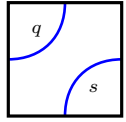
\begin{center}
\begin{tikzpicture}[scale=1.5]
  \draw[lightgray, very thin, opacity=0.3] (0,0) -- (1,0);
  \draw[lightgray, very thin, opacity=0.3] (0,1) -- (1,1);
  \draw[lightgray, very thin, opacity=0.3] (0,0) -- (0,1);
  \draw[lightgray, very thin, opacity=0.3] (1,0) -- (1,1);
  
  \draw[blue, line width=1.0pt] (0,0.5) .. controls (0.3,0.5) and (0.5,0.7) .. (0.5,1);
  
  \draw[blue, line width=1.0pt] (0.5,0) .. controls (0.5,0.3) and (0.7,0.5) .. (1,0.5);
  
  \node[font=\scriptsize, text=black] at (0.25,0.75) {$q$};
  \node[font=\scriptsize, text=black] at (0.75,0.25) {$s$};
  
  \draw[black, line width=1.2pt] (0,0) rectangle (1,1);
\end{tikzpicture}
\end{center}

We observe that in the above RC graph, the position $(1,5)$ does not have this configuration. Placing a crossing there would create a negative root, and the pipes would cross twice. \emph{Note that this results in a collection of ordered pairs that is not an RC graph, if such a crossing exists.} For a valid RC graphs, only positive roots occur as crossings, and this happens if and only if pipes cross at most once.

\begin{figure}[h]
\centering
\caption{An invalid set of crossings causing pipes to cross more than once, caused by inserting a crossing at a negative root}
\begin{tikzpicture}[scale=1.0]
  \draw[lightgray, very thin, opacity=0.3] (0,1) -- (6,1);
  \draw[lightgray, very thin, opacity=0.3] (0,2) -- (6,2);
  \draw[lightgray, very thin, opacity=0.3] (0,3) -- (6,3);
  \draw[lightgray, very thin, opacity=0.3] (0,4) -- (6,4);
  \draw[lightgray, very thin, opacity=0.3] (0,5) -- (6,5);
  \draw[lightgray, very thin, opacity=0.3] (0,6) -- (6,6);
  \draw[lightgray, very thin, opacity=0.3] (0,1) -- (0,6);
  \draw[lightgray, very thin, opacity=0.3] (1,1) -- (1,6);
  \draw[lightgray, very thin, opacity=0.3] (2,1) -- (2,6);
  \draw[lightgray, very thin, opacity=0.3] (3,1) -- (3,6);
  \draw[lightgray, very thin, opacity=0.3] (4,1) -- (4,6);
  \draw[lightgray, very thin, opacity=0.3] (5,1) -- (5,6);
  \draw[lightgray, very thin, opacity=0.3] (6,1) -- (6,6);
  \draw[blue, line width=1.0pt, line cap=round] (0,5.5) -- (1,5.5);
  \draw[blue, line width=1.0pt, line cap=round] (0.5,5) -- (0.5,6);
  \draw[blue, line width=1.0pt, line cap=round] (1,5.5) -- (2,5.5);
  \draw[blue, line width=1.0pt, line cap=round] (1.5,5) -- (1.5,6);
  \draw[blue, line width=1.0pt] (2,5.5) .. controls (2.3,5.5) and (2.5,5.7) .. (2.5,6);
  \draw[blue, line width=1.0pt] (2.5,5) .. controls (2.5,5.3) and (2.7,5.5) .. (3,5.5);
  \draw[blue, line width=1.0pt] (3,5.5) .. controls (3.3,5.5) and (3.5,5.7) .. (3.5,6);
  \draw[blue, line width=1.0pt] (3.5,5) .. controls (3.5,5.3) and (3.7,5.5) .. (4,5.5);
  \draw[blue, line width=1.0pt, line cap=round] (4,5.5) -- (5,5.5);
  \draw[blue, line width=1.0pt, line cap=round] (4.5,5) -- (4.5,6);
  \draw[blue, line width=1.0pt] (5,5.5) .. controls (5.3,5.5) and (5.5,5.7) .. (5.5,6);
  \draw[blue, line width=1.0pt, line cap=round] (0,4.5) -- (1,4.5);
  \draw[blue, line width=1.0pt, line cap=round] (0.5,4) -- (0.5,5);
  \draw[blue, line width=1.0pt] (1,4.5) .. controls (1.3,4.5) and (1.5,4.7) .. (1.5,5);
  \draw[blue, line width=1.0pt] (1.5,4) .. controls (1.5,4.3) and (1.7,4.5) .. (2,4.5);
  \draw[blue, line width=1.0pt] (2,4.5) .. controls (2.3,4.5) and (2.5,4.7) .. (2.5,5);
  \draw[blue, line width=1.0pt] (2.5,4) .. controls (2.5,4.3) and (2.7,4.5) .. (3,4.5);
  \draw[blue, line width=1.0pt] (3,4.5) .. controls (3.3,4.5) and (3.5,4.7) .. (3.5,5);
  \draw[blue, line width=1.0pt] (3.5,4) .. controls (3.5,4.3) and (3.7,4.5) .. (4,4.5);
  \draw[blue, line width=1.0pt] (4,4.5) .. controls (4.3,4.5) and (4.5,4.7) .. (4.5,5);
  \draw[blue, line width=1.0pt, line cap=round] (0,3.5) -- (1,3.5);
  \draw[blue, line width=1.0pt, line cap=round] (0.5,3) -- (0.5,4);
  \draw[blue, line width=1.0pt] (1,3.5) .. controls (1.3,3.5) and (1.5,3.7) .. (1.5,4);
  \draw[blue, line width=1.0pt] (1.5,3) .. controls (1.5,3.3) and (1.7,3.5) .. (2,3.5);
  \draw[blue, line width=1.0pt, line cap=round] (2,3.5) -- (3,3.5);
  \draw[blue, line width=1.0pt, line cap=round] (2.5,3) -- (2.5,4);
  \draw[blue, line width=1.0pt] (3,3.5) .. controls (3.3,3.5) and (3.5,3.7) .. (3.5,4);
  \draw[blue, line width=1.0pt] (0,2.5) .. controls (0.3,2.5) and (0.5,2.7) .. (0.5,3);
  \draw[blue, line width=1.0pt] (0.5,2) .. controls (0.5,2.3) and (0.7,2.5) .. (1,2.5);
  \draw[blue, line width=1.0pt] (1,2.5) .. controls (1.3,2.5) and (1.5,2.7) .. (1.5,3);
  \draw[blue, line width=1.0pt] (1.5,2) .. controls (1.5,2.3) and (1.7,2.5) .. (2,2.5);
  \draw[blue, line width=1.0pt] (2,2.5) .. controls (2.3,2.5) and (2.5,2.7) .. (2.5,3);
  \draw[blue, line width=1.0pt] (0,1.5) .. controls (0.3,1.5) and (0.5,1.7) .. (0.5,2);
  \draw[blue, line width=1.0pt] (0.5,1) .. controls (0.5,1.3) and (0.7,1.5) .. (1,1.5);
  \draw[blue, line width=1.0pt] (1,1.5) .. controls (1.3,1.5) and (1.5,1.7) .. (1.5,2);
  \node[font=\Large\bfseries, text=green!60!black, fill=white, inner sep=0.5pt, circle, transform shape] at (1.5,5.5) {2};
  \node[font=\Large\bfseries, text=green!60!black, fill=white, inner sep=0.5pt, circle, transform shape] at (0.5,4.5) {2};
  \node[font=\Large\bfseries, text=green!60!black, fill=white, inner sep=0.5pt, circle, transform shape] at (4.5,5.5) {5};
  \node[font=\Large\bfseries, text=green!60!black, fill=white, inner sep=0.5pt, circle, transform shape] at (0.5,3.5) {3};
  \node[font=\Large\bfseries, text=green!60!black, fill=white, inner sep=0.5pt, circle, transform shape] at (0.5,5.5) {1};
  \node[font=\Large\bfseries, text=green!60!black, fill=white, inner sep=0.5pt, circle, transform shape] at (2.5,3.5) {5};
  \node[font=\scriptsize, text=black, anchor=south] at (0.5,6.3) {4};
  \node[font=\scriptsize, text=black, anchor=south] at (1.5,6.3) {2};
  \node[font=\scriptsize, text=black, anchor=south] at (2.5,6.3) {1};
  \node[font=\scriptsize, text=black, anchor=south] at (3.5,6.3) {3};
  \node[font=\scriptsize, text=black, anchor=south] at (4.5,6.3) {5};
  \node[font=\scriptsize, text=black, anchor=south] at (5.5,6.3) {6};
  \node[font=\scriptsize, text=black, anchor=east] at (-0.3,5.5) {1};
  \node[font=\scriptsize, text=black, anchor=east] at (-0.3,4.5) {2};
  \node[font=\scriptsize, text=black, anchor=east] at (-0.3,3.5) {3};
  \node[font=\scriptsize, text=black, anchor=east] at (-0.3,2.5) {4};
  \node[font=\scriptsize, text=black, anchor=east] at (-0.3,1.5) {5};
  \draw[black, line width=1.5pt] (0,1) rectangle (6,6);
\end{tikzpicture}
\end{figure}
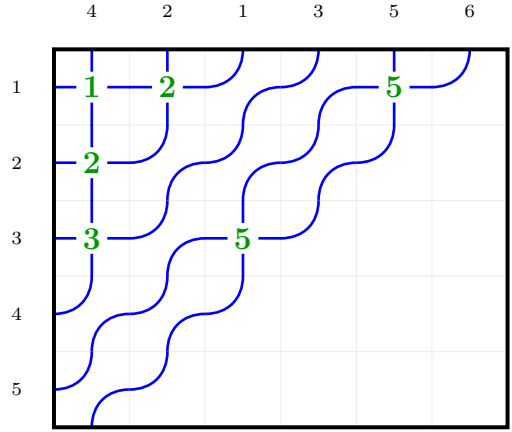

\subsection{Zeroing out the last row}\label{section:bpd}

We proceed now to define a product on $\BRC$ turning it into a ring. To do this, we need a map $\zeromap:\BRC\to\BRC$ trimming empty rows from the bottom instead of from the top, which is far more delicate than trimming from the top. We define $\zeromap$ by Little's bumping algorithm \cite{little2003combinatorial} in its pipe dream form \cite{billey2019bijective}, and we prove its properties by transporting it along the Gao--Huang bijection \cite{GaoHuang2023} to bumpless pipe dreams, where deleting a row is literal. In this subsection an RC graph is a finite set $R\subseteq\mathbb P\times\mathbb P$ with $\ell(\wof R)=|R|$ (a reduced pipe dream); the height enters only where stated. We write $R_{\leq p}=\{(i,j)\in R:i\leq p\}$.

\begin{figure}[htbp]
\centering

\begin{minipage}{0.45\textwidth}
\centering
\begin{tikzpicture}[scale=0.8]
  \draw[lightgray, very thin, opacity=0.3] (0,1) -- (4,1);
  \draw[lightgray, very thin, opacity=0.3] (0,2) -- (4,2);
  \draw[lightgray, very thin, opacity=0.3] (0,3) -- (4,3);
  \draw[lightgray, very thin, opacity=0.3] (0,4) -- (4,4);
  \draw[lightgray, very thin, opacity=0.3] (0,1) -- (0,4);
  \draw[lightgray, very thin, opacity=0.3] (1,1) -- (1,4);
  \draw[lightgray, very thin, opacity=0.3] (2,1) -- (2,4);
  \draw[lightgray, very thin, opacity=0.3] (3,1) -- (3,4);
  \draw[lightgray, very thin, opacity=0.3] (4,1) -- (4,4);
  \draw[blue, line width=1.0pt] (0,3.5) .. controls (0.3,3.5) and (0.5,3.7) .. (0.5,4);
  \draw[blue, line width=1.0pt] (0.5,3) .. controls (0.5,3.3) and (0.7,3.5) .. (1,3.5);
  \draw[blue, line width=1.0pt, line cap=round] (1,3.5) -- (2,3.5);
  \draw[blue, line width=1.0pt, line cap=round] (1.5,3) -- (1.5,4);
  \draw[blue, line width=1.0pt, line cap=round] (2,3.5) -- (3,3.5);
  \draw[blue, line width=1.0pt, line cap=round] (2.5,3) -- (2.5,4);
  \draw[blue, line width=1.0pt] (3,3.5) .. controls (3.3,3.5) and (3.5,3.7) .. (3.5,4);
  \draw[blue, line width=1.0pt] (0,2.5) .. controls (0.3,2.5) and (0.5,2.7) .. (0.5,3);
  \draw[blue, line width=1.0pt] (0.5,2) .. controls (0.5,2.3) and (0.7,2.5) .. (1,2.5);
  \draw[blue, line width=1.0pt, line cap=round] (1,2.5) -- (2,2.5);
  \draw[blue, line width=1.0pt, line cap=round] (1.5,2) -- (1.5,3);
  \draw[blue, line width=1.0pt] (2,2.5) .. controls (2.3,2.5) and (2.5,2.7) .. (2.5,3);
  \draw[blue, line width=1.0pt] (0,1.5) .. controls (0.3,1.5) and (0.5,1.7) .. (0.5,2);
  \draw[blue, line width=1.0pt] (0.5,1) .. controls (0.5,1.3) and (0.7,1.5) .. (1,1.5);
  \draw[blue, line width=1.0pt] (1,1.5) .. controls (1.3,1.5) and (1.5,1.7) .. (1.5,2);
  \node[font=\Large\bfseries, text=green!60!black, fill=white, inner sep=0.5pt, circle, transform shape] at (1.5,3.5) {2};
  \node[font=\Large\bfseries, text=green!60!black, fill=white, inner sep=0.5pt, circle, transform shape] at (2.5,3.5) {3};
  \node[font=\Large\bfseries, text=green!60!black, fill=white, inner sep=0.5pt, circle, transform shape] at (1.5,2.5) {3};
  \node[font=\scriptsize, text=black, anchor=south] at (0.5,4.3) {1};
  \node[font=\scriptsize, text=black, anchor=south] at (1.5,4.3) {4};
  \node[font=\scriptsize, text=black, anchor=south] at (2.5,4.3) {3};
  \node[font=\scriptsize, text=black, anchor=south] at (3.5,4.3) {2};
  \node[font=\scriptsize, text=black, anchor=east] at (-0.3,3.5) {1};
  \node[font=\scriptsize, text=black, anchor=east] at (-0.3,2.5) {2};
  \node[font=\scriptsize, text=black, anchor=east] at (-0.3,1.5) {3};
  \draw[black, line width=1.5pt] (0,1) rectangle (4,4);
\end{tikzpicture}
\caption*{(a) The RC graph $R$; the crossing $(2,2)$ has right root $(3,4)$}
\end{minipage}
\hfill
\begin{minipage}{0.45\textwidth}
\centering
\begin{tikzpicture}[scale=0.8]
  \draw[lightgray, very thin, opacity=0.3] (0,2) -- (4,2);
  \draw[lightgray, very thin, opacity=0.3] (0,3) -- (4,3);
  \draw[lightgray, very thin, opacity=0.3] (0,4) -- (4,4);
  \draw[lightgray, very thin, opacity=0.3] (0,2) -- (0,4);
  \draw[lightgray, very thin, opacity=0.3] (1,2) -- (1,4);
  \draw[lightgray, very thin, opacity=0.3] (2,2) -- (2,4);
  \draw[lightgray, very thin, opacity=0.3] (3,2) -- (3,4);
  \draw[lightgray, very thin, opacity=0.3] (4,2) -- (4,4);
  \draw[blue, line width=1.0pt, line cap=round] (0,3.5) -- (1,3.5);
  \draw[blue, line width=1.0pt, line cap=round] (0.5,3) -- (0.5,4);
  \draw[blue, line width=1.0pt] (1,3.5) .. controls (1.3,3.5) and (1.5,3.7) .. (1.5,4);
  \draw[blue, line width=1.0pt] (1.5,3) .. controls (1.5,3.3) and (1.7,3.5) .. (2,3.5);
  \draw[blue, line width=1.0pt, line cap=round] (2,3.5) -- (3,3.5);
  \draw[blue, line width=1.0pt, line cap=round] (2.5,3) -- (2.5,4);
  \draw[blue, line width=1.0pt] (3,3.5) .. controls (3.3,3.5) and (3.5,3.7) .. (3.5,4);
  \draw[blue, line width=1.0pt, line cap=round] (0,2.5) -- (1,2.5);
  \draw[blue, line width=1.0pt, line cap=round] (0.5,2) -- (0.5,3);
  \draw[blue, line width=1.0pt] (1,2.5) .. controls (1.3,2.5) and (1.5,2.7) .. (1.5,3);
  \draw[blue, line width=1.0pt] (1.5,2) .. controls (1.5,2.3) and (1.7,2.5) .. (2,2.5);
  \draw[blue, line width=1.0pt] (2,2.5) .. controls (2.3,2.5) and (2.5,2.7) .. (2.5,3);
  \node[font=\Large\bfseries, text=green!60!black, fill=white, inner sep=0.5pt, circle, transform shape] at (0.5,3.5) {1};
  \node[font=\Large\bfseries, text=green!60!black, fill=white, inner sep=0.5pt, circle, transform shape] at (2.5,3.5) {3};
  \node[font=\Large\bfseries, text=green!60!black, fill=white, inner sep=0.5pt, circle, transform shape] at (0.5,2.5) {2};
  \node[font=\scriptsize, text=black, anchor=south] at (0.5,4.3) {3};
  \node[font=\scriptsize, text=black, anchor=south] at (1.5,4.3) {1};
  \node[font=\scriptsize, text=black, anchor=south] at (2.5,4.3) {4};
  \node[font=\scriptsize, text=black, anchor=south] at (3.5,4.3) {2};
  \node[font=\scriptsize, text=black, anchor=east] at (-0.3,3.5) {1};
  \node[font=\scriptsize, text=black, anchor=east] at (-0.3,2.5) {2};
  \draw[black, line width=1.5pt] (0,2) rectangle (4,4);
\end{tikzpicture}
\caption*{(b) $\zeromap(R)$, of height $2$}
\end{minipage}

\caption{The map $\zeromap$ on $R=\{(1,2),(1,3),(2,2)\}$, of height $3$ and permutation $1432$. The crossing $(2,2)$ moves to $(2,1)$; the resulting word $3\,2\,2$ is not reduced, so $(1,2)$ moves to $(1,1)$, giving the reduced word $3\,1\,2$ of $1432\,t_{34}t_{13}=2413$. The empty row $3$ is then deleted.}
\label{fig:zeroing}
\end{figure}

\subsubsection*{Bumpless pipe dreams}
Fix $N$ larger than every row and column index that occurs below. A \emph{bumpless pipe dream} \cite{lam2021back} is a tiling of the $N\times N$ grid by six tiles: the blank tile, the crossing tile, the horizontal and vertical tiles, and two elbow tiles, the \emph{$r$-tile}, in which a pipe enters from the south and leaves to the east, and the \emph{$j$-tile}, in which a pipe enters from the west and leaves to the north, such that the resulting pipes enter the grid from the south edge, one in each column, travel north and east, leave through the east edge, one in each row, and no two pipes cross twice. The pipe entering at the bottom of column $j$ is \emph{pipe $j$}. The permutation $w$ of $B$ is defined by: the pipe leaving through row $i$ is pipe $w(i)$. We write $\mathrm{BPD}(w)$ for the set of bumpless pipe dreams of $w$, $\mathrm{blank}(B)$ for the set of blank tiles of $B$ and $B_{\leq p}$ for rows $1,\ldots,p$ of $B$. By \cite[Theorem~5.13]{lam2021back}, $\sch_w=\sum_{B\in\mathrm{BPD}(w)}\prod_{(i,j)\in\mathrm{blank}(B)}x_i$. Enlarging $N$ appends pipes that travel straight north and turn once, and appends horizontal tiles to the rows already present, so statements about $B_{\leq p}$ do not depend on $N$. Since a pipe travels north and east, the pipe leaving through row $i$ lies entirely in rows $\geq i$; we use this repeatedly.

Gao and Huang \cite{GaoHuang2023} construct, for every $w$, a bijection between $\mathrm{BPD}(w)$ and the RC graphs with permutation $w$ that sends $\mathrm{blank}(B)$ row by row to the crossings, i.e.\ row $i$ of $B$ has as many blank tiles as row $i$ of the RC graph has crossings \cite[Theorem~3.6]{GaoHuang2023}. We write $\Phi(R)\in\mathrm{BPD}(\wof R)$ for the bumpless pipe dream corresponding to the RC graph $R$; $\Phi$ is injective and depends only on the set $R$.

We recall Little's bumping operation on reduced words \cite{little2003combinatorial}.

\begin{definition}[Little bumps]\label{definition:littlebump}
  Given a reduced word $a_1\cdots a_m$ for a permutation $w$ and an index $1\leq p\leq m$, the \emph{Little bump} $\ltb_p(a_1\cdots a_m)$ is defined recursively as follows. If $a_p>1$, replace $a_p$ with $a_p-1$, and if $a_p=1$ then instead replace all $a_q$ with $q\neq p$ with $a_q+1$. Denote the new word by $a_1'\cdots a_m'$. If this resulting word is reduced, we are done. Otherwise, there is a unique index $j\neq p$ such that the word obtained by deleting $a_j'$ from the new word is reduced by the deletion property of Coxeter groups. We then define
  $$\ltb_p(a_1\cdots a_m) = \ltb_j(a_1'\cdots a_m')$$
\end{definition}

\subsubsection*{Monk moves}
Monk's rule in the form
\begin{equation}\label{eq:monk}
x_\alpha\sch_w+\sum_{\substack{k<\alpha\\ wt_{k\alpha}\gtrdot w}}\sch_{wt_{k\alpha}}=\sum_{\substack{l>\alpha\\ wt_{\alpha l}\gtrdot w}}\sch_{wt_{\alpha l}}
\end{equation}
is realized bijectively on RC graphs by Billey--Holroyd--Young \cite{billey2019bijective} and on bumpless pipe dreams by Huang \cite{huang2020monk}: there are maps
$$x_\alpha\!\rightsquigarrow:\RC(w)\to\bigsqcup_{l}\RC(wt_{\alpha l})\text{,}\qquad m_{k,\alpha}:\RC(wt_{k\alpha})\to\bigsqcup_{l}\RC(wt_{\alpha l})$$
($l$ ranging over the right side of \eqref{eq:monk}), which together form a bijection from the disjoint union of the sets on the left of \eqref{eq:monk} to the disjoint union of those on the right, with $x_\alpha\!\rightsquigarrow$ adding one crossing in row $\alpha$ and $m_{k,\alpha}$ preserving the number of crossings in every row, and maps with the same names and properties (with blank tiles in place of crossings) on bumpless pipe dreams \cite[Definitions~4.2 and~4.4]{GaoHuang2023}. Gao and Huang prove that $\Phi$ intertwines them \cite[Theorem~4.5]{GaoHuang2023}:
\begin{equation}\label{eq:GHintertwine}
\Phi(x_\alpha\!\rightsquigarrow R)=x_\alpha\!\rightsquigarrow\Phi(R)\text{,}\qquad \Phi(m_{k,\alpha}R)=m_{k,\alpha}\Phi(R)\text{.}
\end{equation}
We recall the two descriptions to the extent needed. On RC graphs \cite[Definition~4.2]{GaoHuang2023}, $x_\alpha\!\rightsquigarrow R$ is computed by placing a crossing at the leftmost elbow of row $\alpha$; whenever the crossing just placed makes two pipes cross twice, the other crossing of that pair of pipes is replaced by an elbow and a crossing is placed at the nearest elbow to its right in its row, until the result is reduced. For $D\in\RC(wt_{k\alpha})$, $m_{k,\alpha}D$ is computed in the same way after first replacing by an elbow the unique crossing $c\in D$ with $\rtt_D(c)=(k,\alpha)$, i.e.\ the crossing whose removal yields $w$, and placing a crossing at the nearest elbow to its right in its row. Consequently, the inverse of the Monk bijection on an element $X$ of the right side is computed by removing the crossing $c\in X$ whose removal yields $w$, placing a crossing at the nearest elbow to its left in the same row, and whenever this makes two pipes cross twice, removing the other crossing of that pair and placing a crossing at the nearest elbow to its left in its row; if a crossing to be moved has no elbow to its left, it is simply removed and the process stops, and the result lies in $\RC(w)$ (this is the $x_\alpha\!\rightsquigarrow$ case). Moving a crossing one column to the left decrements its letter by one, and moving it across a block of crossings to the nearest elbow decrements the letters of the block one after the other, the two deletable positions of the intermediate words being adjacent equal letters; hence, when no crossing is removed (the $m_{k,\alpha}$ case), the inverse Monk move is, in terms of the reading word, the Little bump $\ltb_c$ of Definition~\ref{definition:littlebump} at the position of $c$, with the row sequence unchanged. We call it the \emph{Little bump of $X$ at $c$} in either case. Here we use Little's lemma \cite[Lemma~4]{little2003combinatorial} in the following form: if a word $\mathbf c$ is not reduced but $\mathbf c\setminus k$ is, then exactly one pair of wires of $\mathbf c$ crosses twice, the positions whose deletion leaves a reduced word are exactly the two crossings of this pair, and since the right root of a position records the two wires crossing there, these two positions have the same right root up to order.

On bumpless pipe dreams \cite[Algorithm~1 and proof of Theorem~2]{huang2020monk}, $x_\alpha\!\rightsquigarrow B$ is computed by repeatedly modifying the pipe $p$ that leaves through row $\alpha$: $p$ is drooped from its current $r$-shaped turn into the nearest available tile to the southeast, and when the drooped pipe meets a pipe $q$ in a bump tile the bump is either replaced by a crossing (if $p$ and $q$ do not cross, in which case the algorithm stops) or exchanged with the existing crossing of $p$ and $q$, after which $p$ is drooped again. For $B\in\mathrm{BPD}(wt_{k\alpha})$, $m_{k,\alpha}B$ is computed by replacing the crossing of the two pipes that leave through rows $k$ and $\alpha$ by a bump tile, after which the pipe leaving through row $\alpha$ has the $r$-shaped turn in this tile and is the pipe $p$ to which the same algorithm is applied. In both cases the permutation of the diagram is $w$ throughout, except for the final step, so $p$ leaves through row $\alpha$ throughout.

\begin{lemma}[Row confinement]\label{lemma:bpdconfine}
Let $B\in\mathrm{BPD}(w)$ and $B'\in\mathrm{BPD}(wt_{k\alpha})$ with $k<\alpha$ and $wt_{k\alpha}\gtrdot w$. Then $x_\alpha\!\rightsquigarrow B$ agrees with $B$ in rows $1,\ldots,\alpha-1$, and $m_{k,\alpha}B'$ agrees with $B'$ in rows $1,\ldots,\alpha-1$.
\end{lemma}
\begin{proof}
Every tile changed by Huang's algorithm lies either on the pipe $p$ that leaves through row $\alpha$ (the tiles of a bump and of a crossing that are exchanged, the final bump replaced by a crossing, and the initial crossing replaced by a bump, which lies on both pipes leaving through rows $k$ and $\alpha$) or in the rectangle spanned by a droop of $p$, whose northwest corner is a turn of $p$. Since $p$ travels north and east and leaves through row $\alpha$, all of $p$ lies in rows $\geq\alpha$, and so does every such rectangle.
\end{proof}

\begin{lemma}[Little bumps and top rows]\label{lemma:bumpconfine}
Let $X$ be an RC graph with $\wof X=\tau$ and let $c\in X$ be a crossing with $\rtt_X(c)=(s,\beta)$, $s<\beta$, such that $X\setminus\{c\}$ is reduced. Let $X'$ be the RC graph obtained from $X$ by the Little bump at $c$ (the inverse Monk move described above). Then $\Phi(X)$ and $\Phi(X')$ agree in rows $1,\ldots,s-1$.
\end{lemma}
\begin{proof}
Removing $c$ yields $\pi:=\tau t_{s\beta}$ with $\ell(\pi)=\ell(\tau)-1$, so $X\in\RC(\pi t_{s\beta})$ lies on the right side of the Monk bijection for $(\pi,s)$, and $X'$ is its preimage: $X=x_s\!\rightsquigarrow X'$ with $X'\in\RC(\pi)$, or $X=m_{k,s}X'$ with $X'\in\RC(\pi t_{ks})$ for some $k<s$. By \eqref{eq:GHintertwine} the same relation holds between $\Phi(X)$ and $\Phi(X')$, and Lemma~\ref{lemma:bpdconfine} gives the claim.
\end{proof}

\begin{lemma}[Bounded completion]\label{lemma:bpdcompletion}
Let $B\in\mathrm{BPD}(w)$ have no blank tiles in rows $>p$, and suppose $\maxd(w)\leq p$. Then rows $>p$ of $B$ contain no $j$-shaped tiles. Consequently, in rows $>p$ every pipe $c$ travels straight north along column $c$, until it reaches row $p$ or until it reaches its exit row $w^{-1}(c)>p$, where it turns east and travels straight to the east edge; in particular $B$ is determined by $B_{\leq p}$.
\end{lemma}
\begin{proof}
Suppose $(i,c)$ is a $j$-shaped tile with $i>p$, and take $i$ maximal. No pipe enters $(i,c)$ from the south, so pipe $c$ turns east at some tile $(i'',c)$ with $i''>i$; as rows $>i$ contain no $j$-shaped tiles and no blanks, pipe $c$ then travels east along row $i''$ to the edge, so $w(i'')=c$. The pipe $P$ entering $(i,c)$ from the west travels along row $i$ from an $r$-shaped turn $(i,c_0)$, $c_0<c$. The pipe leaving through row $i$ is not $P$; its last turn is an $r$-shaped tile $(i,c_1)$ from which it travels east along row $i$ to the edge. The tiles $(i,c_0),\ldots,(i,c)$ are occupied by $P$ and an $r$-shaped tile cannot be entered from the west, so $c_1>c$. Below row $i$ there are no $j$-shaped tiles, so this pipe traveled straight north along column $c_1$ and is pipe $c_1$: $w(i)=c_1>c=w(i'')$ with $p<i<i''$, contradicting $\maxd(w)\leq p$. The remaining assertions follow: with no blanks and no $j$-shaped tiles in rows $>p$, a pipe in these rows can only travel north in its own column and, after one $r$-shaped turn, east.
\end{proof}

\begin{proposition}[Top rows]\label{proposition:toprows}
Let $R$ be an RC graph and $p\geq 0$. Then $\Phi(R)$ and $\Phi(R_{\leq p})$ agree in rows $1,\ldots,p$.
\end{proposition}
\begin{proof}
It suffices to prove: if $\alpha$ is the last nonempty row of $R$ and $x$ is the leftmost crossing in row $\alpha$, then $\Phi(R)$ and $\Phi(R\setminus\{x\})$ agree in rows $1,\ldots,\alpha-1$. Indeed, removing the crossings of rows $>p$ one at a time in this manner never changes rows $1,\ldots,p$ of $\Phi$ and ends at $R_{\leq p}$.

Let $x=(\alpha,j)$, $a=\alpha+j-1$ its letter, $R^-=R\setminus\{x\}$ and $\pi=\wof{R^-}$. As $x$ is the last letter of the reading word, $R^-$ is reduced and $\wof R=\pi s_a$, so $\rtt_R(x)=(a,a+1)$. We induct on $j$. If $j=1$ then $a=\alpha$, the leftmost elbow of $R^-$ in row $\alpha$ is $(\alpha,1)$, and $R=x_\alpha\!\rightsquigarrow R^-$ because $R$ is reduced; by \eqref{eq:GHintertwine} and Lemma~\ref{lemma:bpdconfine}, $\Phi(R)$ and $\Phi(R^-)$ agree in rows $<\alpha$.

Let $j\geq 2$ and let $X'$ be the Little bump of $R$ at $x$. Since $a>\alpha$, row $a$ of $R$ is empty, so $R$ is not in the image of $x_a\!\rightsquigarrow$; hence $R=m_{k,a}X'$ for some $k<a$, and $X'$ has as many crossings as $R$ in every row. By Lemma~\ref{lemma:bumpconfine}, $\Phi(R)$ and $\Phi(X')$ agree in rows $<a$, hence in rows $<\alpha$. During the bump, a crossing moves only within its row and only to the nearest elbow on its left, so the crossings of row $\alpha$ other than $x$ remain to the right of $x$: the last nonempty row of $X'$ is $\alpha$, its leftmost crossing is $x'=(\alpha,j')$ with $j'<j$, and $X'\setminus\{x'\}$ is reduced. By induction, $\Phi(X')$ and $\Phi(X'\setminus\{x'\})$ agree in rows $<\alpha$. It remains to show that $\Phi(X'\setminus\{x'\})$ and $\Phi(R^-)$ agree in rows $<\alpha$.

Write the reading word of $R$ as $\mathbf a\,a$ with $\mathbf a$ the reading word of $R^-$, and follow the Little bump letter by letter. We use the following restriction property, in which $\mathbf a\mathbf b$ is a word that is not reduced but becomes reduced when a position $k$ of $\mathbf a$ is deleted, and $\mathbf a^-$ denotes $\mathbf a$ with the letter at $k$ decremented: (i) if $\mathbf a^-$ is not reduced, then its defect (the unique position $j\neq k$ with $\mathbf a^-\setminus j$ reduced) is also the defect of $\mathbf a^-\mathbf b$; (ii) if $\mathbf a^-$ is reduced, then $\mathbf a^-\mathbf b$ is reduced or its defect lies in $\mathbf b$. For (i), the exchange property applied to the reduced word $\mathbf a^-\setminus k=\mathbf a\setminus k$ gives $\wof{\mathbf a^-\setminus j}=\wof{\mathbf a\setminus k}$, so $(\mathbf a^-\setminus j)\mathbf b$ has the permutation and the length of the reduced word $(\mathbf a\setminus k)\mathbf b$ and is reduced, and $j$ is the defect of $\mathbf a^-\mathbf b$ by uniqueness. For (ii), a defect $j$ of $\mathbf a^-\mathbf b$ in $\mathbf a$ would give $\wof{\mathbf a^-\setminus j}=\wof{\mathbf a\setminus k}=\wof{\mathbf a^-\setminus k}$, so the distinct positions $j,k$ of the reduced word $\mathbf a^-$ would have the same right root. The same holds with the roles of prefix and suffix exchanged.

The bump begins by decrementing the last letter, $a\mapsto a-1$. If the word is reduced we are done with $X'\setminus\{x'\}=R^-$. Otherwise the defect $k$ lies in $\mathbf a$ by (ii) for the suffix, and by Little's lemma the positions $k$ and the last position have the same right root up to order; the last letter $a-1$ has right root $(a-1,a)$, so $\rtt_{\mathbf a}(k)=\{a-1,a\}$ as a set, and $\mathbf a\setminus k$ is reduced, being a prefix of the reduced word $(\mathbf a\,(a-1))\setminus k$. By (i) and (ii) the bump now performs the complete Little bump of $\mathbf a$ at $k$, producing a reduced word $\mathbf a^{(1)}$ (the RC graph $A^{(1)}$, with the rows of $R^-$), and then either stops or continues by decrementing the last letter again, to $a-2$, after which the same analysis applies to $\mathbf a^{(1)}(a-2)$, and so on. Thus $X'\setminus\{x'\}=A^{(m)}$ where $A^{(0)}=R^-$ and $A^{(i+1)}$ is the Little bump of $A^{(i)}$ at a crossing with right root $\{a-i-1,a-i\}$ whose removal leaves a reduced word. The last letter is decremented to $a-i-1$ only if $x$ moves to column $j-i-1\geq 1$, so $a-i-1\geq\alpha$. By Lemma~\ref{lemma:bumpconfine}, $\Phi(A^{(i)})$ and $\Phi(A^{(i+1)})$ agree in rows $<a-i-1$, hence in rows $<\alpha$, for every $i$, and therefore so do $\Phi(R^-)$ and $\Phi(X'\setminus\{x'\})$.
\end{proof}

\subsubsection*{Transition steps}
\begin{definition}
For a permutation $w$ with $\maxd(w)\le n$ define
\[
P_1(w)=\{v : w\downvar{1}v\},\qquad
T_n(w)=\{w' : w\downvar{n}w',\ \ell(w')=\ell(w)\}.
\]
\end{definition}

By Proposition~\ref{proposition:pullindex} with $i=n$ (this is the identity used in the proof of Theorem~\ref{theorem:zero_bijection}),
\begin{equation}\label{E:Tn}
\sch_w(x_1,\dots,x_{n-1},0)=\sum_{w'\in T_n(w)}\sch_{w'}(x_1,\dots,x_{n-1}),
\end{equation}
and by Theorem~\ref{theorem:modactrc} together with Corollary~\ref{corollary:evschub} (or Proposition~\ref{proposition:pullindex} with $i=1$, $y=0$),
\begin{equation}\label{E:P1}
\sch_w(x_1,\dots,x_n)=\sum_{v\in P_1(w)}x_1^{\ell(w)-\ell(v)}\,\sch_v(x_2,\dots,x_n).
\end{equation}

Let $u$ be a permutation with $\maxd(u)=n$. Let $s=\max\{j>n : u(j)<u(n)\}$. A \emph{transition step} at $n$ replaces $u$ by
\[
u^{+}=u\,t_{ns}\,t_{in}\qquad\text{for some } i<n \text{ with } \ell(u\,t_{ns}t_{in})=\ell(u\,t_{ns})+1 .
\]
The Lascoux--Sch\"utzenberger transition formula (Macdonald, \emph{Notes on Schubert polynomials}, (4.16)) reads
$\sch_u=x_n\sch_{u t_{ns}}+\sum_i \sch_{u t_{ns}t_{in}}$, the sum over all such $i$.

\begin{lemma}\label{L:trim}
Let $w=c\cdot{\uparrow}v$ with $\ell(w)=\ell(c)+\ell(v)$ and $\maxd(w)\le n$. Then $\maxd(v)\le n-1$.
In particular, if $R\in\RC(w,n)^0$ with $n\ge2$, then $\trm(R)\in\RC(v,n-1)^0$ with $v=\wof{\trm R}\in P_1(w)$;
and every $v\in P_1(w)$ has $\maxd(v)\le n-1$.
\end{lemma}
\begin{proof}
Let $u={\uparrow}v$ and suppose $u(j)>u(j+1)$. Write $u=u''s_j$ with $\ell(u'')=\ell(u)-1$. Then
$w=(cu'')s_j$ and $\ell(cu'')\le \ell(c)+\ell(u)-1<\ell(w)$, so $w(j)>w(j+1)$. Hence every right descent of
${\uparrow}v$ is a right descent of $w$, so $\maxd({\uparrow}v)\le n$, i.e.\ $\maxd(v)\le n-1$.

For the second statement, the reading word of $R$ is (word of row $1$)$\cdot$(word of rows $\ge2$), which is reduced,
so $w=c\cdot{\uparrow}\wof{\trm R}$ with lengths adding, where $c$ is the permutation of row $1$; and $w\downvar{1}\wof{\trm R}$
by the proof of Theorem~\ref{theorem:modactrc}. Row $n-1$ of $\trm R$ is row $n$ of $R$, which is empty.
The last statement follows from Lemmas~\ref{lemma:d1desc} and~\ref{lemma:factor}, which give $w=c\cdot{\uparrow}v$ with lengths adding for every $v\in P_1(w)$.
\end{proof}

\begin{lemma}\label{L:leaves}
Let $\maxd(w)\le n$. If $\maxd(w)<n$ then $T_n(w)=\{w\}$. Otherwise $T_n(w)$ is the set of permutations $w'$ with
$\maxd(w')<n$ that can be reached from $w$ by a sequence of transition steps at $n$.
Moreover, for a single step $u\mapsto u^+=u\,t_{ns}t_{in}$, writing $\zeta=u(n)$ and $\eta=u(s)$:
\begin{enumerate}
\item[(O1)] $u^+(n)=u(i)<\eta<\zeta$; in particular the value in position $n$ strictly decreases at every step;
\item[(O2)] the only position $l<n$ with $u^+(l)\ne u(l)$ is $l=i$, and $u(i)<u^+(i)=\eta<\zeta$.
\end{enumerate}
\end{lemma}
\begin{proof}
(O1) and (O2) are read off from the definition ($\eta<\zeta$ by choice of $s$, and $u(i)<\eta$ by the cover criterion).
Since $\maxd(u)=n$, the values $u(n+1)<u(n+2)<\cdots$ increase, those in positions $n+1,\dots,s$ are $<\zeta$, and those in positions $>s$ are $>\zeta$.
So after replacing $u(s)$ by $\zeta$ the values in positions $>n$ still increase; thus $\maxd(u^+)\le n$.
The number of positions $j>n$ with $u^+(j)<u^+(n)$ is at most $s-1-n$, which is less than $s-n=\#\{j>n:u(j)<u(n)\}$,
so every sequence of steps terminates. Setting $x_n=0$ in the transition formula and iterating gives
$\sch_w(x_1,\dots,x_{n-1},0)=\sum\sch_{w'}$ over the end points $w'$ of all step sequences (a priori with multiplicity).
Comparing with \eqref{E:Tn} and using that Schubert polynomials are linearly independent gives the claim.
\end{proof}

\subsubsection*{The map $\zeromap$}

\begin{definition}\label{definition:zeromap}
Let $R$ be a bounded RC graph of height $n\geq 1$ whose row $n$ is empty, and let $w=\wof R$. If $\maxd(w)<n$, let $\zeromap(R)$ be $R$ with its last row deleted, a bounded RC graph of height $n-1$. If $\maxd(w)=n$, let $s=\max\{j>n: w(j)<w(n)\}$ and let $c\in R$ be the unique crossing with $\rtt_R(c)=(n,s)$. Let $R^{(1)}$ be the RC graph obtained from $R$ by the Little bump at $c$, i.e.\ the inverse Monk move, and define $\zeromap(R)=\zeromap(R^{(1)})$, where $R^{(1)}$ is regarded as a bounded RC graph of height $n$.
\end{definition}

\begin{lemma}\label{lemma:zeroweight}
Definition~\ref{definition:zeromap} is well posed: for $R$ as in the definition with $\maxd(w)=n$, the crossing $c$ exists and is unique, $R\setminus\{c\}$ is reduced, $R^{(1)}$ is a bounded RC graph of height $n$ with row $n$ empty, $\wtt(R^{(1)})=\wtt(R)$, and $\wof{R^{(1)}}=w\,t_{ns}t_{in}$ is a transition step at $n$ for some $i<n$. The recursion terminates, and $\zeromap(R)$ is a bounded RC graph of height $n-1$ with $\wtt(\zeromap R)=\wtt(R)$.
\end{lemma}
\begin{proof}
Since $\maxd(w)=n$ we have $w(n+1)<w(n+2)<\cdots$, so $(n,s)$ is an inversion of $w$ and $w(k)<w(s)$ for $n<k<s$; hence $wt_{ns}\lessdot w$. The right roots of the positions of a reduced word are its inversions, each occurring once, which gives $c$, and $R\setminus\{c\}\in\RC(wt_{ns})$ is reduced. Put $\pi=wt_{ns}$. Then $R\in\RC(\pi t_{ns})$ lies on the right side of the Monk bijection \eqref{eq:monk} for $(\pi,n)$. As row $n$ of $R$ is empty, $R$ is not of the form $x_n\!\rightsquigarrow D$, so $R=m_{i,n}(R^{(1)})$ for a unique $i<n$ with $\pi t_{in}\gtrdot\pi$ and a unique $R^{(1)}\in\RC(\pi t_{in})$, namely the inverse Monk move, which has the same number of crossings as $R$ in every row; in particular no crossing is removed, so it is the Little bump at $c$. By Lemma~\ref{L:leaves}, $\maxd(wt_{ns}t_{in})\leq n$, so $R^{(1)}$ is a bounded RC graph of height $n$ with row $n$ empty, and the number of positions $j>n$ with $u(j)<u(n)$ strictly decreases along the recursion, which therefore terminates at a permutation with $\maxd<n$.
\end{proof}

\begin{lemma} \label{lemma:zerotransition}
Let $R$ be a bounded RC graph with $\hr(R)=n\geq 1$ such that row $n$ is empty. Then $\wof{R}\downvar{n}\wof{\zeromap R}$ and $\ell(\wof{\zeromap R})=\ell(\wof R)$, i.e.\ $\wof{\zeromap R}\in T_n(\wof R)$.
\end{lemma}
\begin{proof}
By Lemma~\ref{lemma:zeroweight}, $\wof{\zeromap R}$ is reached from $\wof R$ by a sequence of transition steps at $n$ and has $\maxd<n$; by Lemma~\ref{L:leaves} it lies in $T_n(\wof R)$.
\end{proof}

\begin{proposition}[$\zeromap$ erases the last row]\label{proposition:zerotrunc}
Let $R$ be a bounded RC graph of height $n$ whose last row is empty. Then $\Phi(\zeromap R)$ and $\Phi(R)$ agree in rows $1,\ldots,n-1$. Consequently $\zeromap(R)$ is the unique bounded RC graph $S$ of height $n-1$ such that $\Phi(S)$ agrees with $\Phi(R)$ in rows $1,\ldots,n-1$.
\end{proposition}
\begin{proof}
Each step $R\mapsto R^{(1)}$ of Definition~\ref{definition:zeromap} is a Little bump at a crossing with right root $(n,s)$ whose removal leaves a reduced word, so by Lemma~\ref{lemma:bumpconfine} it does not change rows $1,\ldots,n-1$ of $\Phi$; deleting the empty last row does not change $\Phi$ at all. For the second statement, if $S$ is a bounded RC graph of height $n-1$ then $\Phi(S)$ has no blank tiles in rows $>n-1$ and $\maxd(\wof S)\leq n-1$, so $\Phi(S)$ is determined by its rows $1,\ldots,n-1$ by Lemma~\ref{lemma:bpdcompletion}, and $S$ is determined by $\Phi(S)$.
\end{proof}

\begin{figure}[htbp]
\centering
\begin{minipage}{0.45\textwidth}
\centering
\begin{tikzpicture}[scale=0.7]
  \draw[lightgray, very thin] (0,0) -- (4,0);
  \draw[lightgray, very thin] (0,0) -- (0,4);
  \draw[lightgray, very thin] (0,1) -- (4,1);
  \draw[lightgray, very thin] (1,0) -- (1,4);
  \draw[lightgray, very thin] (0,2) -- (4,2);
  \draw[lightgray, very thin] (2,0) -- (2,4);
  \draw[lightgray, very thin] (0,3) -- (4,3);
  \draw[lightgray, very thin] (3,0) -- (3,4);
  \draw[lightgray, very thin] (0,4) -- (4,4);
  \draw[lightgray, very thin] (4,0) -- (4,4);
  \fill[gray!20] (0.1,3.1) rectangle (0.9,3.9);
  \fill[gray!20] (1.1,3.1) rectangle (1.9,3.9);
  \draw[blue, line width=1.0pt] (2.5,3) .. controls (2.5,3.3) and (2.7,3.5) .. (3,3.5);
  \draw[blue, line width=1.0pt] (3,3.5) -- (4,3.5);
  \fill[gray!20] (0.1,2.1) rectangle (0.9,2.9);
  \draw[blue, line width=1.0pt] (1.5,2) .. controls (1.5,2.3) and (1.7,2.5) .. (2,2.5);
  \draw[blue, line width=1.0pt] (2,2.5) .. controls (2.3,2.5) and (2.5,2.7) .. (2.5,3);
  \draw[blue, line width=1.0pt] (3.5,2) .. controls (3.5,2.3) and (3.7,2.5) .. (4,2.5);
  \draw[blue, line width=1.0pt] (0.5,1) .. controls (0.5,1.3) and (0.7,1.5) .. (1,1.5);
  \draw[blue, line width=1.0pt] (1,1.5) .. controls (1.3,1.5) and (1.5,1.7) .. (1.5,2);
  \draw[blue, line width=1.0pt] (2.5,1) .. controls (2.5,1.3) and (2.7,1.5) .. (3,1.5);
  \draw[blue, line width=1.0pt] (3,1.5) -- (4,1.5);
  \draw[blue, line width=1.0pt] (3.5,1) -- (3.5,2);
  \draw[blue, line width=1.0pt] (0.5,0) -- (0.5,1);
  \draw[blue, line width=1.0pt] (1.5,0) .. controls (1.5,0.3) and (1.7,0.5) .. (2,0.5);
  \draw[blue, line width=1.0pt] (2,0.5) -- (3,0.5);
  \draw[blue, line width=1.0pt] (2.5,0) -- (2.5,1);
  \draw[blue, line width=1.0pt] (3,0.5) -- (4,0.5);
  \draw[blue, line width=1.0pt] (3.5,0) -- (3.5,1);
  \node[font=\scriptsize] at (0.5,-0.35) {1};
  \node[font=\scriptsize] at (1.5,-0.35) {2};
  \node[font=\scriptsize] at (2.5,-0.35) {3};
  \node[font=\scriptsize] at (3.5,-0.35) {4};
  \node[font=\scriptsize, anchor=west] at (4.15,3.5) {1};
  \node[font=\scriptsize, anchor=west] at (4.15,2.5) {4};
  \node[font=\scriptsize, anchor=west] at (4.15,1.5) {3};
  \node[font=\scriptsize, anchor=west] at (4.15,0.5) {2};
  \draw[black, line width=1pt] (0,0) rectangle (4,4);
\end{tikzpicture}
\caption*{(a) $\Phi(R)$, $\wof R=1432$}
\end{minipage}
\hfill
\begin{minipage}{0.45\textwidth}
\centering
\begin{tikzpicture}[scale=0.7]
  \draw[lightgray, very thin] (0,0) -- (4,0);
  \draw[lightgray, very thin] (0,0) -- (0,4);
  \draw[lightgray, very thin] (0,1) -- (4,1);
  \draw[lightgray, very thin] (1,0) -- (1,4);
  \draw[lightgray, very thin] (0,2) -- (4,2);
  \draw[lightgray, very thin] (2,0) -- (2,4);
  \draw[lightgray, very thin] (0,3) -- (4,3);
  \draw[lightgray, very thin] (3,0) -- (3,4);
  \draw[lightgray, very thin] (0,4) -- (4,4);
  \draw[lightgray, very thin] (4,0) -- (4,4);
  \fill[gray!20] (0.1,3.1) rectangle (0.9,3.9);
  \fill[gray!20] (1.1,3.1) rectangle (1.9,3.9);
  \draw[blue, line width=1.0pt] (2.5,3) .. controls (2.5,3.3) and (2.7,3.5) .. (3,3.5);
  \draw[blue, line width=1.0pt] (3,3.5) -- (4,3.5);
  \fill[gray!20] (0.1,2.1) rectangle (0.9,2.9);
  \draw[blue, line width=1.0pt] (1.5,2) .. controls (1.5,2.3) and (1.7,2.5) .. (2,2.5);
  \draw[blue, line width=1.0pt] (2,2.5) .. controls (2.3,2.5) and (2.5,2.7) .. (2.5,3);
  \draw[blue, line width=1.0pt] (3.5,2) .. controls (3.5,2.3) and (3.7,2.5) .. (4,2.5);
  \draw[blue, line width=1.0pt] (0.5,1) .. controls (0.5,1.3) and (0.7,1.5) .. (1,1.5);
  \draw[blue, line width=1.0pt] (1,1.5) -- (2,1.5);
  \draw[blue, line width=1.0pt] (1.5,1) -- (1.5,2);
  \draw[blue, line width=1.0pt] (2,1.5) -- (3,1.5);
  \draw[blue, line width=1.0pt] (3,1.5) -- (4,1.5);
  \draw[blue, line width=1.0pt] (3.5,1) -- (3.5,2);
  \draw[blue, line width=1.0pt] (0.5,0) -- (0.5,1);
  \draw[blue, line width=1.0pt] (1.5,0) -- (1.5,1);
  \draw[blue, line width=1.0pt] (2.5,0) .. controls (2.5,0.3) and (2.7,0.5) .. (3,0.5);
  \draw[blue, line width=1.0pt] (3,0.5) -- (4,0.5);
  \draw[blue, line width=1.0pt] (3.5,0) -- (3.5,1);
  \node[font=\scriptsize] at (0.5,-0.35) {1};
  \node[font=\scriptsize] at (1.5,-0.35) {2};
  \node[font=\scriptsize] at (2.5,-0.35) {3};
  \node[font=\scriptsize] at (3.5,-0.35) {4};
  \node[font=\scriptsize, anchor=west] at (4.15,3.5) {2};
  \node[font=\scriptsize, anchor=west] at (4.15,2.5) {4};
  \node[font=\scriptsize, anchor=west] at (4.15,1.5) {1};
  \node[font=\scriptsize, anchor=west] at (4.15,0.5) {3};
  \draw[black, line width=1pt] (0,0) rectangle (4,4);
\end{tikzpicture}
\caption*{(b) $\Phi(\zeromap R)$, $\wof{\zeromap R}=2413$}
\end{minipage}
\caption{The bumpless pipe dreams of the RC graph $R=\{(1,2),(1,3),(2,2)\}$ of Figure~\ref{fig:zeroing} and of $\zeromap(R)$, drawn on a $4\times4$ grid; blank tiles are shaded and the pipe leaving through each row is labeled at the right. Row $i$ has as many blank tiles as row $i$ of the RC graph has crossings. The two diagrams agree in rows $1$ and $2$ (Proposition~\ref{proposition:zerotrunc}); rows $3$ and $4$ of $\Phi(\zeromap R)$ contain no blank tile and no $j$-tile, so every pipe there travels straight north and then straight east (Lemma~\ref{lemma:bpdcompletion}).}
\label{fig:bpdzero}
\end{figure}

We also record the behavior of $\Phi$ under deleting the first row. Gao and Huang define $\Phi^{-1}$ by iterating an operation $\nabla$ on bumpless pipe dreams \cite[Definitions~3.1 and~3.3]{GaoHuang2023}: if $B=\Phi(R)$ and $x$ is the first crossing of $R$ in reading order (the rightmost crossing of the first nonempty row $r$, with letter $a$), then $\nabla B=\Phi(R\setminus\{x\})$, and the pair $(a,r)$ is read off from the computation of $\nabla B$ as described below; $\Phi(R)$ is the unique diagram whose successive pairs are the reading word and row sequence of $R$. The operation $\nabla B$ is computed as follows \cite[Definition~3.1, Figures~2 and~3]{GaoHuang2023}. A mark is placed on the rightmost blank tile of the first row of $B$ containing a blank tile, and a sequence of \emph{column moves} follows, each of which modifies only a rectangle of width two whose northwest corner $(x,y)$ is the current position of the mark, moved first to the rightmost blank tile of its block of consecutive blank tiles. The tile $(x,y+1)$ carries a pipe $p$ traveling north. If $p$ entered column $y+1$ at a $j$-tile $(x',y+1)$, $x'>x$, the rectangle has southeast corner $(x',y+1)$, and afterwards the mark is at $(x',y+1)$; otherwise $p$ is pipe $y+1$, the rectangle has southeast corner $(x',y+1)$ where $(x',y)$ is the $r$-tile at which pipe $y$ turns east, this is the last column move, and $a=y$. In both cases, in each row $z$ with $x<z<x'$ the pair of tiles $(z,y),(z,y+1)$ is replaced by a pair depending only on the tile $(z,y)$: a blank, horizontal, $r$- or $j$-tile at $(z,y)$ is replaced by a vertical, crossing, vertical, or crossing tile, respectively, and $(z,y+1)$ by a blank, horizontal, $r$- or $j$-tile, respectively, while a vertical or crossing tile at $(z,y)$ leaves the pair unchanged. The corner $(x,y)$ becomes an $r$-tile and $(x,y+1)$ becomes a horizontal or $j$-tile according as it was an $r$-tile or vertical. In the first case $(x',y+1)$ becomes blank and $(x',y)$ becomes vertical or a $j$-tile according as it was an $r$-tile or horizontal; in the second case $(x',y)$ becomes vertical and $(x',y+1)$ an $r$-tile, which is the same replacement as the row rule applied to the $r$-tile $(x',y)$.

\begin{lemma}[Locality of $\nabla$]\label{lemma:poplocal}
Let $B$ be a bumpless pipe dream with at least one blank tile and let $m\geq 1$. Then rows $1,\ldots,m$ of $\nabla B$ are determined by rows $1,\ldots,m$ of $B$. Moreover, if $B'$ is obtained from $B$ by adding a new first row and a new first column, with pipe $1$ traveling north along the new column and east along the new row and all other tiles shifted by $(1,1)$, then $\nabla B'$ is obtained from $\nabla B$ in the same way, and the pair $(a,r)$ read off from $B'$ is the pair read off from $B$ plus $(1,1)$.
\end{lemma}
\begin{proof}
If rows $1,\ldots,m$ of $B$ contain no blank tile, the mark starts below row $m$ and every rectangle lies in rows $>m$. Otherwise the initial position of the mark is determined by $B_{\leq m}$. Consider a column move with northwest corner $(x,y)$, $x\leq m$. Scan column $y+1$ downward from row $x+1$ through rows $\leq m$. If a $j$-tile is met, in row $x'\leq m$, we are in the first case with this $x'$ (in the second case column $y+1$ carries pipe $y+1$ straight north), and the rectangle, its modification and the new position of the mark are determined by $B_{\leq m}$. If only vertical and crossing tiles are met, then in rows $x+1,\ldots,m$ the two columns are modified by the row rule: this is so in the first case, where $x'>m$; in the second case with $x'>m$; and in the second case with $x'\leq m$, because the corner replacement at row $x'$ agrees with the row rule and rows $x'+1,\ldots,m$ of column $y$ carry pipe $y$ vertically, so the row rule leaves them unchanged. In all these cases the mark is afterwards below row $m$ or the computation ends, so rows $1,\ldots,m$ are not changed again. For the second statement, the new row and column contain no blank tile and are never entered by a rectangle, since rectangles extend south and east from the mark, so every column move in $B'$ is the translate of the corresponding column move in $B$, and the last rectangle's left column index and the first blank row both increase by one.
\end{proof}

\begin{lemma} \label{lemma:ztrim}
Let $R$ be a bounded RC graph of height $n\geq 2$ such that row $n$ is empty. Then
$$\zeromap(\trm(R)) = \trm(\zeromap(R))\text{.}$$
\end{lemma}
\begin{proof}
Put $S=\zeromap(R)$, of height $n-1$, and $T=\trm R$, of height $n-1$ with empty last row (Lemma~\ref{L:trim}). Both $\trm S$ and $\zeromap T$ are bounded RC graphs of height $n-2$ (Lemma~\ref{L:trim} and Lemma~\ref{lemma:zeroweight}), so by Proposition~\ref{proposition:zerotrunc} it suffices to show that $\Phi(\trm S)$ and $\Phi(T)$ agree in rows $1,\ldots,n-2$.

Let $X$ be a bounded RC graph of height $n$ with $k$ crossings in row $1$. The first $k$ crossings of $X$ in reading order are those of row $1$, so $\nabla^k\Phi(X)=\Phi(X\setminus\text{row }1)=\Phi({\uparrow}\trm X)$. The reading word and row sequence of ${\uparrow}\trm X$ are those of $\trm X$ with every entry increased by one, so by the second part of Lemma~\ref{lemma:poplocal}, iterated, $\Phi({\uparrow}\trm X)$ is obtained from $\Phi(\trm X)$ by adding a first row and a first column. Hence rows $1,\ldots,n-2$ of $\Phi(\trm X)$ are determined by rows $2,\ldots,n-1$ of $\nabla^k\Phi(X)$, and these are determined by rows $1,\ldots,n-1$ of $\Phi(X)$ by the first part of Lemma~\ref{lemma:poplocal}, applied $k$ times.

Apply this to $X=R$ and to $X=S$, regarded as a bounded RC graph of height $n$ with empty last row; both have $k=\wtt_1(R)$ crossings in row $1$ by Lemma~\ref{lemma:zeroweight}, and $\Phi(S)$ agrees with $\Phi(R)$ in rows $1,\ldots,n-1$ by Proposition~\ref{proposition:zerotrunc}. Therefore $\Phi(\trm S)$ and $\Phi(T)$ agree in rows $1,\ldots,n-2$.
\end{proof}

\begin{theorem} \label{theorem:zero_bijection}
Let $w$ be a permutation with last descent at most $n$. Let $\mathcal{RC}(w,n)^0$ be the set of bounded RC graphs $R$ with $\hr(R)=n$ such that $\wof{R}=w$ and row $n$ is empty. Then
$$\zeromap:\mathcal{RC}(w, n)^{0}\to \bigcup_{\substack{w\downvar{n} w'\\\ell(w')=\ell(w)}}\mathcal{RC}(w', n-1)$$
is a weight-preserving bijection.
\end{theorem}

The proof occupies the rest of this subsection.

\begin{lemma}\label{L:disjoint}
Let $\maxd(w)\le n$ and let $w'\neq w''$ be two elements of $T_n(w)$. Then $P_1(w')\cap P_1(w'')=\varnothing$.
\end{lemma}
\begin{proof}
\emph{Step 1: what $w'$ and $w''$ look like.}
By Lemma~\ref{L:leaves}, $w'$ and $w''$ are end points of two sequences of transition steps starting at $w$.
The choice of $s$ in a step is forced, so the two sequences agree up to some permutation $u$ and then make different
choices $i\neq i'$; say $i<i'<n$. Put $\tilde v=u\,t_{ns}$, $\eta=\tilde v(n)$, $z=\tilde v(i)=u(i)$, $z'=\tilde v(i')=u(i')$.
Both $t_{in}$ and $t_{i'n}$ are covers of $\tilde v$, so $z<\eta$ and $z'<\eta$; and since $i<i'<n$, the cover criterion for $t_{in}$
forbids $z<z'<\eta$. Hence
\[
z'<z<\eta .
\]
Along the first sequence, right after the branching step position $i$ holds $\eta$, position $i'$ holds $z'$, and position $n$ holds $z$.
By (O1) and (O2), a position $l<n$ is changed at a later step only if its current value is smaller than the current value in position $n$
(which is $\le z$ from now on), and its new value is then also smaller than $z$. Therefore
\begin{equation}\label{E:first}
w'(i)=\eta,\qquad w'(i')<z .
\end{equation}
Along the second sequence, right after the branching step position $i'$ holds $\eta$, position $i$ holds $z$, and position $n$ holds $z'<z$.
By the same reasoning neither position is ever changed again:
\begin{equation}\label{E:second}
w''(i)=z,\qquad w''(i')=\eta .
\end{equation}

\emph{Step 2: what a $1$-chain looks like.}
Fix $M$ with $w',w''\in S_M$ (the relation $\downvar{1}$ does not depend on this choice).
Suppose $y\in S_M$ and $y\downvar{1}u'$, and put $Y=\varphi_{1,M}(u')$. By Definitions~\ref{definition:pieri} and~\ref{definition:polysequence} there are
$y=Y_0,Y_1,\dots,Y_q=Y$ with $Y_t=Y_{t-1}t_{1b_t}$ and $\ell(Y_t)=\ell(Y_{t-1})+1$. Put $x_t=Y_t(1)$, so $x_0<x_1<\dots<x_q$.
A position that has been used holds a value smaller than the current value in position $1$, so it is never used again. Consequently:
\begin{enumerate}
\item[(F1)] for every $l\ge2$: either $Y(l)=y(l)$, or $l=b_t$ for a unique $t\ge1$, and then $y(l)=x_t$ and $Y(l)=x_{t-1}\ge x_0=y(1)$. In particular $Y(l)\le y(l)$.
\item[(F2)] if $l=b_t$, there is no $l_0$ with $1<l_0<l$ and $x_{t-1}<y(l_0)<x_t$.
(The cover criterion says this for $Y_{t-1}$ in place of $y$; and $Y_{t-1}(l_0)\neq y(l_0)$ only if $l_0=b_{t'}$ with $t'<t$,
in which case $y(l_0)=x_{t'}\le x_{t-1}$.)
\end{enumerate}

\emph{Step 3: the contradiction.}
Suppose $u'\in P_1(w')\cap P_1(w'')$ and let $Y=\varphi_{1,M}(u')$. Note $i'\ge2$.
Apply (F1) with $y=w'$: by \eqref{E:first}, $Y(i')\le w'(i')<z$.
Now apply (F1) with $y=w''$: by \eqref{E:second}, $w''(i')=\eta>z>Y(i')$, so $i'=b_t$ is a chain position for $w''$, with
$x_t=\eta$ and $x_{t-1}=Y(i')<z$.

If $i=1$, then $x_0=w''(1)=z$ by \eqref{E:second}, contradicting $x_{t-1}\ge x_0$.

If $i\ge2$, then $l_0=i$ satisfies $1<l_0<i'$ and $x_{t-1}<w''(l_0)=z<\eta=x_t$, contradicting (F2).
\end{proof}

\begin{lemma}\label{L:unique}
Let $n\ge 3$, $\maxd(w)\le n$, and let $u'$ be a permutation. Then there is at most one $v\in P_1(w)$ with $u'\in T_{n-1}(v)$.
\end{lemma}
\begin{proof}
By Lemma~\ref{L:trim} every $v\in P_1(w)$ has $\maxd(v)\le n-1$, so $T_{n-1}(v)$ makes sense.
Set $x_n=0$ in \eqref{E:P1} and expand each $\sch_v(x_2,\dots,x_{n-1},0)$ by \eqref{E:Tn}:
\[
\sch_w(x_1,\dots,x_{n-1},0)=\sum_{v\in P_1(w)}\ \sum_{u'\in T_{n-1}(v)}x_1^{\ell(w)-\ell(v)}\,\sch_{u'}(x_2,\dots,x_{n-1}).
\]
Alternatively, expand by \eqref{E:Tn} first and then apply \eqref{E:P1} to each $w'$ (which has $\maxd(w')\le n-1$):
\[
\sch_w(x_1,\dots,x_{n-1},0)=\sum_{w'\in T_n(w)}\ \sum_{u'\in P_1(w')}x_1^{\ell(w')-\ell(u')}\,\sch_{u'}(x_2,\dots,x_{n-1}).
\]
The polynomials $x_1^{a}\sch_{u'}(x_2,\dots,x_{n-1})$ ($a\ge0$, $\maxd(u')\le n-2$) are linearly independent. Comparing coefficients and summing over $a$,
\[
\#\{v\in P_1(w): u'\in T_{n-1}(v)\}=\#\{w'\in T_n(w): u'\in P_1(w')\},
\]
and the right side is at most $1$ by Lemma~\ref{L:disjoint}.
\end{proof}

\begin{lemma} \label{lemma:zero_injective}
Let $w$ be a permutation with last descent at most $n$. Let $\mathcal{RC}(w,n)^0$ be the set of bounded RC graphs $R$ with $\hr(R) = n$ such that $\wof{R}=w$ and row $n$ is empty. Then
$$\zeromap:\mathcal{RC}(w, n)^{0}\to \mathcal{RC}_{n-1}$$
is injective.
\end{lemma}
\begin{proof}
Induction on $n$. For $n=2$: an element of $\RC(w,2)^0$ is a single row, and a single row is determined by its permutation
(Theorem~\ref{theorem:modactrc} applied to the empty graph), so $\RC(w,2)^0$ has at most one element.

Let $n\ge3$ and assume the statement for $n-1$ \emph{and every permutation}. Let $R,R'\in\RC(w,n)^0$ with $\zeromap(R)=\zeromap(R')=:S$.
Put $T=\trm R$, $T'=\trm R'$, $v=\wof{T}$, $v'=\wof{T'}$.
\begin{enumerate}
\item By Lemma~\ref{L:trim}, $T\in\RC(v,n-1)^0$ and $T'\in\RC(v',n-1)^0$ with $v,v'\in P_1(w)$. So $\zeromap(T)$, $\zeromap(T')$ are defined and
Lemma~\ref{lemma:ztrim} applies: $\zeromap(T)=\trm S=\zeromap(T')$.
\item Let $u'=\wof{\trm S}$. By Lemma~\ref{lemma:zerotransition}, $v\downvar{n-1}u'$ and $v'\downvar{n-1}u'$; by Lemma~\ref{lemma:zeroweight}, $\wtt(\zeromap(T))=\wtt(T)$ and $\wtt(\zeromap(T'))=\wtt(T')$,
so $\ell(u')=\ell(v)=\ell(v')$. Thus $u'\in T_{n-1}(v)\cap T_{n-1}(v')$.
\item By Lemma~\ref{L:unique}, $v=v'$.
\item Now $T,T'\in\RC(v,n-1)^0$ with $\zeromap(T)=\zeromap(T')$, so $T=T'$ by the inductive hypothesis.
\item By Theorem~\ref{theorem:modactrc}, $R=\mathrm{row}_1(Q_1(v,w))\cup{\uparrow}T=R'$.\qedhere
\end{enumerate}
\end{proof}

\begin{proof}[Proof of Theorem \ref{theorem:zero_bijection}]
By the previous lemma, it suffices to show that $\zeromap$ is surjective. We have by the transition formula that if $\maxd(w)\leq n$, then
$$\sch_w(x_1,\ldots,x_{n-1},0) = \sum_{\substack{w\downvar{n} w'\\\ell(w)=\ell(w')}}\sch_{w'}(x_1,\ldots,x_{n-1})$$
Note that this sum is without multiplicity, by Lemma~\ref{L:leaves}. The right hand side is equal to
$$\sum_{\substack{w\downvar{n} w'\\\ell(w)=\ell(w')}}\sum_{R'\in \RC_{n-1}(w')}x^{\wtt(R')}$$
and the left hand side is equal to
$$\sum_{R\in \mathcal{RC}(w, n)^0}x^{\wtt(R)}$$
since the only RC graphs that contribute to the left hand side are those with row $n$ empty. This implies that
$$|\RC(w,n)^0| = \sum_{\substack{w\downvar{n} w'\\\ell(w)=\ell(w')}}|\RC_{n-1}(w')|$$
Now, the set 
$$\{\zeromap(R) \mid R\in \mathcal{RC}(w,n)^0\}$$
has the same cardinality as $\mathcal{RC}(w,n)^0$ by injectivity (Lemma \ref{lemma:zero_injective}), and is a subset of $\bigcup_{w\downvar{n} w',\ \ell(w')=\ell(w)}\mathcal{RC}_{n-1}(w')$ by Lemmas~\ref{lemma:zerotransition} and~\ref{lemma:zeroweight}. Since the right hand side has the same cardinality as $\mathcal{RC}(w,n)^0$, we conclude that $\zeromap$ is surjective, as desired.
\end{proof}

We may extend $\zeromap$ to an endomorphism $\BRC\to\BRC$ by defining $\zeromap(R) = 0$ if the last row of $R$ is not empty. Then we have the following result.

\begin{corollary}[Combinatorial Lascoux-Sc\"utzenberger transition formula]
	Let $w$ be a permutation with last descent at most $n$. Then
	$$\zeromap(\mathcal{S}_w(n))=\sum_{\substack{w\downvar{n} w'\\\ell(w)=\ell(w')}}\mathcal{S}_{w'}(n-1)$$
\end{corollary}

\subsection{Definition of the ring product}\label{section:ringproduct}

We first make the maps $\trm^p$ and $\clp^p$ of the introduction precise. Throughout, $\zeromap^k$ denotes the $k$-fold iterate of $\zeromap$, extended to $\BRC$ as above by declaring $\zeromap(R)=0$ when the last row of $R$ is nonempty. For $a\geq 0$ write $\mathtt{row}(a)=\{(1,1),\ldots,(1,a)\}$ for the bounded RC graph of height $1$ with reading word $s_a s_{a-1}\cdots s_1$.

\begin{lemma}[The last row is determined by its weight]\label{lemma:lastrow}
Let $R$ be a bounded RC graph with $\hr(R)=n\geq 1$ and let $a=\wtt_n(R)$. Then row $n$ of $R$ is $\{(n,1),\ldots,(n,a)\}$, i.e.\ its word is $s_{n+a-1}s_{n+a-2}\cdots s_n$. In particular $\mathtt{row}(a)$ is the unique bounded RC graph of height $1$ with $a$ crossings, and $R$ is determined by $R_{\leq n-1}$ together with $\wtt_n(R)$.
\end{lemma}
\begin{proof}
By Lemma~\ref{L:trim} applied $n-1$ times, $\trm^{n-1}(R)$ is a bounded RC graph of height $1$; its underlying set is $\{(1,j):(n,j)\in R\}$, so it suffices to treat $n=1$. Let $R$ have height $1$ with crossings in columns $j_1>j_2>\cdots>j_a$. Its reading word $s_{j_1}s_{j_2}\cdots s_{j_a}$ is reduced, and $w=\wof{R}$ satisfies $\maxd(w)\leq 1$, so $w(2)<w(3)<\cdots$ and $w$ has code $(a,0,0,\ldots)$, i.e.\ $w=s_a s_{a-1}\cdots s_1$. This permutation is a cycle with a unique reduced word, so $j_k=a+1-k$ for all $k$.
\end{proof}

\begin{definition}\label{definition:cliptrim}
Let $R$ be a bounded RC graph with $\hr(R)=n$ and let $0\leq p\leq n$. For $0\leq k\leq n$ write $R_{\leq k}=\{(i,j)\in R: i\leq k\}$.
\begin{enumerate}
\item $\trm^p(R)=\trm(\trm(\cdots\trm(R)))$ ($p$ times) is the bounded RC graph of height $n-p$ with underlying set $\{(i-p,j):(i,j)\in R,\ i>p\}$. (Each single trim is again a bounded RC graph by Lemma~\ref{L:trim}.)
\item Let $h=\max\bigl(p,\maxd(\wof{R_{\leq p}})\bigr)$ and regard $R_{\leq p}$ as a bounded RC graph of height $h$; its rows $p+1,\ldots,h$ are empty. Then
$$\clp^p(R)=\zeromap^{\,h-p}(R_{\leq p})\in\BRC_p\text{.}$$
\end{enumerate}
\end{definition}

For $\clp^p(R)$ to be a single bounded RC graph (rather than $0$) the last row must be empty at every intermediate step. This holds because $\zeromap$ preserves weight vectors (Lemma~\ref{lemma:zeroweight}): if $R'$ has height $h'$ and rows $p+1,\ldots,h'$ empty, then $\zeromap(R')$ has height $h'-1$ and the same weight, so its rows $p+1,\ldots,h'-1$ are empty. Consequently $\clp^p(R)$ is a bounded RC graph of height $p$ with $\wtt(\clp^p(R))=(\wtt_1(R),\ldots,\wtt_p(R))$; it depends only on $R_{\leq p}$, and $\clp^p(R)=\zeromap^{\,N-p}(R_{\leq p})$ for every $N\geq h$, since $\zeromap$ deletes the empty last row whenever the height exceeds $\maxd$. In particular $\clp^p(R)=R_{\leq p}$ whenever $\maxd(\wof{R_{\leq p}})\leq p$; $\clp^0(R)=\emptyset$ (the empty RC graph of height $0$); and $\clp^1(R)=\mathtt{row}(\wtt_1(R))$ by Lemma~\ref{lemma:lastrow}. We use the same formula to define $\clp^p(X)$ for any bounded RC graph $X$ of height $\geq p$, so that $\clp^p(X)=\clp^p(X_{\leq k})$ for all $k\geq p$.

\begin{definition}\label{definition:ringproduct}
Suppose we have two bounded RC graphs $R_1$ with $\hr(R_1)=m$ and $R_2$ with $\hr(R_2)=n$. The product of these is defined by
$$R_1\sqcupplus R_2 = \sum_{\substack{\clp^m(R')=R_1\\\trm^m(R')=R_2}}R'$$  
where $R'$ ranges over bounded RC graphs with $\hr(R')=m+n$.
\end{definition}

Associativity of $\sqcupplus$ rests on two structural identities relating $\clp^p$ and $\trm^p$: a commutativity of $\clp$ with $\trm$, and a functoriality of $\clp$ under iteration.

\begin{lemma}[Commutativity of $\clp$ and $\trm$] \label{lemma:clptrm}
Let $R$ be a bounded RC graph with $\hr(R)=p+q+r$. Then
$$\trm^p(\clp^{p+q}(R))=\clp^q(\trm^p(R)).$$
\end{lemma}
\begin{proof}
Since $\trm^p$ is a composite of single trims, it suffices to show $\trm(\clp^{k}(X))=\clp^{k-1}(\trm X)$ for $X$ of height $\geq k\geq 1$. Let $h=\max(k,\maxd(\wof{X_{\leq k}}))$. Iterating Lemma~\ref{lemma:ztrim} (the last row stays empty at each step by weight preservation),
$$\trm(\clp^{k}(X))=\trm\bigl(\zeromap^{h-k}(X_{\leq k})\bigr)=\zeromap^{h-k}\bigl(\trm(X_{\leq k})\bigr)\text{,}$$
and $\trm(X_{\leq k})=(\trm X)_{\leq k-1}$ regarded at height $h-1$, which is bounded by Lemma~\ref{L:trim}. Hence $h-1\geq\max(k-1,\maxd(\wof{(\trm X)_{\leq k-1}}))$ and the right side is $\clp^{k-1}(\trm X)$ by the padding independence noted after Definition~\ref{definition:cliptrim}.
\end{proof}

\begin{lemma} \label{lemma:clpfunc}
Let $R$ be a bounded RC graph with $\hr(R)=n$, and let $0\leq p\leq p+q\leq n$. Then
$$\clp^p(\clp^{p+q}(R))=\clp^p(R).$$
\end{lemma}
\begin{proof}
For bounded RC graphs $X$ of height $\geq p$, Definition~\ref{definition:cliptrim} gives $\clp^p(X)=\zeromap^{\,h-p}(X_{\leq p})$ with $X_{\leq p}$ at height $h\geq p$, each intermediate RC graph having an empty last row. By Propositions~\ref{proposition:zerotrunc} and~\ref{proposition:toprows}, $\Phi(\clp^pX)$ agrees with $\Phi(X_{\leq p})$, and hence with $\Phi(X)$, in rows $1,\ldots,p$. Applying this to $X=\clp^{p+q}(R)$, then to $X=R$ with $p+q$ in place of $p$, and then to $X=R$, we find that $\Phi(\clp^p(\clp^{p+q}R))$ agrees with $\Phi(\clp^{p+q}R)$ in rows $1,\ldots,p$, that $\Phi(\clp^{p+q}R)$ agrees with $\Phi(R)$ in rows $1,\ldots,p+q$, and that $\Phi(R)$ agrees with $\Phi(\clp^pR)$ in rows $1,\ldots,p$. Hence $\Phi(\clp^p(\clp^{p+q}R))$ and $\Phi(\clp^pR)$ agree in rows $1,\ldots,p$. Both $\clp^p(\clp^{p+q}R)$ and $\clp^pR$ are bounded RC graphs of height $p$, so their images under $\Phi$ have no blank tiles in rows $>p$ and permutations with $\maxd\leq p$. By Lemma~\ref{lemma:bpdcompletion} these images are equal, and since $\Phi$ is injective, so are the two RC graphs.
\end{proof}

\begin{theorem}
	 The product $\sqcupplus$ is associative and $\BRC$ is a ring when equipped with $\sqcupplus$.
\end{theorem}
\begin{proof}
Let $R_1,R_2,R_3$ be bounded RC graphs of heights $p,q,r$ respectively, and let $R$ be a bounded RC graph with $\hr(R)=p+q+r$. We claim that $R$ appears as a term in $(R_1\sqcupplus R_2)\sqcupplus R_3$ if and only if it appears as a term in $R_1\sqcupplus(R_2\sqcupplus R_3)$, and in both cases with multiplicity $1$. Unwinding the definition of $\sqcupplus$ in each bracketing yields the following conditions on $R$:
\begin{center}
\begin{tabular}{l|l|l}
& Left bracketing $(R_1\sqcupplus R_2)\sqcupplus R_3$ & Right bracketing $R_1\sqcupplus(R_2\sqcupplus R_3)$ \\
\hline
(i) & $\clp^p(\clp^{p+q}(R))=R_1$ & $\clp^p(R)=R_1$ \\
(ii) & $\trm^p(\clp^{p+q}(R))=R_2$ & $\clp^q(\trm^p(R))=R_2$ \\
(iii) & $\trm^{p+q}(R)=R_3$ & $\trm^q(\trm^p(R))=R_3$ \\
\end{tabular}
\end{center}
Row (i) is equivalent on both sides by Lemma \ref{lemma:clpfunc}. Row (ii) is equivalent on both sides by Lemma \ref{lemma:clptrm}. Row (iii) is equivalent on both sides because $\trm^{p+q}=\trm^q\circ\trm^p$ by definition of $\trm$. Thus the term $R$ contributes to the same bracketed product on both sides. Since the membership conditions are equalities (not weighted by any combinatorial multiplicity) and $R$ is determined to lie in either product by these three conditions, the multiplicity is $1$ in each case. We conclude that $(R_1\sqcupplus R_2)\sqcupplus R_3=R_1\sqcupplus(R_2\sqcupplus R_3)$.

To tie up the claim that $\BRC$ is a ring, distributivity is directly imposed by the definition of the product as extending from the definition bilinearly, and the identity element is the empty RC graph of height $0$.
\end{proof}

\begin{example}
We compute the product of two bounded RC graphs $(R_1, 2)\in\BRC(2413,2)$ and $(R_2, 2)\in\BRC(231,2)$:
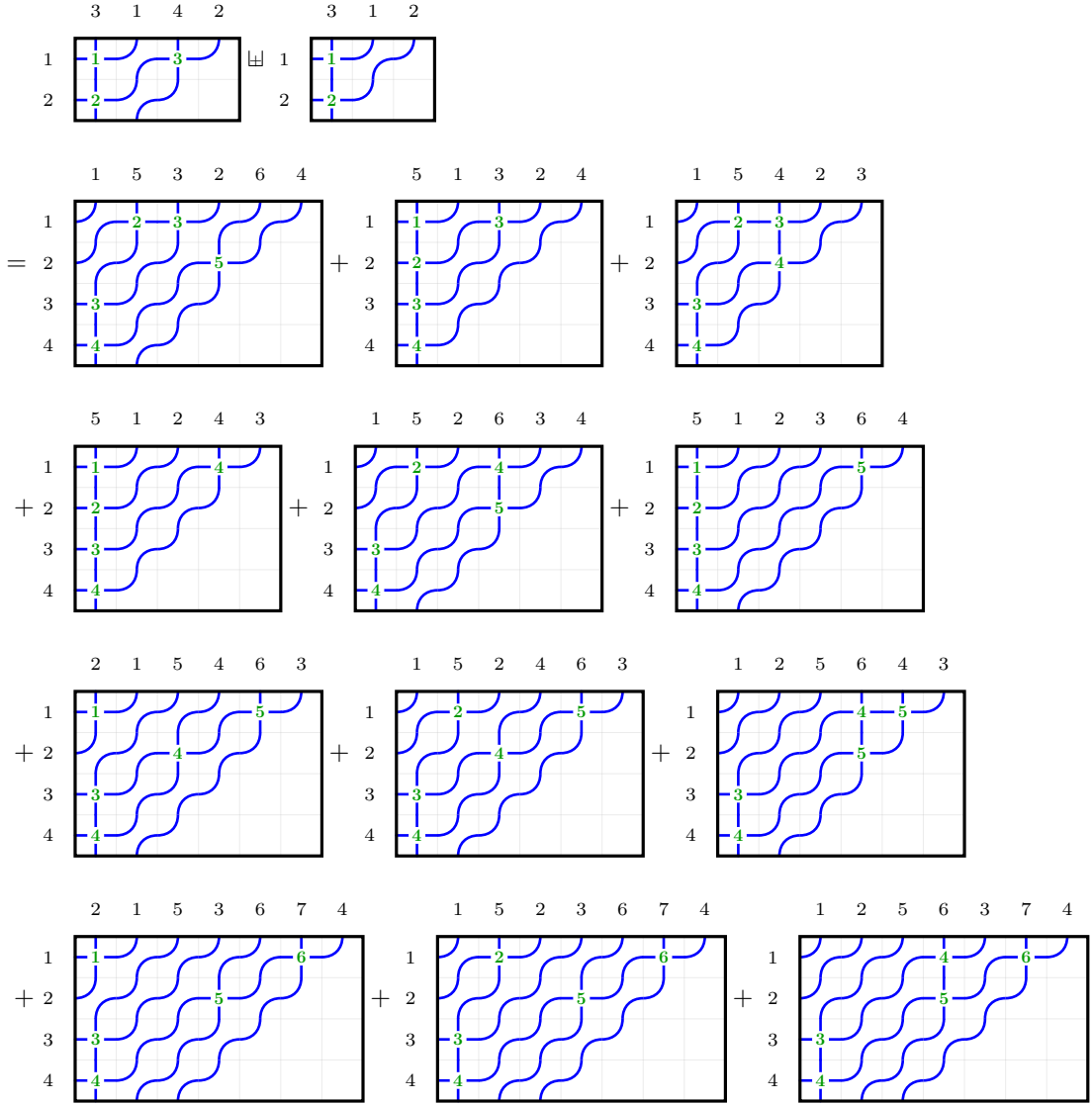
\begin{align*}
  &
  \vcenter{\hbox{
}}
\end{align*}

\end{example}

\subsection{Proof of the Littlewood-Richardson rule for the dual Schubert basis}

We prove Theorem~\ref{theorem:LR} by comparing two expansions of $\sch_w(x_1,\ldots,x_{p+q})$: the monomial expansion over $\RC_{p+q}(w)$ (Corollary~\ref{corollary:evschub}) and the branching expansion \eqref{eq:branching}. The bridge is the following description of the fibers of $\clp^p$.

For a permutation $u$ and $p\geq 0$, let $h=\max(p,\maxd(u))$ and let $\RC_{\leq p}(u)\subseteq\RC_h(u)$ be the set of RC graphs of $u$ all of whose crossings lie in rows $1,\ldots,p$. For $k\geq 1$ and a permutation $z$ with $\maxd(z)\leq k$ recall $T_k(z)=\{z':z\downvar{k}z',\ \ell(z')=\ell(z)\}$; by \eqref{E:Tn} and linear independence of Schubert polynomials, $T_k(z)=\{z\}$ when $\maxd(z)<k$. For $u'$ with $\maxd(u')\leq p$ let $N_p(u,u')$ be the number of chains
$$u=u_h,\quad u_{h-1}\in T_h(u_h),\quad u_{h-2}\in T_{h-1}(u_{h-1}),\quad\ldots,\quad u_p\in T_{p+1}(u_{p+1}),\qquad u_p=u'\text{,}$$
so that $N_p(u,u')=\delta_{u,u'}$ when $h=p$.

\begin{lemma}\label{lemma:clipfibers}
Let $u$, $p$, $h$ be as above. For every $u'$ with $\maxd(u')\leq p$ and every $U\in\RC_p(u')$,
$$\#\{T\in\RC_{\leq p}(u):\zeromap^{h-p}(T)=U\}=N_p(u,u')\text{.}$$
Consequently
$$\sum_{T\in\RC_{\leq p}(u)}x^{\wtt(T)}=\sch_u(x_1,\ldots,x_p,0,\ldots,0)=\sum_{u'}N_p(u,u')\,\sch_{u'}(x_1,\ldots,x_p)\text{.}$$
\end{lemma}
\begin{proof}
Induct on $h-p$; if $h=p$ then $\zeromap^0$ is the identity and $\RC_{\leq p}(u)=\RC_p(u)$. Suppose $h>p$. Every $T\in\RC_{\leq p}(u)$ has empty last row, so by Theorem~\ref{theorem:zero_bijection} and Lemma~\ref{lemma:zeroweight}, $\zeromap$ restricts to a weight-preserving injection from $\RC_{\leq p}(u)$ into $\bigsqcup_{u_1\in T_h(u)}\RC_{h-1}(u_1)$ (the length condition $\ell(u_1)=\ell(u)$ is forced by weight preservation). By weight preservation the image consists of graphs whose rows $p+1,\ldots,h-1$ are empty, and conversely every such graph $T'\in\RC_{h-1}(u_1)$ is $\zeromap(T)$ for a unique $T\in\RC_h(u)^0$ by Theorem~\ref{theorem:zero_bijection}, and this $T$ has rows $p+1,\ldots,h$ empty by weight preservation, so $T\in\RC_{\leq p}(u)$. Thus $\zeromap$ is a bijection from $\RC_{\leq p}(u)$ onto the set of $T'\in\bigsqcup_{u_1\in T_h(u)}\RC_{h-1}(u_1)$ with rows $p+1,\ldots,h-1$ empty. For such $T'$ with $\wof{T'}=u_1$, put $h_1=\max(p,\maxd(u_1))\leq h-1$; the first $h-1-h_1$ applications of $\zeromap$ to $T'$ delete empty last rows (Definition~\ref{definition:zeromap}), so $\zeromap^{h-1-p}(T')=\zeromap^{h_1-p}(T')$ with $T'$ regarded as an element of $\RC_{\leq p}(u_1)$. Hence
$$\#\{T\in\RC_{\leq p}(u):\zeromap^{h-p}(T)=U\}=\sum_{u_1\in T_h(u)}\#\{T'\in\RC_{\leq p}(u_1):\zeromap^{h_1-p}(T')=U\}=\sum_{u_1\in T_h(u)}N_p(u_1,u')=N_p(u,u')$$
by induction. For the second statement, Corollary~\ref{corollary:evschub} gives $\sum_{T\in\RC_h(u)}x^{\wtt(T)}=\sch_u(x_1,\ldots,x_h)$, and setting $x_{p+1}=\cdots=x_h=0$ kills exactly the terms with a crossing in rows $p+1,\ldots,h$; this is the first equality. For the second, sum the fiber count over $U\in\RC_p(u')$ and all $u'$, using $\wtt(\zeromap^{h-p}(T))=\wtt(T)$ and Corollary~\ref{corollary:evschub} for $\RC_p(u')$.
\end{proof}

\begin{definition}
Given RC graphs $U\in \RC_p(u)$ and $V\in\RC_q(v)$ and a permutation $w$, define
$$d_{U, V}^w = \#\{R\in \RC_{p+q}(w): \clp^p(R)=U,\trm^p(R)=V\}$$
\end{definition}

\begin{definition}\label{definition:omega}
Let $\Omega$ be the $\mathbb{Z}$-submodule of $\BRC$ generated by the RC graphs $R_1-R_2$ such that $\wof{R_1}=\wof{R_2}$ and $\hr(R_1)=\hr(R_2)$, and let $\mathcal{R}$ be the subring of $\BRC$ generated by the single-row RC graphs $\mathtt{row}(a)$, $a\geq 0$.
\end{definition}

\begin{proof}[Proof of Theorem \ref{theorem:LR}]
Fix $u,v\in S_\infty$ with $\maxd(u)\leq p$, $\maxd(v)\leq q$, and fix $U\in\RC_p(u)$, $V\in\RC_q(v)$. By Lemma~\ref{lemma:dualbranch}, $d_{u,v}^w(p,q)=b_{u,v}^w(p,q)$, so we must show $d_{U,V}^w=b_{u,v}^w(p,q)$ for every $w$ with $\maxd(w)\leq p+q$.

Let $R\in\RC_{p+q}(w)$ and write $R=T\cup\shup^pV'$ with $T=R_{\leq p}$ and $V'=\trm^p(R)\in\RC_q$ (Lemma~\ref{L:trim}). Reading rows from top to bottom, the word of $R$ is the word of $T$ followed by the word of $\shup^pV'$, so
$$w=\wof{T}\cdot\shup^p\wof{V'}\qquad\text{with}\qquad\ell(w)=\ell(\wof T)+\ell(\wof{V'})\text{.}$$
Thus $\wof{T}=u_{v'}:=w\,(\shup^pv')^{-1}$ where $v'=\wof{V'}$, and $\ell(u_{v'})=\ell(w)-\ell(v')$. Conversely, for any $v'$ with $\maxd(v')\leq q$ and $\ell(w(\shup^pv')^{-1})=\ell(w)-\ell(v')$, any $V'\in\RC_q(v')$, and any $T\in\RC_{\leq p}(u_{v'})$, the union $T\cup\shup^pV'$ is a reduced RC graph of $w$ with $p+q$ rows, hence lies in $\RC_{p+q}(w)$. Finally, by Definition~\ref{definition:cliptrim}, $\clp^p(T\cup\shup^pV')=\zeromap^{h-p}(T)$ with $h=\max(p,\maxd(u_{v'}))$.

Applying this with $V'=V$ and Lemma~\ref{lemma:clipfibers},
$$d_{U,V}^w=\#\{T\in\RC_{\leq p}(u_v):\zeromap^{h-p}(T)=U\}=N_p(u_v,u)$$
if $\ell(u_v)=\ell(w)-\ell(v)$, and $d_{U,V}^w=0$ otherwise. On the other hand, by Corollary~\ref{corollary:evschub}, the decomposition $R=T\cup\shup^pV'$, and Lemma~\ref{lemma:clipfibers},
\begin{align*}
\sch_w(x_1,\ldots,x_{p+q})&=\sum_{R\in\RC_{p+q}(w)}x^{\wtt(R)}
=\sum_{\substack{v':\ \maxd(v')\leq q\\ \ell(u_{v'})=\ell(w)-\ell(v')}}\ \sum_{V'\in\RC_q(v')}x_{p+1}^{\wtt_1(V')}\cdots x_{p+q}^{\wtt_q(V')}\sum_{T\in\RC_{\leq p}(u_{v'})}x^{\wtt(T)}\\
&=\sum_{\substack{v':\ \maxd(v')\leq q\\ \ell(u_{v'})=\ell(w)-\ell(v')}}\ \sum_{u'}N_p(u_{v'},u')\,\sch_{u'}(x_1,\ldots,x_p)\,\sch_{v'}(x_{p+1},\ldots,x_{p+q})\text{.}
\end{align*}
Comparing with \eqref{eq:branching} and using the linear independence of the products $\sch_{u'}(x_1,\ldots,x_p)\sch_{v'}(x_{p+1},\ldots,x_{p+q})$ gives $b_{u,v}^w(p,q)=N_p(u_v,u)$ when $\ell(u_v)=\ell(w)-\ell(v)$ and $b_{u,v}^w(p,q)=0$ otherwise. Hence $d_{U,V}^w=b_{u,v}^w(p,q)=d_{u,v}^w(p,q)$, which is independent of the choice of $U$ and $V$ and nonnegative.
\end{proof}

We now record the ring-theoretic consequences.

\begin{definition}
Recall the map $\alpha:\BRC\to\dcoma$, $\alpha(R)=\dsch_{\wof{R}}^{\hr(R)}$. (The weight map $\omega(R)=[\wtt(R)]$ is additive but not multiplicative: every term of $R_1\sqcupplus R_2$ has weight $\wtt(R_1)\wtt(R_2)$, so $\omega(R_1\sqcupplus R_2)$ is a multiple of $\omega(R_1)\sqcupplus\omega(R_2)$ by the number of terms.) Let $a$ be a sequence of elements of a ring indexed by the nonnegative integers. For a permutation $w$ and an integer $n\geq \maxd(w)$, define a function $\dsch_w^n(a)$ by
$$\dsch_w^n(a) = \sum_{\alpha} d_{\alpha, w}^n a_{\alpha_1} \cdots a_{\alpha_n}$$
where the sum is over all weak compositions $\alpha$ of length $n$, and $d_{\alpha, w}^n$ is the coefficient of $\alpha$ in $\dsch_w^n\in\dcoma$.
\end{definition}

\begin{lemma} \label{lemma:schubertrow}
We have that the elements
$$\dsch_w^n(\mathtt{row})$$
for integers $n\geq 0$ and $w$ that ranges over all permutations in $S_\infty$ with $n\geq \maxd(w)$ form a $\mathbb{Z}$-basis of $\mathcal{R}$, and $\alpha:\mathcal{R}\to \dcoma$ is an isomorphism of graded rings.
\end{lemma}
\begin{proof}
By Theorem~\ref{theorem:LR}, $\alpha$ is multiplicative: for $R_1\in\RC_p(u)$, $R_2\in\RC_q(v)$,
$$\alpha(R_1\sqcupplus R_2)=\sum_w d_{R_1,R_2}^w\,\dsch_w^{p+q}=\sum_w d_{u,v}^w(p,q)\,\dsch_w^{p+q}=\dsch_u^p\sqcupplus\dsch_v^q=\alpha(R_1)\sqcupplus\alpha(R_2)\text{.}$$
Restricting to $\mathcal{R}$ gives a graded ring map $\alpha|_{\mathcal{R}}:\mathcal{R}\to\dcoma$. Since $\alpha(\mathtt{row}(i))=\dsch^1_{s_i\cdots s_1}=[i]$, the image contains the free generators $[i]$ of $\dcoma$, so $\alpha|_{\mathcal{R}}$ is surjective; by freeness of $\dcoma$ there is a unique graded ring homomorphism $\iota:\dcoma\to\mathcal{R}$ with $\iota([i])=\mathtt{row}(i)$, and $(\alpha|_{\mathcal{R}})\circ\iota=\mathrm{id}_{\dcoma}$, so $\alpha|_{\mathcal{R}}$ is also injective. The family $\{\dsch_w^n\}_{n\ge 0,\, n\ge\maxd(w)}$ is a $\mathbb{Z}$-basis of $\dcoma$ by construction, and $\dsch_w^n(\mathtt{row})=\iota(\dsch_w^n)$ since $\iota$ is a ring homomorphism; applying $\iota$ to the basis gives the claim.
\end{proof}

The section $\iota$ has an explicit description in the RC graph basis.

\begin{lemma}\label{lemma:principalsection}
Let $\maxd(u)\leq n$ and let $P_u\in\RC_n(u)$ be the principal RC graph of $u$, the unique RC graph of $u$ with weight $\code(u)$. Then
$$\iota(\dsch_u^n)=\sum_{R\in\RC_n}c_{\wtt(R),u}\,R\text{,}$$
where $c_{\alpha,u}$ is the coefficient of $\sch_u$ in the Schubert expansion of $x_\alpha$, and the only RC graph of permutation $u$ occurring is $P_u$, with coefficient $1$. In particular $\iota(\dsch_u^n)-P_u\in\Omega$ has no term of permutation $u$.
\end{lemma}
\begin{proof}
We have $\iota([\alpha])=\mathtt{row}(\alpha_1)\sqcupplus\cdots\sqcupplus\mathtt{row}(\alpha_n)=\sum_{\wtt(R)=\alpha}R$: by induction on $n$, the terms of $\mathtt{row}(\alpha_1)\sqcupplus\sum_{\wtt(R')=(\alpha_2,\ldots,\alpha_n)}R'$ are the $R$ with $\trm(R)$ of weight $(\alpha_2,\ldots,\alpha_n)$ and $\clp^1(R)=\mathtt{row}(\alpha_1)$, i.e.\ $\wtt_1(R)=\alpha_1$ (Lemma~\ref{lemma:lastrow}). Since $\dsch_u^n=\sum_\alpha c_{\alpha,u}[\alpha]$ by definition, the formula follows. The Schubert polynomials are unitriangular with respect to the monomials in reverse lexicographic order, $\sch_v=x^{\code(v)}+\text{(lower terms)}$ \cite{notes}, so $c_{\alpha,u}\neq0$ forces $\code(u)\leq\alpha$, while $\alpha=\wtt(R)$ for $R\in\RC_n(u)$ means that $x^\alpha$ occurs in $\sch_u$ (Corollary~\ref{corollary:evschub}), forcing $\alpha\leq\code(u)$. Hence $\alpha=\code(u)$, $R=P_u$, and $c_{\code(u),u}=1$.
\end{proof}

\begin{theorem} \label{theorem:semidirect}
  The map $\alpha:\BRC\to \dcoma$ is a surjective homomorphism of rings with kernel $\Omega$; in particular $\Omega$ is a two-sided ideal of $\BRC$. The ring homomorphism $\iota:\dcoma\to\BRC$ sending $[i]\mapsto \mathtt{row}(i)$ is a right inverse to $\alpha$, with image $\mathcal{R}$, and $\BRC=\mathcal{R}\oplus\Omega$. Consequently $\BRC\cong \dcoma\ltimes \Omega$ as rings.
\end{theorem}
\begin{proof}
By Lemma~\ref{lemma:schubertrow} (whose first sentence is the multiplicativity of $\alpha$), $\alpha$ is a ring homomorphism; it is surjective since every $\dsch_w^n$ is $\alpha(R)$ for any $R\in\RC_n(w)$, and $\RC_n(w)\neq\emptyset$ when $\maxd(w)\leq n$.

The kernel of $\alpha$ is $\Omega$: clearly $\Omega\subseteq\ker\alpha$. Conversely, if $x=\sum_R c_RR\in\ker\alpha$, group the terms by $(\wof{R},\hr(R))$; then $\alpha(x)=\sum_{(w,n)}\bigl(\sum_{R\in\RC_n(w)}c_R\bigr)\dsch_w^n=0$ forces $\sum_{R\in\RC_n(w)}c_R=0$ for each $(w,n)$ by linear independence of the $\dsch_w^n$, and an element $\sum_{R\in\RC_n(w)}c_RR$ with $\sum c_R=0$ lies in $\Omega$. As the kernel of a ring homomorphism, $\Omega$ is a two-sided ideal.

By Lemma~\ref{lemma:schubertrow}, $\iota$ is a ring isomorphism $\dcoma\to\mathcal{R}$ with $\alpha\circ\iota=\mathrm{id}$. The split exact sequence $0\to\Omega\to\BRC\xrightarrow{\alpha}\dcoma\to0$ with section $\iota$ gives $\BRC=\mathcal{R}\oplus\Omega$ as abelian groups, with $\mathcal{R}$ a subring and $\Omega$ a two-sided ideal; this is precisely the statement $\BRC\cong\dcoma\ltimes\Omega$.
\end{proof}






\subsection{A general Littlewood--Richardson principle}\label{section:template}

Theorem~\ref{theorem:LR} and the dual key, dual forest, and dual slide rules of the following sections are instances of one mechanism, which we isolate here. For a bounded RC graph $R$ of height $p+q$ write
$$\Psi_p(R)=(\clp^pR,\trm^pR)\in\RC_p\times\RC_q\text{;}$$
by the remarks following Definition~\ref{definition:cliptrim}, $\wtt(R)$ is the concatenation of $\wtt(\clp^pR)$ and $\wtt(\trm^pR)$.

\begin{definition}\label{definition:admissible}
Let $\sim$ be an equivalence relation on the set $\bigsqcup_n\RC_n$ of bounded RC graphs; its equivalence classes are called \emph{classes}. For a class $X$ of height $p+q$ and $(U,V)\in\RC_p\times\RC_q$ put
$$m_X(U,V)=\#\{R\in X:\ \Psi_p(R)=(U,V)\}\text{.}$$
Let $\lambda$ be a function from the set of classes to a set $J$, with classes of different heights receiving different labels. We say that $\sim$ is \emph{admissible} with respect to $\lambda$ if:
\begin{itemize}
\item[(H1)] $R_1\sim R_2$ implies $\hr(R_1)=\hr(R_2)$ and $\wof{R_1}=\wof{R_2}$.
\item[(H2)] (Fiber constancy.) For all classes $X$, $D_1$, $D_2$ of heights $p+q$, $p$, $q$, the function $m_X$ is constant on $D_1\times D_2$; its value is denoted $m_X(D_1,D_2)$.
\item[(H3)] (Characters.) For every class $D$ of height $n$, the polynomial $\sum_{R\in D}x^{\wtt(R)}\in\coma_n$ depends only on $\lambda(D)$; writing it $B_{\lambda(D)}$, the family $\{B_\lambda\}$, with $\lambda$ ranging over the labels of the classes of height $n$, is a $\mathbb Z$-basis of $\coma_n$.
\end{itemize}
Given such data, let $\{B_\lambda^*\}$ be the dual basis of $\dcoma_n$, let $I_\sim=\mathrm{span}_{\mathbb Z}\{R_1-R_2:R_1\sim R_2\}\subseteq\BRC$, write $[D]\in\BRC/I_\sim$ for the common image of the elements of a class $D$, and put
$$E_\lambda=\sum_{D:\ \lambda(D)=\lambda}[D]\in\BRC/I_\sim\text{.}$$
The classes $[D]$ form a $\mathbb Z$-basis of $\BRC/I_\sim$.
\end{definition}

\begin{theorem}[General LR principle]\label{theorem:general_lr_template}
Let $\sim$ be admissible with respect to $\lambda$. Then:
\begin{enumerate}
\item $I_\sim$ is a two-sided ideal of $\BRC$ contained in $\Omega=\ker\alpha$, and $\alpha$ induces a surjective ring homomorphism $\overline\alpha:\BRC/I_\sim\to\dcoma$. (This part uses only (H1) and (H2).)
\item $\overline\alpha(E_\lambda)=B_\lambda^*$ for every label $\lambda$.
\item Let $\mu,\nu$ label classes of heights $p,q$ and define $b_{\mu,\nu}^\lambda$ by $B_\mu^*\sqcupplus B_\nu^*=\sum_\lambda b_{\mu,\nu}^\lambda B_\lambda^*$. For every class $X$ with $\lambda(X)=\lambda$,
$$b_{\mu,\nu}^\lambda=\sum_{\substack{D_1:\ \lambda(D_1)=\mu\\ D_2:\ \lambda(D_2)=\nu}}m_X(D_1,D_2)\text{.}$$
Equivalently, for any function $\mathrm{rep}$ that chooses one element in each class, $b_{\mu,\nu}^\lambda$ is the number of $R\in X$ such that $\clp^pR=\mathrm{rep}(\clp^pR)$, $\trm^pR=\mathrm{rep}(\trm^pR)$, $\lambda([\clp^pR])=\mu$ and $\lambda([\trm^pR])=\nu$. In particular $b_{\mu,\nu}^\lambda\in\mathbb Z_{\geq0}$.
\item $E_\mu\sqcupplus E_\nu=\sum_\lambda b_{\mu,\nu}^\lambda E_\lambda$ in $\BRC/I_\sim$; hence the $E_\lambda$ span a subring of $\BRC/I_\sim$ that $\overline\alpha$ maps isomorphically onto $\dcoma$.
\end{enumerate}
\end{theorem}
\begin{proof}
(1) For bounded RC graphs $S_1,S_2$ of heights $p,q$, Definition~\ref{definition:ringproduct} and grouping the terms by class give, modulo $I_\sim$,
$$S_1\sqcupplus S_2\equiv\sum_{X}m_X(S_1,S_2)\,[X]\text{,}$$
the sum over the classes $X$ of height $p+q$. If $S_1\sim S_1'$ then $m_X(S_1,S_2)=m_X(S_1',S_2)$ for every $X$ by (H2), so $S_1\sqcupplus S_2-S_1'\sqcupplus S_2\in I_\sim$; the same holds in the second argument. Thus $I_\sim$ is a two-sided ideal, and by (H1) each generator $R_1-R_2$ lies in $\Omega$.

(2) By (H1) each class $D$ of height $n$ is contained in a single $\RC_n(w)$, and $\overline\alpha([D])=\dsch_w^n$. Hence $\langle\overline\alpha(E_\lambda),\sch_w^n\rangle$ is the number of classes $D\subseteq\RC_n(w)$ with $\lambda(D)=\lambda$. On the other hand, by Corollary~\ref{corollary:evschub} and (H3),
$$\sch_w^n=\sum_{R\in\RC_n(w)}x^{\wtt(R)}=\sum_{D\subseteq\RC_n(w)}B_{\lambda(D)}\text{,}$$
so $\langle B_\lambda^*,\sch_w^n\rangle$ is the same number. Since the $\sch_w^n$ form a basis of $\coma_n$, $\overline\alpha(E_\lambda)=B_\lambda^*$.

(3) As in the proof of Lemma~\ref{lemma:dualbranch}, $\langle\beta\sqcupplus\gamma,f\rangle=\langle\beta\otimes\gamma,\Delta_{p,q}(f)\rangle$, so $b_{\mu,\nu}^\lambda$ is the coefficient of the product $B_\mu(x_1,\ldots,x_p)\,B_\nu(x_{p+1},\ldots,x_{p+q})$ in the polynomial $B_\lambda(x_1,\ldots,x_{p+q})$; by (H3) these products form a basis of $\coma_{p+q}$. Let $X$ be a class with $\lambda(X)=\lambda$. Grouping the elements of $X$ according to $\Psi_p$, and then the pairs $(U,V)$ according to their classes,
\begin{align*}
B_\lambda=\sum_{R\in X}x^{\wtt(R)}&=\sum_{(U,V)}m_X(U,V)\,x^{\wtt(U)}\,x_{p+1}^{\wtt_1(V)}\cdots x_{p+q}^{\wtt_q(V)}\\
&=\sum_{D_1,D_2}m_X(D_1,D_2)\,B_{\lambda(D_1)}(x_1,\ldots,x_p)\,B_{\lambda(D_2)}(x_{p+1},\ldots,x_{p+q})
\end{align*}
by (H2) and (H3). Comparing coefficients gives the first formula. For the second, the set described is the disjoint union, over the pairs $(D_1,D_2)$ with $\lambda(D_1)=\mu$ and $\lambda(D_2)=\nu$, of the fibers $X\cap\Psi_p^{-1}(\mathrm{rep}\,D_1,\mathrm{rep}\,D_2)$, which have cardinalities $m_X(D_1,D_2)$.

(4) By the display in (1) and by (3), $E_\mu\sqcupplus E_\nu=\sum_{D_1,D_2}\sum_Xm_X(D_1,D_2)[X]=\sum_Xb_{\mu,\nu}^{\lambda(X)}[X]=\sum_\lambda b_{\mu,\nu}^\lambda E_\lambda$. The $E_\lambda$ have disjoint supports in the basis of classes, hence are linearly independent, and by (2) they map onto the basis $\{B_\lambda^*\}$ of $\dcoma$.
\end{proof}

\begin{remark}\label{remark:general_lr_instances}
Theorem~\ref{theorem:LR} is the instance in which $R_1\sim R_2$ means $\hr(R_1)=\hr(R_2)$ and $\wof{R_1}=\wof{R_2}$: the classes are the sets $\RC_n(w)$, labeled by $(w,n)$, with $B_{(w,n)}=\sch_w^n$ by Corollary~\ref{corollary:evschub} and $I_\sim=\Omega$; (H2) is the content of Lemma~\ref{lemma:clipfibers}, which shows that $m_{\RC_{p+q}(w)}(U,V)=N_p(u_v,u)$ for all $U\in\RC_p(u)$ and $V\in\RC_q(v)$ when $\ell(u_v)=\ell(w)-\ell(v)$, and $0$ otherwise. Note that the finer relation of equality on $\RC$ is not admissible: its classes are singletons, whose characters are monomials, and the sums $\mathcal S_w(n)$ of Corollary~\ref{corollary:evschub} do not span a subring of $\BRC$. The next three sections verify (H1)--(H3) for Demazure crystal equivalence, forest equivalence, and quasi-Yamanouchi equivalence.
\end{remark}

\section{The ring of RC graph crystals}\label{section:crystals}

The key polynomials $\kappa_a$ (also called \emph{Demazure characters} in type $A$) form a $\mathbb{Z}$-basis of $\mathbb{Z}[x_1,x_2,\ldots]$ indexed by weak compositions $a$. Their origin lies in the work of Demazure \cite{demazure1974desingularisation,demazure1974nouvelle}, who introduced isobaric divided difference operators $\pi_i$ in order to compute the characters of the $B$-modules $H^0(X_w,\mathcal{L}_\lambda)$ associated to Schubert varieties $X_w$ in $G/B$. In type $A$, the resulting characters become polynomials in the $x_i$, and one obtains $\kappa_a$ for any weak composition $a$ by setting $\kappa_\lambda = x^\lambda$ when $\lambda$ is a partition (a dominant weight) and applying the appropriate operator $\pi_i$ at each descent of $a$ to interpolate between rearrangements of $\lambda$.

A purely combinatorial development of key polynomials was carried out by Lascoux and Schützenberger \cite{lascoux1990keys}, who introduced the \emph{right key} $K_+(T)$ and \emph{left key} $K_-(T)$ of a semistandard Young tableau $T$ and proved that
$$\kappa_a = \sum_{T} x^{\wtt(T)},$$
where the sum runs over semistandard Young tableaux $T$ of shape $\mathrm{sort}(a)$ whose right key $K_+(T)$ is bounded above (entrywise) by the key tableau associated to $a$. They additionally proved that Schubert polynomials decompose as a positive sum of key polynomials, given a combinatorial proof in terms of compatible sequences and the right-key map by Reiner and Shimozono \cite{reiner1995key}. The intermediate basis of \emph{Demazure atoms} (the ``standard bases'' of \cite{lascoux1990keys}) was later given an explicit nonnegative monomial expansion by Mason \cite{mason2009explicit} via semi-skyline augmented fillings.

From a representation-theoretic perspective, Kashiwara \cite{kashiwara1993crystal} introduced \emph{Demazure crystals}: for a dominant integral weight $\lambda$ with crystal $B(\lambda)$ and a Weyl group element $w$ with reduced expression $w=s_{i_1}\cdots s_{i_\ell}$, the Demazure crystal $B_w(\lambda)\subseteq B(\lambda)$ is defined inductively from an extremal weight element by closure under the operators $\{e_{i_k}^n : n\geq 0\}$, and its character coincides with the Demazure character. Littelmann \cite{littelmann1995crystal} obtained parallel results in the path model. Assaf and Schilling \cite{assaf2018demazure}, building on the crystal structure on reduced factorizations of Morse and Schilling \cite{morse2016crystal}, exhibited an explicit Demazure crystal structure on the set of RC graphs of a permutation $w$, thereby giving a representation-theoretic proof of Lascoux--Schützenberger key-positivity. It is this Demazure crystal structure on RC graphs that we will repeatedly invoke below.

The nonnegativity of the dual key structure constants $k_{a,b}^c$ of Theorem~\ref{theorem:LRkey} is, in retrospect, a consequence of the theory of \emph{excellent filtrations} for $B$-modules introduced by Polo \cite{polo1989excellent} and van der Kallen \cite{vanderkallen1989longest}, with existence of such filtrations on tensor products of Demazure modules established in general by Mathieu \cite{mathieu1990filtrations}; see also Joseph \cite{joseph1985demazure}. Restricting a Demazure module $V_c(\lambda)$ for $GL_{p+q}$ to the Borel of the Levi $GL_p\times GL_q$ yields a $B$-module with an excellent filtration whose Demazure subquotients account for the $k_{a,b}^c$. This was pointed out by the anonymous user ``Henry V'' on MathOverflow \cite{henrymath}. While this abstract nonnegativity is classical, no explicit positive combinatorial rule for these coefficients appears in the literature to our knowledge.

\subsection{Little bumps and the crystal action on RC graphs}

\begin{definition}[Crystal action on RC graphs, highest weights, extremal weights] \label{definition:crystal}
Following Morse and Schilling \cite{morse2016crystal}, who defined an $A_{\ell-1}$-crystal structure on reduced factorizations of a permutation $w$, and the adaptation to RC graphs by Assaf and Schilling \cite{assaf2018demazure}, we equip $\RC(w)$ with Kashiwara raising and lowering operators $e_i, f_i$ for $i\geq 1$. The operator pair $(e_i, f_i)$ acts only on rows $i$ and $i+1$ of an RC graph $R$, viewed as adjacent factors in the corresponding reduced factorization.

The action is governed by a bracketing procedure on the column indices appearing in rows $i$ and $i+1$. Process the entries of row $i$ in decreasing order of column index: for each column $b$ in row $i$, pair it with the smallest column $a>b$ appearing in row $i+1$ that has not yet been paired in an earlier step; if no such $a$ exists, $b$ is left unpaired. Let $R_i(R)$ denote the set of unpaired column indices in row $i$, and $L_i(R)$ the set of unpaired column indices in row $i+1$.

If $R_i(R)$ is nonempty, $f_i(R)$ is obtained by removing the smallest unpaired entry $b\in R_i(R)$ from row $i$ and inserting an entry into row $i+1$ at the largest column $b' < b$ such that $b'$ does not already appear in row $i+1$; if $R_i(R) = \emptyset$ or no such $b'$ exists, $f_i(R) = \emptyset$. Dually, if $L_i(R)$ is nonempty, $e_i(R)$ is obtained by removing the largest unpaired entry $a\in L_i(R)$ from row $i+1$ and inserting an entry into row $i$ at the smallest column $a' > a$ such that $a'$ does not already appear in row $i$; if $L_i(R) = \emptyset$ or no such $a'$ exists, $e_i(R) = \emptyset$.

The \emph{weight} of $R$ as an element of the crystal is simply $\wtt(R)$, the row weight itself. A \emph{highest weight} element is an RC graph $R$ with $e_i(R) = \emptyset$ for all $i\geq 1$. Every $R\in \RC(w)$ corresponds to a unique highest weight element $\mathbb{Y}(R)$, which can always be obtained by greedily applying available $e_i$ until none apply.The \emph{Demazure crystal} $\mathrm{Dem}(R)$ generated by a highest weight $\mathbb{Y}(R)$ is the closure of $\{\mathbb{Y}(R)\}$ under the operators $\{f_i^k : i\geq 1, k\geq 0\}$ subject to the Demazure truncation rule of \cite{assaf2018demazure}. The \emph{extremal weight} $\mathrm{extwt}(R)$ of $R$ is the weak composition that is the lexicographically smallest weight of any element of $\mathrm{Dem}(R)$.
\end{definition}

The significance of this elaborate structure is the following.
\begin{theorem}[{\cite{lascoux1990keys,assaf2018demazure}}] \label{theorem:keykey}
If $A\in\BRC$ is such that $\mathrm{extwt}(A) = a$, then
$$\kappa_a = \sum_{R\in\mathrm{Dem}(A)} x^{\wtt(R)}$$
\end{theorem}

We invoke the following theorem of Hamaker and Young.

\begin{theorem}[{\cite[Theorem~4.4]{hamaker2014relating}}] \label{theorem:littlecrystal}
Let $\mathrm{w}$ be a reduced word for a permutation $w\in S_\infty$. Then
$$Q(\mathrm{w}) = Q(\ltb_p(\mathrm{w}))$$
for any valid $p$, where $Q$ denotes the Edelman-Green recording tableau \cite{edelman1987balanced}.
\end{theorem}

\begin{corollary}[$\zeromap$ intertwines with crystal operators] \label{corollary:crystaliso}
For any bounded RC graph $R$ and any positive integer $i$ satisfying $1\leq i\leq \hr(R) - 1$, we have 
$$\zeromap(e_i(R)) = e_i(\zeromap(R))$$
and
$$\zeromap(f_i(R)) = f_i(\zeromap(R))$$
whenever the left-hand sides are defined. Moreover, for $R\in\RC(w,n)^0$, $e_i(\zeromap R)$ is defined if and only if $e_i(R)$ is, and likewise for $f_i$; thus $\zeromap$ is an isomorphism of crystals from $\RC(w,n)^0$ onto $\bigcup_{w'\in T_n(w)}\RC(w',n-1)$.
\end{corollary}
\begin{proof}
Insert the letters of the reading word of an RC graph $R$ of height $m$ from right to left by Edelman--Greene insertion \cite{edelman1987balanced}, let $P(R)$ be the insertion tableau, and let $T^-(R)$ be the recording tableau with each entry replaced by the negative of the row index of the corresponding crossing (Figure~\ref{fig:egexample}). Then $T^-(R)$ is a semistandard tableau in English notation (rows weakly increasing to the right, columns strictly increasing downward) in the alphabet $-m<\cdots<-1$, the pair $(P(R),T^-(R))$ determines $R$, and the operators $e_i,f_i$ of Definition~\ref{definition:crystal} correspond to the tableau crystal operators $f_{-(i+1)}$ and $e_{-(i+1)}$ on $T^-(R)$, leaving $P(R)$ unchanged; this is the compatibility of the Morse--Schilling crystal on reduced factorizations with Edelman--Greene insertion \cite{morse2016crystal}, in which the factors are indexed from right to left, whence the reversal of the alphabet. Each step of $\zeromap$ (Definition~\ref{definition:zeromap}) replaces the reading word by its Little bump at the crossing of the two wires ending in positions $n$ and $s$, where $n$ and $s$ depend only on the permutation, and keeps the row sequence. By Theorem~\ref{theorem:littlecrystal} this leaves the recording tableau, hence $T^-(R)$, unchanged, and by \cite[Lemma~3.3]{hamaker2014relating} Little bumps at the crossing of a fixed pair of wires commute with Coxeter--Knuth moves, so the insertion tableau of the bumped word depends only on $P(R)$ \cite[Theorem~6.24]{edelman1987balanced}. Consequently, if $R$ corresponds to $(P,T)$ and $e_i(R)$ to $(P,f_{-(i+1)}T)$, then after each step the two computations of $\zeromap$ correspond to $(P',T)$ and $(P',f_{-(i+1)}T)$ for the same $P'$; in particular the permutations agree at every stage, the computations stop simultaneously, and $\zeromap(e_iR)=e_i\zeromap(R)$. The same argument applies to $f_i$.

For the last statement put $X=\RC(w,n)^0$ and $Y=\bigcup_{w'\in T_n(w)}\RC(w',n-1)$. By Theorem~\ref{theorem:zero_bijection}, $\zeromap$ is a weight-preserving bijection $X\to Y$, and by the first part it maps each $i$-string of $X$ (a maximal set of elements connected by $e_i,f_i$) onto a set of consecutive elements of an $i$-string of $Y$. Since the weights of $X$ and $Y$ agree as multisets, $X$ and $Y$ have the same number of $i$-strings of each length. An $i$-string of maximal length $L$ in $X$ is mapped into an $i$-string of $Y$ of length at least $L$, hence of length $L$, and onto it; distinct such strings of $X$ have distinct images. Removing these strings from $X$ and their images from $Y$ leaves sets with equal weight multisets, and iterating shows that $\zeromap$ maps every $i$-string of $X$ onto an $i$-string of $Y$, which is the claim.
\end{proof}

\begin{figure}[htbp]
\centering
\begin{minipage}{0.48\textwidth}
\centering
\raisebox{-0.5\height}{\begin{tikzpicture}[scale=0.6]
  \draw[lightgray, very thin, opacity=0.3] (0,1) -- (4,1);
  \draw[lightgray, very thin, opacity=0.3] (0,2) -- (4,2);
  \draw[lightgray, very thin, opacity=0.3] (0,3) -- (4,3);
  \draw[lightgray, very thin, opacity=0.3] (0,4) -- (4,4);
  \draw[lightgray, very thin, opacity=0.3] (0,1) -- (0,4);
  \draw[lightgray, very thin, opacity=0.3] (1,1) -- (1,4);
  \draw[lightgray, very thin, opacity=0.3] (2,1) -- (2,4);
  \draw[lightgray, very thin, opacity=0.3] (3,1) -- (3,4);
  \draw[lightgray, very thin, opacity=0.3] (4,1) -- (4,4);
  \draw[blue, line width=1.5pt] (0,3.5) .. controls (0.3,3.5) and (0.5,3.7) .. (0.5,4);
  \draw[blue, line width=1.5pt] (0.5,3) .. controls (0.5,3.3) and (0.7,3.5) .. (1,3.5);
  \draw[blue, line width=1.5pt, line cap=round] (1,3.5) -- (2,3.5);
  \draw[blue, line width=1.5pt, line cap=round] (1.5,3) -- (1.5,4);
  \draw[blue, line width=1.5pt, line cap=round] (2,3.5) -- (3,3.5);
  \draw[blue, line width=1.5pt, line cap=round] (2.5,3) -- (2.5,4);
  \draw[blue, line width=1.5pt] (3,3.5) .. controls (3.3,3.5) and (3.5,3.7) .. (3.5,4);
  \draw[blue, line width=1.5pt] (0,2.5) .. controls (0.3,2.5) and (0.5,2.7) .. (0.5,3);
  \draw[blue, line width=1.5pt] (0.5,2) .. controls (0.5,2.3) and (0.7,2.5) .. (1,2.5);
  \draw[blue, line width=1.5pt, line cap=round] (1,2.5) -- (2,2.5);
  \draw[blue, line width=1.5pt, line cap=round] (1.5,2) -- (1.5,3);
  \draw[blue, line width=1.5pt] (2,2.5) .. controls (2.3,2.5) and (2.5,2.7) .. (2.5,3);
  \draw[blue, line width=1.5pt] (0,1.5) .. controls (0.3,1.5) and (0.5,1.7) .. (0.5,2);
  \draw[blue, line width=1.5pt] (0.5,1) .. controls (0.5,1.3) and (0.7,1.5) .. (1,1.5);
  \draw[blue, line width=1.5pt] (1,1.5) .. controls (1.3,1.5) and (1.5,1.7) .. (1.5,2);
  \node[font=\Large\bfseries, text=green!60!black, fill=white, inner sep=0.5pt, circle, transform shape] at (1.5,3.5) {2};
  \node[font=\Large\bfseries, text=green!60!black, fill=white, inner sep=0.5pt, circle, transform shape] at (2.5,3.5) {3};
  \node[font=\Large\bfseries, text=green!60!black, fill=white, inner sep=0.5pt, circle, transform shape] at (1.5,2.5) {3};
  \node[font=\scriptsize, text=black, anchor=east] at (-0.3,3.5) {1};
  \node[font=\scriptsize, text=black, anchor=east] at (-0.3,2.5) {4};
  \node[font=\scriptsize, text=black, anchor=east] at (-0.3,1.5) {3};
  \node[font=\scriptsize, text=black, anchor=south] at (0.5,4.3) {1};
  \node[font=\scriptsize, text=black, anchor=south] at (1.5,4.3) {2};
  \node[font=\scriptsize, text=black, anchor=south] at (2.5,4.3) {3};
  \node[font=\scriptsize, text=black, anchor=south] at (3.5,4.3) {4};
  \draw[black, line width=1.7999999999999998pt] (0,1) rectangle (4,4);
\end{tikzpicture}}\quad
\raisebox{-0.5\height}{\begin{tikzpicture}[scale=1]
  \draw (0.00,0.00) rectangle (0.55,-0.55);
  \node[font=\small] at (0.28,-0.28) {$2$};
  \draw (0.55,0.00) rectangle (1.10,-0.55);
  \node[font=\small] at (0.83,-0.28) {$3$};
  \draw (0.00,-0.55) rectangle (0.55,-1.10);
  \node[font=\small] at (0.28,-0.83) {$3$};
\end{tikzpicture}}\quad
\raisebox{-0.5\height}{\begin{tikzpicture}[scale=1]
  \draw (0.00,0.00) rectangle (0.55,-0.55);
  \node[font=\small] at (0.28,-0.28) {$-2$};
  \draw (0.55,0.00) rectangle (1.10,-0.55);
  \node[font=\small] at (0.83,-0.28) {$-1$};
  \draw (0.00,-0.55) rectangle (0.55,-1.10);
  \node[font=\small] at (0.28,-0.83) {$-1$};
\end{tikzpicture}}
\caption*{(a) $R=\{(1,2),(1,3),(2,2)\}$, reading word $3\,2\,3$ with row sequence $1\,1\,2$; $P(R)$ and $T^-(R)$}
\end{minipage}
\hfill
\begin{minipage}{0.48\textwidth}
\centering
\raisebox{-0.5\height}{\begin{tikzpicture}[scale=0.6]
  \draw[lightgray, very thin, opacity=0.3] (0,1) -- (4,1);
  \draw[lightgray, very thin, opacity=0.3] (0,2) -- (4,2);
  \draw[lightgray, very thin, opacity=0.3] (0,3) -- (4,3);
  \draw[lightgray, very thin, opacity=0.3] (0,4) -- (4,4);
  \draw[lightgray, very thin, opacity=0.3] (0,1) -- (0,4);
  \draw[lightgray, very thin, opacity=0.3] (1,1) -- (1,4);
  \draw[lightgray, very thin, opacity=0.3] (2,1) -- (2,4);
  \draw[lightgray, very thin, opacity=0.3] (3,1) -- (3,4);
  \draw[lightgray, very thin, opacity=0.3] (4,1) -- (4,4);
  \draw[blue, line width=1.5pt] (0,3.5) .. controls (0.3,3.5) and (0.5,3.7) .. (0.5,4);
  \draw[blue, line width=1.5pt] (0.5,3) .. controls (0.5,3.3) and (0.7,3.5) .. (1,3.5);
  \draw[blue, line width=1.5pt, line cap=round] (1,3.5) -- (2,3.5);
  \draw[blue, line width=1.5pt, line cap=round] (1.5,3) -- (1.5,4);
  \draw[blue, line width=1.5pt] (2,3.5) .. controls (2.3,3.5) and (2.5,3.7) .. (2.5,4);
  \draw[blue, line width=1.5pt] (2.5,3) .. controls (2.5,3.3) and (2.7,3.5) .. (3,3.5);
  \draw[blue, line width=1.5pt] (3,3.5) .. controls (3.3,3.5) and (3.5,3.7) .. (3.5,4);
  \draw[blue, line width=1.5pt, line cap=round] (0,2.5) -- (1,2.5);
  \draw[blue, line width=1.5pt, line cap=round] (0.5,2) -- (0.5,3);
  \draw[blue, line width=1.5pt, line cap=round] (1,2.5) -- (2,2.5);
  \draw[blue, line width=1.5pt, line cap=round] (1.5,2) -- (1.5,3);
  \draw[blue, line width=1.5pt] (2,2.5) .. controls (2.3,2.5) and (2.5,2.7) .. (2.5,3);
  \draw[blue, line width=1.5pt] (0,1.5) .. controls (0.3,1.5) and (0.5,1.7) .. (0.5,2);
  \draw[blue, line width=1.5pt] (0.5,1) .. controls (0.5,1.3) and (0.7,1.5) .. (1,1.5);
  \draw[blue, line width=1.5pt] (1,1.5) .. controls (1.3,1.5) and (1.5,1.7) .. (1.5,2);
  \node[font=\Large\bfseries, text=green!60!black, fill=white, inner sep=0.5pt, circle, transform shape] at (1.5,3.5) {2};
  \node[font=\Large\bfseries, text=green!60!black, fill=white, inner sep=0.5pt, circle, transform shape] at (0.5,2.5) {2};
  \node[font=\Large\bfseries, text=green!60!black, fill=white, inner sep=0.5pt, circle, transform shape] at (1.5,2.5) {3};
  \node[font=\scriptsize, text=black, anchor=east] at (-0.3,3.5) {1};
  \node[font=\scriptsize, text=black, anchor=east] at (-0.3,2.5) {4};
  \node[font=\scriptsize, text=black, anchor=east] at (-0.3,1.5) {3};
  \node[font=\scriptsize, text=black, anchor=south] at (0.5,4.3) {1};
  \node[font=\scriptsize, text=black, anchor=south] at (1.5,4.3) {2};
  \node[font=\scriptsize, text=black, anchor=south] at (2.5,4.3) {3};
  \node[font=\scriptsize, text=black, anchor=south] at (3.5,4.3) {4};
  \draw[black, line width=1.7999999999999998pt] (0,1) rectangle (4,4);
\end{tikzpicture}}\quad
\raisebox{-0.5\height}{\begin{tikzpicture}[scale=1]
  \draw (0.00,0.00) rectangle (0.55,-0.55);
  \node[font=\small] at (0.28,-0.28) {$2$};
  \draw (0.55,0.00) rectangle (1.10,-0.55);
  \node[font=\small] at (0.83,-0.28) {$3$};
  \draw (0.00,-0.55) rectangle (0.55,-1.10);
  \node[font=\small] at (0.28,-0.83) {$3$};
\end{tikzpicture}}\quad
\raisebox{-0.5\height}{\begin{tikzpicture}[scale=1]
  \draw (0.00,0.00) rectangle (0.55,-0.55);
  \node[font=\small] at (0.28,-0.28) {$-2$};
  \draw (0.55,0.00) rectangle (1.10,-0.55);
  \node[font=\small] at (0.83,-0.28) {$-2$};
  \draw (0.00,-0.55) rectangle (0.55,-1.10);
  \node[font=\small] at (0.28,-0.83) {$-1$};
\end{tikzpicture}}
\caption*{(b) $f_1(R)=\{(1,1),(2,2),(2,3)\}$, reading word $2\,3\,2$ with row sequence $1\,2\,2$; $P(f_1R)=P(R)$ and $T^-(f_1R)=e_{-2}T^-(R)$}
\end{minipage}
\caption{Edelman--Greene data of an RC graph. The letters of the reading word are inserted from right to left; $P$ is the insertion tableau and $T^-$ is the recording tableau with each entry replaced by the negative of the row of the corresponding crossing. In (a), inserting $3$, then $2$ (which bumps $3$), then $3$ gives $P(R)$ with recording tableau $\begin{smallmatrix}1&3\\2\end{smallmatrix}$; the crossings inserted first, second and third lie in rows $2,1,1$, so $T^-(R)=\begin{smallmatrix}-2&-1\\-1\end{smallmatrix}$. The operator $f_1$ moves the crossing $(1,2)$ of $R$ to $(2,3)$ and, via the special bump of $2$ into the row $2\,3$, leaves $P$ unchanged while changing the entry $-1$ in the first row of $T^-$ to $-2$.}
\label{fig:egexample}
\end{figure}

\subsection{The dual key basis and the ring of RC graph crystals}

\begin{definition}[The dual key basis]
  For a weak composition $a$ of length $n$, define $\kappa_a^*\in \dcoma_n$ to be the unique element such that $\langle \kappa_a^*, \kappa_b\rangle = \delta_{a,b}$ for all weak compositions $b$.
\end{definition}




\begin{theorem} \label{theorem:keyformula}
  Let $c$ be a weak composition. Then the dual key polynomial $\kappa^*_c$ is equal to
$$\kappa^*_c = \sum_{\substack{R=\mathbb{Y}(R)\\ \mathrm{extwt}(R)=c}} \dsch_{\wof{R}}^{\hr(R)}$$
\end{theorem}
\begin{proof}
  The statement of the theorem is equivalent to Lascoux and Sch\"utzenberger's formula for Schubert polynomials in terms of key polynomials through Theorem \ref{theorem:keykey}.
\end{proof}


\begin{definition}[The submodule $\Theta\subseteq \BRC$]\label{definition:theta}
  Let $\Theta$ be the $\mathbb{Z}$-submodule of $\BRC$ generated by the elements $R_1-R_2$ such that $\mathbb{Y}(R_1)=\mathbb{Y}(R_2)$.

\end{definition}

\begin{lemma}\label{lemma:thetaideal}
  $\Theta$ is a two-sided ideal of $\BRC$.
\end{lemma}
\begin{proof}
Let $R_1,R_2$ be bounded RC graphs of height $n$ with $\mathbb{Y}(R_1)=\mathbb{Y}(R_2)$, and let $R$ be a bounded RC graph of height $m$. Since crystal operators preserve the permutation, $\wof{R_1}=\wof{R_2}$.

\emph{Left ideal.} A bounded RC graph $R'$ of height $m+n$ is a term of $R\sqcupplus R_1$ if and only if $\trm^m(R')=R_1$ and $\clp^m(R')=R$. Writing $R'=T\cup\shup^mR_1$ with $T=R'_{\leq m}$, the second condition depends only on $T$ (Definition~\ref{definition:cliptrim}), and $T\cup\shup^mR_1$ is a bounded RC graph if and only if the word of $T$ followed by the word of $\shup^m R_1$ is reduced and the resulting permutation has last descent at most $m+n$; this depends on $R_1$ only through $\wof{R_1}=\wof{R_2}$. Hence $T\cup\shup^mR_1\mapsto T\cup\shup^mR_2$ is a bijection from the terms of $R\sqcupplus R_1$ to those of $R\sqcupplus R_2$. For $i>m$ the operators $e_i$ act on rows $i,i+1$ only, so $e_i(T\cup\shup^mR_1)=T\cup\shup^me_{i-m}(R_1)$, with either side defined if and only if the other is. Applying the raising operators that carry $R_1$ to $\mathbb{Y}(R_1)$ and using that the highest weight of a connected component is unique,
$$\mathbb{Y}(T\cup\shup^mR_1)=\mathbb{Y}(T\cup\shup^m\mathbb{Y}(R_1))=\mathbb{Y}(T\cup\shup^m\mathbb{Y}(R_2))=\mathbb{Y}(T\cup\shup^mR_2)\text{.}$$
Thus $R\sqcupplus R_1-R\sqcupplus R_2=\sum_T\bigl((T\cup\shup^mR_1)-(T\cup\shup^mR_2)\bigr)\in\Theta$.

\emph{Right ideal.} Since $R_1$ and $R_2$ lie in the same connected component, it suffices to treat $R_2=e_i(R_1)$ for some $1\leq i\leq n-1$ (the case $R_2=f_i(R_1)$ is symmetric). For $R'$ of height $n+m$ with $\trm^n(R')=R$, the operator $e_i$ ($i\leq n-1$) acts on rows $i,i+1\leq n$, so $\trm^n(e_iR')=R$ when $e_iR'$ is defined, and $\clp^n(e_iR')=e_i(\clp^nR')$ by Corollary~\ref{corollary:crystaliso} applied to each $\zeromap$ step in Definition~\ref{definition:cliptrim} (erasing rows of index $>n$ commutes with $e_i$, and $\zeromap$ is an isomorphism of crystals onto its image by Corollary~\ref{corollary:crystaliso}, so $e_i$ is defined on $\clp^nR'$ if and only if it is defined on $R'$). Hence $e_i$ restricts to a bijection from the terms of $R_1\sqcupplus R$ onto the terms of $e_i(R_1)\sqcupplus R$, and corresponding terms $R'$, $e_i(R')$ have the same highest weight. Thus $R_1\sqcupplus R-R_2\sqcupplus R\in\Theta$.
\end{proof}

\begin{lemma}
  The homomorphism $\alpha:\BRC\to\dcoma$ factors through the quotient $\BRC/\Theta$.
\end{lemma}
\begin{proof}
This is because the generators of $\Theta$ are all in $\Omega$.
\end{proof}

\begin{definition}
For a composition $a$, define  an element $K_a\in\BRC/\Theta$ by
$$K_a =\sum_{\substack{R\in\mathrm{Dem}_a\\ R=\mathbb{Y}(R)}} R$$
\end{definition}

\begin{lemma}[Fiber constancy for crystal components]\label{lemma:keyfibers}
Let $X\subseteq\RC_{p+q}$, $D_1\subseteq\RC_p$, and $D_2\subseteq\RC_q$ be connected components for the operators of Definition~\ref{definition:crystal}. Then $m_X(U,V)=\#\{R\in X:\clp^pR=U,\ \trm^pR=V\}$ is constant for $(U,V)\in D_1\times D_2$.
\end{lemma}
\begin{proof}
Let $R\in X$ with $\clp^pR=U$ and $\trm^pR=V$. For $1\leq i\leq p-1$ the operators $e_i,f_i$ act on rows $i,i+1\leq p$, so $e_i(R)$ is defined if and only if $e_i(R_{\leq p})$ is, and in that case $e_i(R)\in X$ and $\trm^p(e_iR)=V$. Moreover $\clp^p(e_iR)=e_i(\clp^pR)$, and $e_i$ is defined on $\clp^pR$ if and only if it is defined on $R_{\leq p}$: this is Corollary~\ref{corollary:crystaliso} applied to each $\zeromap$ step of Definition~\ref{definition:cliptrim}, as in the proof of Lemma~\ref{lemma:thetaideal}. Hence $e_i$ maps $X\cap\Psi_p^{-1}(U,V)$ into $X\cap\Psi_p^{-1}(e_iU,V)$ and $f_i$ maps back, so $m_X(U,V)=m_X(e_iU,V)$ whenever $e_iU$ is defined, and likewise for $f_i$. For $i\geq p+1$ the operators act on rows $>p$: $e_i(R)$ is defined if and only if $e_{i-p}(V)$ is, and then $\Psi_p(e_iR)=(U,e_{i-p}V)$, so $m_X(U,V)=m_X(U,e_{i-p}V)$, and likewise for $f_i$. Since $D_1$ and $D_2$ are connected, $m_X$ is constant on $D_1\times D_2$.
\end{proof}

\begin{proposition}\label{proposition:keyadmissible}
The relation $R_1\sim R_2\iff\mathbb{Y}(R_1)=\mathbb{Y}(R_2)$ is admissible (Definition~\ref{definition:admissible}) with respect to $\lambda=\mathrm{extwt}$, with $B_a=\kappa_a$ and $E_a=K_a$. Consequently $\alpha(K_a)=\kappa_a^*$, and
$$K_a\sqcupplus K_b=\sum_ck_{a,b}^c\,K_c\quad\text{in }\BRC/\Theta\text{,}\qquad\text{where }\kappa_a^*\sqcupplus\kappa_b^*=\sum_ck_{a,b}^c\,\kappa_c^*\text{;}$$
the $K_a$ span a subring of $\BRC/\Theta$ mapped isomorphically onto $\dcoma$ by $\alpha$.
\end{proposition}
\begin{proof}
The classes are the connected components. (H1) holds because the crystal operators preserve height and permutation; (H2) is Lemma~\ref{lemma:keyfibers}; (H3) is Theorem~\ref{theorem:keykey} together with the fact that the key polynomials $\kappa_a$, $a\in\mathrm{Comp}_n$, form a $\mathbb{Z}$-basis of $\coma_n$. Each component contains exactly one highest weight element, so $E_a=K_a$. Now apply Theorem~\ref{theorem:general_lr_template}.
\end{proof}

\begin{theorem}[LR rule for dual keys]\label{theorem:LRkey}
Let $a,b$ be weak compositions of lengths $p$ and $q$, and let $c$ be a weak composition of length $p+q$. Define integers $k_{a,b}^c$ by
$$\kappa_a^* \sqcupplus \kappa_b^* = \sum_{c} k_{a,b}^c\,\kappa_c^*.$$
For any highest-weight RC graph $C$ with $\mathrm{extwt}(C)=c$, let $N(C)$ be the number of RC graphs $R$ with $\mathbb{Y}(R)=C$ such that $\clp^p(R)$ and $\trm^p(R)$ are both highest weights and $\mathrm{extwt}(\clp^p(R)) = a$, $\mathrm{extwt}(\trm^p(R)) = b$. Then $N(C)$ depends only on $(a,b,c)$, and
$$k_{a,b}^c \;=\; N(C);$$
in particular $k_{a,b}^c\in\mathbb{Z}_{\geq 0}$.
\end{theorem}

\begin{proof}[Proof of Theorem \ref{theorem:LRkey}]
This is Theorem~\ref{theorem:general_lr_template}(3) for the admissible relation of Proposition~\ref{proposition:keyadmissible}, with $\mathrm{rep}=\mathbb{Y}$: the class of $C$ is $\{R:\mathbb{Y}(R)=C\}$, and $\clp^p(R)$ is the representative of its class with label $a$ exactly when it is a highest weight of extremal weight $a$, and likewise for $\trm^p(R)$.
\end{proof}

\begin{example}[Computing $\kappa_a^*\sqcupplus\kappa_b^*$ for $a=(0,1,2)$, $b=(0,1,0,2)$]\label{example:LRkey}
Let $a=(0,1,2)$ (a $3$-row composition) and $b=(0,1,0,2)$ (a $4$-row composition), so $p=3$, $q=4$, and $p+q=7$. We illustrate the coefficient $k_{a,b}^c$ for $c=(0,0,0,1,2,0,3)$.

The permutation $w_c$ whose principal RC graph has Demazure crystal highest weight $c$ is
$$w_c = (1,2,3,5,7,4,10,6,8,9).$$
The key element $K_c$ is the sum of all $7$-row RC graphs whose Demazure crystal highest weight equals $c$. Among these, exactly two satisfy $\clp^3(R)\in\mathrm{Dem}_a$ and $\trm^3(R)\in\mathrm{Dem}_b$:

\begin{figure}[h]
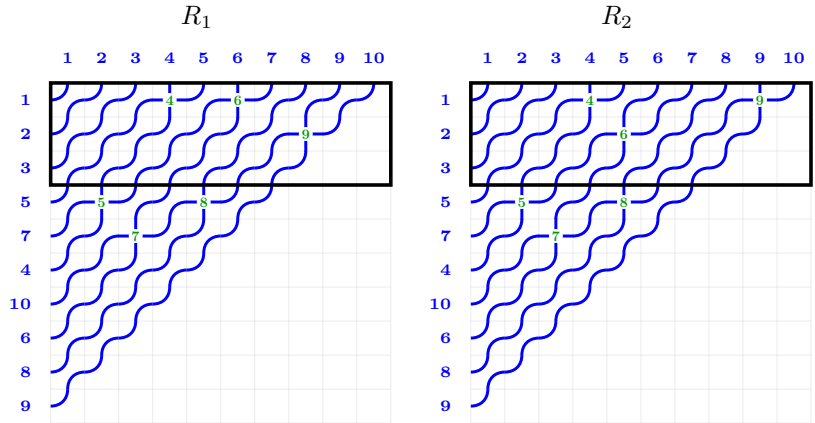

\centering
\begin{tabular}{cc}
$R_1$ & $R_2$ \\
\end{tabular}
\caption{The two $7$-row RC graphs $R_1$ (left) and $R_2$ (right) in $\mathrm{Dem}_c$ satisfying $\clp^3(R_i)\in\mathrm{Dem}_a$ and $\trm^3(R_i)\in\mathrm{Dem}_b$. The outlined box marks the top $3$ rows (the clip region).}
\label{fig:LRkey_example}
\end{figure}

Both graphs have permutation $w_c=(1,2,3,5,7,4,10,6,8,9)$. They differ only in their first two rows: $R_1$ has row data $[(6,4),(9,),()]$ in the top three rows, while $R_2$ has $[(9,4),(6,),()]$. Applying $\clp^3$ (retaining the top $3$ rows and applying $\zeromap$ until the height is $3$) to each yields the same $3$-row graph with extremal weight $a=(0,1,2)$:
$$\clp^3(R_1)=\clp^3(R_2)=
\begin{tikzpicture}[scale=0.45, baseline=(current bounding box.center)]
  \draw[lightgray, very thin, opacity=0.3] (5.3,0) -- (0.3,0);
  \draw[lightgray, very thin, opacity=0.3] (5.3,0) -- (5.3,5);
  \draw[lightgray, very thin, opacity=0.3] (5.3,1) -- (0.3,1);
  \draw[lightgray, very thin, opacity=0.3] (4.3,0) -- (4.3,5);
  \draw[lightgray, very thin, opacity=0.3] (5.3,2) -- (0.3,2);
  \draw[lightgray, very thin, opacity=0.3] (3.3,0) -- (3.3,5);
  \draw[lightgray, very thin, opacity=0.3] (5.3,3) -- (0.3,3);
  \draw[lightgray, very thin, opacity=0.3] (2.3,0) -- (2.3,5);
  \draw[lightgray, very thin, opacity=0.3] (5.3,4) -- (0.3,4);
  \draw[lightgray, very thin, opacity=0.3] (1.3,0) -- (1.3,5);
  \draw[lightgray, very thin, opacity=0.3] (5.3,5) -- (0.3,5);
  \draw[lightgray, very thin, opacity=0.3] (0.3,0) -- (0.3,5);
  \draw[blue, line width=1.125pt] (0.3,4.5) .. controls (0.6,4.5) and (0.8,4.7) .. (0.8,5);
  \draw[blue, line width=1.125pt] (0.8,4) .. controls (0.8,4.3) and (1,4.5) .. (1.3,4.5);
  \draw[blue, line width=1.125pt, line cap=round] (1.3,4.5) -- (2.3,4.5);
  \draw[blue, line width=1.125pt, line cap=round] (1.8,4) -- (1.8,5);
  \draw[blue, line width=1.125pt] (2.3,4.5) .. controls (2.6,4.5) and (2.8,4.7) .. (2.8,5);
  \draw[blue, line width=1.125pt] (2.8,4) .. controls (2.8,4.3) and (3,4.5) .. (3.3,4.5);
  \draw[blue, line width=1.125pt, line cap=round] (3.3,4.5) -- (4.3,4.5);
  \draw[blue, line width=1.125pt, line cap=round] (3.8,4) -- (3.8,5);
  \draw[blue, line width=1.125pt] (4.3,4.5) .. controls (4.6,4.5) and (4.8,4.7) .. (4.8,5);
  \draw[blue, line width=1.125pt] (0.3,3.5) .. controls (0.6,3.5) and (0.8,3.7) .. (0.8,4);
  \draw[blue, line width=1.125pt] (0.8,3) .. controls (0.8,3.3) and (1,3.5) .. (1.3,3.5);
  \draw[blue, line width=1.125pt, line cap=round] (1.3,3.5) -- (2.3,3.5);
  \draw[blue, line width=1.125pt, line cap=round] (1.8,3) -- (1.8,4);
  \draw[blue, line width=1.125pt] (2.3,3.5) .. controls (2.6,3.5) and (2.8,3.7) .. (2.8,4);
  \draw[blue, line width=1.125pt] (2.8,3) .. controls (2.8,3.3) and (3,3.5) .. (3.3,3.5);
  \draw[blue, line width=1.125pt] (3.3,3.5) .. controls (3.6,3.5) and (3.8,3.7) .. (3.8,4);
  \draw[blue, line width=1.125pt] (0.3,2.5) .. controls (0.6,2.5) and (0.8,2.7) .. (0.8,3);
  \draw[blue, line width=1.125pt] (0.8,2) .. controls (0.8,2.3) and (1,2.5) .. (1.3,2.5);
  \draw[blue, line width=1.125pt] (1.3,2.5) .. controls (1.6,2.5) and (1.8,2.7) .. (1.8,3);
  \draw[blue, line width=1.125pt] (1.8,2) .. controls (1.8,2.3) and (2,2.5) .. (2.3,2.5);
  \draw[blue, line width=1.125pt] (2.3,2.5) .. controls (2.6,2.5) and (2.8,2.7) .. (2.8,3);
  \draw[blue, line width=1.125pt] (0.3,1.5) .. controls (0.6,1.5) and (0.8,1.7) .. (0.8,2);
  \draw[blue, line width=1.125pt] (0.8,1) .. controls (0.8,1.3) and (1,1.5) .. (1.3,1.5);
  \draw[blue, line width=1.125pt] (1.3,1.5) .. controls (1.6,1.5) and (1.8,1.7) .. (1.8,2);
  \draw[blue, line width=1.125pt] (0.3,0.5) .. controls (0.6,0.5) and (0.8,0.7) .. (0.8,1);
  \node[font=\tiny\bfseries, text=blue, anchor=south] at (0.8,5.3) {1};
  \node[font=\tiny\bfseries, text=blue, anchor=south] at (1.8,5.3) {2};
  \node[font=\tiny\bfseries, text=blue, anchor=south] at (2.8,5.3) {3};
  \node[font=\tiny\bfseries, text=blue, anchor=south] at (3.8,5.3) {4};
  \node[font=\tiny\bfseries, text=blue, anchor=south] at (4.8,5.3) {5};
  \node[font=\tiny\bfseries, text=blue, anchor=east] at (0,4.5) {1};
  \node[font=\tiny\bfseries, text=blue, anchor=east] at (0,3.5) {3};
  \node[font=\tiny\bfseries, text=blue, anchor=east] at (0,2.5) {5};
  \node[font=\tiny\bfseries, text=blue, anchor=east] at (0,1.5) {2};
  \node[font=\tiny\bfseries, text=blue, anchor=east] at (0,0.5) {4};
  \node[font=\Large\bfseries, text=green!60!black, fill=white, inner sep=0.5pt, circle, transform shape] at (1.8,4.5) {2};
  \node[font=\Large\bfseries, text=green!60!black, fill=white, inner sep=0.5pt, circle, transform shape] at (3.8,4.5) {4};
  \node[font=\Large\bfseries, text=green!60!black, fill=white, inner sep=0.5pt, circle, transform shape] at (1.8,3.5) {3};
\end{tikzpicture}
$$
Applying $\trm^3$ (keeping the bottom $4$ rows) gives the same $4$-row graph with extremal weight $b=(0,1,0,2)$:
$$\trm^3(R_1)=\trm^3(R_2)=
\begin{tikzpicture}[scale=0.45, baseline=(current bounding box.center)]
  \draw[lightgray, very thin, opacity=0.3] (6.3,0) -- (0.3,0);
  \draw[lightgray, very thin, opacity=0.3] (6.3,0) -- (6.3,6);
  \draw[lightgray, very thin, opacity=0.3] (6.3,1) -- (0.3,1);
  \draw[lightgray, very thin, opacity=0.3] (5.3,0) -- (5.3,6);
  \draw[lightgray, very thin, opacity=0.3] (6.3,2) -- (0.3,2);
  \draw[lightgray, very thin, opacity=0.3] (4.3,0) -- (4.3,6);
  \draw[lightgray, very thin, opacity=0.3] (6.3,3) -- (0.3,3);
  \draw[lightgray, very thin, opacity=0.3] (3.3,0) -- (3.3,6);
  \draw[lightgray, very thin, opacity=0.3] (6.3,4) -- (0.3,4);
  \draw[lightgray, very thin, opacity=0.3] (2.3,0) -- (2.3,6);
  \draw[lightgray, very thin, opacity=0.3] (6.3,5) -- (0.3,5);
  \draw[lightgray, very thin, opacity=0.3] (1.3,0) -- (1.3,6);
  \draw[lightgray, very thin, opacity=0.3] (6.3,6) -- (0.3,6);
  \draw[lightgray, very thin, opacity=0.3] (0.3,0) -- (0.3,6);
  \draw[blue, line width=1.125pt] (0.3,5.5) .. controls (0.6,5.5) and (0.8,5.7) .. (0.8,6);
  \draw[blue, line width=1.125pt] (0.8,5) .. controls (0.8,5.3) and (1,5.5) .. (1.3,5.5);
  \draw[blue, line width=1.125pt, line cap=round] (1.3,5.5) -- (2.3,5.5);
  \draw[blue, line width=1.125pt, line cap=round] (1.8,5) -- (1.8,6);
  \draw[blue, line width=1.125pt] (2.3,5.5) .. controls (2.6,5.5) and (2.8,5.7) .. (2.8,6);
  \draw[blue, line width=1.125pt] (2.8,5) .. controls (2.8,5.3) and (3,5.5) .. (3.3,5.5);
  \draw[blue, line width=1.125pt] (3.3,5.5) .. controls (3.6,5.5) and (3.8,5.7) .. (3.8,6);
  \draw[blue, line width=1.125pt] (3.8,5) .. controls (3.8,5.3) and (4,5.5) .. (4.3,5.5);
  \draw[blue, line width=1.125pt, line cap=round] (4.3,5.5) -- (5.3,5.5);
  \draw[blue, line width=1.125pt, line cap=round] (4.8,5) -- (4.8,6);
  \draw[blue, line width=1.125pt] (5.3,5.5) .. controls (5.6,5.5) and (5.8,5.7) .. (5.8,6);
  \draw[blue, line width=1.125pt] (0.3,4.5) .. controls (0.6,4.5) and (0.8,4.7) .. (0.8,5);
  \draw[blue, line width=1.125pt] (0.8,4) .. controls (0.8,4.3) and (1,4.5) .. (1.3,4.5);
  \draw[blue, line width=1.125pt] (1.3,4.5) .. controls (1.6,4.5) and (1.8,4.7) .. (1.8,5);
  \draw[blue, line width=1.125pt] (1.8,4) .. controls (1.8,4.3) and (2,4.5) .. (2.3,4.5);
  \draw[blue, line width=1.125pt, line cap=round] (2.3,4.5) -- (3.3,4.5);
  \draw[blue, line width=1.125pt, line cap=round] (2.8,4) -- (2.8,5);
  \draw[blue, line width=1.125pt] (3.3,4.5) .. controls (3.6,4.5) and (3.8,4.7) .. (3.8,5);
  \draw[blue, line width=1.125pt] (3.8,4) .. controls (3.8,4.3) and (4,4.5) .. (4.3,4.5);
  \draw[blue, line width=1.125pt] (4.3,4.5) .. controls (4.6,4.5) and (4.8,4.7) .. (4.8,5);
  \draw[blue, line width=1.125pt] (0.3,3.5) .. controls (0.6,3.5) and (0.8,3.7) .. (0.8,4);
  \draw[blue, line width=1.125pt] (0.8,3) .. controls (0.8,3.3) and (1,3.5) .. (1.3,3.5);
  \draw[blue, line width=1.125pt] (1.3,3.5) .. controls (1.6,3.5) and (1.8,3.7) .. (1.8,4);
  \draw[blue, line width=1.125pt] (1.8,3) .. controls (1.8,3.3) and (2,3.5) .. (2.3,3.5);
  \draw[blue, line width=1.125pt] (2.3,3.5) .. controls (2.6,3.5) and (2.8,3.7) .. (2.8,4);
  \draw[blue, line width=1.125pt] (2.8,3) .. controls (2.8,3.3) and (3,3.5) .. (3.3,3.5);
  \draw[blue, line width=1.125pt] (3.3,3.5) .. controls (3.6,3.5) and (3.8,3.7) .. (3.8,4);
  \draw[blue, line width=1.125pt] (0.3,2.5) .. controls (0.6,2.5) and (0.8,2.7) .. (0.8,3);
  \draw[blue, line width=1.125pt] (0.8,2) .. controls (0.8,2.3) and (1,2.5) .. (1.3,2.5);
  \draw[blue, line width=1.125pt] (1.3,2.5) .. controls (1.6,2.5) and (1.8,2.7) .. (1.8,3);
  \draw[blue, line width=1.125pt] (1.8,2) .. controls (1.8,2.3) and (2,2.5) .. (2.3,2.5);
  \draw[blue, line width=1.125pt] (2.3,2.5) .. controls (2.6,2.5) and (2.8,2.7) .. (2.8,3);
  \draw[blue, line width=1.125pt] (0.3,1.5) .. controls (0.6,1.5) and (0.8,1.7) .. (0.8,2);
  \draw[blue, line width=1.125pt] (0.8,1) .. controls (0.8,1.3) and (1,1.5) .. (1.3,1.5);
  \draw[blue, line width=1.125pt] (1.3,1.5) .. controls (1.6,1.5) and (1.8,1.7) .. (1.8,2);
  \draw[blue, line width=1.125pt] (0.3,0.5) .. controls (0.6,0.5) and (0.8,0.7) .. (0.8,1);
  \node[font=\tiny\bfseries, text=blue, anchor=south] at (0.8,6.3) {1};
  \node[font=\tiny\bfseries, text=blue, anchor=south] at (1.8,6.3) {2};
  \node[font=\tiny\bfseries, text=blue, anchor=south] at (2.8,6.3) {3};
  \node[font=\tiny\bfseries, text=blue, anchor=south] at (3.8,6.3) {4};
  \node[font=\tiny\bfseries, text=blue, anchor=south] at (4.8,6.3) {5};
  \node[font=\tiny\bfseries, text=blue, anchor=south] at (5.8,6.3) {6};
  \node[font=\tiny\bfseries, text=blue, anchor=east] at (0,5.5) {1};
  \node[font=\tiny\bfseries, text=blue, anchor=east] at (0,4.5) {3};
  \node[font=\tiny\bfseries, text=blue, anchor=east] at (0,3.5) {2};
  \node[font=\tiny\bfseries, text=blue, anchor=east] at (0,2.5) {6};
  \node[font=\tiny\bfseries, text=blue, anchor=east] at (0,1.5) {4};
  \node[font=\tiny\bfseries, text=blue, anchor=east] at (0,0.5) {5};
  \node[font=\Large\bfseries, text=green!60!black, fill=white, inner sep=0.5pt, circle, transform shape] at (1.8,5.5) {2};
  \node[font=\Large\bfseries, text=green!60!black, fill=white, inner sep=0.5pt, circle, transform shape] at (4.8,5.5) {5};
  \node[font=\Large\bfseries, text=green!60!black, fill=white, inner sep=0.5pt, circle, transform shape] at (2.8,4.5) {4};
\end{tikzpicture}
$$

Since there are exactly $2$ such RC graphs, Theorem~\ref{theorem:LRkey} gives $k_{a,b}^c = 2$.
\end{example}

\section{The ring of forest classes of RC graphs}\label{section:forest}

The \emph{forest polynomials} $\forest_a$ were introduced by Nadeau and Tewari in \cite{nadeau2024forest} and are closely related to Schubert polynomials. They can be defined by reduced words and compatible sequences, but the invariant preserved by these words is distinct from the permutation. Given a reduced word/compatible sequence pair $(\mathbf{r},\mathbf{s})$, there is an insertion algorithm that produces a corresponding invariant $\finv(\mathbf{r},\mathbf{s})$ that represents an equivalence relation on the words, where the compatible sequence can vary arbitrarily. We therefore shorten the notation to $\finv(\mathbf{r})$. The equivalence classes consist of words that are all for the same permutation, so that if we define
$$R\equiv_{\forest} R'$$
to mean that $\finv(\mathrm{rev}(\mathrm{word}(R))) = \finv(\mathrm{rev}(\mathrm{word}(R')))$, where $\mathrm{rev}$ reverses a word (see Definition~\ref{definition:forestequiv} for the reason for the reversal), then the forest polynomials partition the set of RC graphs for the given permutation. The forest polynomials are indexed in the original work by ``indexed forests,'' hence the name. These correspond bijectively to weak compositions, so we can define a weak composition $\fcode(R)$ by associating an indexed forest to this equivalence class as in the article. Define $\mathcal{C}_\forest(a)$ for a composition $a$ to be the set of equivalence classes of RC graphs under this relation such that $\fcode(R)=a$.

\subsection{Indexed forests, forest invariant}

\begin{definition}[Indexed forests and codes]
  An \emph{indexed forest} $F$ is a pair $(S, T_\bullet)$ where $S$ is a set of positive integers, called the \emph{support} of $F$. Writing
  $$S=I_1\cup I_2\cup \cdots \cup I_k$$
  where $I_1,I_2,\ldots, I_k$ are the maximal intervals contained in $S$, we have $T_\bullet$ as a sequence of binary trees $T_1,T_2,\ldots, T_k$ where $T_j$ has $|I_j|$ internal nodes for each $j$. The internal nodes of $F$ are denoted by $\mathrm{IN}(F)$. For $v\in \mathrm{IN}(F)$, we denote by $\mathrm{left}(v)$ and $\mathrm{right}(v)$ the left and right children of $v$ respectively, which may not be in $\mathrm{IN}(F)$ if they are leaves, or may be empty.
  
  There is a natural labeling $\lambda_F:\mathrm{IN}(F)\to S$ sending each internal node, say $v\in T_j$, to the unique element of $I_j$ that is at the same index as $v\in T_j$ in the inorder traversal of $T_j$.
  
  The function $\rho_F:\mathrm{IN}(F)\to S$ is defined for an internal node $v\in T_j$ by 
  $$\rho_F(v) = \begin{cases}
  \lambda_F(v)&\mbox{ if }\mathrm{left}(v)\notin\mathrm{IN}(F)\\
  \lambda_F(\mathrm{left}(v))&\mbox{ if }\mathrm{left}(v)\in\mathrm{IN}(F).
  \end{cases}$$
  For each indexed forest $F$, we assign a weak composition $\code(F)$ by
  $$\code_j(F)=\#\{v\in \mathrm{IN}(F): \rho_F(v)=j\}$$
  Then the \emph{forest polynomial} $\forest_F(x)$ is defined by
  $$\forest_F(x) = \sum_{\tau}\prod_{v\in\mathrm{IN}(F)}x_{\tau(v)}$$
  where the sum is over all functions $\tau:\mathrm{IN}(F)\to \mathbb{Z}_{>0}$ such that $\tau(v)\leq \rho_F(v)$ for all $v\in \mathrm{IN}(F)$, $\tau(v)\leq\tau(\mathrm{left}(v))$ for all $v\in \mathrm{IN}(F)$ such that $\mathrm{left}(v)\in \mathrm{IN}(F)$, and $\tau(v)<\tau(\mathrm{right}(v))$ for all $v\in \mathrm{IN}(F)$ such that $\mathrm{right}(v)\in \mathrm{IN}(F)$.

  It is easiest for our purposes to index forest polynomials by weak compositions rather than indexed forests, so we define $\forest_a(x) = \forest_F(x)$ where $F$ is the unique indexed forest such that $\code(F) = a$.
\end{definition}

\begin{definition}[Injective words, forest invariant]
In \cite{nadeau2024ppartitions}, an alphabet $\overline{\mathbb{Z}}$ is defined as the set of pairs of integers $(i,j)$ in lexicographical order, written with the notation $i^{[j]}$. They define $\mathrm{val}(i^{[j]})=i$. An \emph{injective word} is a word $w=w_1w_2\cdots w_k$ such that $w_i\neq w_j$ for all $i\neq j$. An \emph{LBS (local binary search)  labeling} of an indexed forest $F$ is an injective function $\rho:\mathrm{IN}(F)\to \overline{\mathbb{Z}}$ such that
\begin{enumerate}
\item $\rho(\mathrm{left}(v))<\rho(v)$ for all $v\in \mathrm{IN}(F)$ such that $\mathrm{left}(v)\in \mathrm{IN}(F)$.
\item $\rho(v)<\rho(\mathrm{right}(v))$ for all $v\in \mathrm{IN}(F)$ such that $\mathrm{right}(v)\in \mathrm{IN}(F)$.
\item For all $v\in\mathrm{IN}(F)$, $\mathrm{val}(\rho(v)) = \lambda_F(u)$ for some $u\in\mathrm{IN}(F)$.
\end{enumerate}

Following \cite{nadeau2024forest}, given any injective word $w = w_1\cdots w_m$ in $\overline{\mathbb{Z}}$, we associate to it inductively an indexed forest $F(w)$, an LBS labeling $P(w)$ of $F(w)$, and an injective ``time-stamp'' labeling $Q(w):\mathrm{IN}(F(w))\to\{1,\ldots,m\}$ as follows. If $m=0$, then $F(w)$ is the empty forest. Otherwise write $w = w'a$ with $w' = w_1\cdots w_{m-1}$, and assume $F' := F(w')$, $P' := P(w')$, and $Q' := Q(w')$ have already been constructed; let $S'$ be the support of $F'$, and let $S' = I_1\cup\cdots\cup I_k$ be its decomposition into maximal intervals (listed left to right). For each $j$, denote by $u_j$ the root of the tree of $F'$ supported on $I_j$, and set $a_j := P'(u_j)$. We obtain $F(w)$ by adjoining a new internal node $u$ to $F'$, with attachments determined by $i := \mathrm{val}(a)$:
\begin{itemize}
\item If $i\notin S'$, then $u$ is the root of a new tree, with left child $u_j$ when $i-1\in I_j$ (otherwise no left child), and right child $u_j$ when $i+1\in I_j$ (otherwise no right child).
\item If $i\in I_j$, then we compare $a$ with $a_j$ in the order on $\overline{\mathbb{Z}}$:
\begin{itemize}
\item if $a>a_j$, then $u_j$ becomes the left child of $u$; additionally, if $\max(I_j)+2\in S'$ (necessarily $\max(I_j)+2 = \min(I_{j+1})$), then $u_{j+1}$ becomes the right child of $u$;
\item if $a<a_j$, then $u_j$ becomes the right child of $u$; additionally, if $\min(I_j)-2\in S'$ (necessarily $\min(I_j)-2=\max(I_{j-1})$), then $u_{j-1}$ becomes the left child of $u$.
\end{itemize}
\end{itemize}
The remaining roots of $F'$ become roots of $F(w)$ unchanged. We then extend $P'$ to $P(w)$ by setting $P(w)(u) := a$, and extend $Q'$ to $Q(w)$ by setting $Q(w)(u) := m$. The pair $(P(w), Q(w))$ is the \emph{$\Omega$-insertion} of $w$.
\end{definition}

The importance of this algorithm is

\begin{theorem}[{\cite[Theorem~5.1]{nadeau2024forest}}]
The correspondence $w\mapsto (P(w), Q(w))$ is a bijection between the following two sets:
\begin{enumerate}
\item Injective words $w$ with letters in $\overline{\mathbb{Z}}$, and
\item Pairs $(P, Q)$ such that $P \in \mathrm{LBS}(F)$ and $Q \in \mathrm{Dec}(F)$ for a common indexed forest $F$.
\end{enumerate}
\end{theorem}

\begin{definition}[Injectification, $\equiv_\forest$]\label{definition:forestequiv}
  For a word $\mathbf{r}$ in the alphabet $\mathbb{Z}$, we may canonically turn $\mathbf{r}$ into an injective word $\overline{\mathbf r}$ in the alphabet $\overline{\mathbb{Z}}$ by assigning superscripts consecutively starting at $1$ to break ties between equal letters.

For words $\mathbf{r}_1$ and $\mathbf{r}_2$ we say that $\mathbf{r}_1\equiv_\forest\mathbf{r}_2$ if $P(\overline{\mathbf{r}_1}) = P(\overline{\mathbf{r}_2})$. For RC graphs $R_1$ and $R_2$, we say that $R_1\equiv_\forest R_2$ if $\mathrm{rev}(\mathrm{word}(R_1))\equiv_\forest \mathrm{rev}(\mathrm{word}(R_2))$, where $\mathrm{rev}$ reverses a word: the reduced words of \cite{nadeau2024forest} are read in the order opposite to our reading order, and the reversal is forced by Theorem~\ref{theorem:forestformula}(I) below, which fails for the unreversed reading word already for $w=2413$. For a permutation $w$, we denote by $\mathcal{C}_\forest(w)$ the set of equivalence classes of reduced words for $w$ under $\equiv_\forest$.
\end{definition}

\begin{theorem}[{\cite[Theorem~1.4]{nadeau2024forest}}] \label{theorem:forestforest}
  Let $w\in S_\infty$. Then we have

  $$\sch_w(x) = \sum_{C\in \mathcal{C}_\forest(w)} \forest_{\code_\forest(C)}(x)$$
\end{theorem}

\subsection{The dual forest basis and the ring of forest classes of RC graphs}

\begin{definition}[The dual forest basis]
  For a weak composition $a$ of length $n$, define $\forest_a^*\in \dcoma_n$ to be the unique element such that $\langle \forest_a^*, \forest_b\rangle = \delta_{a,b}$ for all weak compositions $b$.
\end{definition}

We record the result that the coefficients of $\forest^*_c$ in the monomial-dual basis are exactly the forest expansion coefficients of monomials computed by Nadeau, Spink, and Tewari, and that under the Borel-type isomorphism $H^\bullet(\mathrm{QFl}_n)\cong \mathbb{Z}[x_1,\ldots,x_n]/\mathrm{QSym}_n^+$ of \cite{bergeron2025qsymflag} the same integers compute the cap product against the geometric homology basis $\{[X(F)]\}$ of $H_\bullet(\mathrm{QFl}_n)$.

\begin{proposition}\label{proposition:forestdualcoefficients}
Let $a$ be a weak composition of length $n$, and let $F_a$ denote the corresponding indexed forest. Let $\{e_{\mathbf{b}}\}_{\mathbf{b}\in \mathbb{Z}_{\geq 0}^n}$ denote the basis of $\dcoma_n$ dual to the monomial basis $\{x^{\mathbf{b}}\}$ of $\coma_n$. Then
$$\forest_a^* \;=\; \sum_{\mathbf{b}\in \mathbb{Z}_{\geq 0}^n} \epsilon_{F_a}(\mathbf{b})\,e_{\mathbf{b}}\quad\in\;\dcoma_n,$$
where $\epsilon_{F_a}(\mathbf{b})\in\{-1,0,+1\}$ is the path sign of \cite[Proposition~10.15]{nadeau2024quasisymmetric}.

The same integers admit two further interpretations:
\begin{enumerate}
\item (Volume polynomial / Postnikov--Stanley-style dual.) The volume polynomial $V_{F_a}(\boldsymbol\lambda)\in\mathbb{Q}[\lambda_1,\ldots,\lambda_n]$ of \cite[Section~10.5]{nadeau2024quasisymmetric}, which is dual to $\forest_{F_a}$ under the divided-power $D$-pairing of \emph{loc.~cit.}, expands as
$$V_{F_a}(\boldsymbol\lambda) \;=\; \sum_{\mathbf{b}\in\mathbb{Z}_{\geq 0}^n} \epsilon_{F_a}(\mathbf{b})\,\frac{\boldsymbol\lambda^{\mathbf{b}}}{\mathbf{b}!}\qquad\bigl(\text{\cite[Proposition~10.14]{nadeau2024quasisymmetric}}\bigr).$$
Up to the divided-power normalization, $V_{F_a}(\boldsymbol\lambda)$ and $\forest_a^*$ are the same linear functional on $\coma_n$, in exact analogy with the relationship between the Postnikov--Stanley dual Schubert polynomials \cite{postnikov2009chains} and our $\dsch_u^n$ noted in the introduction.

\item (Geometric Kronecker dual on $\mathrm{QFl}_n$.) Under the Borel isomorphism $\pi_n:\coma_n\twoheadrightarrow \coma_n/\mathrm{QSym}_n^+\cdot\coma_n \cong H^\bullet(\mathrm{QFl}_n)$ of \cite[Theorem~A and Corollary~12.5]{bergeron2025qsymflag}, the dual functional $\forest_a^*$ kills the kernel ideal and descends to the Kronecker dual of $[\forest_{F_a}]\in H^\bullet(\mathrm{QFl}_n)$ paired against the homology basis $\{[X(F)]:F\in\mathsf{Forest}_n\}$ of \cite[Corollary~11.6]{bergeron2025qsymflag}. Concretely, for $F_a\in\mathsf{Forest}_n$,
$$\forest_a^*\;=\;\pi_n^*\bigl([X(F_a)]\bigr) \quad\text{in}\quad \dcoma_n,$$
and equivalently
$$\int_{X(F_a)}\!\overline{x^{\mathbf{b}}} \;=\; \epsilon_{F_a}(\mathbf{b})\quad \text{in } H^\bullet(\mathrm{QFl}_n)$$
for all $\mathbf{b}\in\mathbb{Z}_{\geq 0}^n$, where $\overline{x^{\mathbf{b}}}=\pi_n(x^{\mathbf{b}})$. For $a$ with $F_a\not\in\mathsf{Forest}_n$ both sides vanish: $\forest_{F_a}\in \mathrm{QSym}_n^+\cdot\coma_n$ by \cite[Theorem~9.7]{nadeau2024quasisymmetric}, so $\forest_a^*$ vanishes on this submodule.
\end{enumerate}
\end{proposition}

\begin{proof}
The first identity is immediate: applying $\forest_a^*$ to the forest expansion of monomials $x^{\mathbf{b}}=\sum_G \epsilon_G(\mathbf{b})\,\forest_G$ \cite[Proposition~10.15]{nadeau2024quasisymmetric} gives $\langle \forest_a^*, x^{\mathbf{b}}\rangle = \sum_G \epsilon_G(\mathbf{b})\langle \forest_a^*,\forest_G\rangle = \epsilon_{F_a}(\mathbf{b})$ by the defining duality $\langle \forest_a^*,\forest_b\rangle=\delta_{a,b}$.

Item~(1) is then a direct comparison of \cite[Proposition~10.14]{nadeau2024quasisymmetric} with the displayed expansion.

For item~(2), the splitting $\coma_n = \bigoplus_{F_a\in\mathsf{Forest}_n}\mathbb{Z}\forest_a \oplus \mathrm{QSym}_n^+\cdot\coma_n$ \cite[Theorem~9.7]{nadeau2024quasisymmetric} forces $\forest_a^*$ to annihilate the kernel of $\pi_n$ when $F_a\in\mathsf{Forest}_n$, hence to factor through $\pi_n$. The induced functional pairs with $[\forest_b]$ to $\delta_{a,b}$, which is precisely the defining property of $[X(F_a)]$ \cite[Theorem~12.12]{bergeron2025qsymflag}, so $\forest_a^*=\pi_n^*([X(F_a)])$. The second equality of item~(2) is then the previous display read modulo $\mathrm{QSym}_n^+\cdot\coma_n$.
\end{proof}

In particular, the algebraic dual basis $\{\forest_a^*\}\subset\dcoma_n$ and the geometric homology basis $\{[X(F)]\}\subset H_\bullet(\mathrm{QFl}_n)$ are not merely Kronecker-dual to the same basis: they are the same linear functional on $\coma_n$, viewed once before and once after passing to the Borel quotient. The bialgebra structure on $\dcoma$ and the geometric coproduct on $\bigoplus_n H_\bullet(\mathrm{QFl}_n)$ induced by the variable-splitting comorphism $\iota_{p,q}^*$ of Section~\ref{subsection:qflbranch} are correspondingly related: $\pi^* = \bigoplus_n \pi_n^*$ identifies the geometric coproduct with the restriction of the concatenation coproduct on $\dcoma$ to $\bigoplus_n \pi_n^*\!\bigl(H_\bullet(\mathrm{QFl}_n)\bigr)$.

\begin{theorem}\label{theorem:forestformula}
  We have the following.
  \begin{itemize}
\item[I] Let $R\in \BRC$ be a bounded RC graph. Then the forest polynomial $\forest_{\fcode(R)}$ is equal to
$$\forest_{\fcode(R)}(x)=\sum_{R'\equiv_{\forest} R} x^{\wtt(R')}$$

\item[II] Let $c$ be a weak composition. Then the dual forest polynomial $\forest^*_c$ is equal to
$$\forest^*_c = \sum_{[R]\in\mathcal{C}_\forest(c)} \dsch_{\wof{R}}^{\hr(R)}$$
  \end{itemize}
\end{theorem}

\begin{proof}[Proof of Theorem \ref{theorem:forestformula}]
  Part I of this theorem follows from the fact that for any equivalence class of words $W$ that have the same $P$-LBS, $\forest_{\code_\forest(W)}(x)$ is the sum of the compatible sequences on the words in $W$ (\cite[Proposition~5.10]{nadeau2024forest}), which are simply the RC graphs in the corresponding equivalence class of RC graphs; the words of \cite{nadeau2024forest} are the reversals of our reading words, with the compatible sequences reversed accordingly. Part II follows from Theorem \ref{theorem:forestforest} and the definition of the dual Schubert basis.
\end{proof}

Guo and Woodruff \cite{guo2026forest} also published an alternative RC graph formula for forest polynomials.

\begin{proposition}
Let $c$ be a weak composition of length $p$. Let $\RC_\forest(c)$ be the set of RC graphs with $\wtt(R) = c$ such that $\mathrm{word}(\mathbb{Y}_\forest(R))$ is minimal in lexicographical order among all RC graphs $R'$ such that $\code_\forest(R)=\code_\forest(R')$. Then
$$[c] = \sum_{R\in\RC_\forest(c)} \forest_{\code_\forest(R)}^*$$
\end{proposition}

\begin{remark}
The coefficients of the expansion of a monomial into forest polynomials, or equivalently the coefficients of the expansion of a dual forest polynomial into weak composition monomials, is given in \cite[Proposition~10.15]{nadeau2024forest}.
\end{remark}

\begin{definition}[The submodule $\Gamma\subseteq \BRC$]\label{definition:gamma}
  Let $\Gamma$ be the $\mathbb{Z}$-submodule of $\BRC$ generated by the elements $R_1-R_2$ such that $\hr(R_1)=\hr(R_2)$ and $R_1\equiv_\forest R_2$. A \emph{forest class} is a set of bounded RC graphs of a common height $n$ that is an equivalence class for $\equiv_\forest$; its \emph{forest code} $\fcode(D)$ is the common value of $\fcode(R)$, $R\in D$, regarded as a weak composition of length $n$.
\end{definition}

\begin{lemma}[$\equiv_\forest$ and concatenation]\label{lemma:forestconcat}
Let $\mathbf r_1\equiv_\forest\mathbf r_2$ be words.
\begin{enumerate}
\item[(a)] $\mathbf r_1\mathbf s\equiv_\forest\mathbf r_2\mathbf s$ for every word $\mathbf s$.
\item[(b)] $\mathbf r_1+m\equiv_\forest\mathbf r_2+m$ for every $m\geq0$, where $\mathbf r+m$ adds $m$ to every letter.
\end{enumerate}
\end{lemma}
\begin{proof}
(a) The labeling $P(\overline{\mathbf r})$ records the letters of $\overline{\mathbf r}$ as a set, so $P(\overline{\mathbf r_1})=P(\overline{\mathbf r_2})$ implies that each value occurs equally often in $\mathbf r_1$ and $\mathbf r_2$; hence $\overline{\mathbf r_1\mathbf s}=\overline{\mathbf r_1}\,\mathbf s'$ and $\overline{\mathbf r_2\mathbf s}=\overline{\mathbf r_2}\,\mathbf s'$ with the same injective word $\mathbf s'$. Inserting a letter $a$ into $(F(w),P(w))$ depends only on $(F(w),P(w))$ and $a$, so $P(\overline{\mathbf r_1}\mathbf s')=P(\overline{\mathbf r_2}\mathbf s')$ by induction on the length of $\mathbf s'$.

(b) This is \cite[Proposition~5.8]{nadeau2024forest}.
\end{proof}

\begin{lemma}[Deleting the first crossing commutes with the zero map]\label{lemma:forestzeropop}
For a nonempty reduced RC graph $R$, let $\operatorname{pop}R$ delete its first crossing in reading order, without changing its declared height. If $R$ is bounded of height $n\geq2$ with empty last row, then
$$\zeromap(\operatorname{pop}R)=\operatorname{pop}(\zeromap R).$$
The removed crossings lie in the same row, although their letters need not be equal.
\end{lemma}
\begin{proof}
Write the reading word of $R$ as $a\mathbf r$. Deleting its first crossing gives the reduced word $\mathbf r$ of $s_a\wof R$. Multiplication on the left by $s_a$ interchanges the consecutive values $a,a+1$, which occur in inverted order because $a$ is a left descent. Comparisons at adjacent positions are unchanged except when those two values occupy adjacent positions, in which case their descent is removed. Thus no right descent is created, and $\operatorname{pop}R$ remains bounded at height $n$, with empty last row.

Put $Y=\zeromap R$, $R^-=\operatorname{pop}R$, and $Y^-=\operatorname{pop}Y$. The same argument makes $Y^-$ bounded at height $n-1$. The row of the first crossing is the first nonzero entry of the weight vector, so it agrees for $R,Y$ by weight preservation. The Gao--Huang pop correspondence gives
$$\Phi(R^-)=\nabla\Phi(R),\qquad \Phi(Y^-)=\nabla\Phi(Y).$$
Proposition~\ref{proposition:zerotrunc} says that $\Phi(R)$ and $\Phi(Y)$ agree in rows $1,\ldots,n-1$. Lemma~\ref{lemma:poplocal} therefore gives the same agreement for $\Phi(R^-)$ and $\Phi(Y^-)$. Apply Proposition~\ref{proposition:zerotrunc} to $R^-$: $\Phi(\zeromap R^-)$ also agrees with these rows. Both $\zeromap R^-$ and $Y^-$ are bounded at height $n-1$, so bounded completion and injectivity of $\Phi$ identify them.
\end{proof}

\begin{lemma}[Leading-letter transport in a reversed transition]\label{lemma:forestzeroleadingstep}
Let $\maxd(u)=n$, $s=\max\{j>n:u(j)<u(n)\}$, $\pi=ut_{ns}$, and $v=\pi t_{in}\gtrdot\pi$ with $i<n$. Set $\eta=\pi(n)$. If a graph $D\in\RC(v,n)^0$ has reading word beginning with $b$, then the reading word of its forward Monk reconstruction $m_{i,n}D\in\RC(u,n)^0$ begins with
$$f_\eta(b)=
\begin{cases}
b+1,&b=\eta-1,\\
b,&b\neq\eta-1.
\end{cases}$$
In particular, the reconstructed leading letter is independent of the remaining reduced word and the row sequence.
\end{lemma}
\begin{proof}
For any permutation $\rho$ with increasing tail in positions $>n$, there is exactly one right-hand Monk cover $\rho t_{nq}\gtrdot\rho$: $q$ is the first position $>n$ with $\rho(q)>\rho(n)$. Earlier tail values are smaller than $\rho(n)$; later choices are blocked by $\rho(q)$. For $\rho=\pi$ this unique cover is $u$, at $q=s$.

Write the input word as $b\mathbf r$. If its marked exchange occurrence is the first position, then deleting that position gives $\pi$, so $v=s_b\pi$. The root $(i,n)$ at that occurrence means $\pi(i)=b$ and $\pi(n)=b+1$. Thus this case occurs only when $b=\eta-1$.

Otherwise the marked occurrence lies in $\mathbf r$. Deleting it gives a reduced word for $\bar\pi=s_b\pi$, so $b$ is a left descent of $\pi$. No right descent is created by left multiplication at a left descent, as proved in Lemma~\ref{lemma:forestzeropop}; hence $\bar\pi$ also has increasing tail past $n$. Its suffix Monk input has permutation $s_bv=\bar\pi t_{in}$ and the same marked right root $(i,n)$.

Suppose $b\neq\eta-1$. Swapping the consecutive values $b,b+1$ leaves the position $s$ of the first tail value above the value in position $n$ unchanged. If $\eta=b$, the value in position $n$ increases by one, but the value $b+1$ was earlier than $n$ because $b$ is a left descent; thus no tail value crosses the threshold. If $\eta\notin\{b,b+1\}$, comparisons with the threshold do not change. Consequently the unique suffix Monk output is
$$q=\bar\pi t_{ns}=s_bu.$$
Also $b$ remains a left descent of $u$. The transition from $\pi$ to $u$ can affect the order of $b,b+1$ only if one is in position $n$ or $s$: the case $\eta=b$ moves $b$ later, and the case $\pi(s)=b$ moves $b$ to $n$ while $b+1$ is earlier than $n$; $\pi(s)=b+1$ is impossible for a left descent with increasing tail and $\eta\neq b$. All other cases leave that order unchanged. Thus prefixing $b$ to the reduced suffix output is reduced and gives $u$. No push reaches the first position, and the leading letter stays $b$.

Now suppose $b=\eta-1$ and the initial mark is in the suffix. The left-descent condition places $b$ at a position $r>n$. In $\bar\pi$, position $n$ contains $b$ and position $r$ contains $b+1$. Its first larger tail value is therefore at $r$, and the unique suffix Monk output is $\bar\pi t_{nr}=\pi$. Prefixing $b$ is not reduced, so the other exchange defect is the first position. The next push changes this letter to $b+1=\eta$. In the initial-mark-first case, this is already the first push, with suffix permutation $\pi$.

If $\eta$ is not a left descent of $\pi$, the word $\eta\mathbf r'$ with suffix permutation $\pi$ is reduced and terminates. Otherwise $\pi^{-1}(\eta+1)=i'<n$, since $\pi(n)=\eta$ and the tail is increasing. The other exchange defect is in the suffix, with right root $(i',n)$; the remaining computation is the suffix Monk move with base $s_\eta\pi$. This base has increasing tail, and its unique right cover is at the same $s$, because the value $\eta+1$ occurs before $n$, not in the tail. Its output is $s_\eta u$. The value $\eta$ is in position $s$ of $u$ while $\eta+1$ is before $n$, so $\eta$ is a left descent of $u$. Prefixing $\eta$ to this suffix output is therefore reduced and gives $u$; the first position is not pushed again. Both subcases yield the leading letter $b+1$.

These suffix computations describe the actual marked bump trajectory: until a suffix Monk move terminates, deleting its current marked occurrence leaves a reduced suffix, so the corresponding whole word cannot terminate or acquire its other exchange defect at the first position. At suffix termination, either prefixing the leading letter is reduced, or the first position is precisely the other defect. The exchange property therefore justifies each decomposition used above.
\end{proof}

\begin{corollary}[The lifted leading letter is unique]\label{corollary:forestzeroleadingunique}
Fix $\maxd(u)=n$ and $w'\in T_n(u)$. There is a function $F_{u,w'}$ on the possible terminal leading letters such that every $Y\in\RC_{n-1}(w')$ whose reading word starts with $b$ has lift $L_{u,w'}(Y)$ starting with $F_{u,w'}(b)$. For the unique transition path $u=u_0,\ldots,u_m=w'$, put $\eta_k=(u_{k-1}t_{n,s_k})(n)$. Then
$$F_{u,w'}=f_{\eta_1}\circ f_{\eta_2}\circ\cdots\circ f_{\eta_m}.$$
\end{corollary}
\begin{proof}
Reverse the unique transition path and apply Lemma~\ref{lemma:forestzeroleadingstep} at each step. Empty-row padding does not affect the word. This gives the formula and the uniqueness assertion.
\end{proof}

\begin{lemma}[A high-supported forest has a row-supported representative]\label{lemma:foresthighrows}
Let $D$ be a bounded RC graph of height $h$, and suppose
$$\operatorname{supp}P\!\left(\overline{\mathrm{rev}(\mathrm{word}(D))}\right)
  \subseteq[J+1,\infty)$$
for an integer $0\leq J<h$. There is a bounded height-$h$ graph $\widehat D\equiv_\forest D$ whose crossings all lie in rows $J+1,\ldots,h$.
\end{lemma}
\begin{proof}
Every node value belongs to its component support, so every reading letter of $D$ exceeds $J$. Its permutation has the form $1^J\times q$, where $\maxd(q)\leq h-J$. Translate its reversed word and its labeled forest by $-J$. The translated forest has positive support. Its forest polynomial is nonzero, and \cite[Proposition~5.10]{nadeau2024forest}, applied to this single labeled $\Omega$-class, expresses that polynomial as the sum of the compatible-sequence polynomials of the words in the class. Consequently some word in this class has a positive compatible sequence and hence gives an RC graph $E$. Forest swaps are Coxeter commutations, so $E$ has permutation $q$; they also preserve the order of equal-valued occurrences, so canonical injectification is retained.

Every RC graph for $q$ has its crossings in rows at most $\maxd(q)$. Indeed, if $r$ is its last nonempty row, its decreasing row word is a length-additive suffix of the full reduced word and has a right descent at least $r$; that descent is also a right descent of $q$. Shift $E$ down by $J$ rows, so that all its reading letters increase by $J$, and pad with empty rows to height $h$. Translation equivariance gives the required graph $\widehat D$.
\end{proof}

\begin{lemma}[Separated-support shuffles]\label{lemma:forestseparatedshuffle}
Let $\mathbf l,\mathbf u$ be integer words, and let $J$ be an integer. Here the support of $P$ means the support of its indexed forest, not the set of letters in the word. Suppose
$$\operatorname{supp}P(\overline{\mathbf l})\subseteq(-\infty,J-1],
\qquad
\operatorname{supp}P(\overline{\mathbf u})\subseteq[J+1,\infty).$$
Then every shuffle $\mathbf w$ of $\mathbf l$ and $\mathbf u$ satisfies
$$P(\overline{\mathbf w})=
P(\overline{\mathbf l})\sqcup P(\overline{\mathbf u}),$$
where the right side is the disjoint union of the two labeled forests. In particular, $\mathbf w\equiv_\forest\mathbf u\mathbf l$.
Moreover, if
$$\operatorname{supp}P(\overline{\mathbf w})
 \subseteq(-\infty,J-1]\cup[J+1,\infty)$$
and no component meets both sides, then restricting to either value side gives the $P$-symbol of the corresponding subword.
\end{lemma}
\begin{proof}
Each inserted letter has its value in the support of its component. Thus the two word alphabets are disjoint, and canonical injectification commutes with taking the two subwords. Support only grows under $\Omega$-insertion, so the same support bounds hold for every prefix of $\mathbf l$ and $\mathbf u$.

We induct on the length of a shuffle prefix, maintaining that its inserted forest is the disjoint union of the forests obtained from the two projected prefixes separately. Insert a next letter $a<J$ from the lower word. If $a$ is absent from the current support, neither neighboring value $a-1,a+1$ belongs to an upper component, so the new node cannot attach to one. If $a$ lies in a lower support interval $I$, the only possible attachment to an upper component is the right-hand bridge in the $\Omega$-insertion rule: $a>P'(u_I)$ and $\max(I)+2$ belongs to the upper support. The support bounds force $\max(I)=J-1$ and the upper interval to start at $J+1$. Without that upper component, the same insertion makes the old tree the left child of the new node, whose inorder index is $\max(I)+1=J$. This contradicts the assumed support bound for the corresponding prefix of $\mathbf l$. Thus no such bridge occurs, and the lower insertion has exactly its standalone effect. The argument for a next letter $a>J$ is symmetric: a bridge to a lower component could occur only when $a<P'(u_I)$ and $\min(I)-2$ belongs to the lower support, forcing $\min(I)=J+1$; without that lower component the new node has inorder index $\min(I)-1=J$, contradicting the support bound for the corresponding prefix of $\mathbf u$. This proves the induction and the disjoint-union assertion.

For the final statement, components that have become attached can never separate under later insertions. Thus, if the final support avoids $J$ and no component meets both sides, no insertion attached a component from one side to the other. Deleting the letters on either side therefore leaves precisely the insertion history, and hence the $P$-symbol, of the other side.
\end{proof}

\begin{proposition}[Fixed-chain lifts with a compatible right factor]\label{proposition:forestcompatiblechain}
Let $T,T'$ be reduced graphs supported in rows $1,\ldots,p$. Choose $h\geq p$ so that both are bounded at height $h$, and put
$$T_k=\zeromap^{h-k}T,\qquad T'_k=\zeromap^{h-k}T'
\quad(p\leq k\leq h).$$
Suppose $\wof{T_k}=\wof{T'_k}=:u_k$ for every $k$, and $T_p\equiv_\forest T'_p$. Let $V$ be a bounded graph of height $q$ such that
$$R=T\cup\shup^pV,\qquad R'=T'\cup\shup^pV$$
are bounded reduced graphs of height $p+q$. Then $R\equiv_\forest R'$.
\end{proposition}
\begin{proof}
We induct strongly on the number $d$ of crossings of $T$, simultaneously for all $p,q,h$, chains, and compatible right factors. The case $d=0$ is immediate. Let $P$ be the common forest of $T_p,T'_p$ and put $\mathbf a=\mathrm{rev}(\mathrm{word}(V))+p$.

If $P$ has one nonempty tree, both terminal reading words begin with its root value $b$. Reverse the common permutation chain. At every nontrivial step, Corollary~\ref{corollary:forestzeroleadingunique} gives a common leading letter; empty-row padding leaves it unchanged. Thus $T_k,T'_k$ have common leading letters $a_k$ at every height. Lemma~\ref{lemma:forestzeropop}, iterated along the chain, gives
$$\zeromap(\operatorname{pop}T_k)=\operatorname{pop}T_{k-1},
\qquad
\zeromap(\operatorname{pop}T'_k)=\operatorname{pop}T'_{k-1}.$$
The popped chains have common permutations $s_{a_k}u_k$, and their terminal symbols agree after deletion of the common root of $P$. The full popped graphs $\operatorname{pop}T\cup\shup^pV$ and $\operatorname{pop}T'\cup\shup^pV$ are reduced and bounded at height $p+q$: the removed crossing is the first reading occurrence, and multiplication on the left at a left descent creates no right descent. Apply the induction hypothesis with $d-1$ crossings, then append the common letter $a_h$ to the reversed contextual words. Lemma~\ref{lemma:forestconcat}(a) proves $R\equiv_\forest R'$.

Suppose $P$ has several nonempty trees. Choose the vacant cut $J$ immediately before its rightmost component, so that its lower components have support below $J$ and its rightmost component has support above $J$. Every terminal letter avoids $J$. Moreover $J<p$: the terminal permutation preserves $\{1,\ldots,J\}$, and if $J\geq p$ its increasing tail would force the nonempty right factor to be the identity.

All reversed transitions in the chain act strictly above this cut. Indeed, at height $k>p>J$, write the child as $v=u_{k-1}$ and the parent as $u_k$. If $v$ preserves the cut and $v=\pi t_{ik}$ with $\pi t_{ik}\gtrdot\pi$, then
$$v(i)=\pi(k)>v(k)>J,$$
so $i>J$. Since $k>J$ and the parent transposition is $(k,s)$ with $s>k$, both transition transpositions act inside the upper parabolic block. The marked crossing has right root $(i,k)$ with both entries above $J$. At each Little-bump push the other exchange defect has the same right root, by the exchange property recalled above; hence it too belongs to the upper factor. Every incremented letter remains above $J$, and the row sequence is unchanged. Thus reconstruction leaves the lower graph unchanged and restricts on the upper projection to the same inverse zero-map step. Empty steps just delete an empty declared row. Consequently, writing the letter projections with subscripts $-$ and $+$,
$$T_-= (T_p)_-,\qquad T'_-= (T'_p)_-,$$
and the upper projections follow a common projected permutation chain: at each height $k$, the projected permutations are equal for $T_k,T'_k$, and $T_{k-1,+}=\zeromap(T_{k,+})$ (with empty-row padding when the step is empty). By the component-restriction assertion of Lemma~\ref{lemma:forestseparatedshuffle}, their terminal graphs $(T_p)_+,(T'_p)_+$ are forest equivalent; the lower terminal projections are forest equivalent as well. Both terminal factors are nonempty, so their crossing degrees are strictly smaller than $d$.

Every letter of the shifted right factor exceeds $p>J$. Consequently $T_+\cup\shup^pV$ and $T'_+\cup\shup^pV$ are the right parabolic projections of $R,R'$; they are reduced and bounded at height $p+q$. Apply the smaller-degree induction to obtain equality of their contextual upper forest symbols.

It remains to show that this common symbol has support above $J$. Lemma~\ref{lemma:foresthighrows} gives a terminal representative
$$\widehat Y_+\equiv_\forest(T_p)_+$$
at height $p$, with all crossings in rows greater than $J$. Lift it along the same projected permutation chain to $\widehat T_+$. Each inverse zero map exists by Theorem~\ref{theorem:zero_bijection}, and weight preservation keeps all crossings in rows $J+1,\ldots,p$. The source permutation of $\widehat T_+$ equals that of $T_+$. Hence $\widehat T_+\cup\shup^pV$ has the same permutation and crossing number as $T_+\cup\shup^pV$, and is reduced and bounded at height $p+q$. All its rows at most $J$ are empty. Trimming these rows gives a positive compatible RC graph, whose forest support is positive by the compatible-forest correspondence. Translating back shows that its forest support lies in $[J+1,\infty)$.

Apply the smaller-degree induction once more, this time to $T_+,\widehat T_+$ with the same right factor $V$. Their terminal symbols agree and their permutation chains agree, so their contextual source symbols agree. This transfers the support bound to $T_+\cup\shup^pV$ and $T'_+\cup\shup^pV$.

The lower projected words have forest symbols supported in $(-\infty,J-1]$, and the contextual upper words have equal forest symbols supported in $[J+1,\infty)$. The full reversed words are shuffles of these lower and upper words, with disjoint value alphabets. Lemma~\ref{lemma:forestseparatedshuffle} therefore identifies each full forest symbol with the disjoint union of its lower and contextual upper symbols. These unions agree, completing the induction.
\end{proof}

The auxiliary height $h$ in Proposition~\ref{proposition:forestcompatiblechain} need not be at most $p+q$.

\begin{lemma}[Fiber constancy for forest classes]\label{lemma:forestfibers}
Let $X$, $D_1$, $D_2$ be forest classes of heights $p+q$, $p$, $q$. Then $m_X(U,V)$ is constant for $(U,V)\in D_1\times D_2$.
\end{lemma}
\begin{proof}
Let $w$, $u'$, $v$ be the permutations of the elements of $X$, $D_1$, $D_2$, fix $C\in X$, and write $R\in\RC_{p+q}(w)$ as $R=T\cup\shup^pV$ with $T=R_{\leq p}$, $V=\trm^pR$, as in the proof of Theorem~\ref{theorem:LR}. Then
$$\mathrm{rev}(\mathrm{word}(R))=\bigl(\mathrm{rev}(\mathrm{word}(V))+p\bigr)\,\mathrm{rev}(\mathrm{word}(T))\text{,}$$
and $R\in X$ if and only if $\mathrm{rev}(\mathrm{word}(R))\equiv_\forest\mathrm{rev}(\mathrm{word}(C))$.

\emph{Varying $V$.} Let $V,V'\in D_2$. For $R=T\cup\shup^pV$ with $\Psi_p(R)=(U,V)$, the graph $T\cup\shup^pV'$ is a bounded RC graph of height $p+q$ with $\Psi_p(T\cup\shup^pV')=(U,V')$: reducedness of $\mathrm{word}(T)(\mathrm{word}(V')+p)$ and the last descent of the resulting permutation depend on $V'$ only through $\wof{V'}=v$, and $\clp^p$ depends only on $T$. By Lemma~\ref{lemma:forestconcat}(b) and then (a), $\mathrm{rev}(\mathrm{word}(T\cup\shup^pV'))\equiv_\forest\mathrm{rev}(\mathrm{word}(T\cup\shup^pV))$, so $T\cup\shup^pV'\in X$. Hence $T\cup\shup^pV\mapsto T\cup\shup^pV'$ is a bijection $X\cap\Psi_p^{-1}(U,V)\to X\cap\Psi_p^{-1}(U,V')$, and $m_X(U,V)=m_X(U,V')$.

\emph{Varying $U$.} Fix $V\in D_2$, put $u_v=w(\shup^pv)^{-1}$ and $h=\max(p,\maxd(u_v))$, and let $\mathbf a=\mathrm{rev}(\mathrm{word}(V))+p$. If $\ell(u_v)\neq\ell(w)-\ell(v)$ then $m_X(\,\cdot\,,V)=0$. Otherwise, by the proof of Theorem~\ref{theorem:LR}, $T\cup\shup^pV$ is a bounded RC graph of height $p+q$ for every $T\in\RC_{\leq p}(u_v)$, and, writing $F(U)=\{T\in\RC_{\leq p}(u_v):\zeromap^{h-p}(T)=U\}$,
$$X\cap\Psi_p^{-1}(U,V)=\{T\cup\shup^pV:\ T\in F(U),\ \mathbf a\,\mathrm{rev}(\mathrm{word}(T))\equiv_\forest\mathrm{rev}(\mathrm{word}(C))\}\text{,}$$
and the last condition depends on $T$ only through the forest class of the concatenated word $\mathbf a\,\mathrm{rev}(\mathrm{word}(T))$. It therefore suffices to construct, for $U,U'\in D_1$, a bijection $\theta\colon F(U)\to F(U')$ such that
$$\mathbf a\,\mathrm{rev}(\mathrm{word}(\theta(T)))\equiv_\forest
\mathbf a\,\mathrm{rev}(\mathrm{word}(T)).$$

For $T\in F(U)$ consider the sequence $T=T_h,T_{h-1},\ldots,T_p=U$ with $T_{k-1}=\zeromap(T_k)$, and let $u_k=\wof{T_k}$. If $\maxd(u_k)<k$ then $\zeromap$ merely deletes the empty last row; otherwise $u_{k-1}\in T_k(u_k)$. By Theorem~\ref{theorem:zero_bijection} and Lemma~\ref{lemma:zeroweight}, as in the proof of Lemma~\ref{lemma:clipfibers}, $T\mapsto(u_h,u_{h-1},\ldots,u_p)$ is a bijection from $F(U)$ onto the set of chains counted by $N_p(u_v,u')$: the chain and $U$ determine $T$, and every chain arises by surjectivity. The same holds for $F(U')$; define $\theta(T)\in F(U')$ to be the element with the same chain. Both assembled graphs are bounded and reduced by the preceding length-additivity condition. Proposition~\ref{proposition:forestcompatiblechain}, applied to the common chain and $U\equiv_\forest U'$, gives
$$T\cup\shup^pV\equiv_\forest\theta(T)\cup\shup^pV.$$
Thus $\theta$ restricts to a bijection between the indicated fibers inside $X$, proving $m_X(U,V)=m_X(U',V)$.
\end{proof}

\begin{lemma} \label{lemma:forestideal}
  $\Gamma$ is a two-sided ideal of $\BRC$.
\end{lemma}
\begin{proof}
This is Theorem~\ref{theorem:general_lr_template}(1) for the relation ``same height and $\equiv_\forest$'': (H1) holds because forest-equivalent words are reduced words for the same permutation, and (H2) is Lemma~\ref{lemma:forestfibers}.
\end{proof}

\begin{lemma}
  The homomorphism $\alpha:\BRC\to\dcoma$ factors through the quotient $\BRC/\Gamma$.
\end{lemma}
\begin{proof}
This is because the generators of $\Gamma$ are all in $\Omega$.
\end{proof}

\begin{definition}
For a composition $a$, define  an element $F_a\in\BRC/\Gamma$ by
$$F_a =\sum_{[R]\in\mathcal{C}_\forest(a)} [R]$$
\end{definition}

\begin{proposition}\label{proposition:forestadmissible}
The relation ``$\hr(R_1)=\hr(R_2)$ and $R_1\equiv_\forest R_2$'' is admissible (Definition~\ref{definition:admissible}) with respect to $\lambda=\fcode$, with $B_a=\forest_a$ and $E_a=F_a$. Consequently $\alpha(F_a)=\forest_a^*$, and
$$F_a\sqcupplus F_b=\sum_cf_{a,b}^c\,F_c\quad\text{in }\BRC/\Gamma\text{,}\qquad\text{where }\forest_a^*\sqcupplus\forest_b^*=\sum_cf_{a,b}^c\,\forest_c^*\text{;}$$
the $F_a$ span a subring of $\BRC/\Gamma$ mapped isomorphically onto $\dcoma$ by $\alpha$.
\end{proposition}
\begin{proof}
(H1) and (H2) were verified in the proof of Lemma~\ref{lemma:forestideal}. For (H3), the character of a forest class $D$ of height $n$ is $\forest_{\fcode(D)}(x_1,\ldots,x_n)$ by Theorem~\ref{theorem:forestformula}(I), and the forest polynomials $\forest_a$, $a$ a weak composition of length $n$, form a $\mathbb{Z}$-basis of $\coma_n$ \cite{nadeau2024forest}. Now apply Theorem~\ref{theorem:general_lr_template}.
\end{proof}

\begin{theorem}[LR rule for dual forests]\label{theorem:LRforest}
Let $a,b$ be weak compositions of lengths $p$ and $q$, and let $c$ be a weak composition of length $p+q$. Define integers $f_{a,b}^c$ by
$$\forest_a^* \sqcupplus \forest_b^* = \sum_{c} f_{a,b}^c\,\forest_c^*.$$
For any $C\in\mathcal{C}_\forest(c)$, let $N(C)$ be the number of $R\in[C]$ such that $\clp^p(R)\in \mathcal{C}_\forest(a)$, $\trm^p(R)\in \mathcal{C}_\forest(b)$, $\wtt(\clp^p(R))=a$, and $\wtt(\trm^p(R))=b$. Then $N(C)$ depends only on $(a,b,c)$, and
$$f_{a,b}^c \;=\; N(C);$$
in particular $f_{a,b}^c\in\mathbb{Z}_{\geq 0}$.
\end{theorem}

\begin{proof}[Proof of Theorem \ref{theorem:LRforest}]
The monomial $x^{\fcode(D)}$ occurs with coefficient $1$ in $\forest_{\fcode(D)}$ \cite{nadeau2024forest}, so by Theorem~\ref{theorem:forestformula}(I) every forest class $D$ contains exactly one element of weight $\fcode(D)$; let $\mathrm{rep}(D)$ be this element. The statement is Theorem~\ref{theorem:general_lr_template}(3) for the admissible relation of Proposition~\ref{proposition:forestadmissible} with this choice of $\mathrm{rep}$: $\clp^p(R)$ is the representative of its class with label $a$ exactly when $\clp^p(R)\in\mathcal{C}_\forest(a)$ and $\wtt(\clp^p(R))=a$, and likewise for $\trm^p(R)$.
\end{proof}

\begin{example}[Computing $\forest_a^*\sqcupplus\forest_b^*$ for $a=(0,0,2,0,2)$, $b=(0,1,0,0,2)$]\label{example:LRforest}
Let
$$a=(0,0,2,0,2),\qquad b=(0,1,0,0,2),\qquad c=(0,0,1,0,1,0,2,0,0,3).$$
Here $p=q=5$, so by Theorem~\ref{theorem:LRforest} the coefficient $f_{a,b}^c$ counts the $R$ in a fixed forest class $\mathcal{C}_\forest(c)$ of $10$-row RC graphs satisfying
$$\clp^{5}(R)\in\mathcal{C}_\forest(a),\quad \trm^{5}(R)\in\mathcal{C}_\forest(b),\quad \wtt(\clp^{5}(R))=a,\quad \wtt(\trm^{5}(R))=b.$$

Writing an RC graph as the set of ordered pairs $(i,j)\in\mathbb{Z}_{>0}\times\mathbb{Z}_{>0}$ recording the (row,\,column) of each crossing, there are exactly two such $R$, namely
\begin{align*}
R_1 &= \{(3,1),(3,2),(5,3),(5,8),(7,1),(10,1),(10,2)\},\\
R_2 &= \{(3,1),(3,6),(5,1),(5,8),(7,1),(10,1),(10,2)\}.
\end{align*}
Both have permutation
$$\wof{R_1}=\wof{R_2}=(1,2,4,3,6,5,9,7,8,13,10,11,12),$$
and they share a common clip/trim pair
$$\clp^{5}(R_i)=\{(3,1),(3,2),(5,1),(5,2)\},\qquad \trm^{5}(R_i)=\{(2,1),(5,1),(5,2)\},$$
with $\wtt(\clp^{5}(R_i))=a$ and $\wtt(\trm^{5}(R_i))=b$.

\begin{figure}[h]
\centering
\begin{tabular}{cc}
\resizebox{0.43\textwidth}{!}{\input{forest_lr_R1.tikz}} &
\resizebox{0.43\textwidth}{!}{\input{forest_lr_R2.tikz}} \\
$(R_1)$ & $(R_2)$
\end{tabular}
\caption{The two $10$-row witness pipe dreams in the forest class for $c$.}
\label{fig:forest_lr_witnesses}
\end{figure}
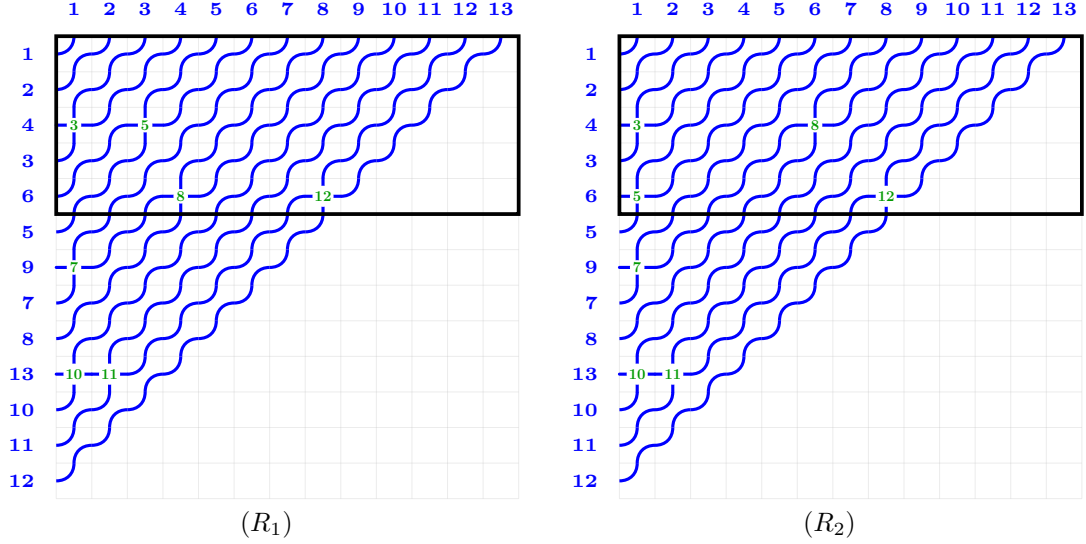

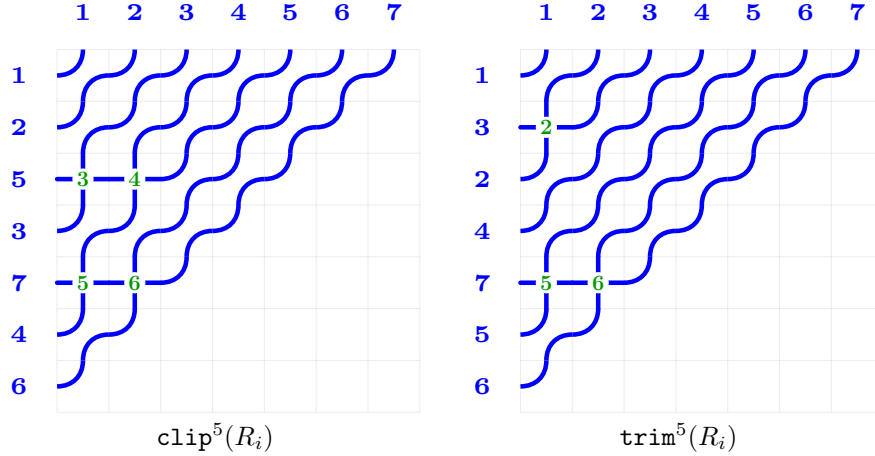
\begin{figure}[h]
\centering
\begin{tabular}{cc}
\resizebox{0.35\textwidth}{!}{\input{forest_lr_clip.tikz}} &
\resizebox{0.35\textwidth}{!}{\input{forest_lr_trim.tikz}} \\
$\clp^5(R_i)$ & $\trm^5(R_i)$
\end{tabular}
\caption{Common clip/trim pair for both witnesses $R_1,R_2$.}
\label{fig:forest_lr_cliptrim}
\end{figure}

Therefore Theorem~\ref{theorem:LRforest} gives
$$f_{a,b}^c=2.$$ 
\end{example}

\subsection{Geometric interpretation: branching on the quasisymmetric flag variety}\label{subsection:qflbranch}

The forest polynomials acquire a geometric meaning through the work of Bergeron, Gagnon, Nadeau, Spink, and Tewari \cite{bergeron2025qsymflag}, who introduce the \emph{quasisymmetric flag variety} $\mathrm{QFl}_n\subset \mathrm{Fl}_n$, a (reducible) toric complex built as a union of left-translated Richardson varieties indexed by $n$-leaf planar binary trees. Their main results give a Borel-type presentation
$$H^\bullet(\mathrm{QFl}_n) \;\cong\; \mathbb{Z}[x_1,\ldots,x_n]/\mathrm{QSym}_n^+,$$
where $\mathrm{QSym}_n^+$ is the ideal of quasisymmetric polynomials with no constant term \cite[Theorem~A and Corollary~12.5]{bergeron2025qsymflag}, and identify the forest polynomials $\{\forest_F : F\in \mathsf{Forest}_n\}$ as a free $\mathbb{Z}$-basis of this ring \cite[Corollary~12.5]{bergeron2025qsymflag}, Kronecker dual to the homology basis $\{[X(F)] : F\in\mathsf{Forest}_n\}$ of $H_\bullet(\mathrm{QFl}_n)$ \cite[Corollary~11.6, Theorem~12.12]{bergeron2025qsymflag}. Under the standard bijection between $\mathsf{Forest}_n$ and weak compositions of length $n$ (compare \cite{nadeau2024forest}), the basis $\forest_F$ matches our composition-indexed basis $\forest_a$; we write $[\forest_a]=\pi_n(\forest_a)$ for the corresponding cohomology class.

The dual forest LR rule (Theorem~\ref{theorem:LRforest}) admits a geometric reading on $\mathrm{QFl}_n$. The variable-splitting substitution
$$\rho_{p,q}:\;\mathbb{Z}[x_1,\ldots,x_{p+q}]\longrightarrow \mathbb{Z}[x_1,\ldots,x_p]\otimes \mathbb{Z}[x_{p+1},\ldots,x_{p+q}]$$
descends to a well-defined ring homomorphism
$$\iota_{p,q}^*:\;H^\bullet(\mathrm{QFl}_{p+q})\longrightarrow H^\bullet(\mathrm{QFl}_p)\otimes H^\bullet(\mathrm{QFl}_q).$$
Indeed, for $f\in\mathrm{QSym}_{p+q}^+$ the standard alphabet-sum identity for the QSym coproduct writes $\rho_{p,q}(f) = \sum f_{(1)}(x_1,\ldots,x_p)\otimes f_{(2)}(x_{p+1},\ldots,x_{p+q})$ with each $f_{(1)}$, $f_{(2)}$ quasisymmetric in its variables; since $\deg f>0$ and the degrees add, each term has at least one tensor factor in $\mathrm{QSym}^+$, so $\rho_{p,q}(f)\in \mathrm{QSym}_p^+\otimes\mathbb{Z}[y] + \mathbb{Z}[x]\otimes \mathrm{QSym}_q^+$. The splitting reformulation of Theorem~\ref{theorem:LRforest} then says exactly that, in the forest-polynomial basis,
$$\iota_{p,q}^*\,[\forest_c] \;=\; \sum_{a,b} f_{a,b}^c\;[\forest_a]\otimes [\forest_b]\qquad\text{in } H^\bullet(\mathrm{QFl}_p)\otimes H^\bullet(\mathrm{QFl}_q),$$
and Theorem~\ref{theorem:LRforest} gives an explicit positive cardinality formula for each $f_{a,b}^c$.

The map $\iota_{p,q}^*$ is the natural quasisymmetric-flag analogue of the type-A Levi comorphism $H^\bullet(\mathrm{Fl}_{p+q})\to H^\bullet(\mathrm{Fl}_p)\otimes H^\bullet(\mathrm{Fl}_q)$ of \cite{ressayre2009branching}, and the dual forest LR rule plays for it the role that Theorem~\ref{theorem:LR} plays for the Schubert basis. We do not know whether $\iota_{p,q}^*$ is induced by a geometric embedding $\mathrm{QFl}_p\times \mathrm{QFl}_q\hookrightarrow \mathrm{QFl}_{p+q}$ compatible with the Levi inclusion $\mathrm{Fl}_p\times \mathrm{Fl}_q\hookrightarrow \mathrm{Fl}_{p+q}$; the natural concatenation operation $\mathsf{Forest}_p\times \mathsf{Forest}_q\to \mathsf{Forest}_{p+q}$ on indexed forests, and the resulting matching $X(F)\times X(F')\subset \mathrm{QFl}_{p+q}$ one expects on the toric Richardson side, suggest that such an embedding should exist, but we do not pursue the geometric verification here. At the level of cohomology rings $\iota_{p,q}^*$ is in any case a well-defined ring homomorphism, and Theorem~\ref{theorem:LRforest} is, in this language, a manifestly positive branching Schubert calculus rule on the quasisymmetric flag variety, parallel to the one Theorem~\ref{theorem:LR} provides on $\mathrm{Fl}_n$. Dually, the family $\{\iota_{p,q}^*\}_{p,q\geq 0}$ assembles into a graded coproduct on $\bigoplus_n H_\bullet(\mathrm{QFl}_n)$ whose structure constants in the geometric forest basis $\{[X(F)]\}$ are again the $f_{a,b}^c$ of Theorem~\ref{theorem:LRforest}.


\section{Product rule for dual slide polynomials}\label{section:slide}
\subsection{Slide polynomials}

\begin{definition}
For a word $\mathbf{r}=\mathbf{r}_1\mathbf{r}_2\cdots \mathbf{r}_k$, a sequence of positive integers $a_1a_2\cdots a_k$ is \emph{compatible} with $\mathbf{r}$ if $a_i\leq \mathbf{r}_i$ for all $i$, $a_i\leq a_{i+1}$ for all $i$, and $a_i<a_{i+1}$ if $\mathbf{r}_i<\mathbf{r}_{i+1}$.

Slide polynomials $\mathfrak{F}(\mathbf{r})$ for a word $\mathbf{r}$ are defined by summing the monomials $x^a$ over all compatible sequences $a$ of $\mathbf{r}$.

Assaf and Searles \cite{assaf2017slide} derived an expansion of Schubert polynomials in terms of slide polynomials as follows. A \emph{quasi-Yamanouchi} RC graph is an RC graph $R$ such that if $R'$ is another RC graph with $\mathrm{word}(R)=\mathrm{word}(R')$ and $\hr(R)=\hr(R')$, then $\wtt(R)\leq \wtt(R')$ in dominance order and $\mathrm{flat}(\wtt(R'))$ refines $\mathrm{flat}(\wtt(R))$. The set of RC graphs $A$ such that $\mathrm{word}(A)=\mathrm{word}(R)$ and $\hr(A)=\hr(R)$ is a partially ordered set with $R$ as its smallest element. We then define
$$\QY(A)=R$$
\end{definition}

\begin{theorem}[{\cite[Theorem~3.13]{assaf2017slide}}]
  For a permutation $w\in S_\infty$, we have the formula
  $$\sch_w(x) = \sum_{R} \mathfrak{F}_{\wtt(R)}(x)$$
  where the sum is over all quasi-Yamanouchi RC graphs $R$ for $w$.
\end{theorem}
\begin{corollary}
  For a weak composition $c$ of length $p$, we have the formula
  $$\mathfrak{F}_c^* = \sum_{R=\QY(R)} \dsch_{\wof{R}}^p$$
  where the sum is over all quasi-Yamanouchi RC graphs $R$ such that $\wtt(R) = c$.
\end{corollary}

\begin{remark}
The expansion of dual slides in the weak composition basis is given by the expansion of monomials into slides in \cite[Theorem~5.2]{nadeau2024ppartitions}.
\end{remark}

\begin{definition}[The submodule $\Sigma\subseteq \BRC$]
Let $\Sigma$ be the $\mathbb{Z}$-submodule of $\BRC$ generated by the elements $R_1-R_2$ such that $\QY(R_1)=\QY(R_2)$.
\end{definition}

\begin{lemma}\label{lemma:slideideal}
$\Sigma$ is a two-sided ideal of $\BRC$.
\end{lemma}
\begin{proof}
By definition $\QY(R_1)=\QY(R_2)$ means $\mathrm{word}(R_1)=\mathrm{word}(R_2)$ and $\hr(R_1)=\hr(R_2)=:n$; in particular $\wof{R_1}=\wof{R_2}$. Let $R$ be a bounded RC graph of height $m$.

\emph{Left ideal.} As in the proof of Lemma~\ref{lemma:thetaideal}, $T\cup\shup^mR_1\mapsto T\cup\shup^mR_2$ is a bijection from the terms of $R\sqcupplus R_1$ to those of $R\sqcupplus R_2$, since membership depends on $R_1$ only through $\wof{R_1}$. Corresponding terms have height $m+n$ and reading word $\mathrm{word}(T)$ followed by $\mathrm{word}(R_1)$ with every letter increased by $m$, which is the same for $R_1$ and $R_2$. Hence $R\sqcupplus R_1-R\sqcupplus R_2\in\Sigma$.

\emph{Right ideal.} We first record that $\zeromap$ and its inverse act at the level of reading words. A step of Definition~\ref{definition:zeromap} on $X\in\RC(u,h)^0$ with $\maxd(u)=h$ replaces $\mathrm{word}(X)$ by its Little bump at the position with right root $(h,s)$, where $s$ is determined by $u$, and keeps the row sequence; so $\mathrm{word}(X^{(1)})$ depends only on $\mathrm{word}(X)$. Conversely, by Lemma~\ref{lemma:zeroweight} this step is the inverse of the Monk move $m_{i,h}$, and $m_{i,h}$ is computed (see the description preceding Lemma~\ref{lemma:bpdconfine}) by replacing the crossing with right root $(i,h)$ and moving crossings to the right within their rows, which in terms of the reading word is an incrementing Little bump at a word-determined position, again with the row sequence unchanged. Consequently, if $X,X'\in\RC(u,h)^0$ have the same reading word, then $\zeromap(X)$ and $\zeromap(X')$ have the same reading word; and if $Y,Y'\in\RC_{h-1}$ have the same reading word and $Y=\zeromap(X)$ with $X\in\RC(u,h)^0$, then applying to $Y'$, with an empty row $h$ appended, the sequence of Monk moves $m_{i_k,h},\ldots,m_{i_1,h}$ that recovers $X$ from $Y$ produces at each stage an RC graph with the same reading word as the corresponding stage for $X$, hence with the same permutation, and yields $X'\in\RC(u,h)^0$ with $\mathrm{word}(X')=\mathrm{word}(X)$ and $\zeromap(X')=Y'$.

Now a bounded RC graph $R'$ of height $n+m$ is a term of $R_1\sqcupplus R$ if and only if $R'=T\cup\shup^nR$ with $T=R'_{\leq n}$, where $T\cup\shup^nR$ is a bounded RC graph and $\zeromap^{h-n}(T)=R_1$ for $h=\max(n,\maxd(\wof T))$ (Definition~\ref{definition:cliptrim}). The first condition depends only on $\mathrm{word}(T)$, and so does $h$. Iterating the previous paragraph $h-n$ times gives a bijection $T\mapsto T'$ from $\{T:\zeromap^{h-n}(T)=R_1\}$ onto $\{T':\zeromap^{h-n}(T')=R_2\}$ with $\mathrm{word}(T')=\mathrm{word}(T)$: at each stage the two computations pass through RC graphs with equal reading words, and the row sequences of $R_1$ and $R_2$ are carried along unchanged. Hence $T\cup\shup^nR\mapsto T'\cup\shup^nR$ is a bijection from the terms of $R_1\sqcupplus R$ onto the terms of $R_2\sqcupplus R$ preserving reading words and heights, and $R_1\sqcupplus R-R_2\sqcupplus R\in\Sigma$.
\end{proof}

\begin{lemma}
The homomorphism $\alpha:\BRC\to\dcoma$ factors through the quotient $\BRC/\Sigma$.
\end{lemma}
\begin{proof}
This is because the generators of $\Sigma$ are all in $\Omega$: if $\QY(R_1)=\QY(R_2)$ then $\wof{R_1}=\wof{\QY(R_1)}=\wof{\QY(R_2)}=\wof{R_2}$ and $\hr(R_1)=\hr(R_2)$, so $\alpha(R_1)=\alpha(R_2)$.
\end{proof}

\begin{definition}
For a weak composition $a$, define an element $\mathcal{F}_a\in\BRC/\Sigma$ by
$$\mathcal{F}_a=\sum_{R: \QY(R)=R,\ \wtt(R)=a} R.$$
\end{definition}

\begin{lemma}[Fiber constancy for word classes]\label{lemma:slidefibers}
Let $X$, $D_1$, $D_2$ be classes of heights $p+q$, $p$, $q$ for the relation $\QY(R_1)=\QY(R_2)$. Then $m_X(U,V)$ is constant for $(U,V)\in D_1\times D_2$.
\end{lemma}
\begin{proof}
The two bijections constructed in the proof of Lemma~\ref{lemma:slideideal} preserve reading words and heights. For $V,V'\in D_2$ the first, $T\cup\shup^pV\mapsto T\cup\shup^pV'$, restricts to a bijection $X\cap\Psi_p^{-1}(U,V)\to X\cap\Psi_p^{-1}(U,V')$, since it fixes $\clp^p$ and sends $X$ to $X$. For $U,U'\in D_1$ the second, $T\cup\shup^pV\mapsto T'\cup\shup^pV$ with $\zeromap^{h-p}(T')=U'$ and $\mathrm{word}(T')=\mathrm{word}(T)$, restricts to a bijection $X\cap\Psi_p^{-1}(U,V)\to X\cap\Psi_p^{-1}(U',V)$.
\end{proof}

\begin{proposition}\label{proposition:slidebasis}
The relation $\QY(R_1)=\QY(R_2)$ is admissible (Definition~\ref{definition:admissible}) with respect to $\lambda(D)=\wtt(\QY(D))$, with $B_a=\mathfrak{F}_a$ and $E_a=\mathcal{F}_a$. Consequently $\alpha(\mathcal{F}_a)=\mathfrak{F}_a^*$, and
$$\mathcal{F}_a\sqcupplus\mathcal{F}_b=\sum_cs_{a,b}^c\,\mathcal{F}_c\quad\text{in }\BRC/\Sigma\text{,}\qquad\text{where }\mathfrak{F}_a^*\sqcupplus\mathfrak{F}_b^*=\sum_cs_{a,b}^c\,\mathfrak{F}_c^*\text{;}$$
the $\mathcal{F}_a$ span a subring of $\BRC/\Sigma$ mapped isomorphically onto $\dcoma$ by $\alpha$.
\end{proposition}
\begin{proof}
(H1) holds because a class consists of the RC graphs of a fixed height with a fixed reading word; (H2) is Lemma~\ref{lemma:slidefibers}. For (H3), the character of the class $D$ of height $n$ with reading word $\mathbf r$ is $\mathfrak{F}(\mathbf r)=\mathfrak{F}_{\wtt(\QY(D))}$ by \cite[Theorem~3.13]{assaf2017slide} and its proof, and the fundamental slide polynomials $\mathfrak{F}_a$, $a$ a weak composition of length $n$, form a $\mathbb{Z}$-basis of $\coma_n$ \cite{assaf2017slide}. Each class contains exactly one quasi-Yamanouchi element, so $E_a=\mathcal{F}_a$. Now apply Theorem~\ref{theorem:general_lr_template}.
\end{proof}

\begin{theorem} \label{theorem:lrfslide}
  Let $a,b$ be weak compositions of lengths $p$ and $q$. Write
$$\mathfrak{F}_a^*\sqcupplus\mathfrak{F}_b^*=\sum_c s_{a,b}^c\,\mathfrak{F}_c^*$$
For each $c$, fix an RC graph $C$ with $\wtt(\QY(C)) = c$. Then $s_{a,b}^c$ is the number of distinct RC graphs $R$ with $\mathrm{word}(R)=\mathrm{word}(C)$ with $\wtt(\clp^p(R)) = a,\wtt(\trm^p(R)) = b$ such that $\clp^p(R),\trm^p(R)$ are quasi-Yamanouchi, and is in particular nonnegative.
\end{theorem}
\begin{proof}
This is Theorem~\ref{theorem:general_lr_template}(3) for the admissible relation of Proposition~\ref{proposition:slidebasis}, with $\mathrm{rep}=\QY$: the class of $C$ consists of the RC graphs of height $p+q$ with reading word $\mathrm{word}(C)$, and $\clp^p(R)$ is the representative of its class with label $a$ exactly when it is quasi-Yamanouchi of weight $a$, and likewise for $\trm^p(R)$.
\end{proof}

\section*{Notation and conventions}\label{section:notation}

For the reader's convenience we collect here, in one place, the principal symbols used throughout the article. All notation is introduced in context where it is needed; this list is meant as a navigational aid rather than a definition.

\begingroup
\setlist[description]{
  font=\normalfont,
  leftmargin=2em,
  labelindent=0em,
  labelsep=0.5em,
  itemsep=2pt,
  parsep=0pt,
  topsep=4pt,
}

\subsubsection*{Permutations and combinatorial objects}
\begin{description}
\item[$S_\infty$] The group of permutations of $\mathbb{Z}_{>0}$ fixing all but finitely many elements.
\item[$\ell(w)$] The Coxeter length of $w\in S_\infty$.
\item[$\code(w)$, $\code^*(w)$] The (right) Lehmer code of $w$; the code of $w^{-1}$.
\item[$\maxd(w)$] The position of the last descent of $w$.
\item[$\dom_\lambda$] The dominant permutation with $\code^*(\dom_\lambda)=\lambda$.
\item[$\mathsf{IndexedForests}_n$] The set of indexed forests on $[n]$ in the sense of Nadeau--Tewari \cite{nadeau2024forest}; matched bijectively with $\mathrm{Comp}_n$ via $F\leftrightarrow\code(F)$, so we use $a\in\mathrm{Comp}_n$ and $F_a\in\mathsf{IndexedForests}_n$ interchangeably.
\item[$\mathsf{Forest}_n$] The subset of $\mathsf{IndexedForests}_n$ corresponding to the homology basis $\{[X(F)]\}$ of $H_\bullet(\mathrm{QFl}_n)$ in BGNST \cite[Cor.~11.6]{bergeron2025qsymflag}; in NST language, those $F$ with $n\notin\mathrm{LT}(F)$.
\item[$\RC_n$] The set of bounded RC graphs with $n$ rows (Definition~\ref{definition:rcgraph}).
\item[$\RC_n(w)$] The set of $R\in\RC_n$ with $\wof{R}=w$.
\item[$\RC_n^0$] The set of $R\in\RC_n$ whose last row is empty (\S\ref{section:brc}).
\item[$\wof{R}$] The permutation of an RC graph $R$.
\item[$\hr(R)$] The number of rows of $R$.
\item[$\wtt(R)$] The row weight composition of $R$.
\item[$\mathrm{word}(R)$] The reduced word read from $R$.
\end{description}

\subsubsection*{Row maps on RC graphs}
\begin{description}
\item[$\zeromap$] The transition map $\RC_n^0\to\RC_{n-1}$ of Definition~\ref{definition:zeromap}; it realizes the Lascoux--Sch\"utzenberger transition formula on the last row.
\item[$\trm^p(R)$] ``Truncate'': the RC graph obtained from the last $\hr(R)-p$ rows of $R$ (Definition~\ref{definition:cliptrim}).
\item[$\clp^p(R)$] ``Clip'': the RC graph in $\RC_p$ obtained from the top $p$ rows of $R$ by applying $\zeromap$ until the height is $p$ (Definition~\ref{definition:cliptrim}).
\item[$\shup^p$] Shift indices up by $p$ (applied to permutations or letters).
\end{description}

\subsubsection*{Crystal and forest structures}
\begin{description}
\item[$e_i,f_i$] Crystal raising and lowering operators on RC graphs (Definition~\ref{definition:crystal}).
\item[$\mathrm{Dem}(R)$] The Demazure crystal of $R$ (Assaf--Schilling; Definition~\ref{definition:crystal}).
\item[$\mathrm{Dem}_a$] $\{R:\mathrm{extwt}(R)=a\}$, the set of RC graphs with extremal weight $a$ (\S\ref{section:crystals}).
\item[$\mathrm{extwt}(R)$] The extremal weight of $\mathrm{Dem}(R)$, i.e.\ the weight of its unique lowest weight element (Definition~\ref{definition:crystal}).
\item[$\mathbb{Y}(R)$] The highest-weight representative of $R$ in its Demazure crystal (Definition~\ref{definition:crystal}).
\item[$\mathbb{Y}_\forest(R)$] The forest analog of $\mathbb{Y}(R)$: the canonical representative of $R$ in its forest equivalence class, used in the dual forest LR rule (\S\ref{section:forest}).
\item[$\finv(R)$] The forest invariant of $R$, an indexed forest (Nadeau--Tewari; \S\ref{section:forest}).
\item[$\fcode(R)\,(=\code_\forest(R))$] The composition associated to $\finv(R)$ (\S\ref{section:forest}). The forms $\fcode(R)$ and $\code_\forest(R)$ are used interchangeably.
\item[$\mathcal{C}_\forest(a)$] The set of forest classes of RC graphs with $\fcode=a$ (\S\ref{section:forest}).
\item[$\equiv_\forest$] Forest equivalence: $R\equiv_\forest R'$ iff $\finv(R)=\finv(R')$, computed on the reversed reading word (Definition~\ref{definition:forestequiv}).
\item[$\RC_\forest(c)$] The set of $R\in\RC_n$ with $\wtt(R)=c$ such that $\mathrm{word}(\mathbb{Y}_\forest(R))$ is lex-minimal among $R'$ with $\code_\forest(R')=\code_\forest(R)$ (Theorem~\ref{theorem:forestformula}).
\end{description}

\subsubsection*{The Schubert bialgebra and its dual}
\begin{description}
\item[$\coma_n$] The polynomial ring $\mathbb{Z}[x_1,\ldots,x_n]$ (\S\ref{section:bialgebra}).
\item[$\coma$] The Schubert bialgebra $\bigoplus_{n\geq 0}\coma_n$, with $\coma_m\cdot\coma_n=0$ when $m\neq n>0$ (\S\ref{section:bialgebra}).
\item[$\dcoma$] The graded dual bialgebra of $\coma$ (\S\ref{section:dualalgebra}).
\item[$\sch_u^n$] Schubert polynomial in $\coma_n$, $\maxd(u)\leq n$ (\S\ref{section:bialgebra}).
\item[$\dsch_u^n$] Element of $\dcoma$ dual to $\sch_u^n$ (\S\ref{section:dualalgebra}).
\item[$\mathrm{ev}$] The evaluation map $\BRC\to\coma$ defined by $\mathrm{ev}(R)=x^{\wtt(R)}$ on basis elements (\S\ref{section:brc}).
\item[$\kappa_a$, $\kappa_a^*$] Key polynomial / its dual; $a\in\mathrm{Comp}_n$ (\S\ref{section:crystals}).
\item[$\forest_a$, $\forest_a^*$] Forest polynomial / its dual; $a\in\mathrm{Comp}_n$ (\S\ref{section:forest}).
\item[{$x_a$, $[a]$, $e_{\mathbf b}$}] The monomial $x^{a}=x_1^{a_1}\cdots x_n^{a_n}\in\coma_n$; its dual basis element in $\dcoma_n$, written either $[a]$ or $e_{\mathbf b}$ depending on context.
\item[$\langle -,-\rangle$] The bialgebra pairing $\dcoma\otimes\coma\to\mathbb{Z}$ defined by $\langle [a],x^{b}\rangle=\delta_{a,b}$, equivalently $\langle e_{\mathbf a},x^{\mathbf b}\rangle=\delta_{\mathbf a,\mathbf b}$.
\item[$\mathrm{Comp}_n$] The set of weak compositions of length $n$.
\item[$\Delta$, $\nabla$, $\varepsilon$] Coproduct, product, counit on $\coma$ (\S\ref{section:bialgebra}).
\item[$\Delta^*$] The dual coproduct on $\dcoma$, dual to the polynomial product on $\coma$ (\S\ref{section:dualalgebra}).
\end{description}

\subsubsection*{Structure constants}
\begin{description}
\item[$c_{u,v}^w$] Schubert structure constants: $\sch_u\sch_v=\sum c_{u,v}^w\sch_w$.
\item[$d_{u,v}^w(p,q)$] Dual Schubert structure constants of Theorem~\ref{theorem:LR}.
\item[$k_{a,b}^c$] Dual key structure constants of Theorem~\ref{theorem:LRkey}.
\item[$f_{a,b}^c$] Dual forest structure constants of Theorem~\ref{theorem:LRforest}.
\item[$s_{a,b}^c$] Dual slide structure constants of Theorem~\ref{theorem:lrfslide}.
\item[$b_{\mu,\nu}^\lambda$] The general dual LR structure constants of Theorem~\ref{theorem:general_lr_template}.
\end{description}

\subsubsection*{The ring of RC graphs and its quotients}
\begin{description}
\item[$\BRC$] $\bigoplus_{n\geq 0}\mathbb{Z}^{\oplus\RC_n}$, with the product $\sqcupplus$ (\S\ref{section:brc}).
\item[$\BRC_n$] The graded component $\mathbb{Z}^{\oplus \RC_n}$ of $\BRC$ (RC graphs of height exactly $n$).
\item[$\mathcal{S}_w(n)$] $\sum_{R:\,\wof{R}=w,\,\hr(R)=n}R\in\BRC_n$, the sum of all $n$-row RC graphs of permutation $w$; satisfies $\alpha(\mathcal{S}_w(n))=|\RC_n(w)|\,\dsch_w^n$ and $\mathrm{ev}(\mathcal{S}_w(n))=\sch_w^n$ (\S\ref{section:brc}).
\item[$K_a$] $\sum_{R\in\mathrm{Dem}_a,\,R=\mathbb{Y}(R)} R\in\BRC/\Theta$, the dual key class sum; $\alpha(K_a)=\kappa_a^*$ (\S\ref{section:crystals}).
\item[$F_a$] $\sum_{[R]\in\mathcal{C}_\forest(a)}[R]\in\BRC/\Gamma$, the dual forest class sum; $\alpha(F_a)=\forest_a^*$ (\S\ref{section:forest}).
\item[$\sqcupplus$] The associative concatenation product on $\BRC$, on $\dcoma$, and on the quotients $\BRC/I_\sim$; built from $\clp^p$ and $\trm^p$ on $\BRC$ (\S\ref{section:ringproduct}).
\item[$\alpha:\BRC\to\dcoma$] Ring homomorphism with $\alpha(R)=\dsch_{\wof{R}}^{\hr(R)}$ (\S\ref{section:brc}).
\item[$\Omega$] $\ker\alpha$, the two-sided ideal of $\BRC$ on which $\alpha$ vanishes (Definition~\ref{definition:omega}; Theorem~\ref{theorem:semidirect}).
\item[$\Theta$] The two-sided ideal generated by $R_1-R_2$ with $\mathbb{Y}(R_1)=\mathbb{Y}(R_2)$ (Definition~\ref{definition:theta}); $\BRC/\Theta$ supports the dual key LR rule (Theorem~\ref{theorem:LRkey}).
\item[$\Gamma$] The two-sided ideal generated by $R_1-R_2$ with $\hr(R_1)=\hr(R_2)$ and $R_1\equiv_\forest R_2$ (Definition~\ref{definition:gamma}); $\BRC/\Gamma$ supports the dual forest LR rule (Theorem~\ref{theorem:LRforest}).
\item[$\mathcal{E}_n$] The elementary basis of $\coma_n$ (Theorem~\ref{theorem:elemtrans}).
\end{description}

\subsubsection*{Slide polynomials and the general LR template}
\begin{description}
\item[$\mathfrak{F}_a(x)$, $\mathfrak{F}_a^*$] Slide polynomial in $\coma$ (Assaf--Searles) indexed by a weak composition $a$; its dual in $\dcoma$ (\S\ref{section:slide}).
\item[$\QY(R)$] Quasi-Yamanouchi representative of $R$ in its $(\mathrm{word},\hr)$-class; smallest element under the dominance/refinement order on weights (\S\ref{section:slide}).
\item[$\RC_\QY(c)$] Set of RC graphs $R$ with $\wtt(R)=c$ such that $\mathrm{word}(\QY(R))$ is lex-minimal among $R'$ with $\wtt(\QY(R'))=\wtt(\QY(R))$ (\S\ref{section:slide}).
\item[$\Sigma$] Two-sided ideal of $\BRC$ generated by $R_1-R_2$ with $\QY(R_1)=\QY(R_2)$ (Lemma~\ref{lemma:slideideal}).
\item[$\mathcal{F}_a$] The element $\sum_{R=\QY(R),\,\wtt(R)=a} R\in\BRC/\Sigma$, a representative of the slide basis under $\alpha$ (Proposition~\ref{proposition:slidebasis}).
\item[$\sim$] An admissible equivalence relation on $\bigsqcup_n\RC_n$, i.e.\ one satisfying (H1)--(H3) of Definition~\ref{definition:admissible}.
\item[$I_\sim$] The two-sided ideal $\mathrm{span}_{\mathbb{Z}}\{R_1-R_2:R_1\sim R_2\}\subseteq\BRC$ (Theorem~\ref{theorem:general_lr_template}).
\item[$J$, $\lambda$] Index set, resp.\ labeling map $(\RC/\!\sim)\to J$ in the general LR template (Theorem~\ref{theorem:general_lr_template}).
\item[$E_\lambda$] The class sum $\sum_{[R]:\,\lambda([R])=\lambda}[R]\in\BRC/I_\sim$ (Theorem~\ref{theorem:general_lr_template}).
\item[$B_\lambda^*$, $b_{\mu,\nu}^\lambda$] Image $\alpha(E_\lambda)\in\dcoma$ and the corresponding structure constants in Theorem~\ref{theorem:general_lr_template}.
\end{description}

\subsubsection*{Quasisymmetric flag variety (\S\ref{subsection:qflbranch})}
\begin{description}
\item[$\mathrm{QFl}_n$] The quasisymmetric flag variety of Bergeron--Gagnon--Nadeau--Spink--Tewari \cite{bergeron2025qsymflag}, a toric subcomplex of $\mathrm{Fl}_n$ stratified by $X(F)$ for $F\in\mathsf{Forest}_n$.
\item[{$X(F)$, $[X(F)]$}] The toric Richardson variety attached to $F\in\mathsf{Forest}_n$, and its homology class; $\{[X(F)]\}$ is a $\mathbb{Z}$-basis of $H_\bullet(\mathrm{QFl}_n)$ \cite[Cor.~11.6]{bergeron2025qsymflag}.
\item[$H^\bullet(\mathrm{QFl}_n)$, $H_\bullet(\mathrm{QFl}_n)$] Cohomology and homology of $\mathrm{QFl}_n$, paired by the integration map $\int_{X(F)}(\cdot)$.
\item[$\mathrm{QSym}_n^+$] The ideal of quasisymmetric polynomials in $\mathbb{Z}[x_1,\ldots,x_n]$ with no constant term; $\mathrm{QSym}_n^+\!\cdot\coma_n$ is the kernel of the Borel surjection.
\item[$\pi_n$] The Borel-type surjection $\coma_n\twoheadrightarrow\coma_n/(\mathrm{QSym}_n^+\!\cdot\coma_n)\cong H^\bullet(\mathrm{QFl}_n)$ \cite[Thm.~A and Cor.~12.5]{bergeron2025qsymflag}.
\item[{$[\forest_a]$}] The cohomology class $\pi_n(\forest_a)\in H^\bullet(\mathrm{QFl}_n)$; under the Borel isomorphism, $\{[\forest_a]:F_a\in\mathsf{Forest}_n\}$ is a $\mathbb{Z}$-basis of $H^\bullet(\mathrm{QFl}_n)$, Kronecker dual to $\{[X(F_a)]\}$.
\item[$\rho_{p,q}$, $\iota_{p,q}^*$] The variable-splitting substitution $\coma_{p+q}\to\coma_p\otimes\coma_q$, and the induced ring homomorphism
$$H^\bullet(\mathrm{QFl}_{p+q})\to H^\bullet(\mathrm{QFl}_p)\otimes H^\bullet(\mathrm{QFl}_q)$$
(\S\ref{subsection:qflbranch}).
\item[$\epsilon_F(\mathbf{b})$] The path sign of \cite[Prop.~10.15]{nadeau2024quasisymmetric}: the coefficient of $\forest_F$ in the forest expansion of $x^{\mathbf{b}}$, valued in $\{-1,0,+1\}$. By Proposition~\ref{proposition:forestdualcoefficients}, $\langle\forest_a^*,x^{\mathbf{b}}\rangle=\epsilon_{F_a}(\mathbf{b})$.
\item[$V_F(\boldsymbol\lambda)$] The volume polynomial of \cite[\S 10.5]{nadeau2024quasisymmetric}, dual to $\forest_F$ under the divided-power $D$-pairing; $V_F(\boldsymbol\lambda)=\sum_{\mathbf{b}}\epsilon_F(\mathbf{b})\,\boldsymbol\lambda^{\mathbf{b}}/\mathbf{b}!$.
\end{description}
\endgroup

\appendix
\section{Proof of Proposition \ref{proposition:pullindex}}

\renewcommand{\theproposition2}{\Alph{proposition2}}
\setcounter{lemma}{0}

We dedicate this section to the otherwise distracting proof of Proposition \ref{proposition:pullindex}. Recall from Definition~\ref{definition:divdiff} that $\delta^i=\frac{1}{y_i-y_{i+1}}(1-s_i)$ acts on the $y$ variables, that $\delta^v\sch_w(x;y)=(-1)^{\ell(v)}\sch_{vw}(x;y)$ whenever $\ell(vw)=\ell(w)-\ell(v)$, and that the skew operators $\delta_u^w$ are characterized by the Leibniz formula $\delta^w(pq)=\sum_{u\leq w}\delta^u(p)\,\delta_u^w(q)$.

\begin{proposition2}[Special case of Pieri formula {\cite[Theorem~7.1]{samuelmolev}}] \label{proposition:pieri}
Suppose $k\geq 1$ and $u,w\in S_\infty$. If $u\not\tom{k} w$, then
$$\ddv{y}_u^w\prod_{i=1}^k(x_1 - y_i)=0$$
If $u\tom{k} w$, define
$$Q = \{u(i)\mid i\leq k\mbox{ and }u(i)=w(i)\}$$
then
$$\ddv{y}_u^w\prod_{i=1}^k(x_1 - y_i)=(-1)^{\ell(u,w)}\prod_{q\in Q}(x_1 - y_q)$$
\end{proposition2}
\begin{proof}
	The cited theorem, in the special case $p=k$ (the Molev--Sagan Pieri formula for double Schubert polynomials), is stated for the operators $\partial_u^w$ on the $x$ variables applied to $\prod_{i=1}^k(x_i-y_1)$, with value $\prod_{q\in Q}(x_q-y_1)$. Exchanging the roles of $x$ and $y$ gives $\delta_u^w\prod_{i=1}^k(y_i-x_1)=\prod_{q\in Q}(y_q-x_1)$. Since $\prod_{i=1}^k(y_i-x_1)=(-1)^k\prod_{i=1}^k(x_1-y_i)$, $\prod_{q\in Q}(y_q-x_1)=(-1)^{|Q|}\prod_{q\in Q}(x_1-y_q)$, and $|Q|=k-\ell(u,w)$ (the positions $i\leq k$ moved by the $\ell(u,w)$ transpositions witnessing $u\tom{k}w$ are distinct), the stated sign follows.
\end{proof}

\begin{proof}[Proof of Proposition \ref{proposition:pullindex}]
Suppose $v\in S_n$. Since $\ell(vw_0(n))=\ell(w_0(n))-\ell(v)$, we have
$$\sch_v(x;y)=(-1)^{\ell(vw_0(n))}\,\ddv{y}^{vw_0(n)}(\sch_{w_0(n)}(x;y))$$
This is equal to
$$(-1)^{\ell(vw_0(n))}\,\ddv{y}^{vw_0(n)}\left(\sch_{s_{n+1-i}\cdots s_1w_0(n)}(x^{(i)};y)\prod_{j=1}^{n+1-i}(x_i - y_j)\right)$$
and, applying the Leibniz formula together with $\delta^u\sch_{s_{n+1-i}\cdots s_1w_0(n)}(x^{(i)};y)=(-1)^{\ell(u)}\sch_{v'}(x^{(i)};y)$ for $u=v'w_0(n)s_1\cdots s_{n+1-i}$, is also equal to
$$\sum_{\substack{v'\in S_\infty\\\ell(v'w_0(n)s_1\cdots s_{n+1-i})\\=\ell(w_0(n)s_1\cdots s_{n+1-i})-\ell(v')}}(-1)^{\ell(vw_0(n))+\ell(v'w_0(n)s_1\cdots s_{n+1-i})}\,\sch_{v'}(x^{(i)};y)\ddv{y}_{v'w_0(n)s_1\cdots s_{n+1-i}}^{vw_0(n)}\prod_{j=1}^{n+1-i}(x_i - y_j)$$
By the restricted Pieri formula (Proposition \ref{proposition:pieri}), for this to be nonzero we need $v'w_0(n)s_1\cdots s_{n+1-i}\allowbreak\tom{n+1-i}vw_0(n)$. We note that $v'w_0(n)s_1\cdots s_{n+1-i}\tom{n+1-i}vw_0(n)$ if and only if 
$$v\Tom{i}v'w_0(n)s_1\cdots s_{n+1-i}w_0(n)=v's_ns_{n-1}\cdots s_i$$
We have that $v's_n\cdots s_i$ is exactly $\varphi_{i,n}(v')$ since $v'\in S_n$, so we require that
$$v\Tom{i}\varphi_{i,n}(v')$$
so the sum is over all $v'$ with $v\downvar{i} v'$.

Applying Proposition \ref{proposition:pieri}, the skew operator contributes a further sign $(-1)^{\ell(v'w_0(n)s_1\cdots s_{n+1-i},\,vw_0(n))}$; together with the signs already present this is $(-1)^{2\ell(vw_0(n))}=1$, so the result is equal to
$$\sum_{v\downvar{i} v'} \sch_{v'}(x^{(i)};y)\prod_{a\in A(v',v)}(x_i - y_a)$$
where $A(v',v)$ is the set of all $vw_0(n)(j)$ such that $1\leq j\leq n+1-i$ and $v'w_0(n)s_1\cdots s_{n+1-i}(j)=vw_0(n)(j)$. These values are the same as at the indices that comprise the set of all $1\leq j\leq n+1-i$ such that
$$v'(n+2-s_1\cdots s_{n+1-i}(j))=v(n+2-j)$$
Applying the $s_1\cdots s_{n+1-i}$ to $j$, since $1\leq j\leq n+1-i$ we have that 
$$s_1\cdots s_{n+1-i}(j)=j+1$$
Hence we need
$$v'(n+2-(j+1))=v'(n+1-j)=v(n+2-j)$$
Replacing $j$ with $n+2-p$, the indices are the set of all $p$ such that $i<p\leq n+1$ and
$$v'(p-1)=v(p)$$
Thus $A(v',v)$ is the set of all $v(p)$ such that $p>i$ and $v'(p-1)=v(p)$, which is exactly $Q_i(v',v)$, and we are done.
\end{proof}

\section*{Acknowledgements}

The author is deeply grateful to Zachary Hamaker for many illuminating conversations during the development of this work, and in particular for introducing the author to forest polynomials in the first place, in addition to pointing out the connections of RC graph crystal operators to  Little's bumping algorithm. Thanks are also due to Frank Sottile, whose work, much of it now classical, has been a sustained guiding influence on the author's thinking about Schubert calculus, and whose generosity with his time over the years has shaped the perspective taken here.

\bibliographystyle{acm}
\bibliography{schubert_bialgebra,schubert_bialgebra_bpd}

\end{document}

%% file: forest_lr_R1.tikz
\begin{tikzpicture}[scale=0.4]
  \draw[lightgray, very thin, opacity=0.3] (13,0) -- (0,0);
  \draw[lightgray, very thin, opacity=0.3] (13,0) -- (13,13);
  \draw[lightgray, very thin, opacity=0.3] (13,1) -- (0,1);
  \draw[lightgray, very thin, opacity=0.3] (12,0) -- (12,13);
  \draw[lightgray, very thin, opacity=0.3] (13,2) -- (0,2);
  \draw[lightgray, very thin, opacity=0.3] (11,0) -- (11,13);
  \draw[lightgray, very thin, opacity=0.3] (13,3) -- (0,3);
  \draw[lightgray, very thin, opacity=0.3] (10,0) -- (10,13);
  \draw[lightgray, very thin, opacity=0.3] (13,4) -- (0,4);
  \draw[lightgray, very thin, opacity=0.3] (9,0) -- (9,13);
  \draw[lightgray, very thin, opacity=0.3] (13,5) -- (0,5);
  \draw[lightgray, very thin, opacity=0.3] (8,0) -- (8,13);
  \draw[lightgray, very thin, opacity=0.3] (13,6) -- (0,6);
  \draw[lightgray, very thin, opacity=0.3] (7,0) -- (7,13);
  \draw[lightgray, very thin, opacity=0.3] (13,7) -- (0,7);
  \draw[lightgray, very thin, opacity=0.3] (6,0) -- (6,13);
  \draw[lightgray, very thin, opacity=0.3] (13,8) -- (0,8);
  \draw[lightgray, very thin, opacity=0.3] (5,0) -- (5,13);
  \draw[lightgray, very thin, opacity=0.3] (13,9) -- (0,9);
  \draw[lightgray, very thin, opacity=0.3] (4,0) -- (4,13);
  \draw[lightgray, very thin, opacity=0.3] (13,10) -- (0,10);
  \draw[lightgray, very thin, opacity=0.3] (3,0) -- (3,13);
  \draw[lightgray, very thin, opacity=0.3] (13,11) -- (0,11);
  \draw[lightgray, very thin, opacity=0.3] (2,0) -- (2,13);
  \draw[lightgray, very thin, opacity=0.3] (13,12) -- (0,12);
  \draw[lightgray, very thin, opacity=0.3] (1,0) -- (1,13);
  \draw[lightgray, very thin, opacity=0.3] (13,13) -- (0,13);
  \draw[lightgray, very thin, opacity=0.3] (0,0) -- (0,13);
  \draw[blue, line width=1.0pt] (0,12.5) .. controls (0.3,12.5) and (0.5,12.7) .. (0.5,13);
  \draw[blue, line width=1.0pt] (0.5,12) .. controls (0.5,12.3) and (0.7,12.5) .. (1,12.5);
  \draw[blue, line width=1.0pt] (1,12.5) .. controls (1.3,12.5) and (1.5,12.7) .. (1.5,13);
  \draw[blue, line width=1.0pt] (1.5,12) .. controls (1.5,12.3) and (1.7,12.5) .. (2,12.5);
  \draw[blue, line width=1.0pt] (2,12.5) .. controls (2.3,12.5) and (2.5,12.7) .. (2.5,13);
  \draw[blue, line width=1.0pt] (2.5,12) .. controls (2.5,12.3) and (2.7,12.5) .. (3,12.5);
  \draw[blue, line width=1.0pt] (3,12.5) .. controls (3.3,12.5) and (3.5,12.7) .. (3.5,13);
  \draw[blue, line width=1.0pt] (3.5,12) .. controls (3.5,12.3) and (3.7,12.5) .. (4,12.5);
  \draw[blue, line width=1.0pt] (4,12.5) .. controls (4.3,12.5) and (4.5,12.7) .. (4.5,13);
  \draw[blue, line width=1.0pt] (4.5,12) .. controls (4.5,12.3) and (4.7,12.5) .. (5,12.5);
  \draw[blue, line width=1.0pt] (5,12.5) .. controls (5.3,12.5) and (5.5,12.7) .. (5.5,13);
  \draw[blue, line width=1.0pt] (5.5,12) .. controls (5.5,12.3) and (5.7,12.5) .. (6,12.5);
  \draw[blue, line width=1.0pt] (6,12.5) .. controls (6.3,12.5) and (6.5,12.7) .. (6.5,13);
  \draw[blue, line width=1.0pt] (6.5,12) .. controls (6.5,12.3) and (6.7,12.5) .. (7,12.5);
  \draw[blue, line width=1.0pt] (7,12.5) .. controls (7.3,12.5) and (7.5,12.7) .. (7.5,13);
  \draw[blue, line width=1.0pt] (7.5,12) .. controls (7.5,12.3) and (7.7,12.5) .. (8,12.5);
  \draw[blue, line width=1.0pt] (8,12.5) .. controls (8.3,12.5) and (8.5,12.7) .. (8.5,13);
  \draw[blue, line width=1.0pt] (8.5,12) .. controls (8.5,12.3) and (8.7,12.5) .. (9,12.5);
  \draw[blue, line width=1.0pt] (9,12.5) .. controls (9.3,12.5) and (9.5,12.7) .. (9.5,13);
  \draw[blue, line width=1.0pt] (9.5,12) .. controls (9.5,12.3) and (9.7,12.5) .. (10,12.5);
  \draw[blue, line width=1.0pt] (10,12.5) .. controls (10.3,12.5) and (10.5,12.7) .. (10.5,13);
  \draw[blue, line width=1.0pt] (10.5,12) .. controls (10.5,12.3) and (10.7,12.5) .. (11,12.5);
  \draw[blue, line width=1.0pt] (11,12.5) .. controls (11.3,12.5) and (11.5,12.7) .. (11.5,13);
  \draw[blue, line width=1.0pt] (11.5,12) .. controls (11.5,12.3) and (11.7,12.5) .. (12,12.5);
  \draw[blue, line width=1.0pt] (12,12.5) .. controls (12.3,12.5) and (12.5,12.7) .. (12.5,13);
  \draw[blue, line width=1.0pt] (0,11.5) .. controls (0.3,11.5) and (0.5,11.7) .. (0.5,12);
  \draw[blue, line width=1.0pt] (0.5,11) .. controls (0.5,11.3) and (0.7,11.5) .. (1,11.5);
  \draw[blue, line width=1.0pt] (1,11.5) .. controls (1.3,11.5) and (1.5,11.7) .. (1.5,12);
  \draw[blue, line width=1.0pt] (1.5,11) .. controls (1.5,11.3) and (1.7,11.5) .. (2,11.5);
  \draw[blue, line width=1.0pt] (2,11.5) .. controls (2.3,11.5) and (2.5,11.7) .. (2.5,12);
  \draw[blue, line width=1.0pt] (2.5,11) .. controls (2.5,11.3) and (2.7,11.5) .. (3,11.5);
  \draw[blue, line width=1.0pt] (3,11.5) .. controls (3.3,11.5) and (3.5,11.7) .. (3.5,12);
  \draw[blue, line width=1.0pt] (3.5,11) .. controls (3.5,11.3) and (3.7,11.5) .. (4,11.5);
  \draw[blue, line width=1.0pt] (4,11.5) .. controls (4.3,11.5) and (4.5,11.7) .. (4.5,12);
  \draw[blue, line width=1.0pt] (4.5,11) .. controls (4.5,11.3) and (4.7,11.5) .. (5,11.5);
  \draw[blue, line width=1.0pt] (5,11.5) .. controls (5.3,11.5) and (5.5,11.7) .. (5.5,12);
  \draw[blue, line width=1.0pt] (5.5,11) .. controls (5.5,11.3) and (5.7,11.5) .. (6,11.5);
  \draw[blue, line width=1.0pt] (6,11.5) .. controls (6.3,11.5) and (6.5,11.7) .. (6.5,12);
  \draw[blue, line width=1.0pt] (6.5,11) .. controls (6.5,11.3) and (6.7,11.5) .. (7,11.5);
  \draw[blue, line width=1.0pt] (7,11.5) .. controls (7.3,11.5) and (7.5,11.7) .. (7.5,12);
  \draw[blue, line width=1.0pt] (7.5,11) .. controls (7.5,11.3) and (7.7,11.5) .. (8,11.5);
  \draw[blue, line width=1.0pt] (8,11.5) .. controls (8.3,11.5) and (8.5,11.7) .. (8.5,12);
  \draw[blue, line width=1.0pt] (8.5,11) .. controls (8.5,11.3) and (8.7,11.5) .. (9,11.5);
  \draw[blue, line width=1.0pt] (9,11.5) .. controls (9.3,11.5) and (9.5,11.7) .. (9.5,12);
  \draw[blue, line width=1.0pt] (9.5,11) .. controls (9.5,11.3) and (9.7,11.5) .. (10,11.5);
  \draw[blue, line width=1.0pt] (10,11.5) .. controls (10.3,11.5) and (10.5,11.7) .. (10.5,12);
  \draw[blue, line width=1.0pt] (10.5,11) .. controls (10.5,11.3) and (10.7,11.5) .. (11,11.5);
  \draw[blue, line width=1.0pt] (11,11.5) .. controls (11.3,11.5) and (11.5,11.7) .. (11.5,12);
  \draw[blue, line width=1.0pt, line cap=round] (0,10.5) -- (1,10.5);
  \draw[blue, line width=1.0pt, line cap=round] (0.5,10) -- (0.5,11);
  \draw[blue, line width=1.0pt] (1,10.5) .. controls (1.3,10.5) and (1.5,10.7) .. (1.5,11);
  \draw[blue, line width=1.0pt] (1.5,10) .. controls (1.5,10.3) and (1.7,10.5) .. (2,10.5);
  \draw[blue, line width=1.0pt, line cap=round] (2,10.5) -- (3,10.5);
  \draw[blue, line width=1.0pt, line cap=round] (2.5,10) -- (2.5,11);
  \draw[blue, line width=1.0pt] (3,10.5) .. controls (3.3,10.5) and (3.5,10.7) .. (3.5,11);
  \draw[blue, line width=1.0pt] (3.5,10) .. controls (3.5,10.3) and (3.7,10.5) .. (4,10.5);
  \draw[blue, line width=1.0pt] (4,10.5) .. controls (4.3,10.5) and (4.5,10.7) .. (4.5,11);
  \draw[blue, line width=1.0pt] (4.5,10) .. controls (4.5,10.3) and (4.7,10.5) .. (5,10.5);
  \draw[blue, line width=1.0pt] (5,10.5) .. controls (5.3,10.5) and (5.5,10.7) .. (5.5,11);
  \draw[blue, line width=1.0pt] (5.5,10) .. controls (5.5,10.3) and (5.7,10.5) .. (6,10.5);
  \draw[blue, line width=1.0pt] (6,10.5) .. controls (6.3,10.5) and (6.5,10.7) .. (6.5,11);
  \draw[blue, line width=1.0pt] (6.5,10) .. controls (6.5,10.3) and (6.7,10.5) .. (7,10.5);
  \draw[blue, line width=1.0pt] (7,10.5) .. controls (7.3,10.5) and (7.5,10.7) .. (7.5,11);
  \draw[blue, line width=1.0pt] (7.5,10) .. controls (7.5,10.3) and (7.7,10.5) .. (8,10.5);
  \draw[blue, line width=1.0pt] (8,10.5) .. controls (8.3,10.5) and (8.5,10.7) .. (8.5,11);
  \draw[blue, line width=1.0pt] (8.5,10) .. controls (8.5,10.3) and (8.7,10.5) .. (9,10.5);
  \draw[blue, line width=1.0pt] (9,10.5) .. controls (9.3,10.5) and (9.5,10.7) .. (9.5,11);
  \draw[blue, line width=1.0pt] (9.5,10) .. controls (9.5,10.3) and (9.7,10.5) .. (10,10.5);
  \draw[blue, line width=1.0pt] (10,10.5) .. controls (10.3,10.5) and (10.5,10.7) .. (10.5,11);
  \draw[blue, line width=1.0pt] (0,9.5) .. controls (0.3,9.5) and (0.5,9.7) .. (0.5,10);
  \draw[blue, line width=1.0pt] (0.5,9) .. controls (0.5,9.3) and (0.7,9.5) .. (1,9.5);
  \draw[blue, line width=1.0pt] (1,9.5) .. controls (1.3,9.5) and (1.5,9.7) .. (1.5,10);
  \draw[blue, line width=1.0pt] (1.5,9) .. controls (1.5,9.3) and (1.7,9.5) .. (2,9.5);
  \draw[blue, line width=1.0pt] (2,9.5) .. controls (2.3,9.5) and (2.5,9.7) .. (2.5,10);
  \draw[blue, line width=1.0pt] (2.5,9) .. controls (2.5,9.3) and (2.7,9.5) .. (3,9.5);
  \draw[blue, line width=1.0pt] (3,9.5) .. controls (3.3,9.5) and (3.5,9.7) .. (3.5,10);
  \draw[blue, line width=1.0pt] (3.5,9) .. controls (3.5,9.3) and (3.7,9.5) .. (4,9.5);
  \draw[blue, line width=1.0pt] (4,9.5) .. controls (4.3,9.5) and (4.5,9.7) .. (4.5,10);
  \draw[blue, line width=1.0pt] (4.5,9) .. controls (4.5,9.3) and (4.7,9.5) .. (5,9.5);
  \draw[blue, line width=1.0pt] (5,9.5) .. controls (5.3,9.5) and (5.5,9.7) .. (5.5,10);
  \draw[blue, line width=1.0pt] (5.5,9) .. controls (5.5,9.3) and (5.7,9.5) .. (6,9.5);
  \draw[blue, line width=1.0pt] (6,9.5) .. controls (6.3,9.5) and (6.5,9.7) .. (6.5,10);
  \draw[blue, line width=1.0pt] (6.5,9) .. controls (6.5,9.3) and (6.7,9.5) .. (7,9.5);
  \draw[blue, line width=1.0pt] (7,9.5) .. controls (7.3,9.5) and (7.5,9.7) .. (7.5,10);
  \draw[blue, line width=1.0pt] (7.5,9) .. controls (7.5,9.3) and (7.7,9.5) .. (8,9.5);
  \draw[blue, line width=1.0pt] (8,9.5) .. controls (8.3,9.5) and (8.5,9.7) .. (8.5,10);
  \draw[blue, line width=1.0pt] (8.5,9) .. controls (8.5,9.3) and (8.7,9.5) .. (9,9.5);
  \draw[blue, line width=1.0pt] (9,9.5) .. controls (9.3,9.5) and (9.5,9.7) .. (9.5,10);
  \draw[blue, line width=1.0pt] (0,8.5) .. controls (0.3,8.5) and (0.5,8.7) .. (0.5,9);
  \draw[blue, line width=1.0pt] (0.5,8) .. controls (0.5,8.3) and (0.7,8.5) .. (1,8.5);
  \draw[blue, line width=1.0pt] (1,8.5) .. controls (1.3,8.5) and (1.5,8.7) .. (1.5,9);
  \draw[blue, line width=1.0pt] (1.5,8) .. controls (1.5,8.3) and (1.7,8.5) .. (2,8.5);
  \draw[blue, line width=1.0pt] (2,8.5) .. controls (2.3,8.5) and (2.5,8.7) .. (2.5,9);
  \draw[blue, line width=1.0pt] (2.5,8) .. controls (2.5,8.3) and (2.7,8.5) .. (3,8.5);
  \draw[blue, line width=1.0pt, line cap=round] (3,8.5) -- (4,8.5);
  \draw[blue, line width=1.0pt, line cap=round] (3.5,8) -- (3.5,9);
  \draw[blue, line width=1.0pt] (4,8.5) .. controls (4.3,8.5) and (4.5,8.7) .. (4.5,9);
  \draw[blue, line width=1.0pt] (4.5,8) .. controls (4.5,8.3) and (4.7,8.5) .. (5,8.5);
  \draw[blue, line width=1.0pt] (5,8.5) .. controls (5.3,8.5) and (5.5,8.7) .. (5.5,9);
  \draw[blue, line width=1.0pt] (5.5,8) .. controls (5.5,8.3) and (5.7,8.5) .. (6,8.5);
  \draw[blue, line width=1.0pt] (6,8.5) .. controls (6.3,8.5) and (6.5,8.7) .. (6.5,9);
  \draw[blue, line width=1.0pt] (6.5,8) .. controls (6.5,8.3) and (6.7,8.5) .. (7,8.5);
  \draw[blue, line width=1.0pt, line cap=round] (7,8.5) -- (8,8.5);
  \draw[blue, line width=1.0pt, line cap=round] (7.5,8) -- (7.5,9);
  \draw[blue, line width=1.0pt] (8,8.5) .. controls (8.3,8.5) and (8.5,8.7) .. (8.5,9);
  \draw[blue, line width=1.0pt] (0,7.5) .. controls (0.3,7.5) and (0.5,7.7) .. (0.5,8);
  \draw[blue, line width=1.0pt] (0.5,7) .. controls (0.5,7.3) and (0.7,7.5) .. (1,7.5);
  \draw[blue, line width=1.0pt] (1,7.5) .. controls (1.3,7.5) and (1.5,7.7) .. (1.5,8);
  \draw[blue, line width=1.0pt] (1.5,7) .. controls (1.5,7.3) and (1.7,7.5) .. (2,7.5);
  \draw[blue, line width=1.0pt] (2,7.5) .. controls (2.3,7.5) and (2.5,7.7) .. (2.5,8);
  \draw[blue, line width=1.0pt] (2.5,7) .. controls (2.5,7.3) and (2.7,7.5) .. (3,7.5);
  \draw[blue, line width=1.0pt] (3,7.5) .. controls (3.3,7.5) and (3.5,7.7) .. (3.5,8);
  \draw[blue, line width=1.0pt] (3.5,7) .. controls (3.5,7.3) and (3.7,7.5) .. (4,7.5);
  \draw[blue, line width=1.0pt] (4,7.5) .. controls (4.3,7.5) and (4.5,7.7) .. (4.5,8);
  \draw[blue, line width=1.0pt] (4.5,7) .. controls (4.5,7.3) and (4.7,7.5) .. (5,7.5);
  \draw[blue, line width=1.0pt] (5,7.5) .. controls (5.3,7.5) and (5.5,7.7) .. (5.5,8);
  \draw[blue, line width=1.0pt] (5.5,7) .. controls (5.5,7.3) and (5.7,7.5) .. (6,7.5);
  \draw[blue, line width=1.0pt] (6,7.5) .. controls (6.3,7.5) and (6.5,7.7) .. (6.5,8);
  \draw[blue, line width=1.0pt] (6.5,7) .. controls (6.5,7.3) and (6.7,7.5) .. (7,7.5);
  \draw[blue, line width=1.0pt] (7,7.5) .. controls (7.3,7.5) and (7.5,7.7) .. (7.5,8);
  \draw[blue, line width=1.0pt, line cap=round] (0,6.5) -- (1,6.5);
  \draw[blue, line width=1.0pt, line cap=round] (0.5,6) -- (0.5,7);
  \draw[blue, line width=1.0pt] (1,6.5) .. controls (1.3,6.5) and (1.5,6.7) .. (1.5,7);
  \draw[blue, line width=1.0pt] (1.5,6) .. controls (1.5,6.3) and (1.7,6.5) .. (2,6.5);
  \draw[blue, line width=1.0pt] (2,6.5) .. controls (2.3,6.5) and (2.5,6.7) .. (2.5,7);
  \draw[blue, line width=1.0pt] (2.5,6) .. controls (2.5,6.3) and (2.7,6.5) .. (3,6.5);
  \draw[blue, line width=1.0pt] (3,6.5) .. controls (3.3,6.5) and (3.5,6.7) .. (3.5,7);
  \draw[blue, line width=1.0pt] (3.5,6) .. controls (3.5,6.3) and (3.7,6.5) .. (4,6.5);
  \draw[blue, line width=1.0pt] (4,6.5) .. controls (4.3,6.5) and (4.5,6.7) .. (4.5,7);
  \draw[blue, line width=1.0pt] (4.5,6) .. controls (4.5,6.3) and (4.7,6.5) .. (5,6.5);
  \draw[blue, line width=1.0pt] (5,6.5) .. controls (5.3,6.5) and (5.5,6.7) .. (5.5,7);
  \draw[blue, line width=1.0pt] (5.5,6) .. controls (5.5,6.3) and (5.7,6.5) .. (6,6.5);
  \draw[blue, line width=1.0pt] (6,6.5) .. controls (6.3,6.5) and (6.5,6.7) .. (6.5,7);
  \draw[blue, line width=1.0pt] (0,5.5) .. controls (0.3,5.5) and (0.5,5.7) .. (0.5,6);
  \draw[blue, line width=1.0pt] (0.5,5) .. controls (0.5,5.3) and (0.7,5.5) .. (1,5.5);
  \draw[blue, line width=1.0pt] (1,5.5) .. controls (1.3,5.5) and (1.5,5.7) .. (1.5,6);
  \draw[blue, line width=1.0pt] (1.5,5) .. controls (1.5,5.3) and (1.7,5.5) .. (2,5.5);
  \draw[blue, line width=1.0pt] (2,5.5) .. controls (2.3,5.5) and (2.5,5.7) .. (2.5,6);
  \draw[blue, line width=1.0pt] (2.5,5) .. controls (2.5,5.3) and (2.7,5.5) .. (3,5.5);
  \draw[blue, line width=1.0pt] (3,5.5) .. controls (3.3,5.5) and (3.5,5.7) .. (3.5,6);
  \draw[blue, line width=1.0pt] (3.5,5) .. controls (3.5,5.3) and (3.7,5.5) .. (4,5.5);
  \draw[blue, line width=1.0pt] (4,5.5) .. controls (4.3,5.5) and (4.5,5.7) .. (4.5,6);
  \draw[blue, line width=1.0pt] (4.5,5) .. controls (4.5,5.3) and (4.7,5.5) .. (5,5.5);
  \draw[blue, line width=1.0pt] (5,5.5) .. controls (5.3,5.5) and (5.5,5.7) .. (5.5,6);
  \draw[blue, line width=1.0pt] (0,4.5) .. controls (0.3,4.5) and (0.5,4.7) .. (0.5,5);
  \draw[blue, line width=1.0pt] (0.5,4) .. controls (0.5,4.3) and (0.7,4.5) .. (1,4.5);
  \draw[blue, line width=1.0pt] (1,4.5) .. controls (1.3,4.5) and (1.5,4.7) .. (1.5,5);
  \draw[blue, line width=1.0pt] (1.5,4) .. controls (1.5,4.3) and (1.7,4.5) .. (2,4.5);
  \draw[blue, line width=1.0pt] (2,4.5) .. controls (2.3,4.5) and (2.5,4.7) .. (2.5,5);
  \draw[blue, line width=1.0pt] (2.5,4) .. controls (2.5,4.3) and (2.7,4.5) .. (3,4.5);
  \draw[blue, line width=1.0pt] (3,4.5) .. controls (3.3,4.5) and (3.5,4.7) .. (3.5,5);
  \draw[blue, line width=1.0pt] (3.5,4) .. controls (3.5,4.3) and (3.7,4.5) .. (4,4.5);
  \draw[blue, line width=1.0pt] (4,4.5) .. controls (4.3,4.5) and (4.5,4.7) .. (4.5,5);
  \draw[blue, line width=1.0pt, line cap=round] (0,3.5) -- (1,3.5);
  \draw[blue, line width=1.0pt, line cap=round] (0.5,3) -- (0.5,4);
  \draw[blue, line width=1.0pt, line cap=round] (1,3.5) -- (2,3.5);
  \draw[blue, line width=1.0pt, line cap=round] (1.5,3) -- (1.5,4);
  \draw[blue, line width=1.0pt] (2,3.5) .. controls (2.3,3.5) and (2.5,3.7) .. (2.5,4);
  \draw[blue, line width=1.0pt] (2.5,3) .. controls (2.5,3.3) and (2.7,3.5) .. (3,3.5);
  \draw[blue, line width=1.0pt] (3,3.5) .. controls (3.3,3.5) and (3.5,3.7) .. (3.5,4);
  \draw[blue, line width=1.0pt] (0,2.5) .. controls (0.3,2.5) and (0.5,2.7) .. (0.5,3);
  \draw[blue, line width=1.0pt] (0.5,2) .. controls (0.5,2.3) and (0.7,2.5) .. (1,2.5);
  \draw[blue, line width=1.0pt] (1,2.5) .. controls (1.3,2.5) and (1.5,2.7) .. (1.5,3);
  \draw[blue, line width=1.0pt] (1.5,2) .. controls (1.5,2.3) and (1.7,2.5) .. (2,2.5);
  \draw[blue, line width=1.0pt] (2,2.5) .. controls (2.3,2.5) and (2.5,2.7) .. (2.5,3);
  \draw[blue, line width=1.0pt] (0,1.5) .. controls (0.3,1.5) and (0.5,1.7) .. (0.5,2);
  \draw[blue, line width=1.0pt] (0.5,1) .. controls (0.5,1.3) and (0.7,1.5) .. (1,1.5);
  \draw[blue, line width=1.0pt] (1,1.5) .. controls (1.3,1.5) and (1.5,1.7) .. (1.5,2);
  \draw[blue, line width=1.0pt] (0,0.5) .. controls (0.3,0.5) and (0.5,0.7) .. (0.5,1);
  \node[font=\tiny\bfseries, text=blue, anchor=south] at (0.5,13.3) {1};
  \node[font=\tiny\bfseries, text=blue, anchor=south] at (1.5,13.3) {2};
  \node[font=\tiny\bfseries, text=blue, anchor=south] at (2.5,13.3) {3};
  \node[font=\tiny\bfseries, text=blue, anchor=south] at (3.5,13.3) {4};
  \node[font=\tiny\bfseries, text=blue, anchor=south] at (4.5,13.3) {5};
  \node[font=\tiny\bfseries, text=blue, anchor=south] at (5.5,13.3) {6};
  \node[font=\tiny\bfseries, text=blue, anchor=south] at (6.5,13.3) {7};
  \node[font=\tiny\bfseries, text=blue, anchor=south] at (7.5,13.3) {8};
  \node[font=\tiny\bfseries, text=blue, anchor=south] at (8.5,13.3) {9};
  \node[font=\tiny\bfseries, text=blue, anchor=south] at (9.5,13.3) {10};
  \node[font=\tiny\bfseries, text=blue, anchor=south] at (10.5,13.3) {11};
  \node[font=\tiny\bfseries, text=blue, anchor=south] at (11.5,13.3) {12};
  \node[font=\tiny\bfseries, text=blue, anchor=south] at (12.5,13.3) {13};
  \node[font=\tiny\bfseries, text=blue, anchor=east] at (-0.3,12.5) {1};
  \node[font=\tiny\bfseries, text=blue, anchor=east] at (-0.3,11.5) {2};
  \node[font=\tiny\bfseries, text=blue, anchor=east] at (-0.3,10.5) {4};
  \node[font=\tiny\bfseries, text=blue, anchor=east] at (-0.3,9.5) {3};
  \node[font=\tiny\bfseries, text=blue, anchor=east] at (-0.3,8.5) {6};
  \node[font=\tiny\bfseries, text=blue, anchor=east] at (-0.3,7.5) {5};
  \node[font=\tiny\bfseries, text=blue, anchor=east] at (-0.3,6.5) {9};
  \node[font=\tiny\bfseries, text=blue, anchor=east] at (-0.3,5.5) {7};
  \node[font=\tiny\bfseries, text=blue, anchor=east] at (-0.3,4.5) {8};
  \node[font=\tiny\bfseries, text=blue, anchor=east] at (-0.3,3.5) {13};
  \node[font=\tiny\bfseries, text=blue, anchor=east] at (-0.3,2.5) {10};
  \node[font=\tiny\bfseries, text=blue, anchor=east] at (-0.3,1.5) {11};
  \node[font=\tiny\bfseries, text=blue, anchor=east] at (-0.3,0.5) {12};
  \draw[black, line width=1.2000000000000002pt] (13,8) rectangle (0,13);
  \node[font=\Large\bfseries, text=green!60!black, fill=white, inner sep=0.5pt, circle, transform shape] at (0.5,10.5) {3};
  \node[font=\Large\bfseries, text=green!60!black, fill=white, inner sep=0.5pt, circle, transform shape] at (2.5,10.5) {5};
  \node[font=\Large\bfseries, text=green!60!black, fill=white, inner sep=0.5pt, circle, transform shape] at (3.5,8.5) {8};
  \node[font=\Large\bfseries, text=green!60!black, fill=white, inner sep=0.5pt, circle, transform shape] at (7.5,8.5) {12};
  \node[font=\Large\bfseries, text=green!60!black, fill=white, inner sep=0.5pt, circle, transform shape] at (0.5,6.5) {7};
  \node[font=\Large\bfseries, text=green!60!black, fill=white, inner sep=0.5pt, circle, transform shape] at (0.5,3.5) {10};
  \node[font=\Large\bfseries, text=green!60!black, fill=white, inner sep=0.5pt, circle, transform shape] at (1.5,3.5) {11};
\end{tikzpicture}

%% file: forest_lr_R2.tikz
\begin{tikzpicture}[scale=0.4]
  \draw[lightgray, very thin, opacity=0.3] (13,0) -- (0,0);
  \draw[lightgray, very thin, opacity=0.3] (13,0) -- (13,13);
  \draw[lightgray, very thin, opacity=0.3] (13,1) -- (0,1);
  \draw[lightgray, very thin, opacity=0.3] (12,0) -- (12,13);
  \draw[lightgray, very thin, opacity=0.3] (13,2) -- (0,2);
  \draw[lightgray, very thin, opacity=0.3] (11,0) -- (11,13);
  \draw[lightgray, very thin, opacity=0.3] (13,3) -- (0,3);
  \draw[lightgray, very thin, opacity=0.3] (10,0) -- (10,13);
  \draw[lightgray, very thin, opacity=0.3] (13,4) -- (0,4);
  \draw[lightgray, very thin, opacity=0.3] (9,0) -- (9,13);
  \draw[lightgray, very thin, opacity=0.3] (13,5) -- (0,5);
  \draw[lightgray, very thin, opacity=0.3] (8,0) -- (8,13);
  \draw[lightgray, very thin, opacity=0.3] (13,6) -- (0,6);
  \draw[lightgray, very thin, opacity=0.3] (7,0) -- (7,13);
  \draw[lightgray, very thin, opacity=0.3] (13,7) -- (0,7);
  \draw[lightgray, very thin, opacity=0.3] (6,0) -- (6,13);
  \draw[lightgray, very thin, opacity=0.3] (13,8) -- (0,8);
  \draw[lightgray, very thin, opacity=0.3] (5,0) -- (5,13);
  \draw[lightgray, very thin, opacity=0.3] (13,9) -- (0,9);
  \draw[lightgray, very thin, opacity=0.3] (4,0) -- (4,13);
  \draw[lightgray, very thin, opacity=0.3] (13,10) -- (0,10);
  \draw[lightgray, very thin, opacity=0.3] (3,0) -- (3,13);
  \draw[lightgray, very thin, opacity=0.3] (13,11) -- (0,11);
  \draw[lightgray, very thin, opacity=0.3] (2,0) -- (2,13);
  \draw[lightgray, very thin, opacity=0.3] (13,12) -- (0,12);
  \draw[lightgray, very thin, opacity=0.3] (1,0) -- (1,13);
  \draw[lightgray, very thin, opacity=0.3] (13,13) -- (0,13);
  \draw[lightgray, very thin, opacity=0.3] (0,0) -- (0,13);
  \draw[blue, line width=1.0pt] (0,12.5) .. controls (0.3,12.5) and (0.5,12.7) .. (0.5,13);
  \draw[blue, line width=1.0pt] (0.5,12) .. controls (0.5,12.3) and (0.7,12.5) .. (1,12.5);
  \draw[blue, line width=1.0pt] (1,12.5) .. controls (1.3,12.5) and (1.5,12.7) .. (1.5,13);
  \draw[blue, line width=1.0pt] (1.5,12) .. controls (1.5,12.3) and (1.7,12.5) .. (2,12.5);
  \draw[blue, line width=1.0pt] (2,12.5) .. controls (2.3,12.5) and (2.5,12.7) .. (2.5,13);
  \draw[blue, line width=1.0pt] (2.5,12) .. controls (2.5,12.3) and (2.7,12.5) .. (3,12.5);
  \draw[blue, line width=1.0pt] (3,12.5) .. controls (3.3,12.5) and (3.5,12.7) .. (3.5,13);
  \draw[blue, line width=1.0pt] (3.5,12) .. controls (3.5,12.3) and (3.7,12.5) .. (4,12.5);
  \draw[blue, line width=1.0pt] (4,12.5) .. controls (4.3,12.5) and (4.5,12.7) .. (4.5,13);
  \draw[blue, line width=1.0pt] (4.5,12) .. controls (4.5,12.3) and (4.7,12.5) .. (5,12.5);
  \draw[blue, line width=1.0pt] (5,12.5) .. controls (5.3,12.5) and (5.5,12.7) .. (5.5,13);
  \draw[blue, line width=1.0pt] (5.5,12) .. controls (5.5,12.3) and (5.7,12.5) .. (6,12.5);
  \draw[blue, line width=1.0pt] (6,12.5) .. controls (6.3,12.5) and (6.5,12.7) .. (6.5,13);
  \draw[blue, line width=1.0pt] (6.5,12) .. controls (6.5,12.3) and (6.7,12.5) .. (7,12.5);
  \draw[blue, line width=1.0pt] (7,12.5) .. controls (7.3,12.5) and (7.5,12.7) .. (7.5,13);
  \draw[blue, line width=1.0pt] (7.5,12) .. controls (7.5,12.3) and (7.7,12.5) .. (8,12.5);
  \draw[blue, line width=1.0pt] (8,12.5) .. controls (8.3,12.5) and (8.5,12.7) .. (8.5,13);
  \draw[blue, line width=1.0pt] (8.5,12) .. controls (8.5,12.3) and (8.7,12.5) .. (9,12.5);
  \draw[blue, line width=1.0pt] (9,12.5) .. controls (9.3,12.5) and (9.5,12.7) .. (9.5,13);
  \draw[blue, line width=1.0pt] (9.5,12) .. controls (9.5,12.3) and (9.7,12.5) .. (10,12.5);
  \draw[blue, line width=1.0pt] (10,12.5) .. controls (10.3,12.5) and (10.5,12.7) .. (10.5,13);
  \draw[blue, line width=1.0pt] (10.5,12) .. controls (10.5,12.3) and (10.7,12.5) .. (11,12.5);
  \draw[blue, line width=1.0pt] (11,12.5) .. controls (11.3,12.5) and (11.5,12.7) .. (11.5,13);
  \draw[blue, line width=1.0pt] (11.5,12) .. controls (11.5,12.3) and (11.7,12.5) .. (12,12.5);
  \draw[blue, line width=1.0pt] (12,12.5) .. controls (12.3,12.5) and (12.5,12.7) .. (12.5,13);
  \draw[blue, line width=1.0pt] (0,11.5) .. controls (0.3,11.5) and (0.5,11.7) .. (0.5,12);
  \draw[blue, line width=1.0pt] (0.5,11) .. controls (0.5,11.3) and (0.7,11.5) .. (1,11.5);
  \draw[blue, line width=1.0pt] (1,11.5) .. controls (1.3,11.5) and (1.5,11.7) .. (1.5,12);
  \draw[blue, line width=1.0pt] (1.5,11) .. controls (1.5,11.3) and (1.7,11.5) .. (2,11.5);
  \draw[blue, line width=1.0pt] (2,11.5) .. controls (2.3,11.5) and (2.5,11.7) .. (2.5,12);
  \draw[blue, line width=1.0pt] (2.5,11) .. controls (2.5,11.3) and (2.7,11.5) .. (3,11.5);
  \draw[blue, line width=1.0pt] (3,11.5) .. controls (3.3,11.5) and (3.5,11.7) .. (3.5,12);
  \draw[blue, line width=1.0pt] (3.5,11) .. controls (3.5,11.3) and (3.7,11.5) .. (4,11.5);
  \draw[blue, line width=1.0pt] (4,11.5) .. controls (4.3,11.5) and (4.5,11.7) .. (4.5,12);
  \draw[blue, line width=1.0pt] (4.5,11) .. controls (4.5,11.3) and (4.7,11.5) .. (5,11.5);
  \draw[blue, line width=1.0pt] (5,11.5) .. controls (5.3,11.5) and (5.5,11.7) .. (5.5,12);
  \draw[blue, line width=1.0pt] (5.5,11) .. controls (5.5,11.3) and (5.7,11.5) .. (6,11.5);
  \draw[blue, line width=1.0pt] (6,11.5) .. controls (6.3,11.5) and (6.5,11.7) .. (6.5,12);
  \draw[blue, line width=1.0pt] (6.5,11) .. controls (6.5,11.3) and (6.7,11.5) .. (7,11.5);
  \draw[blue, line width=1.0pt] (7,11.5) .. controls (7.3,11.5) and (7.5,11.7) .. (7.5,12);
  \draw[blue, line width=1.0pt] (7.5,11) .. controls (7.5,11.3) and (7.7,11.5) .. (8,11.5);
  \draw[blue, line width=1.0pt] (8,11.5) .. controls (8.3,11.5) and (8.5,11.7) .. (8.5,12);
  \draw[blue, line width=1.0pt] (8.5,11) .. controls (8.5,11.3) and (8.7,11.5) .. (9,11.5);
  \draw[blue, line width=1.0pt] (9,11.5) .. controls (9.3,11.5) and (9.5,11.7) .. (9.5,12);
  \draw[blue, line width=1.0pt] (9.5,11) .. controls (9.5,11.3) and (9.7,11.5) .. (10,11.5);
  \draw[blue, line width=1.0pt] (10,11.5) .. controls (10.3,11.5) and (10.5,11.7) .. (10.5,12);
  \draw[blue, line width=1.0pt] (10.5,11) .. controls (10.5,11.3) and (10.7,11.5) .. (11,11.5);
  \draw[blue, line width=1.0pt] (11,11.5) .. controls (11.3,11.5) and (11.5,11.7) .. (11.5,12);
  \draw[blue, line width=1.0pt, line cap=round] (0,10.5) -- (1,10.5);
  \draw[blue, line width=1.0pt, line cap=round] (0.5,10) -- (0.5,11);
  \draw[blue, line width=1.0pt] (1,10.5) .. controls (1.3,10.5) and (1.5,10.7) .. (1.5,11);
  \draw[blue, line width=1.0pt] (1.5,10) .. controls (1.5,10.3) and (1.7,10.5) .. (2,10.5);
  \draw[blue, line width=1.0pt] (2,10.5) .. controls (2.3,10.5) and (2.5,10.7) .. (2.5,11);
  \draw[blue, line width=1.0pt] (2.5,10) .. controls (2.5,10.3) and (2.7,10.5) .. (3,10.5);
  \draw[blue, line width=1.0pt] (3,10.5) .. controls (3.3,10.5) and (3.5,10.7) .. (3.5,11);
  \draw[blue, line width=1.0pt] (3.5,10) .. controls (3.5,10.3) and (3.7,10.5) .. (4,10.5);
  \draw[blue, line width=1.0pt] (4,10.5) .. controls (4.3,10.5) and (4.5,10.7) .. (4.5,11);
  \draw[blue, line width=1.0pt] (4.5,10) .. controls (4.5,10.3) and (4.7,10.5) .. (5,10.5);
  \draw[blue, line width=1.0pt, line cap=round] (5,10.5) -- (6,10.5);
  \draw[blue, line width=1.0pt, line cap=round] (5.5,10) -- (5.5,11);
  \draw[blue, line width=1.0pt] (6,10.5) .. controls (6.3,10.5) and (6.5,10.7) .. (6.5,11);
  \draw[blue, line width=1.0pt] (6.5,10) .. controls (6.5,10.3) and (6.7,10.5) .. (7,10.5);
  \draw[blue, line width=1.0pt] (7,10.5) .. controls (7.3,10.5) and (7.5,10.7) .. (7.5,11);
  \draw[blue, line width=1.0pt] (7.5,10) .. controls (7.5,10.3) and (7.7,10.5) .. (8,10.5);
  \draw[blue, line width=1.0pt] (8,10.5) .. controls (8.3,10.5) and (8.5,10.7) .. (8.5,11);
  \draw[blue, line width=1.0pt] (8.5,10) .. controls (8.5,10.3) and (8.7,10.5) .. (9,10.5);
  \draw[blue, line width=1.0pt] (9,10.5) .. controls (9.3,10.5) and (9.5,10.7) .. (9.5,11);
  \draw[blue, line width=1.0pt] (9.5,10) .. controls (9.5,10.3) and (9.7,10.5) .. (10,10.5);
  \draw[blue, line width=1.0pt] (10,10.5) .. controls (10.3,10.5) and (10.5,10.7) .. (10.5,11);
  \draw[blue, line width=1.0pt] (0,9.5) .. controls (0.3,9.5) and (0.5,9.7) .. (0.5,10);
  \draw[blue, line width=1.0pt] (0.5,9) .. controls (0.5,9.3) and (0.7,9.5) .. (1,9.5);
  \draw[blue, line width=1.0pt] (1,9.5) .. controls (1.3,9.5) and (1.5,9.7) .. (1.5,10);
  \draw[blue, line width=1.0pt] (1.5,9) .. controls (1.5,9.3) and (1.7,9.5) .. (2,9.5);
  \draw[blue, line width=1.0pt] (2,9.5) .. controls (2.3,9.5) and (2.5,9.7) .. (2.5,10);
  \draw[blue, line width=1.0pt] (2.5,9) .. controls (2.5,9.3) and (2.7,9.5) .. (3,9.5);
  \draw[blue, line width=1.0pt] (3,9.5) .. controls (3.3,9.5) and (3.5,9.7) .. (3.5,10);
  \draw[blue, line width=1.0pt] (3.5,9) .. controls (3.5,9.3) and (3.7,9.5) .. (4,9.5);
  \draw[blue, line width=1.0pt] (4,9.5) .. controls (4.3,9.5) and (4.5,9.7) .. (4.5,10);
  \draw[blue, line width=1.0pt] (4.5,9) .. controls (4.5,9.3) and (4.7,9.5) .. (5,9.5);
  \draw[blue, line width=1.0pt] (5,9.5) .. controls (5.3,9.5) and (5.5,9.7) .. (5.5,10);
  \draw[blue, line width=1.0pt] (5.5,9) .. controls (5.5,9.3) and (5.7,9.5) .. (6,9.5);
  \draw[blue, line width=1.0pt] (6,9.5) .. controls (6.3,9.5) and (6.5,9.7) .. (6.5,10);
  \draw[blue, line width=1.0pt] (6.5,9) .. controls (6.5,9.3) and (6.7,9.5) .. (7,9.5);
  \draw[blue, line width=1.0pt] (7,9.5) .. controls (7.3,9.5) and (7.5,9.7) .. (7.5,10);
  \draw[blue, line width=1.0pt] (7.5,9) .. controls (7.5,9.3) and (7.7,9.5) .. (8,9.5);
  \draw[blue, line width=1.0pt] (8,9.5) .. controls (8.3,9.5) and (8.5,9.7) .. (8.5,10);
  \draw[blue, line width=1.0pt] (8.5,9) .. controls (8.5,9.3) and (8.7,9.5) .. (9,9.5);
  \draw[blue, line width=1.0pt] (9,9.5) .. controls (9.3,9.5) and (9.5,9.7) .. (9.5,10);
  \draw[blue, line width=1.0pt, line cap=round] (0,8.5) -- (1,8.5);
  \draw[blue, line width=1.0pt, line cap=round] (0.5,8) -- (0.5,9);
  \draw[blue, line width=1.0pt] (1,8.5) .. controls (1.3,8.5) and (1.5,8.7) .. (1.5,9);
  \draw[blue, line width=1.0pt] (1.5,8) .. controls (1.5,8.3) and (1.7,8.5) .. (2,8.5);
  \draw[blue, line width=1.0pt] (2,8.5) .. controls (2.3,8.5) and (2.5,8.7) .. (2.5,9);
  \draw[blue, line width=1.0pt] (2.5,8) .. controls (2.5,8.3) and (2.7,8.5) .. (3,8.5);
  \draw[blue, line width=1.0pt] (3,8.5) .. controls (3.3,8.5) and (3.5,8.7) .. (3.5,9);
  \draw[blue, line width=1.0pt] (3.5,8) .. controls (3.5,8.3) and (3.7,8.5) .. (4,8.5);
  \draw[blue, line width=1.0pt] (4,8.5) .. controls (4.3,8.5) and (4.5,8.7) .. (4.5,9);
  \draw[blue, line width=1.0pt] (4.5,8) .. controls (4.5,8.3) and (4.7,8.5) .. (5,8.5);
  \draw[blue, line width=1.0pt] (5,8.5) .. controls (5.3,8.5) and (5.5,8.7) .. (5.5,9);
  \draw[blue, line width=1.0pt] (5.5,8) .. controls (5.5,8.3) and (5.7,8.5) .. (6,8.5);
  \draw[blue, line width=1.0pt] (6,8.5) .. controls (6.3,8.5) and (6.5,8.7) .. (6.5,9);
  \draw[blue, line width=1.0pt] (6.5,8) .. controls (6.5,8.3) and (6.7,8.5) .. (7,8.5);
  \draw[blue, line width=1.0pt, line cap=round] (7,8.5) -- (8,8.5);
  \draw[blue, line width=1.0pt, line cap=round] (7.5,8) -- (7.5,9);
  \draw[blue, line width=1.0pt] (8,8.5) .. controls (8.3,8.5) and (8.5,8.7) .. (8.5,9);
  \draw[blue, line width=1.0pt] (0,7.5) .. controls (0.3,7.5) and (0.5,7.7) .. (0.5,8);
  \draw[blue, line width=1.0pt] (0.5,7) .. controls (0.5,7.3) and (0.7,7.5) .. (1,7.5);
  \draw[blue, line width=1.0pt] (1,7.5) .. controls (1.3,7.5) and (1.5,7.7) .. (1.5,8);
  \draw[blue, line width=1.0pt] (1.5,7) .. controls (1.5,7.3) and (1.7,7.5) .. (2,7.5);
  \draw[blue, line width=1.0pt] (2,7.5) .. controls (2.3,7.5) and (2.5,7.7) .. (2.5,8);
  \draw[blue, line width=1.0pt] (2.5,7) .. controls (2.5,7.3) and (2.7,7.5) .. (3,7.5);
  \draw[blue, line width=1.0pt] (3,7.5) .. controls (3.3,7.5) and (3.5,7.7) .. (3.5,8);
  \draw[blue, line width=1.0pt] (3.5,7) .. controls (3.5,7.3) and (3.7,7.5) .. (4,7.5);
  \draw[blue, line width=1.0pt] (4,7.5) .. controls (4.3,7.5) and (4.5,7.7) .. (4.5,8);
  \draw[blue, line width=1.0pt] (4.5,7) .. controls (4.5,7.3) and (4.7,7.5) .. (5,7.5);
  \draw[blue, line width=1.0pt] (5,7.5) .. controls (5.3,7.5) and (5.5,7.7) .. (5.5,8);
  \draw[blue, line width=1.0pt] (5.5,7) .. controls (5.5,7.3) and (5.7,7.5) .. (6,7.5);
  \draw[blue, line width=1.0pt] (6,7.5) .. controls (6.3,7.5) and (6.5,7.7) .. (6.5,8);
  \draw[blue, line width=1.0pt] (6.5,7) .. controls (6.5,7.3) and (6.7,7.5) .. (7,7.5);
  \draw[blue, line width=1.0pt] (7,7.5) .. controls (7.3,7.5) and (7.5,7.7) .. (7.5,8);
  \draw[blue, line width=1.0pt, line cap=round] (0,6.5) -- (1,6.5);
  \draw[blue, line width=1.0pt, line cap=round] (0.5,6) -- (0.5,7);
  \draw[blue, line width=1.0pt] (1,6.5) .. controls (1.3,6.5) and (1.5,6.7) .. (1.5,7);
  \draw[blue, line width=1.0pt] (1.5,6) .. controls (1.5,6.3) and (1.7,6.5) .. (2,6.5);
  \draw[blue, line width=1.0pt] (2,6.5) .. controls (2.3,6.5) and (2.5,6.7) .. (2.5,7);
  \draw[blue, line width=1.0pt] (2.5,6) .. controls (2.5,6.3) and (2.7,6.5) .. (3,6.5);
  \draw[blue, line width=1.0pt] (3,6.5) .. controls (3.3,6.5) and (3.5,6.7) .. (3.5,7);
  \draw[blue, line width=1.0pt] (3.5,6) .. controls (3.5,6.3) and (3.7,6.5) .. (4,6.5);
  \draw[blue, line width=1.0pt] (4,6.5) .. controls (4.3,6.5) and (4.5,6.7) .. (4.5,7);
  \draw[blue, line width=1.0pt] (4.5,6) .. controls (4.5,6.3) and (4.7,6.5) .. (5,6.5);
  \draw[blue, line width=1.0pt] (5,6.5) .. controls (5.3,6.5) and (5.5,6.7) .. (5.5,7);
  \draw[blue, line width=1.0pt] (5.5,6) .. controls (5.5,6.3) and (5.7,6.5) .. (6,6.5);
  \draw[blue, line width=1.0pt] (6,6.5) .. controls (6.3,6.5) and (6.5,6.7) .. (6.5,7);
  \draw[blue, line width=1.0pt] (0,5.5) .. controls (0.3,5.5) and (0.5,5.7) .. (0.5,6);
  \draw[blue, line width=1.0pt] (0.5,5) .. controls (0.5,5.3) and (0.7,5.5) .. (1,5.5);
  \draw[blue, line width=1.0pt] (1,5.5) .. controls (1.3,5.5) and (1.5,5.7) .. (1.5,6);
  \draw[blue, line width=1.0pt] (1.5,5) .. controls (1.5,5.3) and (1.7,5.5) .. (2,5.5);
  \draw[blue, line width=1.0pt] (2,5.5) .. controls (2.3,5.5) and (2.5,5.7) .. (2.5,6);
  \draw[blue, line width=1.0pt] (2.5,5) .. controls (2.5,5.3) and (2.7,5.5) .. (3,5.5);
  \draw[blue, line width=1.0pt] (3,5.5) .. controls (3.3,5.5) and (3.5,5.7) .. (3.5,6);
  \draw[blue, line width=1.0pt] (3.5,5) .. controls (3.5,5.3) and (3.7,5.5) .. (4,5.5);
  \draw[blue, line width=1.0pt] (4,5.5) .. controls (4.3,5.5) and (4.5,5.7) .. (4.5,6);
  \draw[blue, line width=1.0pt] (4.5,5) .. controls (4.5,5.3) and (4.7,5.5) .. (5,5.5);
  \draw[blue, line width=1.0pt] (5,5.5) .. controls (5.3,5.5) and (5.5,5.7) .. (5.5,6);
  \draw[blue, line width=1.0pt] (0,4.5) .. controls (0.3,4.5) and (0.5,4.7) .. (0.5,5);
  \draw[blue, line width=1.0pt] (0.5,4) .. controls (0.5,4.3) and (0.7,4.5) .. (1,4.5);
  \draw[blue, line width=1.0pt] (1,4.5) .. controls (1.3,4.5) and (1.5,4.7) .. (1.5,5);
  \draw[blue, line width=1.0pt] (1.5,4) .. controls (1.5,4.3) and (1.7,4.5) .. (2,4.5);
  \draw[blue, line width=1.0pt] (2,4.5) .. controls (2.3,4.5) and (2.5,4.7) .. (2.5,5);
  \draw[blue, line width=1.0pt] (2.5,4) .. controls (2.5,4.3) and (2.7,4.5) .. (3,4.5);
  \draw[blue, line width=1.0pt] (3,4.5) .. controls (3.3,4.5) and (3.5,4.7) .. (3.5,5);
  \draw[blue, line width=1.0pt] (3.5,4) .. controls (3.5,4.3) and (3.7,4.5) .. (4,4.5);
  \draw[blue, line width=1.0pt] (4,4.5) .. controls (4.3,4.5) and (4.5,4.7) .. (4.5,5);
  \draw[blue, line width=1.0pt, line cap=round] (0,3.5) -- (1,3.5);
  \draw[blue, line width=1.0pt, line cap=round] (0.5,3) -- (0.5,4);
  \draw[blue, line width=1.0pt, line cap=round] (1,3.5) -- (2,3.5);
  \draw[blue, line width=1.0pt, line cap=round] (1.5,3) -- (1.5,4);
  \draw[blue, line width=1.0pt] (2,3.5) .. controls (2.3,3.5) and (2.5,3.7) .. (2.5,4);
  \draw[blue, line width=1.0pt] (2.5,3) .. controls (2.5,3.3) and (2.7,3.5) .. (3,3.5);
  \draw[blue, line width=1.0pt] (3,3.5) .. controls (3.3,3.5) and (3.5,3.7) .. (3.5,4);
  \draw[blue, line width=1.0pt] (0,2.5) .. controls (0.3,2.5) and (0.5,2.7) .. (0.5,3);
  \draw[blue, line width=1.0pt] (0.5,2) .. controls (0.5,2.3) and (0.7,2.5) .. (1,2.5);
  \draw[blue, line width=1.0pt] (1,2.5) .. controls (1.3,2.5) and (1.5,2.7) .. (1.5,3);
  \draw[blue, line width=1.0pt] (1.5,2) .. controls (1.5,2.3) and (1.7,2.5) .. (2,2.5);
  \draw[blue, line width=1.0pt] (2,2.5) .. controls (2.3,2.5) and (2.5,2.7) .. (2.5,3);
  \draw[blue, line width=1.0pt] (0,1.5) .. controls (0.3,1.5) and (0.5,1.7) .. (0.5,2);
  \draw[blue, line width=1.0pt] (0.5,1) .. controls (0.5,1.3) and (0.7,1.5) .. (1,1.5);
  \draw[blue, line width=1.0pt] (1,1.5) .. controls (1.3,1.5) and (1.5,1.7) .. (1.5,2);
  \draw[blue, line width=1.0pt] (0,0.5) .. controls (0.3,0.5) and (0.5,0.7) .. (0.5,1);
  \node[font=\tiny\bfseries, text=blue, anchor=south] at (0.5,13.3) {1};
  \node[font=\tiny\bfseries, text=blue, anchor=south] at (1.5,13.3) {2};
  \node[font=\tiny\bfseries, text=blue, anchor=south] at (2.5,13.3) {3};
  \node[font=\tiny\bfseries, text=blue, anchor=south] at (3.5,13.3) {4};
  \node[font=\tiny\bfseries, text=blue, anchor=south] at (4.5,13.3) {5};
  \node[font=\tiny\bfseries, text=blue, anchor=south] at (5.5,13.3) {6};
  \node[font=\tiny\bfseries, text=blue, anchor=south] at (6.5,13.3) {7};
  \node[font=\tiny\bfseries, text=blue, anchor=south] at (7.5,13.3) {8};
  \node[font=\tiny\bfseries, text=blue, anchor=south] at (8.5,13.3) {9};
  \node[font=\tiny\bfseries, text=blue, anchor=south] at (9.5,13.3) {10};
  \node[font=\tiny\bfseries, text=blue, anchor=south] at (10.5,13.3) {11};
  \node[font=\tiny\bfseries, text=blue, anchor=south] at (11.5,13.3) {12};
  \node[font=\tiny\bfseries, text=blue, anchor=south] at (12.5,13.3) {13};
  \node[font=\tiny\bfseries, text=blue, anchor=east] at (-0.3,12.5) {1};
  \node[font=\tiny\bfseries, text=blue, anchor=east] at (-0.3,11.5) {2};
  \node[font=\tiny\bfseries, text=blue, anchor=east] at (-0.3,10.5) {4};
  \node[font=\tiny\bfseries, text=blue, anchor=east] at (-0.3,9.5) {3};
  \node[font=\tiny\bfseries, text=blue, anchor=east] at (-0.3,8.5) {6};
  \node[font=\tiny\bfseries, text=blue, anchor=east] at (-0.3,7.5) {5};
  \node[font=\tiny\bfseries, text=blue, anchor=east] at (-0.3,6.5) {9};
  \node[font=\tiny\bfseries, text=blue, anchor=east] at (-0.3,5.5) {7};
  \node[font=\tiny\bfseries, text=blue, anchor=east] at (-0.3,4.5) {8};
  \node[font=\tiny\bfseries, text=blue, anchor=east] at (-0.3,3.5) {13};
  \node[font=\tiny\bfseries, text=blue, anchor=east] at (-0.3,2.5) {10};
  \node[font=\tiny\bfseries, text=blue, anchor=east] at (-0.3,1.5) {11};
  \node[font=\tiny\bfseries, text=blue, anchor=east] at (-0.3,0.5) {12};
  \draw[black, line width=1.2000000000000002pt] (13,8) rectangle (0,13);
  \node[font=\Large\bfseries, text=green!60!black, fill=white, inner sep=0.5pt, circle, transform shape] at (0.5,10.5) {3};
  \node[font=\Large\bfseries, text=green!60!black, fill=white, inner sep=0.5pt, circle, transform shape] at (5.5,10.5) {8};
  \node[font=\Large\bfseries, text=green!60!black, fill=white, inner sep=0.5pt, circle, transform shape] at (0.5,8.5) {5};
  \node[font=\Large\bfseries, text=green!60!black, fill=white, inner sep=0.5pt, circle, transform shape] at (7.5,8.5) {12};
  \node[font=\Large\bfseries, text=green!60!black, fill=white, inner sep=0.5pt, circle, transform shape] at (0.5,6.5) {7};
  \node[font=\Large\bfseries, text=green!60!black, fill=white, inner sep=0.5pt, circle, transform shape] at (0.5,3.5) {10};
  \node[font=\Large\bfseries, text=green!60!black, fill=white, inner sep=0.5pt, circle, transform shape] at (1.5,3.5) {11};
\end{tikzpicture}

%% file: forest_lr_clip.tikz
\begin{tikzpicture}[scale=0.45]
  \draw[lightgray, very thin, opacity=0.3] (7,0) -- (0,0);
  \draw[lightgray, very thin, opacity=0.3] (7,0) -- (7,7);
  \draw[lightgray, very thin, opacity=0.3] (7,1) -- (0,1);
  \draw[lightgray, very thin, opacity=0.3] (6,0) -- (6,7);
  \draw[lightgray, very thin, opacity=0.3] (7,2) -- (0,2);
  \draw[lightgray, very thin, opacity=0.3] (5,0) -- (5,7);
  \draw[lightgray, very thin, opacity=0.3] (7,3) -- (0,3);
  \draw[lightgray, very thin, opacity=0.3] (4,0) -- (4,7);
  \draw[lightgray, very thin, opacity=0.3] (7,4) -- (0,4);
  \draw[lightgray, very thin, opacity=0.3] (3,0) -- (3,7);
  \draw[lightgray, very thin, opacity=0.3] (7,5) -- (0,5);
  \draw[lightgray, very thin, opacity=0.3] (2,0) -- (2,7);
  \draw[lightgray, very thin, opacity=0.3] (7,6) -- (0,6);
  \draw[lightgray, very thin, opacity=0.3] (1,0) -- (1,7);
  \draw[lightgray, very thin, opacity=0.3] (7,7) -- (0,7);
  \draw[lightgray, very thin, opacity=0.3] (0,0) -- (0,7);
  \draw[blue, line width=1.125pt] (0,6.5) .. controls (0.3,6.5) and (0.5,6.7) .. (0.5,7);
  \draw[blue, line width=1.125pt] (0.5,6) .. controls (0.5,6.3) and (0.7,6.5) .. (1,6.5);
  \draw[blue, line width=1.125pt] (1,6.5) .. controls (1.3,6.5) and (1.5,6.7) .. (1.5,7);
  \draw[blue, line width=1.125pt] (1.5,6) .. controls (1.5,6.3) and (1.7,6.5) .. (2,6.5);
  \draw[blue, line width=1.125pt] (2,6.5) .. controls (2.3,6.5) and (2.5,6.7) .. (2.5,7);
  \draw[blue, line width=1.125pt] (2.5,6) .. controls (2.5,6.3) and (2.7,6.5) .. (3,6.5);
  \draw[blue, line width=1.125pt] (3,6.5) .. controls (3.3,6.5) and (3.5,6.7) .. (3.5,7);
  \draw[blue, line width=1.125pt] (3.5,6) .. controls (3.5,6.3) and (3.7,6.5) .. (4,6.5);
  \draw[blue, line width=1.125pt] (4,6.5) .. controls (4.3,6.5) and (4.5,6.7) .. (4.5,7);
  \draw[blue, line width=1.125pt] (4.5,6) .. controls (4.5,6.3) and (4.7,6.5) .. (5,6.5);
  \draw[blue, line width=1.125pt] (5,6.5) .. controls (5.3,6.5) and (5.5,6.7) .. (5.5,7);
  \draw[blue, line width=1.125pt] (5.5,6) .. controls (5.5,6.3) and (5.7,6.5) .. (6,6.5);
  \draw[blue, line width=1.125pt] (6,6.5) .. controls (6.3,6.5) and (6.5,6.7) .. (6.5,7);
  \draw[blue, line width=1.125pt] (0,5.5) .. controls (0.3,5.5) and (0.5,5.7) .. (0.5,6);
  \draw[blue, line width=1.125pt] (0.5,5) .. controls (0.5,5.3) and (0.7,5.5) .. (1,5.5);
  \draw[blue, line width=1.125pt] (1,5.5) .. controls (1.3,5.5) and (1.5,5.7) .. (1.5,6);
  \draw[blue, line width=1.125pt] (1.5,5) .. controls (1.5,5.3) and (1.7,5.5) .. (2,5.5);
  \draw[blue, line width=1.125pt] (2,5.5) .. controls (2.3,5.5) and (2.5,5.7) .. (2.5,6);
  \draw[blue, line width=1.125pt] (2.5,5) .. controls (2.5,5.3) and (2.7,5.5) .. (3,5.5);
  \draw[blue, line width=1.125pt] (3,5.5) .. controls (3.3,5.5) and (3.5,5.7) .. (3.5,6);
  \draw[blue, line width=1.125pt] (3.5,5) .. controls (3.5,5.3) and (3.7,5.5) .. (4,5.5);
  \draw[blue, line width=1.125pt] (4,5.5) .. controls (4.3,5.5) and (4.5,5.7) .. (4.5,6);
  \draw[blue, line width=1.125pt] (4.5,5) .. controls (4.5,5.3) and (4.7,5.5) .. (5,5.5);
  \draw[blue, line width=1.125pt] (5,5.5) .. controls (5.3,5.5) and (5.5,5.7) .. (5.5,6);
  \draw[blue, line width=1.125pt, line cap=round] (0,4.5) -- (1,4.5);
  \draw[blue, line width=1.125pt, line cap=round] (0.5,4) -- (0.5,5);
  \draw[blue, line width=1.125pt, line cap=round] (1,4.5) -- (2,4.5);
  \draw[blue, line width=1.125pt, line cap=round] (1.5,4) -- (1.5,5);
  \draw[blue, line width=1.125pt] (2,4.5) .. controls (2.3,4.5) and (2.5,4.7) .. (2.5,5);
  \draw[blue, line width=1.125pt] (2.5,4) .. controls (2.5,4.3) and (2.7,4.5) .. (3,4.5);
  \draw[blue, line width=1.125pt] (3,4.5) .. controls (3.3,4.5) and (3.5,4.7) .. (3.5,5);
  \draw[blue, line width=1.125pt] (3.5,4) .. controls (3.5,4.3) and (3.7,4.5) .. (4,4.5);
  \draw[blue, line width=1.125pt] (4,4.5) .. controls (4.3,4.5) and (4.5,4.7) .. (4.5,5);
  \draw[blue, line width=1.125pt] (0,3.5) .. controls (0.3,3.5) and (0.5,3.7) .. (0.5,4);
  \draw[blue, line width=1.125pt] (0.5,3) .. controls (0.5,3.3) and (0.7,3.5) .. (1,3.5);
  \draw[blue, line width=1.125pt] (1,3.5) .. controls (1.3,3.5) and (1.5,3.7) .. (1.5,4);
  \draw[blue, line width=1.125pt] (1.5,3) .. controls (1.5,3.3) and (1.7,3.5) .. (2,3.5);
  \draw[blue, line width=1.125pt] (2,3.5) .. controls (2.3,3.5) and (2.5,3.7) .. (2.5,4);
  \draw[blue, line width=1.125pt] (2.5,3) .. controls (2.5,3.3) and (2.7,3.5) .. (3,3.5);
  \draw[blue, line width=1.125pt] (3,3.5) .. controls (3.3,3.5) and (3.5,3.7) .. (3.5,4);
  \draw[blue, line width=1.125pt, line cap=round] (0,2.5) -- (1,2.5);
  \draw[blue, line width=1.125pt, line cap=round] (0.5,2) -- (0.5,3);
  \draw[blue, line width=1.125pt, line cap=round] (1,2.5) -- (2,2.5);
  \draw[blue, line width=1.125pt, line cap=round] (1.5,2) -- (1.5,3);
  \draw[blue, line width=1.125pt] (2,2.5) .. controls (2.3,2.5) and (2.5,2.7) .. (2.5,3);
  \draw[blue, line width=1.125pt] (0,1.5) .. controls (0.3,1.5) and (0.5,1.7) .. (0.5,2);
  \draw[blue, line width=1.125pt] (0.5,1) .. controls (0.5,1.3) and (0.7,1.5) .. (1,1.5);
  \draw[blue, line width=1.125pt] (1,1.5) .. controls (1.3,1.5) and (1.5,1.7) .. (1.5,2);
  \draw[blue, line width=1.125pt] (0,0.5) .. controls (0.3,0.5) and (0.5,0.7) .. (0.5,1);
  \node[font=\tiny\bfseries, text=blue, anchor=south] at (0.5,7.3) {1};
  \node[font=\tiny\bfseries, text=blue, anchor=south] at (1.5,7.3) {2};
  \node[font=\tiny\bfseries, text=blue, anchor=south] at (2.5,7.3) {3};
  \node[font=\tiny\bfseries, text=blue, anchor=south] at (3.5,7.3) {4};
  \node[font=\tiny\bfseries, text=blue, anchor=south] at (4.5,7.3) {5};
  \node[font=\tiny\bfseries, text=blue, anchor=south] at (5.5,7.3) {6};
  \node[font=\tiny\bfseries, text=blue, anchor=south] at (6.5,7.3) {7};
  \node[font=\tiny\bfseries, text=blue, anchor=east] at (-0.3,6.5) {1};
  \node[font=\tiny\bfseries, text=blue, anchor=east] at (-0.3,5.5) {2};
  \node[font=\tiny\bfseries, text=blue, anchor=east] at (-0.3,4.5) {5};
  \node[font=\tiny\bfseries, text=blue, anchor=east] at (-0.3,3.5) {3};
  \node[font=\tiny\bfseries, text=blue, anchor=east] at (-0.3,2.5) {7};
  \node[font=\tiny\bfseries, text=blue, anchor=east] at (-0.3,1.5) {4};
  \node[font=\tiny\bfseries, text=blue, anchor=east] at (-0.3,0.5) {6};
  \node[font=\Large\bfseries, text=green!60!black, fill=white, inner sep=0.5pt, circle, transform shape] at (0.5,4.5) {3};
  \node[font=\Large\bfseries, text=green!60!black, fill=white, inner sep=0.5pt, circle, transform shape] at (1.5,4.5) {4};
  \node[font=\Large\bfseries, text=green!60!black, fill=white, inner sep=0.5pt, circle, transform shape] at (0.5,2.5) {5};
  \node[font=\Large\bfseries, text=green!60!black, fill=white, inner sep=0.5pt, circle, transform shape] at (1.5,2.5) {6};
\end{tikzpicture}

%% file: forest_lr_trim.tikz
\begin{tikzpicture}[scale=0.45]
  \draw[lightgray, very thin, opacity=0.3] (7,0) -- (0,0);
  \draw[lightgray, very thin, opacity=0.3] (7,0) -- (7,7);
  \draw[lightgray, very thin, opacity=0.3] (7,1) -- (0,1);
  \draw[lightgray, very thin, opacity=0.3] (6,0) -- (6,7);
  \draw[lightgray, very thin, opacity=0.3] (7,2) -- (0,2);
  \draw[lightgray, very thin, opacity=0.3] (5,0) -- (5,7);
  \draw[lightgray, very thin, opacity=0.3] (7,3) -- (0,3);
  \draw[lightgray, very thin, opacity=0.3] (4,0) -- (4,7);
  \draw[lightgray, very thin, opacity=0.3] (7,4) -- (0,4);
  \draw[lightgray, very thin, opacity=0.3] (3,0) -- (3,7);
  \draw[lightgray, very thin, opacity=0.3] (7,5) -- (0,5);
  \draw[lightgray, very thin, opacity=0.3] (2,0) -- (2,7);
  \draw[lightgray, very thin, opacity=0.3] (7,6) -- (0,6);
  \draw[lightgray, very thin, opacity=0.3] (1,0) -- (1,7);
  \draw[lightgray, very thin, opacity=0.3] (7,7) -- (0,7);
  \draw[lightgray, very thin, opacity=0.3] (0,0) -- (0,7);
  \draw[blue, line width=1.125pt] (0,6.5) .. controls (0.3,6.5) and (0.5,6.7) .. (0.5,7);
  \draw[blue, line width=1.125pt] (0.5,6) .. controls (0.5,6.3) and (0.7,6.5) .. (1,6.5);
  \draw[blue, line width=1.125pt] (1,6.5) .. controls (1.3,6.5) and (1.5,6.7) .. (1.5,7);
  \draw[blue, line width=1.125pt] (1.5,6) .. controls (1.5,6.3) and (1.7,6.5) .. (2,6.5);
  \draw[blue, line width=1.125pt] (2,6.5) .. controls (2.3,6.5) and (2.5,6.7) .. (2.5,7);
  \draw[blue, line width=1.125pt] (2.5,6) .. controls (2.5,6.3) and (2.7,6.5) .. (3,6.5);
  \draw[blue, line width=1.125pt] (3,6.5) .. controls (3.3,6.5) and (3.5,6.7) .. (3.5,7);
  \draw[blue, line width=1.125pt] (3.5,6) .. controls (3.5,6.3) and (3.7,6.5) .. (4,6.5);
  \draw[blue, line width=1.125pt] (4,6.5) .. controls (4.3,6.5) and (4.5,6.7) .. (4.5,7);
  \draw[blue, line width=1.125pt] (4.5,6) .. controls (4.5,6.3) and (4.7,6.5) .. (5,6.5);
  \draw[blue, line width=1.125pt] (5,6.5) .. controls (5.3,6.5) and (5.5,6.7) .. (5.5,7);
  \draw[blue, line width=1.125pt] (5.5,6) .. controls (5.5,6.3) and (5.7,6.5) .. (6,6.5);
  \draw[blue, line width=1.125pt] (6,6.5) .. controls (6.3,6.5) and (6.5,6.7) .. (6.5,7);
  \draw[blue, line width=1.125pt, line cap=round] (0,5.5) -- (1,5.5);
  \draw[blue, line width=1.125pt, line cap=round] (0.5,5) -- (0.5,6);
  \draw[blue, line width=1.125pt] (1,5.5) .. controls (1.3,5.5) and (1.5,5.7) .. (1.5,6);
  \draw[blue, line width=1.125pt] (1.5,5) .. controls (1.5,5.3) and (1.7,5.5) .. (2,5.5);
  \draw[blue, line width=1.125pt] (2,5.5) .. controls (2.3,5.5) and (2.5,5.7) .. (2.5,6);
  \draw[blue, line width=1.125pt] (2.5,5) .. controls (2.5,5.3) and (2.7,5.5) .. (3,5.5);
  \draw[blue, line width=1.125pt] (3,5.5) .. controls (3.3,5.5) and (3.5,5.7) .. (3.5,6);
  \draw[blue, line width=1.125pt] (3.5,5) .. controls (3.5,5.3) and (3.7,5.5) .. (4,5.5);
  \draw[blue, line width=1.125pt] (4,5.5) .. controls (4.3,5.5) and (4.5,5.7) .. (4.5,6);
  \draw[blue, line width=1.125pt] (4.5,5) .. controls (4.5,5.3) and (4.7,5.5) .. (5,5.5);
  \draw[blue, line width=1.125pt] (5,5.5) .. controls (5.3,5.5) and (5.5,5.7) .. (5.5,6);
  \draw[blue, line width=1.125pt] (0,4.5) .. controls (0.3,4.5) and (0.5,4.7) .. (0.5,5);
  \draw[blue, line width=1.125pt] (0.5,4) .. controls (0.5,4.3) and (0.7,4.5) .. (1,4.5);
  \draw[blue, line width=1.125pt] (1,4.5) .. controls (1.3,4.5) and (1.5,4.7) .. (1.5,5);
  \draw[blue, line width=1.125pt] (1.5,4) .. controls (1.5,4.3) and (1.7,4.5) .. (2,4.5);
  \draw[blue, line width=1.125pt] (2,4.5) .. controls (2.3,4.5) and (2.5,4.7) .. (2.5,5);
  \draw[blue, line width=1.125pt] (2.5,4) .. controls (2.5,4.3) and (2.7,4.5) .. (3,4.5);
  \draw[blue, line width=1.125pt] (3,4.5) .. controls (3.3,4.5) and (3.5,4.7) .. (3.5,5);
  \draw[blue, line width=1.125pt] (3.5,4) .. controls (3.5,4.3) and (3.7,4.5) .. (4,4.5);
  \draw[blue, line width=1.125pt] (4,4.5) .. controls (4.3,4.5) and (4.5,4.7) .. (4.5,5);
  \draw[blue, line width=1.125pt] (0,3.5) .. controls (0.3,3.5) and (0.5,3.7) .. (0.5,4);
  \draw[blue, line width=1.125pt] (0.5,3) .. controls (0.5,3.3) and (0.7,3.5) .. (1,3.5);
  \draw[blue, line width=1.125pt] (1,3.5) .. controls (1.3,3.5) and (1.5,3.7) .. (1.5,4);
  \draw[blue, line width=1.125pt] (1.5,3) .. controls (1.5,3.3) and (1.7,3.5) .. (2,3.5);
  \draw[blue, line width=1.125pt] (2,3.5) .. controls (2.3,3.5) and (2.5,3.7) .. (2.5,4);
  \draw[blue, line width=1.125pt] (2.5,3) .. controls (2.5,3.3) and (2.7,3.5) .. (3,3.5);
  \draw[blue, line width=1.125pt] (3,3.5) .. controls (3.3,3.5) and (3.5,3.7) .. (3.5,4);
  \draw[blue, line width=1.125pt, line cap=round] (0,2.5) -- (1,2.5);
  \draw[blue, line width=1.125pt, line cap=round] (0.5,2) -- (0.5,3);
  \draw[blue, line width=1.125pt, line cap=round] (1,2.5) -- (2,2.5);
  \draw[blue, line width=1.125pt, line cap=round] (1.5,2) -- (1.5,3);
  \draw[blue, line width=1.125pt] (2,2.5) .. controls (2.3,2.5) and (2.5,2.7) .. (2.5,3);
  \draw[blue, line width=1.125pt] (0,1.5) .. controls (0.3,1.5) and (0.5,1.7) .. (0.5,2);
  \draw[blue, line width=1.125pt] (0.5,1) .. controls (0.5,1.3) and (0.7,1.5) .. (1,1.5);
  \draw[blue, line width=1.125pt] (1,1.5) .. controls (1.3,1.5) and (1.5,1.7) .. (1.5,2);
  \draw[blue, line width=1.125pt] (0,0.5) .. controls (0.3,0.5) and (0.5,0.7) .. (0.5,1);
  \node[font=\tiny\bfseries, text=blue, anchor=south] at (0.5,7.3) {1};
  \node[font=\tiny\bfseries, text=blue, anchor=south] at (1.5,7.3) {2};
  \node[font=\tiny\bfseries, text=blue, anchor=south] at (2.5,7.3) {3};
  \node[font=\tiny\bfseries, text=blue, anchor=south] at (3.5,7.3) {4};
  \node[font=\tiny\bfseries, text=blue, anchor=south] at (4.5,7.3) {5};
  \node[font=\tiny\bfseries, text=blue, anchor=south] at (5.5,7.3) {6};
  \node[font=\tiny\bfseries, text=blue, anchor=south] at (6.5,7.3) {7};
  \node[font=\tiny\bfseries, text=blue, anchor=east] at (-0.3,6.5) {1};
  \node[font=\tiny\bfseries, text=blue, anchor=east] at (-0.3,5.5) {3};
  \node[font=\tiny\bfseries, text=blue, anchor=east] at (-0.3,4.5) {2};
  \node[font=\tiny\bfseries, text=blue, anchor=east] at (-0.3,3.5) {4};
  \node[font=\tiny\bfseries, text=blue, anchor=east] at (-0.3,2.5) {7};
  \node[font=\tiny\bfseries, text=blue, anchor=east] at (-0.3,1.5) {5};
  \node[font=\tiny\bfseries, text=blue, anchor=east] at (-0.3,0.5) {6};
  \node[font=\Large\bfseries, text=green!60!black, fill=white, inner sep=0.5pt, circle, transform shape] at (0.5,5.5) {2};
  \node[font=\Large\bfseries, text=green!60!black, fill=white, inner sep=0.5pt, circle, transform shape] at (0.5,2.5) {5};
  \node[font=\Large\bfseries, text=green!60!black, fill=white, inner sep=0.5pt, circle, transform shape] at (1.5,2.5) {6};
\end{tikzpicture}

%% file: schubert_bialgebra_separate_paper.bbl
\begin{thebibliography}{10}

\bibitem{assaf2018demazure}
{\sc Assaf, S., and Schilling, A.}
\newblock A {D}emazure crystal construction for {S}chubert polynomials.
\newblock {\em Algebraic Combinatorics 1}, 2 (2018), 225--247.

\bibitem{assaf2017slide}
{\sc Assaf, S., and Searles, D.}
\newblock Schubert polynomials, slide polynomials, {S}tanley symmetric
  functions and quasi-{Y}amanouchi pipe dreams.
\newblock {\em Advances in Mathematics 306\/} (2017), 89--122.

\bibitem{bbrc}
{\sc Bergeron, N., and Billey, S.}
\newblock {R}{C}-graphs and {S}chubert polynomials.
\newblock {\em Experimental Math. 2}, 4 (1993), 257--269.

\bibitem{bergeron2025qsymflag}
{\sc Bergeron, N., Gagnon, L., Nadeau, P., Spink, H., and Tewari, V.}
\newblock The quasisymmetric flag variety: a toric complex on noncrossing
  partitions.
\newblock {\em arXiv preprint arXiv:2508.12171\/} (2025).

\bibitem{coprod}
{\sc Bergeron, N., and Sottile, F.}
\newblock Schubert polynomials, the {B}ruhat order, and the geometry of flag
  manifolds.
\newblock {\em Duke Math. J. 95}, 2 (1998), 373--423.

\bibitem{bsskew}
{\sc Bergeron, N., and Sottile, F.}
\newblock Skew {S}chubert functions and the {P}ieri formula for flag manifolds.
\newblock {\em Trans. Amer. Math. Soc. 354}, 2 (2001), 651--673.

\bibitem{billey2019bijective}
{\sc Billey, S.~C., Holroyd, A.~E., and Young, B.~J.}
\newblock A bijective proof of {M}acdonald's reduced word formula.
\newblock {\em Algebraic Combinatorics 2}, 2 (2019), 217--248.

\bibitem{demazure1974desingularisation}
{\sc Demazure, M.}
\newblock D{\'e}singularisation des vari{\'e}t{\'e}s de {S}chubert
  g{\'e}n{\'e}ralis{\'e}es.
\newblock {\em Annales scientifiques de l'{\'E}cole Normale Sup{\'e}rieure 7},
  1 (1974), 53--88.

\bibitem{demazure1974nouvelle}
{\sc Demazure, M.}
\newblock Une nouvelle formule des caract{\`e}res.
\newblock {\em Bulletin des Sciences Math{\'e}matiques. 2e S{\'e}rie 98}, 3
  (1974), 163--172.

\bibitem{edelman1987balanced}
{\sc Edelman, P., and Greene, C.}
\newblock Balanced tableaux.
\newblock {\em Advances in Mathematics 63}, 1 (1987), 42--99.

\bibitem{GaoHuang2023}
{\sc Gao, Y., and Huang, D.}
\newblock The canonical bijection between pipe dreams and bumpless pipe dreams.
\newblock {\em International Mathematics Research Notices 2023}, 21 (November
  2023), 18629--18663.

\bibitem{guo2026forest}
{\sc Guo, A., and Woodruff, D.}
\newblock Forest polynomials and pattern avoidance, 2026.
\newblock arXiv:2602.04036 [math.CO].

\bibitem{hamaker2014relating}
{\sc Hamaker, Z., and Young, B.}
\newblock Relating {E}delman--{G}reene insertion to the {L}ittle map.
\newblock {\em Journal of Algebraic Combinatorics 40}, 3 (2014), 693--710.

\bibitem{henrymath}
{\sc {Henry V}}.
\newblock Positivity of branching of key polynomials: references and
  significance.
\newblock MathOverflow answer, 2026.
\newblock URL: \url{https://mathoverflow.net/q/511128} (version: 2026-05-09).

\bibitem{huang2020monk}
{\sc Huang, D.}
\newblock Bijective proofs of {M}onk's rule for {S}chubert and double
  {S}chubert polynomials with bumpless pipe dreams.
\newblock {\em arXiv preprint arXiv:2010.15048\/} (2020).

\bibitem{HuangShimozonoYu2024}
{\sc Huang, D., Shimozono, M., and Yu, T.}
\newblock Marked bumpless pipedreams and compatible pairs.
\newblock {\em arXiv preprint arXiv:2407.18160\/} (2024).

\bibitem{joseph1985demazure}
{\sc Joseph, A.}
\newblock On the {D}emazure character formula.
\newblock {\em Annales scientifiques de l'{\'E}cole Normale Sup{\'e}rieure 18},
  3 (1985), 389--419.

\bibitem{kashiwara1993crystal}
{\sc Kashiwara, M.}
\newblock The crystal base and {L}ittelmann's refined {D}emazure character
  formula.
\newblock {\em Duke Mathematical Journal 71}, 3 (1993), 839--858.

\bibitem{lam2021back}
{\sc Lam, T., Lee, S.~J., and Shimozono, M.}
\newblock Back stable {S}chubert calculus.
\newblock {\em Compositio Mathematica 157}, 5 (2021), 883--962.

\bibitem{lascoux1990keys}
{\sc Lascoux, A., and Sch{\"u}tzenberger, M.-P.}
\newblock Keys \& standard bases.
\newblock In {\em Invariant theory and tableaux ({M}inneapolis, {MN}, 1988)},
  vol.~19 of {\em IMA Vol. Math. Appl.} Springer, New York, 1990, pp.~125--144.

\bibitem{littelmann1995crystal}
{\sc Littelmann, P.}
\newblock Crystal graphs and {Y}oung tableaux.
\newblock {\em Journal of Algebra 175}, 1 (1995), 65--87.

\bibitem{little2003combinatorial}
{\sc Little, D.~P.}
\newblock Combinatorial aspects of the {L}ascoux--{S}ch{\"u}tzenberger tree.
\newblock {\em Advances in Mathematics 174}, 2 (2003), 236--253.

\bibitem{notes}
{\sc Macdonald, I.~G.}
\newblock {\em Notes on {S}chubert {P}olynomials}.
\newblock Universit{\'e} Du Qu{\'e}bec {\`a} Montr{\'e}al, D{\'e}p. de
  math{\'e}matiques et d{'}informatique, 1991.

\bibitem{mason2009explicit}
{\sc Mason, S.}
\newblock An explicit construction of type {A} {D}emazure atoms.
\newblock {\em Journal of Algebraic Combinatorics 29}, 3 (2009), 295--313.

\bibitem{mathieu1990filtrations}
{\sc Mathieu, O.}
\newblock Filtrations of {$G$}-modules.
\newblock {\em Annales scientifiques de l'{\'E}cole Normale Sup{\'e}rieure 23},
  4 (1990), 625--644.

\bibitem{morse2016crystal}
{\sc Morse, J., and Schilling, A.}
\newblock Crystal approach to affine {S}chubert calculus.
\newblock {\em International Mathematics Research Notices 2016}, 8 (2016),
  2239--2294.

\bibitem{nadeau2024quasisymmetric}
{\sc Nadeau, P., Spink, H., and Tewari, V.}
\newblock Quasisymmetric divided differences.
\newblock {\em arXiv preprint arXiv:2406.01510\/} (2024).

\bibitem{nadeau2024forest}
{\sc Nadeau, P., and Tewari, V.}
\newblock Forest polynomials and the class of the permutahedral variety.
\newblock {\em Advances in Mathematics 453\/} (2024), 109834.

\bibitem{nadeau2024ppartitions}
{\sc Nadeau, P., and Tewari, V.}
\newblock {$P$}-partitions with flags and back stable quasisymmetric functions.
\newblock {\em Combinatorial Theory 4}, 2 (2024).
\newblock hal-04053201.

\bibitem{polo1989excellent}
{\sc Polo, P.}
\newblock Vari{\'e}t{\'e}s de {S}chubert et excellentes filtrations.
\newblock {\em Ast{\'e}risque}, 173--174 (1989), 281--311.

\bibitem{postnikov2009chains}
{\sc Postnikov, A., and Stanley, R.~P.}
\newblock Chains in the {B}ruhat order.
\newblock {\em Journal of Algebraic Combinatorics 29}, 2 (2009), 133--174.

\bibitem{reiner1995key}
{\sc Reiner, V., and Shimozono, M.}
\newblock Key polynomials and a flagged {L}ittlewood-{R}ichardson rule.
\newblock {\em Journal of Combinatorial Theory, Series A 70}, 1 (1995),
  107--143.

\bibitem{ressayre2009branching}
{\sc Ressayre, N., and Richmond, E.}
\newblock Branching {S}chubert calculus and the {B}elkale--{K}umar product on
  cohomology.
\newblock {\em Proceedings of the American Mathematical Society 139}, 3 (2011),
  835--848.
\newblock arXiv:0909.0865.

\bibitem{samuelleibniz}
{\sc Samuel, M.~J.}
\newblock {\em The {L}eibniz formula for divided difference operators
  associated to {K}ac-{M}oody root systems}.
\newblock PhD thesis, Rutgers, The State University of New Jersey, 2014.

\bibitem{samuelmolev}
{\sc Samuel, M.~J.}
\newblock A {M}olev-{S}agan type formula for double {S}chubert polynomials.
\newblock {\em J. Pure Appl. Algebra 228}, 7 (2024), 107636.

\bibitem{samuelforest}
{\sc Samuel, M.~J.}
\newblock A {L}ittlewood--{R}ichardson rule for forest polynomials.
\newblock arXiv preprint arXiv:2606.23876, 2026.

\bibitem{schubmult}
{\sc Samuel, M.~J.}
\newblock {\verb|schubmult| python package}.
\newblock {\url{git@github.com:matthematics/schubmult.git}. Version 5.2.1dev.
  Stable version available on PyPI.}, 2026-10-15.

\bibitem{sottile}
{\sc Sottile, F.}
\newblock Pieri{'}s rule for flag manifolds and {S}chubert polynomials.
\newblock {\em Ann. Inst. Fourier 46}, 1 (1996), 89--110.

\bibitem{vanderkallen1989longest}
{\sc van~der Kallen, W.}
\newblock Longest weight vectors and excellent filtrations.
\newblock {\em Mathematische Zeitschrift 201}, 1 (1989), 19--31.

\bibitem{Weigandt2021}
{\sc Weigandt, A.}
\newblock Bumpless pipe dreams and alternating sign matrices.
\newblock {\em Journal of Combinatorial Theory, Series A 182\/} (August 2021),
  105470.

\bibitem{Yu2025EmbeddingBP}
{\sc Yu, T.}
\newblock Embedding {B}umpless {P}ipedreams as {B}ruhat {C}hains.
\newblock {\em International Mathematics Research Notices 2025}, 22 (2025),
  rnaf336.

\end{thebibliography}
